\documentclass[11pt]{amsart}
\usepackage{enumitem}
\usepackage{amsthm}
\usepackage[paperheight=11in,paperwidth=8.5in,left=1in,top=1.25in,right=1in,bottom=1.25in,]{geometry}
 \usepackage{relsize}
\usepackage{rotating}
\usepackage[table]{xcolor}

\usepackage[super]{nth}
\usepackage[none]{hyphenat}
\usepackage{hyperref}
\hypersetup{
	colorlinks,
	citecolor=black,
	filecolor=black,
	linkcolor=black,
	urlcolor=black
}
\usepackage[capitalize, nameinlink, noabbrev, nosort]{cleveref}
\usepackage{graphicx, amssymb}
\usepackage{epstopdf} 
\usepackage{youngtab}
\usepackage[boxsize=.35 em]{ytableau}
\usepackage{comment}
\usepackage{setspace}
\usepackage{multicol}

\usepackage{tikz}   
\usetikzlibrary{arrows,calc,intersections,matrix,positioning,through}
\usepackage{tikz-cd}
\usepackage{comment}
\usepackage{diagbox}
\usepackage{blkarray}
\tikzset{commutative diagrams/.cd,every label/.append style = {font = \normalsize}}

\usepgflibrary{shapes.geometric}
\usetikzlibrary{shapes.geometric}
\usetikzlibrary{backgrounds}

\DeclareMathOperator{\Fr}{Fr}
\DeclareMathOperator{\co}{co}

\DeclareMathOperator{\rtop}{{\it k}} 

\DeclareMathOperator{\arr}{arr}
\DeclareMathOperator{\rt}{{\it k}} 
\DeclareMathOperator{\Gr}{Gr}
\DeclareMathOperator{\Mat}{Mat}
\DeclareMathOperator{\inn}{in}
\DeclareMathOperator{\out}{out}

\DeclareMathOperator{\sgn}{sgn}
\DeclareMathOperator{\Span}{span}
\DeclareMathOperator{\tw}{tw}

\DeclareMathOperator{\AFacet}{Froz}

\DeclareMathOperator{\refl}{refl}
\DeclareMathOperator{\cyc}{cyc}
\DeclareMathOperator{\op}{op}
\DeclareMathOperator{\pre}{pre}
\DeclareMathOperator{\inc}{inc}
\DeclareMathOperator{\Irr}{\xx}
\DeclareMathOperator{\cl}{cl}

\newcommand{\galp}{{\alpha}}
\newcommand{\gbet}{{\beta}}
\newcommand{\ggam}{{\gamma}}
\newcommand{\gdel}{{\delta}}
\newcommand{\geps}{{\varepsilon}}

\newcommand{\Ampl}{\mathcal{A}}
\def\CD{\mathcal{CD}_{n,k}}
\def\coord{\mathcal{C}}

\newcommand{\Grk}{\Gr_{k,n}^{\scriptscriptstyle \ge 0}}

\def\AA{\mathcal{A}_{n,k,m}(Z)}
\newcommand{\Ank}{\mathcal{A}_{n, k, 4}({Z})}

\def\tZ{\tilde{Z}}
\def\bcfw{\bowtie}

\def\mbcfw{\iota_{\bowtie}}

\def\4biddenprop{4-coindependent}
\newcommand{\lr}[1]{\langle #1 \rangle}
\newcommand{\llrr}[1]{\langle\!\langle #1 \rangle\!\rangle}

\newcommand{\pos}{\mathcal{P}}
\newcommand{\br}{\,|\,}
\newcommand{\gt}[1]{Z_{#1}}
\newcommand{\gto}[1]{Z_{#1}^\circ}
\def\ctop{D_{\rtop}}
\def\rchn{\nearrow}
\def\lchn{\nwarrow}
\DeclareMathOperator{\rightarrowp}{\widetilde{\rchn}}
\DeclareMathOperator{\leftarrowp}{\widetilde{\lchn}}

\newcommand{\rzeta}{\bar{\zeta}}
\newcommand{\ralpha}{\bar{\alpha}}
\newcommand{\rbeta}{\bar{\beta}}
\newcommand{\rgamma}{\bar{\gamma}}
\newcommand{\rdelta}{\bar{\delta}}
\newcommand{\repsilon}{\bar{\varepsilon}}
\DeclareMathOperator{\Mut}{Mut}

\DeclareMathOperator{\after}{after}
\DeclareMathOperator{\below}{below}
\DeclareMathOperator{\sticky}{sticky}
\DeclareMathOperator{\nonsticky}{nonsticky}

\newcommand{\czeta}{{\zeta}}
\newcommand{\calpha}{{\alpha}}
\newcommand{\cbeta}{{\beta}}
\newcommand{\cgamma}{{\gamma}}
\newcommand{\cdelta}{{\delta}}
\newcommand{\cepsilon}{{\varepsilon}}

\newcommand{\OPsi}{\overline{\Psi}}
\newcommand{\rPsi}{\overline{\Psi}}

\newcommand{\chR}{\Gr_{1,5}^{\scriptscriptstyle >0}}

\newcommand{\PP}{\mathbb{M}}
\newcommand{\Z}{\mathbb{Z}}
\newcommand{\C}{\mathbb{C}}
\newcommand{\CC}{\mathbb{C}}
\newcommand{\A}{\mathcal{A}}
\newcommand{\B}{\mathcal{B}}
\newcommand{\T}{\mathcal{T}}

\def\mtx{M}
\def\twmt{M^{\tw}}

\def\O{\mathcal{O}}
\def\Xcal{\mathcal{X}}
\def\Acal{\mathcal{A}}
\def\Rcal{\mathcal{R}}
\def\Fcal{\mathcal{F}}

\def\R{\mathbb{R}}
\def\F{\mathbb{F}}
\def\NN{\mathbb{N}}
\def\TT{\mathbb{T}}

\newcommand{\xx}{\mathbf{x}}
\newcommand{\txx}{\widetilde{\mathbf{x}}}
\newcommand{\bxx}{\overline{\mathbf{x}}}
\newcommand{\tSigma}{\tilde{\Sigma}}
\newcommand{\tQ}{\widetilde{Q}}

\newcommand{\overunder}[2]{
\!\begin{array}{c}
\scriptstyle{#1}\\[-.1in]
-\!\!\!-\!\!\!-\\[-.1in]
\scriptstyle{#2}
\end{array}
\!
}
\def\rcp{\mathfrak{r}}
\def\rcpp{\mathfrak{p}}
\def\st{\text{FStep}}

\newcommand{\bulR}{\triangleleft}
\newcommand{\bulL}{\triangleright}

\def\MSB#1{\textcolor{red}{[MSB: #1]}}

\def\RT#1{\textcolor{blue}{[RT: #1]}}

\newtheorem*{theorem*}{Theorem}
\newtheorem{theorem}{Theorem}[section]

\newtheorem{mainthm}{Theorem}

\newtheorem{lemma}[theorem]{Lemma}
\newtheorem{proposition}[theorem]{Proposition}
\newtheorem{corollary}[theorem]{Corollary}
\newtheorem{conjecture}[theorem]{Conjecture}
\theoremstyle{definition}
\newtheorem{definition}[theorem]{Definition}
\newtheorem{example}[theorem]{Example}

\newtheorem{observation}[theorem]{Observation}
\newtheorem{remark}[theorem]{Remark}
\newtheorem{notation}[theorem]{Notation}
\newtheorem{claim}[theorem]{Claim}

\setlist[itemize]{leftmargin=*}
\setlist[enumerate]{leftmargin=*}

\newcommand\marklessfootnote[1]{
    \addtocounter{footnote}{1} 
    \footnotetext{#1}
}

\begin{document}
	\begin{abstract}
		The \emph{amplituhedron} $\mathcal{A}_{n,k,m}(Z)$ is the
image of the
positive Grassmannian $\Gr_{k,n}^{\geq 0}$ under the
 map $\tilde{Z}: \Gr_{k,n}^{\geq 0} \to \Gr_{k,k+m}$ induced by a positive linear map 
$Z:\R^n \to \R^{k+m}$.
Motivated by a question of Hodges, Arkani-Hamed and Trnka introduced the amplituhedron in 2013 as a geometric object whose \emph{tilings} conjecturally encode the BCFW recursion for computing scattering amplitudes.
More specifically, the expectation was that one 
can compute scattering amplitudes in $\mathcal{N}=4$ SYM 
by tiling the $m=4$ amplituhedron $\A_{n,k,4}(Z)$ --- 
that is, decomposing the amplituhedron into 
`tiles' (closures of images of $4k$-dimensional cells of $\Gr_{k,n}^{\geq 0}$
on which $\tilde{Z}$ is injective) --- and summing the `volumes' of the tiles.  Also in 2013, Golden-Goncharov-Spradlin-Vergu-Volovich 
gave the first link between scattering amplitudes and cluster algebras,
with Drummond-Foster-Gurdogan subsequently formulating the \emph{cluster adjacency conjecture}.
In this article we reveal and prove the deep mechanism behind `cluster phenomena' in tree-level scattering amplitudes.  By connecting the BCFW recursion to a new cluster quasi-homomorphism on the Grassmannian $\Gr_{4,n}$,
we prove 
the \emph{cluster adjacency conjecture} for BCFW tiles, which says that each tile is a semialgebraic subset of the amplituhedron where a collection
of compatible cluster variables take on definite signs.
 In particular, the facets of these tiles are cut out by compatible cluster variables.  Finally, we use the 
cluster description of BCFW tiles to prove the \emph{BCFW tiling conjecture}, resolving the main original conjecture for the 
$m=4$ amplituhedron. 

	\end{abstract}
	
	\title{Cluster algebras and tilings for the $m=4$ amplituhedron}
	\date{\today}
	\author[C. Even-Zohar]{Chaim Even-Zohar}
        \address{Faculty of Mathematics, Technion, Haifa, Israel}
	\email{chaime@technion.ac.il}

	\author[T. Lakrec]{Tsviqa Lakrec}
	\address{Section of Mathematics, University of Geneva, Switzerland}
	\email{tsviqa@gmail.com}
	\author[M. Parisi]{Matteo Parisi}
	\address{MPI Physics, Garching, Germany and OIST, Okinawa, Japan}
 \email{matteo.parisi@oist.jp}
	
 \author[M. Sherman-Bennett]{Melissa Sherman-Bennett}
	\address{Department of Mathematics, UC Davis, Davis CA}
	\email{mshermanbennett@ucdavis.edu}
	\author[R. Tessler]{Ran Tessler}
	\address{Department of Mathematics, Weizmann Institute of Science, Israel}
	\email{ran.tessler@weizmann.ac.il}
	\author[L. Williams]{Lauren Williams}
	\address{Department of Mathematics, Harvard University, Cambridge, MA}
	\email{williams@math.harvard.edu}
    \date{}
	\maketitle

    \marklessfootnote{Parts of this work were previously announced, without proofs, in ``BCFW tilings and cluster adjacency for the amplituhedron'', PNAS, 2025 \cite{doi:10.1073/pnas.2408572122}}

	\setcounter{tocdepth}{1}
	\tableofcontents

\section{Introduction}\label{sec:intro}

The (tree) \emph{amplituhedron} $\mathcal{A}_{n,k,m}(Z)$ is the image of the 
positive Grassmannian $\Gr_{k,n}^{\scriptscriptstyle\geq 0}$ under the 
\emph{amplituhedron map} $\tilde{Z}: \Gr_{k,n}^{\scriptscriptstyle\geq 0} \to \Gr_{k,k+m}$. 
It was introduced by 
Arkani-Hamed and Trnka \cite{arkani-hamed_trnka} in order to give a 
geometric interpretation of the \emph{BCFW recurrence} \cite{BCFW} for 
\emph{scattering amplitudes} in $\mathcal{N}=4$ super Yang Mills theory (SYM). 
In particular, each way of iterating the BCFW recurrence gives rise
to a collection of cells of $\Gr_{k,n}^{\scriptscriptstyle\geq 0}$. Arkani-Hamed and Trnka conjectured that the images of these cells
`tile' the $m=4$ amplituhedron $\A_{n,k,4}(Z)$, and one can compute 
scattering amplitudes  by
summing the `volumes' of the tiles.
While the case $m=4$ is most important for physics\footnote{The 
$m=2$ amplituhedron is also closely
related to `loop-level' amplitudes
 \cite{Kojima:2020tjf} 
and to some correlators of determinant operators and form factors \cite{Caron-Huot:2023wdh,Basso:2023bwv} in planar $\mathcal{N} = 4$ SYM.}, the amplituhedron 
is defined for any $n,k,m$ such that $k+m \leq n$, and 
has a very rich geometric and combinatorial
structure.  It generalizes cyclic polytopes (when $k=1$), 
cyclic hyperplane arrangements \cite{karpwilliams}
(when $m=1$), and the positive Grassmannian (when $k=n-m$), and it
is connected to 
the hypersimplex and the positive tropical Grassmanian \cite{LPW,PSW} (when $m=2$).

Two of the guiding problems about the amplituhedron have been:
\begin{enumerate}
\item the \emph{cluster adjacency conjecture}, which says that 
facets of tiles are cut out by collections of \emph{compatible cluster variables.}

\item the \emph{BCFW tiling conjecture}, 
which says that any way of iterating the
BCFW recurrence gives rise to a collection of cells whose images  tile the
$m=4$ amplituhedron $\Ank$.

\end{enumerate}


\emph{Cluster algebras} are a remarkable family of commutative rings introduced by Fomin and Zelevinsky 
\cite{FZ1}, which have since found applications in Poisson geometry, 
quiver representations,    and mathematical physics.
The first connection of cluster algebras
to scattering 
amplitudes in \mbox{$\mathcal{N}=4$} SYM
was made by Golden--Goncharov--Spradlin--Vergu--Volovich 
\cite{Golden:2013xva}, who demonstrated
that singularities of scattering amplitudes of planar $\mathcal{N}=4$ SYM at loop level can be described using cluster $\mathcal{X}$-variables.
The notion of \emph{cluster adjacency} 
was introduced  by
Drummond--Foster-- G\"urdo\u gan \cite{Drummond:2017ssj, Drummond:2018dfd},
who conjectured that 
the terms in tree-level $\mathcal{N}=4$ SYM amplitudes 
coming from the BCFW recursions are rational functions 
whose poles correspond to \emph{compatible cluster variables} of the Grassmannian,
 see also \cite{Mago:2019waa}. The cluster adjacency
conjecture was subsequently
reformulated in terms of the  amplituhedron 
\cite{Lukowski:2019sxw, Gurdougan2020ClusterPI}.  Roughly speaking, cluster adjacency says that while tiles of the amplituhedron
are not in general polytopes,
their ``facets'' can nevertheless be described by very nice polynomial equations: these polynomials
are {cluster variables} for the Grassmannian, and the cluster variables cutting out the facets of a given tile are all {compatible}.

In this paper we give the first proof that the cluster structure on the Grassmannian $\Gr_{4,n}$ is intimately connected to the BCFW recurrence and to tiles of the amplituhedron 
$\mathcal{A}_{n,k,4}(Z)$.  This allows us to address both of the guiding problems above.  Our main results are summarized below.

\begin{mainthm}  [Cluster algebras and BCFW tiles]  \label{mainthm:A}
Let $S_\rcp$ be an arbitrary ``BCFW cell'' of the positive Grassmannian 
$\Gr_{k,n}^{\scriptscriptstyle\geq 0}$, that is, a cell
arising from the BCFW recurrence.
\begin{enumerate}
\item  The amplituhedron map $\tilde{Z}$ is injective on the cell
$S_{\rcp}$.  In other words, each BCFW cell gives rise to a \emph{BCFW tile} 
 of the amplituhedron. 
\item Each BCFW tile can be described as the subset of the amplituhedron
where certain cluster variables for the Grassmannian $\Gr_{4,n}$, interpreted as twistor coordinates, have a fixed sign, either positive or negative.
\item (Cluster adjacency)
The cluster variables cutting out the 
\emph{facets} of each BCFW tile are all compatible.
\end{enumerate}
\end{mainthm}

Regarding item (1) above, we note that ever since \cite{arkani-hamed_trnka}, it was believed that all BCFW cells give rise to tiles.  This was confirmed for the ``standard'' BCFW cells
in \cite{even2021amplituhedron} but was unknown in all other cases until our work.
Regarding item (2), we note that in the theory of total positivity, there is a general paradigm that the ``totally positive'' part of a variety
is described as the locus where the cluster variables are all positive.  It is a novel feature of amplituhedron tiles that they have a concise description as the locus where cluster variables have fixed signs, which can be positive
\emph{or} negative. 
Regarding item (3), we note that cluster adjacency was previously proved for 
the $m=2$ amplituhedron in \cite{PSW}: that is, each tile of $\A_{n, k, 2}$ has facets given by compatible cluster variables for $\Gr_{2,n}$.  However, while there were 
experimental results for $m=4$, until our work
there were no conceptual or rigorous results connecting cluster algebras
to the $m=4$ amplituhedron or to the terms in tree-level $\mathcal{N}=4$ SYM amplitudes. We note that the case $m=4$ is significantly more difficult than the case $m=2$, in part because the cluster algebra
associated to $\Gr_{4,n}$ has infinite type instead of finite type.

By building on our results connecting amplituhedron tiles
to cluster algebras, we also prove the complete \emph{BCFW tiling conjecture}.
To give some background, in 2005,
Britto, Cachazo, Feng, and Witten \cite{BCFW} gave a recurrence which
expresses scattering amplitudes as sums of rational functions
of momenta. In this recurrence, the individual terms have 
``spurious poles,''  singularities not present in the amplitude. 
Hodges \cite{hodges} observed that in some cases,
the terms of the BCFW recurrence correspond to the simplices in a triangulation of a polytope, with spurious poles arising from internal boundaries of the triangulation,
and the amplitude computing the `volume' of the polytope as a sum over the simplices.
Arkani-Hamed and Trnka \cite{arkani-hamed_trnka} then introduced the amplituhedron in order to generalize Hodges' observation: they
interpreted the BCFW recurrence  
as giving collections of \emph{BCFW cells}, and conjectured that each way of running the recurrence
gives rise to a \emph{tiling} of
$\Ank$.  There is a canonical way of performing the BCFW recurrence which produces the 
``standard'' BCFW cells, and in
\cite{even2021amplituhedron}, Even-Zohar--Lakrec--Tessler proved  
that the images of standard BCFW cells tile the $m=4$ amplituhedron. 
This verified the BCFW tiling conjecture for a \emph{single} way of running the BCFW recursion, resulting in a single tiling of~$\Ank$.

By building on \cref{mainthm:A} -- specifically, our result that each BCFW tile can be described as the locus where certain cluster variables have fixed signs -- we also 
prove the main conjecture of~\cite{arkani-hamed_trnka}. This generalizes \cite{even2021amplituhedron} to arbitrary ways of running the recursion, though we note that our techniques are quite different.

\begin{mainthm}[BCFW tiling theorem]
	Every collection $\{S_{\rcp} \}$ of general BCFW cells obtained by iterating the 
	BCFW recurrence  gives rise to a tiling 
	$\{Z_{\rcp} \}$ 
	of the amplituhedron
$\Ank$.  
The tiles in $\{Z_{\rcp}\}$
have disjoint interiors, and they cover the amplituhedron $\Ank$.
\end{mainthm}

\section{Main results}
We now give definitions and explain our results in more detail. 
Given an $n \times (k+m)$ matrix~$Z$ with all maximal minors positive,
the \emph{amplituhedron map} is the map 
$\tilde{Z}: \Grk \to \Gr_{k,k+m}$  induced by matrix multiplication by $Z$. The \emph{amplituhedron} $\AA$ is the image of 
the nonnegative Grassmannian $\Grk$ under this map.
One 
of the main ideas of 
Arkani-Hamed and Trnka \cite{arkani-hamed_trnka} is 
that the \emph{BCFW recurrence} \cite{BCFW} for computing 
scattering amplitudes in $\mathcal{N}=4$ SYM 
can be interpreted as a recurrence which produces
collections of $N(n-3,k+1)$ 
$4k$-dimensional cells of $\Grk$, where $N(n-3,k+1)$ is the \emph{Narayana number} 
$\frac{1}{k+1}\binom{n-4}{k}\binom{n-3}{k}$. Different ways of running the recurrence produce different collections of cells.
Arkani-Hamed and Trnka conjectured that the $\tilde{Z}$-images of the cells in each of these collections gives a \emph{tiling} of the $m=4$ amplituhedron
$\A_{n,k,4}(Z)$. This means that $\tilde{Z}$ is injective on each cell in the collection, and the images are disjoint and dense in $\Ank.$

We use the terminology \emph{(general) BCFW cells} for the positroid cells in $\Grk$ arising from any way of running the BCFW recursion. More concretely, we construct general BCFW cells recursively, using the \emph{BCFW product}. See Definitions~\ref{def:butterfly} and~\ref{def:BCFW_cell} for details.

As a prerequisite for proving cluster adjacency and 
the BCFW tiling conjecture, we need to show that BCFW cells give rise to tiles.
The following result appears in \cref{thm:BCFW-tile-and-sign-description}.
\begin{theorem*}[BCFW tile theorem]
The amplituhedron map is injective on each general BCFW cell.
That is,
	the closure 
	$\gt{\rcp} :=
	\overline{\tilde{Z}(S_{\rcp})}$ of 
	the image of a general BCFW cell $S_{\rcp}$
 is a \emph{tile}. 
\end{theorem*}

We refer to the closures of the images of general BCFW cells as \emph{general BCFW tiles}.
There is a natural notion of  \emph{facet} of a tile (see \cref{def:facet2}), 
generalizing the notion of facet of a polytope. 
We study facets by describing associated 
\emph{functionaries} -- 
polynomial functions in \emph{twistor coordinates} -- 
which vanish on them.

Our next main result 
(cf.~\cref{thm:clusteradjacency})
is the following. For $I\in {[n]\choose 4}$, we use $\lr{I}$ 
and $\llrr{I}$ to denote the 
corresponding  Pl\"ucker coordinate of $\Gr_{4,n}$
and  
twistor coordinate on $\Gr_{k,k+4}$; see 
\cref{sec:Pluckertwistor} for the connection between 
Pl\"ucker coordinates and twistor coordinates.

\begin{theorem*}
    [Cluster adjacency for BCFW tiles]
	Let $\gt{\rcp}$ be a general BCFW tile of $\Ank$. Then for each facet $\gt{S}$ of $\gt{\rcp}$, there is a functionary $F_S (\llrr{I})$ which vanishes on $\gt{S}$, such that 
	the  set
	\[\{F_S(\lr{I}): \gt{S} \text{ a facet of } \gt{\rcp} \}\]
	 is a collection of compatible cluster variables for $\Gr_{4,n}$.
\end{theorem*}

Strengthening the connection with cluster algebras, 
we associate to each general BCFW tile $\gt{\rcp}$ 
a larger 
collection of compatible cluster variables $\Irr(\rcp)$ for $\Gr_{4,n}$ 
(cf.~\cref{def:generalcluster}). Interpreting each cluster variable as a functionary,
we describe each general BCFW tile as the subset of the Grassmannian
$\Gr_{k,k+4}$  where these cluster variables take on particular signs.
The following theorem appears later as \cref{cor:cluster-sign-description}.
\begin{theorem*}[Sign description of BCFW tiles]
	Let $\gt{\rcp}$ be a general BCFW tile. For each element $x$ of $\Irr(\rcp)$, the functionary $x(Y)$ has a definite sign $s_x$ on $\gto{\rcp}$ and 
	\[\gto{\rcp}= \{Y \in \Gr_{k,k+4}: s_x \, x(Y) >0 \text{ for all } x \in \Irr(\rcp) \}.\]
\end{theorem*}

Our strategy to prove the cluster adjacency theorem above is to show that under the
\emph{BCFW product}, 
 the polynomials cutting out BCFW tiles
 evolve by \emph{product promotion} (cf.~\Cref{pro-twistors2})
$$ \Psi : \C(\widehat{\Gr}_{4,n_L})\times \C(\widehat{\Gr}_{4,n_R}) \to \C(\widehat{\Gr}_{4,n}) .$$
We also show that product promotion is a \emph{cluster 
quasi-homomorphism} (cf.~\Cref{thm:promotion2}).
This fact may be of independent interest
in the study of the cluster structure on $\Gr_{4,n}$, as it allows us
to map cluster variables (respectively, clusters) between coordinate rings of Grassmannians.

We give additional cluster-theoretic results for \emph{standard} BCFW cells, which arise from one particularly natural way to iterate
the BCFW recurrence. The (images of) standard BCFW cells were shown to give a tiling of $\Ank$ in \cite{even2021amplituhedron}.
Several descriptions of standard BCFW cells in terms of Narayana-enumerated objects were given in \cite{karp2020decompositions}. Another description in terms of \emph{chord diagrams} (see Figure~\ref{fig:intr-chord}) was given in \cite{even2021amplituhedron}, and it is that one which we will primarily use here (see \cref{def:cd}). A chord diagram can be interpreted as a recursive procedure for constructing
a standard BCFW cell (cf. \cref{def:standardfromCD}).

\begin{figure}[h]
\begin{center}
\tikz[line width=0.8,scale=0.9]{
\draw (0.5,0) -- (15.5,0);
\foreach \i in {1,2,...,15}{
\def\x{\i}
\draw (\x,-0.1)--(\x,+0.1);
\node at (\x,-0.5) {\i};}
\foreach \i/\j in {1/8, 3/4.8, 5.2/7.75, 8.25/13, 9/12.2, 10/11.8}{
\def\x{\i+0.5}
\def\y{\j+0.5}
\draw[line width=1.5,-stealth] (\x,0) -- (\x,0.25) to[in=90,out=90] (\y,0.25) -- (\y,0);
}
\node at(1.5,1.5) {$D_3$};
\node at(3.5,0.875) {$D_1$};
\node at(5.75,1) {$D_2$};
\node at(9,1.5) {$D_6$};
\node at(9.5,1) {$D_5$};
\node at(10.5,0.8125) {$D_4$};
}
\end{center}
\caption{\label{fig:intr-chord} The  chord diagram for a standard BCFW tile of $\mathcal{A}_{15,6,4}(Z)$.}
\end{figure}
\begin{figure}
\centering
\includegraphics[width=0.5\textwidth]{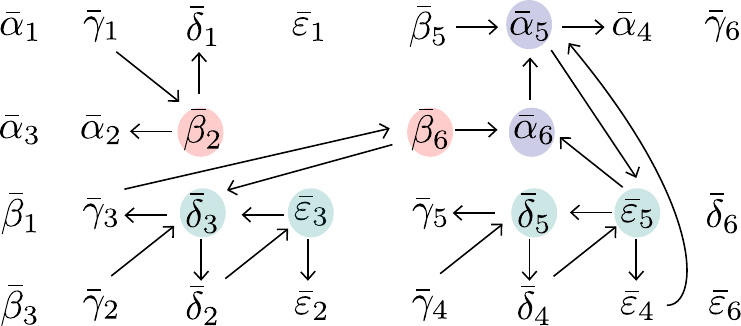}
\caption{The seed associated to the standard BCFW tile in \cref{fig:intr-chord}. The mutable variables are circled; all other variables are frozen.
 The colors (red, green, blue) indicate the different cases of \cref{def:seed}.
	\label{fig:cd94}}
\vspace{1em}
\end{figure}

Using the combinatorics of chord diagrams, 
we give an explicit description of standard BCFW tiles and associated 
cluster structures, including:
a characterization of each facet of a standard BCFW tile, together with 
the functionary which vanishes on the facet (\cref{prop:standard_facets});
a description of the seed associated to a standard BCFW tile, 
including both the quiver (\cref{def:seed}, \cref{thm:quiver})
and the cluster variables (\cref{prop:explicit}, \cref{thm:cluster}).
We expect that most of these results can be extended to general BCFW tiles.

We note that much of the work thus far on the cluster structure of the Grassmannian
has focused on cluster variables which are Pl\"ucker coordinates 
or are polynomials in Pl\"ucker coordinates
with low degree; by constrast, the cluster variables we obtain can have arbitrarily high degree
in Pl\"ucker coordinates, see \cref{prop:explicit} for their explicit descriptions.

In the final section of the paper, we use our understanding of BCFW tiles to generalize the main result of \cite{even2021amplituhedron} and prove the main conjecture of \cite{arkani-hamed_trnka}. In particular, we use a
characterization of the boundaries of BCFW cells
(cf.~\cref{lem:boundaries_before_amp_map}) and analysis of 
signs of cluster variables on general BCFW cells
to prove the following (cf.~\cref{thm:BCFWtiling}):

\begin{theorem*}[BCFW tiling theorem]
	Every collection $\{S_{\rcp} \}$ of general BCFW cells obtained by iterating the 
	BCFW recurrence (cf. \cref{def:BCFWoutput}) gives rise to a tiling 
	$\{Z_{\rcp} \}$ 
	of the amplituhedron
$\Ank$.  That is:
	\begin{itemize}
		\item  the amplituhedron map is injective on each general BCFW cell
			$S_{\rcp}$, i.e. 
			$Z_{\rcp}:=\overline{\tZ(S_{\rcp})}$ 
			is a \emph{tile};
		\item  the open tiles
			$\{\gto{\rcp} \}$  
			are pairwise disjoint;
		\item  and the tiles in 
			$\{Z_{\rcp} \}$  cover the amplituhedron $\Ank$.
	\end{itemize}
\end{theorem*}
We note that \cite{even2021amplituhedron} proved this result for the collection consisting of standard BCFW cells.

\subsection{Further motivation and context}

From the point of view of physics,  
 it is important to  understand  
 facets of amplituhedron tiles, because 
 they correspond to poles of \emph{Yangian invariants}, 
   the rational functions 
   that are the building blocks of scattering amplitudes,
  see \cref{rk:yangianpromotion}.  Studying tiles can then
 give a geometric interpretation of the notion
of  \emph{cluster adjacency} in physics.
   Moreover, it would be interesting to show that 
   tiles are  \emph{positive geometries} 
   \cite{Arkani-Hamed:2017tmz,Lam:2022yly}.

From the point of view of cluster algebras, 
the study of tiles for the amplituhedron $\mathcal{A}_{n,k,m}$ is useful
because it is closely related to the 
cluster structure on the Grassmannian $\Gr_{m,n}$, as was shown for 
$m=2$ in 
\cite{PSW} and as this paper demonstrates for $m=4$.  In particular, for $m=4$,
the \emph{BCFW product} (\cref{sec:bcfw-and-products}) used to recursively build tiles 
(\cref{thm:BCFW-tile-and-sign-description}) has a cluster quasi-homomorphism 
counterpart called \emph{product promotion} (\cref{sec:quasihom}),
that can be used to recursively construct 
cluster variables and seeds in $\Gr_{4,n}$ (\cref{thm:promotion2}). Secondly, inverting the amplituhedron map on tiles is related to solving the `$C \cdot z$ equations' which generates cluster variables for the \emph{symbol alphabet} of loop amplitudes, see \cref{rk:CZ_eq}. 

In the closely related field of \emph{total positivity}, one prototypical 
problem is to give an efficient characterization of
the ``positive part'' of a space as the subset
where a certain minimal collection of functions take on positive values \cite{fz-intel}; 
this is sometimes called a ``positivity test.''
For example,  if we define
the \emph{positive Grassmannian} $\Gr_{k,n}^{\scriptscriptstyle>0}$ as the subset of the real Grassmannian
where all Pl\"ucker coordinates are strictly positive, then 
 for any extended cluster $\txx$ 
for $\Gr_{k,n}$ \cite{scott}, we have 
\begin{equation*} \label{eq:pos_test}
	\Gr^{\scriptscriptstyle> 0}_{k,n}= \{Y \in \Gr_{k, n}: x(Y)>0 \text{ for all }x \in \txx \}.
\end{equation*}
We think of our ``Sign description of BCFW tiles'' theorem 
(see \cref{cor:cluster-sign-description} and  \cref{thm:sign-definite})
as ``positivity tests'' 
for membership in a BCFW tile of the amplituhedron. 
See \cite[Theorem 6.8]{PSW} for an analogous result for $m=2$, 
and \cref{conj:clusteradjmain} for some conjectures for general $m$.

From the point of view of discrete geometry, it is interesting to study tiles 
and more generally
$\tZ$-images of positroid cells because
one can think of them as a generalization of polytopes in the Grassmannian.  
In particular, our sign description of BCFW tiles
is roughly analogous to the hyperplane description of polytopes. 
Note, however, that the inequalities defining ``facets'' of the tiles---which correspond to frozen variables---are not enough to cut out the tiles as a subset of $\Gr_{k, k+4}$.  As for the positivity test for $\Gr^{\scriptscriptstyle>0}_{k,n}$, in addition to the frozen variables, we need to 
include some additional `mutable' cluster variables, hence the crucial role of cluster algebras.

\subsection{Connection to previous work}

Our definition of ``product promotion" (cf.~\cref{pro-twistors2}) is inspired by work on Yangian invariants in physics; see \cref{rk:yangianpromotion} for more details. A special case of promotion also appeared in \cite[Section 4]{even2021amplituhedron}; see \cref{rem:ELTpromotion} for more details.

Cluster adjacency for the $m=2$ amplituhedron was proved in \cite{PSW}. In the present paper we prove cluster adjacency for BCFW tiles of $\Ank$; there is evidence that cluster adjacency should also hold for more general tiles \cite{Gurdougan2020ClusterPI}. Interestingly, if one looks at arbitrary images of positroid cells in $\Ank$ rather than \emph{injective} images (aka tiles), cluster adjacency fails. This is shown in \cite{Gurdougan2020ClusterPI} where facets are cut out by \emph{non-compatible} cluster variables. We note that there are also cluster adjacency conjectures for the \emph{symbol} of loop-level $\mathcal{N}=4$ SYM amplitudes \cite{Drummond:2017ssj}. These conjectures are related to \emph{loop amplituhedra} rather than the tree amplituhedron studied here, and remain open.

BCFW-like tilings of the $m=1$ and $m=2$ amplituhedron 
were proved in \cite{karpwilliams}
and \cite{BaoHe}, building on prior work of \cite{AHTT}
and \cite{karp2020decompositions}. 
Partial progress was made 
on the BCFW tiling conjecture
for $m=4$ in \cite{karp2020decompositions}, including
an explicit description of the \emph{standard} BCFW cells,
those cells obtained by performing the BCFW recurrence in the canonical way. Further progress on the conjecture was made in \cite{even2021amplituhedron}, where they showed the images of standard BCFW cells comprise a single tiling of $\Ank$. The methods of \cite{even2021amplituhedron} for standard BCFW cells are quite different from our methods for general BCFW cells; see \cref{rem:ELT-method-comparison} for more details.

\subsection{Organization}
The structure of this paper is as follows. 
\cref{sec:background-amp} provides background on the positive
Grassmannian and the amplituhedron, and \cref{sec:background-cluster}
provides background on cluster algebras, including the notion of 
cluster quasi-homomorphism, and the cluster algebra structure of 
$\C[\widehat{\Gr}_{k,n}]$.
In \cref{sec:quasihom} we introduce \emph{product promotion} 
and prove that it is a cluster quasi-homomorphism.
In \cref{sec:BCFWmap} we define the \emph{BCFW map} and \emph{BCFW product},
which we then use in 
\cref{def:BCFWcells} to define the standard and general BCFW cells.
In \cref{sec:BCFWtilesfacets} we analyze the images of BCFW cells 
and prove the cluster adjacency theorem for general BCFW tiles.
In \cref{sec:clustervariable} we deepen the connection to 
cluster algebras: we give explicit formulas for the cluster variables
associated to standard BCFW tiles, and describe 
the sign that each such cluster variable takes on the tile.
In \cref{sec:quiver} we then associate an explicit seed
(a quiver plus cluster variables) to each standard BCFW tile, using the 
combinatorics of chord diagrams.
In \cref{sec:BCFWcells} 
and \cref{sec:imagesBCFWcells}
we provide proofs of some 
technical results about BCFW cells and BCFW tiles, including
the fact that 
the amplituhedron map is injective on each (general) BCFW cell.
These results are then used in \cref{sec:BCFWtilings}, where
we show that each collection of BCFW cells  gives rise to a 
tiling of the amplituhedron.
We end with
\cref{app:plabic}, which provides necessary background on 
plabic graphs.

\vskip .2cm

\noindent{\bf Acknowledgements:~} 
The authors would like to thank Nima Arkani-Hamed for many
inspiring conversations.
TL is supported by SNSF grant Dynamical Systems, grant no.~188535.
MP is supported by the CMSA at Harvard University and at the Institute for Advanced Study by the U.S. Department of Energy under the grant number DE-SC0009988. 
MSB is supported by the National Science Foundation under Award No.~DMS-2103282 and DMS-2444020.
RT (incumbent of the Lillian and George Lyttle Career Development Chair) was supported by the ISF grant No.~335/19 and 1729/23.
RT would like to thank Yoel Groman for discussions related to this work.
LW is supported by the National Science Foundation under Award No. 
DMS-1854316 and DMS-2152991. Any opinions, findings, and conclusions or recommendations expressed in this material are
those of the author(s) and do not necessarily reflect the views of the National Science
Foundation.  The authors would also like to thank 
Harvard University, 
the Institute for Advanced Study, 
and the `Research in Paris' program at the 
Institut Henri Poincar\'e, 
where some of this 
work was carried out.

\section{Background on the amplituhedron (tiles, facets, functionaries)}\label{sec:background-amp}

\subsection{The (positive) Grassmannian}\label{sec:posGrass}

The \emph{Grassmannian} $\Gr_{k,n}(\F)$
is the space of all $k$-dimensional subspaces of
an $n$-dimensional vector space $\F^n$.
Let $[n]$ denote $\{1,\dots,n\}$, and $\binom{[n]}{k}$ denote the set of all $k$-element 
subsets of $[n]$.
We can
 represent a point $V \in
\Gr_{k,n}(\F)$  as the row-span 
of
a full-rank $k\times n$ matrix $C$ with entries in
$\F$.  
Then for $I=\{i_1 < \dots < i_k\} \in \binom{[n]}{k}$, we let $\lr{I}_V=\lr{i_1\,i_2\,\dots\,i_k}_V$ be the $k\times k$ minor of $C$ using the columns $I$. 
The $\lr{I}_V$ 
are called the {\itshape Pl\"{u}cker coordinates} of $V$, and are independent of the choice of matrix
representative $C$ (up to common rescaling). The \emph{Pl\"ucker embedding} 
$V \mapsto \{\lr{I}_V\}_{I\in \binom{[n]}{k}}$
embeds  $\Gr_{k,n}(\F)$ into
projective space. When it does not cause confusion, we will identify $C$ with its row-span and drop the subscript $V$ on Pl\"ucker coordinates. 

If $C$ has columns $v_1, \dots, v_n$, we may also identify $\lr{i_1\,i_2\,\dots\,i_k}_V$ with the element $v_{i_1} \wedge v_{i_2} \wedge \dots \wedge v_{i_k}$, hence the Pl\"ucker coordinates are \emph{alternating} in the indices, e.g. $\lr{i_1\,i_2\,\dots\,i_k}=- \lr{i_2\,i_1\,\dots\,i_k}$. 
\begin{remark}\label{rem:convention} 
For convenience, if $I$ is a subset (rather than a sequence) of positive integers,
then we let $\lr{I}$ denote the Pl\"ucker coordinate obtained
by listing 
the elements of $I$ in \emph{increasing} order. 
\end{remark}

In this paper we will often be working with the 
\emph{real} Grassmannian $\Gr_{k,n}=
\Gr_{k,n}(\R)$.

\begin{definition}[Positive Grassmannian]\label{def:positroid}\cite{lusztig, postnikov}
We say that $V\in \Gr_{k,n}$ is \emph{totally nonnegative}
     if (up to a global change of sign)
       $\lr{I}_V \geq 0$ for all $I \in \binom{[n]}{k}$.
Similarly, $V$ is \emph{totally positive} if $\lr{I}_V >0$ for all $I
      \in \binom{[n]}{k}$.
We let $\Grk$ and $\Gr_{k,n}^{>0}$ denote the set of
totally nonnegative and totally positive elements of $\Gr_{k,n}$, respectively.  
$\Grk$ is called the \emph{totally nonnegative}  \emph{Grassmannian}, or
       sometimes just the \emph{positive Grassmannian}.
\end{definition}

If we partition $\Grk$ into strata based on which Pl\"ucker coordinates are strictly
positive and which are $0$, we obtain a cell decomposition of $\Grk$
into \emph{positroid cells} \cite{postnikov}.
Each positroid cell $S$ gives rise to a matroid, whose bases are precisely
the $k$-element subsets $I$ such that the Pl\"ucker coordinate
$\lr{I}$ does not vanish on $S$; this matroid 
is called a \emph{positroid}.

Both $\Gr_{k,n}$ and $\Grk$ admit an action of the dihedral group, 
which will be useful to us.

\begin{definition}[Dihedral group on the Grassmannian]\label{def:dihedral} Let $[\cyc_{k,n}]$ denote the $n \times n$ matrix
	\[[\cyc_{k,n}]= \begin{bmatrix}
		0 & 1 & \dots & 0\\
		\vdots  & 0  &       \ddots   & \vdots\\
	 0 & 0  &      \cdots     & 1\\
		(-1)^{k-1}& 0& \dots & 0\\ 
	\end{bmatrix}.\]
	We define the \emph{cyclic shift}
 as the map $\cyc: \Mat_{k, n} \to \Mat_{k,n}$ which sends
	\begin{align*}
	\begin{bmatrix}
		| & |& & |\\
		v_1 & v_2 & \cdots & v_n\\
		| & |& & |
	\end{bmatrix} \mapsto 	\begin{bmatrix}
	| & |& & |\\
	v_1 & v_2 & \cdots & v_n\\
	| & |& & |
\end{bmatrix} \cdot [\cyc_{k,n}]=
\begin{bmatrix}
	| & |& & |\\
	(-1)^{k-1}v_n & v_1 & \cdots & v_{n-1}\\
	| & |& & |
\end{bmatrix}. \end{align*}

Let $[\refl_n]$ denote the antidiagonal $n \times n$ matrix with 1's along the antidiagonal and let $P_{k,i, s}$ denote the $k \times k$ diagonal matrix with $i$th diagonal entry $(-1)^s$ and all others equal to 1. We define \emph{reflection} as the map $\refl:\Mat_{k, n} \to \Mat_{k,n}$ which sends
	\begin{align*}
	\begin{bmatrix}
	| & |& & |\\
	v_1 & v_2 & \cdots & v_n\\
	| & |& & |
\end{bmatrix} \mapsto 
P_{k,1, \binom{k}{2}} \cdot	\begin{bmatrix}
	| & |& & |\\
	v_1 & v_2 & \cdots & v_n\\
	| & |& & |
\end{bmatrix} \cdot [\refl_n]=
P_{k,1, \binom{k}{2}}\begin{bmatrix}
	| & |& & |\\
	v_n & v_{n-1} & \cdots & v_{1}\\
	| & |& & |
\end{bmatrix} \end{align*}
		The maps $\cyc, \refl$ descend to automorphisms on $\Gr_{k,n}$ and $\Grk$, which we denote in the same way.
 Note that
\begin{align}\label{eq:cyc-refl-on-plucker}\lr{I}_{\cyc M} = \lr{I-1}_{M}   \text{ and } \lr{I}_{\refl M}
	= \lr{n+1-I}_{M}, \end{align}
		where $I-1$ is the subset obtained from $I$ by 
		subtracting $1$ (modulo $n$) from each element\footnote{Note that according to the convention in \cref{rem:convention}, the elements of $I-1$ are listed in \emph{increasing} order in $\lr{I-1}$.} and $n+1-I$ is obtained from $I$ by subtracting each element from $n+1$.
 That is, the pullbacks are
\[\cyc^*: \lr{I} \mapsto \lr{I-1} \quad \text{and} \quad \refl^*:\lr{I} \mapsto \lr{n+1-I}.\]
We use the notation $\cyc^{-*}:=(\cyc^{-1})^*$.
\end{definition}

There is another operation which will be useful to us, which embeds $\Gr_{k,n}$ into a larger Grassmannian.
\begin{definition}\label{def:pre}
	Choose an index set $N$ and let $I \subset \NN$ be disjoint from $N$. The map $\pre_I: \Mat_{k,N} \to \Mat_{k, N \cup I}$ inserts zero columns at  positions $I$, where $\Mat_{k,N}$ consists of matrices with rows indexed by $1, \dots, k$ and columns indexed by $N$. If $I=\{i\}$, we write $\pre_i:= \pre_{\{i\}}$. The map $\pre_I$ descends to a map on $\Gr_{k,N}$ and $\Gr_{k,N}^{\geq 0}$, which we denote in the same way.
\end{definition}

\subsection{Chain polynomials}\label{poly:Plucker}
In what follows, we will be studying cluster variables for the Grassmannian which are certain polynomials
in the Pl\"ucker coordinates, so we introduce some distinguished polynomials here.

\begin{definition}[Chain polynomials]\label{def:chain_polynomials} We introduce the quadratic polynomials
\begin{align}
	\lr{a\,b\,c \br d\,e \br f\,g\,h}  &= 
\lr{a\,b\,c\,d}\,\lr{e\,f\,g\,h} - \lr{a\,b\,c\,e}\,\lr{d\,f\,g\,h} \label{eq1:quadratic} \\  
	&= -\lr{a\,b\,c\,f} \lr{d\,e\,g\,h}+\lr{a\,b\,c\,g} 
        \lr{d\,e\,f\,h}-\lr{a\,b\,c\,h} \lr{d\,e\,f\,g}. \label{eq2:quadratic}
\end{align}
More generally, we define polynomials of degree~$s+1$ as follows:
\begin{align*}
	&  \lr{a_0\,b_0\,c_0 \br d_{1,0}\,d_{1,1} \br b_1\,c_1 \br d_{2,0}\,d_{2,1} \br b_2\,c_2 \br \dots \br d_{s,0}\,d_{s,1} \br b_s\,c_s\,a_s} \\[0.25em] 
	& 
	\;=\; \sum_{t \in \{0,1\}^s}\;(-1)^{t_1+\dots+t_s} \,\lr{a_0\,b_0\,c_0\,d_{1,{t_1}}}\, \lr{d_{1,{1-t_1}}\,b_1\,c_1\,d_{2,{t_2}}} \, \lr{d_{2,{1-t_2}}\,b_2\,c_2\,d_{3,{t_3}}} \, \cdots \, \lr{d_{s,{1-t_s}}\,b_s\,c_s\,a_s}
\end{align*}
In the notation 
	$\lr{a_0\,b_0\,c_0 \br d_{1,0}\,d_{1,1} \br b_1\,c_1 \br d_{2,0}\,d_{2,1} \br b_2\,c_2 \br \dots \br d_{s,0}\,d_{s,1} \br b_s\,c_s\,a_s}$, the vertical bars separate the indices into collections which we call \emph{clauses}, e.g. each of $a_0 b_0 c_0$ and $d_{1,0}\,d_{1,1}$ is a clause.
\end{definition}

\begin{remark} 
	Permuting blocks
	introduces a sign coming from the 
	sign of the permutation, e.g. 
	$$ \lr{a\,b\,c \br d\,e \br f\,g\,h} = 
	- \lr{f\,g\,h \br d\,e \br a\,b\,c}.$$
	Also note
	that permuting the indices in a single block (e.g. permuting the numbers $abc$)
	multiplies the polynomial by the sign of the permutation. 
\end{remark}

\begin{remark}\label{rem:factors}
These polynomials may factor or terms may vanish if there are 
overlaps in the indices.  E.g. in 
\eqref{eq1:quadratic},
$\lr{a\,b\,c \br d\,e \br f\,g\,h}$ will factor as a product of 
two Pl\"ucker coordinates if
$|\{a,b,c\} \cap \{f,g,h\}| = 2$ or $|\{d,e\} \cap (\{a,b,c\} \setminus \{f,g,h\})|=1$. Concretely, if $a<b<c<d<e<f<g<h$,
\begin{align*}
	\lr{a\, b\, d\br d\,e \br f\,g\,h}&= -\lr{abde}\lr{dfgh}\\
	\lr{a\, b\, c\br d\,e \br a\,b\,h}&= \lr{abcd}\lr{eabh}-\lr{abce}\lr{dabh}
	=\lr{abch}\lr{abde}.
\end{align*}
\end{remark}

\begin{remark}[Notations in physics]
In the physics literature, 
	\eqref{eq1:quadratic} 
	is usually
denoted by
        $\langle a, b, c, \, (de)\cap (fgh) \rangle$ or simply
        $\langle a b c \, (de)\cap (fgh) \rangle$,
        where $(de)\cap (fgh)$ denotes the vector $v_d \langle efgh\rangle - v_e \langle dfgh\rangle $, which spans the intersection 
	between the two-plane $(de)$ spanned by $v_d,v_e$ and the three-plane $(fgh)$ spanned by $v_f, v_g, v_h$. 
	Similarly, \eqref{eq1:quadratic} 
	can also be rewritten as $\langle (a b c)\cap (de) fgh \rangle$ and using \eqref{eq2:quadratic} as $\lr{de (fgh) \cap (abc)}$, where $(fgh) \cap (abc)$ denotes $v_f \wedge v_g \lr{h abc}-v_f \wedge v_h \lr{g abc}+v_g \wedge v_h \lr{f abc}$.
This notation was introduced in \cite{Arkani-Hamed:2010zjl}, see also \cite[Section~2]{Golden:2013xva}.
\end{remark}

\begin{definition}
[Pure functions and degree]
\label{def:pure}
Let $\widehat{\Gr}_{k,n}$ denote the affine cone over the Grassmannian in its Pl\"ucker
embedding.
Recall that the homogeneous coordinate ring 
$\CC[\widehat{\Gr}_{k,n}]$ 
of the 
Grassmannian is generated by the Pl\"ucker coordinates, and 
is $\Z_{>0}^n$-graded; equivalently,
its relations are preserved by the torus action
on the Grassmannian. 
Concretely, $F \in \CC[\widehat{\Gr}_{k,n}]$ 
lies in the component with grading $(a_1, \dots, a_n) \in \Z_{>0}^n$ if it can be written as a polynomial in Pl\"ucker coordinates where the index $i$ appears $a_i$ times in each term. In this case, we call $F$ \emph{pure} and define $\deg_i F= a_i$. An element of the ring of rational functions 
$\C(\Gr_{k,n})$ 
of the Grassmannian  
is \emph{pure} if it can be written as the ratio $F/G$ where $F,G \in \C[\widehat{\Gr}_{k,n}]$ are both pure. We set $\deg_i F/G = \deg_i F - \deg_i G$. Throughout
this paper we will only be considering functions which are pure. 
\end{definition}

\subsection{The amplituhedron, tiles, tilings, and facets}

Building on \cite{abcgpt},
Arkani-Hamed and Trnka
\cite{arkani-hamed_trnka}
introduced
the \emph{(tree) amplituhedron}, which they defined as 
the image of the positive Grassmannian under a positive linear map.
Let $\Mat_{n,p}^{>0}$ denote the set of $n\times p$ matrices whose maximal minors
are positive.

\begin{definition}[Amplituhedron]\label{defn_amplituhedron}
Let $Z\in \Mat_{n,k+m}^{>0}$, where $k+m \leq n$. 
    The \emph{amplituhedron map}
$\tilde{Z}:\Gr_{k,n}^{\ge 0} \to \Gr_{k,k+m}$
        is defined by
        $\tilde{Z}(C):=CZ$,
    where
 $C$ is a $k \times n$ matrix representing an element of
        $\Gr_{k,n}^{\ge 0}$,  and  $CZ$ is a $k \times (k+m)$ matrix representing an
         element of $\Gr_{k,k+m}$.
        The \emph{amplituhedron} $\mathcal{A}_{n,k,m}(Z) \subset \Gr_{k,k+m}$ is the image
$\tilde{Z}(\Gr_{k,n}^{\ge 0})$.
\end{definition}

In this paper we will be concerned with the case where $m=4$. 

\begin{notation}\label{rem:Z}
Throughout this paper we will often use the phrase ``for all $Z$''
as shorthand for 
``for all $Z\in \Mat_{n,k+4}^{>0}$,'' when the $k$ and $n$ are understood. 
\end{notation}

\begin{definition}[Tiles]
\label{defn_tile}
	Fix $k, n, m$ with $k+m \leq n$ and choose
	$Z\in \Mat_{n,k+m}^{>0}$.  
	Given a positroid cell $S$ of 
	$\Gr_{k,n}^{\ge 0}$, we let 
	$\gto{S} := \tilde{Z}(S)$ and 
	$\gt{S}: = \overline{
		\tilde{Z}(S)} = \tilde{Z}(\overline{S})$.
	If $D$ is an element of an indexing set and $S_D$ is the corresponding cell, we also write $\gto{D}$ and $\gt{D}$ for $\gto{S_D}$ and $\gt{S_D}$.
	We call
	$\gt{S}$  and $\gto{S}$
	a \emph{tile}  and an \emph{open tile}
	for $\AA$ 
	if $\dim(S) =km$ and 
	$\tilde{Z}$ is injective on $S$. 
\end{definition}

\begin{definition}[Tilings]\label{def:tiling}
	A \emph{tiling} 
                of $\AA$ is
                a collection
		$\{Z_{S} \ \vert \ S\in \mathcal{C}\}$
                 of tiles, such that:
                 \begin{itemize}
                         \item {their union equals $\AA$}, and 
                         \item {the open tiles $\gto{S},\gto{S'}$ are pairwise disjoint.}
                 \end{itemize}
\end{definition}

\begin{definition}[Facet of a cell and a tile]\label{def:facet2}
Given two positroid cells $S'$ and $S$, we say that 
$S'$ is a \emph{facet} of $S$ if 
$S' \subset \overline{S}$ and $S'$ has codimension $1$ in $\overline{S}$.
If $S'$ is a facet of $S$ and $Z_S$ is a tile of $\AA$, we say that $Z_{S'}$ is a
	\emph{facet} 
 of $Z_{S}$ if 
         $\gt{S'} \subset \partial \gt{S}$ and has codimension 1 in $\gt{S}$.
\end{definition}

\subsection{Twistor coordinates and functionaries}

\begin{definition}[Twistor coordinates]
Fix
  $Z \in \Mat_{n,k+m}^{>0}$
   with rows  $Z_1,\dots, Z_n \in \R^{k+m}$.
   Given $Y \in \Gr_{k,k+m}$
   with rows $y_1,\dots,y_k$,
   and $\{i_1,\dots, i_m\} \subset [n]$,
   we define the \emph{twistor coordinate}, denoted
	$$\llrr{Y Z_{i_1} Z_{i_2} \cdots Z_{i_m}} \text{ or }
	\llrr{{i_1} {i_2} \cdots {i_m}}$$
        to be  the determinant of the
        matrix with rows
        $y_1, \dots,y_k, Z_{i_1}, \dots, Z_{i_m}$.
\end{definition}
  Note that using Laplace expansion, we can express 
the twistor coordinates of $Y=CZ$ as linear functions in the Pl\"ucker coordinates of $C$, whose coefficients are the $m \times m$ minors of $Z$.

\begin{lemma}
	[Twistor expansion lemma]
	\label{rk:Lapexpansion} Let $C \in \Grk$ and set $Y:=CZ$.
Then
 \begin{align}\label{eq:cauchy-binet}
	 \llrr{Y~Z_{i_1} \dots Z_{i_m}} &=
                \sum_{J=\{j_1<\dots<j_k\} \in {[n] \choose k}} \lr{J}_C \, \lr{j_1, \dots, j_k, i_1,\dots, i_m}_Z\\ 
		& =\sum_{J \in {[n] \choose k}, \  J \cap I = \emptyset} \lr{J}_C \, 
		\lr{j_1, \dots, j_k, i_1,\dots, i_m}_Z,
        \end{align}
	where the second line follows from the first by noting that 
		$\lr{j_1, \dots, j_k, i_1,\dots, i_m}_Z = 0$
		whenever $\{j_1,\dots,j_k\}$
		and $\{i_1,\dots,i_m\}$ have an element in common.
\end{lemma}

\begin{definition}[Functionaries and rational functionaries]
We refer to a homogeneous polynomial in twistor coordinates as a \emph{functionary}, and 
 to a ratio of functionaries as a \emph{rational functionary}.
\end{definition}
By a slight abuse of notation, we will sometimes consider a functionary as a function on the domain Grassmannian $\Gr_{k,n}$, by composing with the amplituhedron map.

\begin{definition}\label{def:signs}
  For $S \subset \Grk$, we say a rational functionary $F$ has \emph{sign $s \in \{0, \pm 1\}$ on the image of $S$} if for all $Z\in \Mat_{n,k+4}^{>0}$ and for all $Y \in \gto{S}$, $F(Y)$ has sign $s$.  If $F$ has sign 0, we also say it \emph{vanishes} on the image of $S$.
	In these cases, we say that $F$ has \emph{fixed sign} on 
	$\gto{S}$.
\end{definition}

As we will see  in \Cref{sec:Pluckertwistor}, we can identify twistor coordinates with Pl\"ucker coordinates
of a related matrix. We therefore use parallel notation to that of \Cref{poly:Plucker} 
in what follows.

\begin{definition}[Chain functionaries]
We let \begin{equation}\label{eq:quadratic}
\llrr{a\,b\,c \br d\,e \br f\,g\,h} \;=\; 
\llrr{a\,b\,c\,d}\,\llrr{e\,f\,g\,h} - \llrr{a\,b\,c\,e}\,\llrr{d\,f\,g\,h}, 
	\end{equation}
and more generally, we define \emph{chain functionaries} of degree~$s+1$ as follows:
\begin{align*}
&  \llrr{a_0\,b_0\,c_0 \br d_{1,0}\,d_{1,1} \br b_1\,c_1 \br d_{2,0}\,d_{2,1} \br b_2\,c_2 \br \dots \br d_{s,0}\,d_{s,1} \br b_s\,c_s\,a_s} \\[0.25em] 
& 
\;=\; \sum_{t \in \{0,1\}^s}\;(-1)^{t_1+\dots+t_s} \,\llrr{a_0\,b_0\,c_0\,d_{1,{t_1}}}\, \llrr{d_{1,{1-t_1}}\,b_1\,c_1\,d_{2,{t_2}}} \, \llrr{d_{2,{1-t_2}}\,b_2\,c_2\,d_{3,{t_3}}} \, \cdots \, \llrr{d_{s,{1-t_s}}\,b_s\,c_s\,a_s}
\end{align*}
\end{definition}

\subsection{Pl\"ucker coordinates versus twistor coordinates}	
\label{sec:Pluckertwistor}

In this section we explain how we can talk interchangeably about 
twistor coordinates for the amplituhedron and 
Pl\"ucker coordinates for $\Gr_{4,n}$. In particular, 
the amplituhedron is homeomorphic to the \emph{B-amplituhedron},
and the homeomorphism identifies twistor coordinates of the amplituhedron
with Pl\"ucker coordinates of the B-amplituhedron.
The construction of the \emph{B-amplituhedron} is motivated by the observation that 
since $\AA \subset \Gr_{k,k+m}$, when $m$ is small, it is sometimes more convenient
to take orthogonal complements and work with $\Gr_{m,k+m}$ instead of $\Gr_{k,k+m}$.

\begin{definition}[{\cite[Definition 3.8]{karpwilliams}}]
        \label{def:B}
       Choose $Z \in \Mat_{n,k+m}^{>0}$,
        and let
        $W\in \Gr_{k+m,n}^{>0}$
        be the column span of $Z$.
        We define the \emph{B-amplituhedron} to be
        $$\B_{n,k,m}(W):=\{V^{\perp} \cap W \ \vert \ V\in \Grk\} \subset
        \Gr_{m}(W),$$
        where $\Gr_m(W)$ denotes the Grassmannian of $m$-planes
        in $W$.
\end{definition}

\begin{proposition} [{\cite[Lemma 3.10 and Proposition 3.12]{karpwilliams}}]
\label{prop:AB}
     Choose $Z \in \Mat_{n,k+m}^{>0}$,
and let
        $W\in \Gr_{k+m,n}^{>0}$
         be the column span of $Z$.
        We 
         define a map $f_Z: \Gr_m(W) \to \Gr_{k,k+m}$ by
        $$f_Z(X):= Z(X^{\perp}) = \{Z(x): x\in X^{\perp}\}.$$
        Then
$f_Z$ restricts to a homeomorphism from
        $\B_{n,k,m}(W)$ onto $\AA$, sending
        $V^{\perp} \cap W$ to $\tilde{Z}(V)$ for all
        $V\in \Grk$. Moreover, $f_Z$ restricts to a diffeomorphism from an open neighborhood of $\B_{n,k,m}(W)$ in $\Gr_m(W)$ to a neighborhood of $\AA$ in $\Gr_{k,k+m}.$  
\end{proposition}

\cref{lem:coordinates} shows that the
twistor coordinates of the amplituhedron
$\AA \subset \Gr_{k,k+m}$
are equal to the Pl\"ucker coordinates of the B-amplituhedron
$\B_{n,k,m}(W) \subset \Gr_{m,n}$.
\begin{lemma} [{\cite[(3.11)]{karpwilliams}}]
        \label{lem:coordinates}
        If we let $Y:=f_Z(X)$ in \cref{prop:AB},
        we have
        \begin{equation}\label{eq:identify}
        \lr{I}_X = \llrr{Y Z_{I}}
                \qquad \text{ for all } \quad
                I \in \binom{[n]}{m}.
        \end{equation}
\end{lemma}

We say that a rational functionary $F$ is \emph{pure} if the corresponding
rational function in Pl\"ucker coordinates is pure (cf. \Cref{def:pure}) and define $\deg_i F$ in the same way.
Throughout this paper, the rational functionaries that we consider will be pure.

\begin{definition}\label{def:irreduciblefunctionary}
We say that a functionary is \emph{irreducible}
if and only if the corresponding polynomial in
Pl\"ucker coordinates in ${\Gr}_{m,n}$
defined by the map which takes $\llrr{hijl}\to\lr{hijl},$ is  irreducible.
\end{definition}
Our notion of when a functionary is irreducible may not
be the same as the notion of when its expansion
 in the Pl\"ucker coordinates of $C$ or $Y$ and minors of $Z$ is irreducible.

\section{Background on cluster algebras}\label{sec:background-cluster}
	
\subsection{Background on cluster algebras}
Cluster algebras were introduced by Fomin and Zelevinsky in \cite{FZ1}; see \cite{FWZ} for 
an introduction.
We give a quick definition of cluster algebras from quivers here.
All such cluster algebras are cluster algebras of \emph{geometric type}.

\begin{definition}
[\emph{Quiver}]
	A \emph{quiver} $Q$ is an oriented graph given by a finite set of
	vertices $Q_0$, a finite set of arrows $Q_1$, and two maps
$s: Q_1 \to Q_0$ and $t: Q_1 \to Q_0$ taking an arrow to its source
and target, respectively.
\end{definition}

A \emph{loop} of a quiver is an arrow $\alpha$ whose
source and target coincide.  A \emph{$2$-cycle} of a quiver is a pair of distinct
arrows $\beta$ and $\gamma$ such that $s(\beta) = t(\gamma)$
and $t(\beta) = s(\gamma)$.

\begin{definition}
[\emph{Quiver Mutation}]
Let $Q$ be a quiver without loops or $2$-cycles.
Let $k$ be a vertex of $Q$.
Following \cite{FZ1}, we define the
\emph{mutated quiver} $\mu_k(Q)$ as follows:
it has the same set of vertices as $Q$,
and its set of arrows is obtained by the following procedure:
\begin{enumerate}
\item for each subquiver $i \to k \to j$, add a new arrow $i \to j$;
\item reverse all allows with source or target $k$;
\item remove the arrows in a maximal set of pairwise
disjoint $2$-cycles.
\end{enumerate}
\end{definition}

It is not hard to check that mutation is an involution, that is,
$\mu_k^2(Q) = Q$ for each vertex $k$.

\begin{definition}
[\emph{Labeled seeds}]
\label{def:seed0}
Choose $s\geq r$ positive integers.
Let $\Fcal$ be an \emph{ambient field}
of rational functions
in $r$ independent
variables
over
$\CC(x_{r+1},\dots,x_s)$.
A \emph{labeled seed} in~$\Fcal$ is
a pair $(\txx, Q)$, where
\begin{itemize}
\item
$\txx = (x_1, \dots, x_s)$ forms a free generating
set for
$\Fcal$,
and
\item
$Q$ is a quiver on vertices
$1, 2, \dots,r, r+1, \dots, s$,
whose vertices $1,2, \dots, r$ are called
\emph{mutable}, and whose vertices $r+1,\dots, s$ are called \emph{frozen}.
\end{itemize}
We call~$\txx$  the (labeled)
\emph{extended cluster} of a labeled seed $(\txx, Q)$.
The variables $\{x_1,\dots,x_r\}$ are called \emph{cluster
variables}, and the variables $c=\{x_{r+1},\dots,x_s\}$ are called
	\emph{frozen} (or \emph{coefficient variables}).
We let $\PP$ denote the group of Laurent monomials in the frozen variables, which we call the \emph{frozen group}.
\end{definition}

\begin{definition}
[\emph{Seed mutations}]
\label{def:seed-mutation0}
Let $(\txx, Q)$ be a labeled seed in $\Fcal$,
and let $k \in \{1,\dots,r\}$.
The \emph{seed mutation} $\mu_k$ in direction~$k$ transforms
$(\txx, Q)$ into the labeled seed
$\mu_k(\txx,  Q)=(\txx', \mu_k(Q))$, where the cluster
$\txx'=(x'_1,\dots,x'_s)$ is defined as follows:
$x_j'=x_j$ for $j\neq k$,
whereas $x'_k \in \Fcal$ is determined
by the \emph{exchange relation}
\begin{equation}
\label{exchange relation0}
x'_k\ x_k =
 \ \prod_{\substack{\alpha\in Q_1 \\ s(\alpha)=k}} x_{t(\alpha)}
+ \ \prod_{\substack{\alpha\in Q_1 \\ t(\alpha)=k}} x_{s(\alpha)} \, .
\end{equation}
\end{definition}

Note that arrows between two frozen vertices of a quiver do not
affect seed mutation; therefore we often omit
arrows between two frozen vertices. 

\begin{definition}
[\emph{Patterns}]
\label{def:patterns0}
Consider the \emph{$r$-regular tree}~$\TT_r$
whose edges are labeled by the numbers $1, \dots, r$,
so that the $r$ edges emanating from each vertex receive
different labels.
A \emph{cluster pattern}  is an assignment
of a labeled seed $\Sigma_t=(\txx_t, Q_t)$
to every vertex $t \in \TT_n$, such that the seeds assigned to the
endpoints of any edge $t \overunder{k}{} t'$ are obtained from each
other by the seed mutation in direction~$k$.
The components of
$\txx_t$ are written as $\txx_t = (x_{1;t}\,,\dots,x_{s;t}).$
\end{definition}

Clearly, a cluster pattern is uniquely determined
by an arbitrary  seed.

\begin{definition}
[\emph{Cluster algebra}]
\label{def:cluster-algebra0}
Given a cluster pattern, we denote
\begin{equation}
\label{eq:cluster-variables0}
\Xcal
= \bigcup_{t \in \TT_r} \txx_t
= \{ x_{i,t}\,:\, t \in \TT_r\,,\ 1\leq i\leq r \} \ ,
\end{equation}
the union of clusters of all the seeds in the pattern.
The elements $x_{i,t}\in \Xcal$ are called \emph{cluster variables}.
Let $\CC[c^{\pm 1}]$ be the \emph{ground ring} consisting
of Laurent polynomials in the frozen variables. The
\emph{cluster algebra} $\Acal$ associated with a
given pattern is the $\CC[c^{\pm 1}]$-subalgebra of the ambient field $\Fcal$
generated by all cluster variables, 
with coefficients which are Laurent polynomials
in the frozen variables: $\Acal = \CC[c^{\pm 1}] [\Xcal]$.
We denote $\Acal = \Acal(\txx,  Q)$, where
$(\txx,Q)$
is any seed in the underlying cluster pattern.
We say that $\Acal$ has \emph{rank $r$} because each cluster contains
$r$ cluster variables. Cluster (or frozen) variables that belong to a common cluster are said to be \emph{compatible}.
\end{definition}

\begin{remark}\label{rmk:dif-ground-ring}
Another common convention is to choose the ring 
	$\CC[c]$ of polynomials in the frozen variables as the ground ring,
and define the cluster algebra 
as $\overline{\Acal} := \CC[c] [\Xcal]$.
\end{remark}

Some structural results on cluster algebras will be helpful for us.
\begin{theorem}[{\cite[Corollary 3.6]{CKLP}}] \label{thm:GSV}
Let $\Acal$ be a cluster algebra coming from a quiver. Then every seed is uniquely determined by its 
cluster.
\end{theorem}

\begin{definition}
	For a ring $\Rcal$ with $1$, let $\Rcal^{\times}$ be the set of invertible
	elements in $\Rcal$. Recall that a ring without zero divisors is called
	an \emph{integral domain}.  A non-invertible element $a$ in an integral 
	domain $\Rcal$ is \emph{irreducible} if it cannot be written
	as a product $a=bc$ with $b,c\in \Rcal$ both non-invertible.  
\end{definition}
Note that every cluster algebra is an integral domain, since it is by definition
a subring of a field.

\begin{theorem}[{\cite[Theorem 1.3]{GLS}}]\label{thm:GLS}
	Consider a cluster algebra $\Acal = \Acal(\txx,Q)$, with 
	ground ring the Laurent polynomials $\CC[x_{r+1}^{\pm 1}, \dots, x_s^{\pm 1}]$ in the 
	frozen variables.
	\begin{itemize}
		\item We have $\Acal(\txx,Q)^{\times} = 
		\{\lambda x_{r+1}^{a_{r+1}} \dots x_s^{a_s} \ \vert \ 
		\lambda \in \CC^{\times}, a_i \in \Z\}.$
		\item Every cluster variable in a cluster algebra is irreducible.
	\end{itemize}
\end{theorem}

\subsection{The cluster structure on the Grassmannian}\label{sec:Grass}

The Grassmannian
has a cluster structure \cite{scott},
which we recall here, following the exposition of \cite{FW6}.
We also discuss several operations on the Grassmannian which are compatible
with the cluster structure.

\begin{definition}\label{def:frozen}
Given a $k$-element subset $J=\{j_1,j_2, \dots, j_k\}\subset\{1,2,\dots, n\}$
and a positive integer $i$,
we define
\[
(J+i)\bmod n: = \{j_1+i, j_2+i, \dots, j_k+i\},
\]
where the sums are taken modulo~$n$.
We often write $J+i$ if the $n$ is clear
from context.

The set of frozen variables for the cluster structure on the Grassmannian
consists of the $n$ \emph{cyclically consecutive} 
Pl\"ucker coordinates 
	\begin{equation} \label{eq:Grfacets}
	    f_1:=\lr{[k]} \text{ and } f_i:=\lr{[k]+(i-1) \bmod n} \text{ for }2\leq i \leq n.
     \end{equation}
\end{definition}

For example, the frozen variables  for $\Gr_{4,7}$ are the Pl\"ucker coordinates
\[
\lr{1234}, \lr{2345}, \lr{3456}, \lr{4567}, \lr{1567}, \lr{1267}, \lr{1237}.
\]

A particularly nice seed for the Grassmannian cluster structure is 
the 
 \emph{rectangles seed} $\Sigma_{k,n}$.

\begin{definition}[Rectangles seed $\Sigma_{k,n}$] \label{def:rectangles}
We construct a quiver $Q_{k,n}$ whose vertices are labeled by the
rectangles contained in an $k \times (n-k)$ rectangle~$R$,
including the empty rectangle~$\varnothing$.
The frozen vertices of $Q_{k,n}$ are labeled by
the rectangles of sizes $k\times j$ (with $1\le j\le n-k$),
rectangles of sizes $i\times (n-k)$ (with $1\le i\le k$), and the empty rectangle.
The arrows from an $i\times j$ rectangle go to rectangles of sizes $i\times (j+1)$,
$(i+1)\times j$, and $(i-1)\times (j-1)$ (assuming those rectangles have nonzero
dimensions, fit inside~$R$, and the arrow does not connect two frozen vertices).
There is also an arrow from the frozen vertex labeled by~$\varnothing$
to the vertex labeled by the $1\times 1$ rectangle.
See  \Cref{fig:G37-Le-quiver}, left.

We map each rectangle $r$ contained in the $k \times (n-k)$ rectangle
$R$ to an $k$-element subset of
$\{1,2,\dots, n\}$ (representing a Pl\"ucker coordinate), as follows.
We justify $r$
so that its upper left corner coincides with the upper left corner
of $R$.  There is a path of length $n$ from the northeast corner of
$R$ to the southwest corner of $R$ which cuts out the
smaller rectangle $r$;  we label the steps of this path from $1$ to~$n$.
We then map $r$ to the set of labels $J(r)$ of the vertical steps on this path.
This construction allows us to assign to each vertex of the quiver $Q_{k,n}$
a particular Pl\"ucker coordinate.
We set
\begin{equation*}
\txx^{k,n} =
	\{\lr{J(r)} \ \vert \ \text{ $r$ is a rectangle contained in an
$k \times (n-k)$ rectangle}\} ,
\end{equation*}
and then define the \emph{rectangles seed} $\Sigma_{k,n}=(\txx^{k,n}, Q_{k,n})$.
See Figure~\ref{fig:G37-Le-quiver}.
\end{definition}

\begin{figure}[t] 
\vspace{1em}
\begin{center}
\setlength{\unitlength}{2pt}
\hspace{-1.5cm}
\begin{picture}(60,75)(0,-21)
\put(20,20){\makebox(0,0){${\ydiagram{1,1}}$}}
\put(20,40){\makebox(0,0){${\ydiagram{1}}$}}
\put(40,20){\makebox(0,0){${\ydiagram{2,2}}$}}
\put(40,40){\makebox(0,0){${\ydiagram{2}}$}}
\put(60,20){\makebox(0,0){$\boxed{\ydiagram{3,3}}$}}
\put(60,40){\makebox(0,0){$\boxed{\ydiagram{3}}$}}
\put(20,-1){\makebox(0,0){$\ydiagram{1,1,1}$}}
\put(40,-1){\makebox(0,0){$\ydiagram{2,2,2}$}}
\put(60,-1){\makebox(0,0){$\boxed{\ydiagram{3,3,3}}$}}
\put(20, -21){\makebox(0,0){$\boxed{\ydiagram{1,1,1,1}}$}}
\put(40, -21){\makebox(0,0){$\boxed{\ydiagram{2,2,2,2}}$}}
\put(60, -21){\makebox(0,0){$\boxed{\ydiagram{3,3,3,3}}$}}
\put(5,55){\makebox(0,0){$\boxed{\scriptstyle\varnothing}$}}

\put(20,16){\vector(0,-1){12}}
\put(40,36){\vector(0,-1){12}}

\put(20,36){\vector(0,-1){12}}

\put(26,20){\vector(1,0){8}}
\put(26,-1){\vector(1,0){8}}
\put(46,-1){\vector(1,0){8}}
\put(46,40){\vector(1,0){8}}

\put(26,40){\vector(1,0){8}}

\put(46,20){\vector(1,0){8}}

\put(10,50){\vector(1,-1){7}}

\put(40,16){\vector(0,-1){12}}
\put(20,-6){\vector(0,-1){7}}
\put(40,-6){\vector(0,-1){7}}

\put(36, 4){\vector(-1,1){12}}
\put(54, 6){\vector(-1,1){10}}
\put(36, 24){\vector(-1,1){12}}
\put(54, 26){\vector(-1,1){10}}
\put(54, -14){\vector(-1,1){10}}
\put(34, -14){\vector(-1,1){10}}
\end{picture}
\hspace{1cm}
\begin{picture}(60,75)(0,-21)
	\put(20,20){\makebox(0,0){$3467$}}
	\put(20,40){\makebox(0,0){$3567$}}
	\put(40,20){\makebox(0,0){${2367}$}}
	\put(40,40){\makebox(0,0){${2567}$}}
	\put(60,20){\makebox(0,0){$\boxed{1267}$}}
	\put(60,40){\makebox(0,0){$\boxed{1567}$}}
	\put(20,-1){\makebox(0,0){$3457$}}
	\put(40,-1){\makebox(0,0){$2347$}}
	\put(60,-1){\makebox(0,0){$\boxed{1237}$}}
	\put(20, -21){\makebox(0,0){$\boxed{3456}$}}
	\put(40, -21){\makebox(0,0){$\boxed{2345}$}}
	\put(60, -21){\makebox(0,0){$\boxed{1234}$}}

	\put(5,55){\makebox(0,0){$\boxed{4567}$}}

	\put(20,16){\vector(0,-1){12}}
	\put(40,36){\vector(0,-1){12}}
	
	\put(20,36){\vector(0,-1){12}}

	\put(26,20){\vector(1,0){8}}
	\put(26,-1){\vector(1,0){8}}
	\put(46,-1){\vector(1,0){7}}
	\put(46,40){\vector(1,0){7}}
	
	\put(26,40){\vector(1,0){8}}
	
	\put(46,20){\vector(1,0){7}}
	
	\put(10,50){\vector(1,-1){7}}
	
	\put(40,16){\vector(0,-1){12}}
	\put(20,-6){\vector(0,-1){7}}
	\put(40,-6){\vector(0,-1){7}}
	
	\put(36, 4){\vector(-1,1){12}}
	\put(54, 6){\vector(-1,1){10}}
	
	\put(36, 24){\vector(-1,1){12}}
	\put(54, 26){\vector(-1,1){10}}
	\put(54, -14){\vector(-1,1){10}}
	\put(34, -14){\vector(-1,1){10}}
\end{picture}
\hspace{1cm}
\begin{picture}(60,75)(0,-21)
	\put(20,20){\makebox(0,0){$1457$}}
	\put(20,40){\makebox(0,0){$1467$}}
	\put(40,20){\makebox(0,0){${1347}$}}
	\put(40,40){\makebox(0,0){${1367}$}}
	\put(60,20){\makebox(0,0){$\boxed{1237}$}}
	\put(60,40){\makebox(0,0){$\boxed{1267}$}}
	\put(20,-1){\makebox(0,0){$1456$}}
	\put(40,-1){\makebox(0,0){$1345$}}
	\put(60,-1){\makebox(0,0){$\boxed{1234}$}}
	\put(20, -21){\makebox(0,0){$\boxed{4567}$}}
	\put(40, -21){\makebox(0,0){$\boxed{3456}$}}
	\put(60, -21){\makebox(0,0){$\boxed{2345}$}}

	\put(5,55){\makebox(0,0){$\boxed{1567}$}}

	\put(20,16){\vector(0,-1){12}}
	\put(40,36){\vector(0,-1){12}}
	
	\put(20,36){\vector(0,-1){12}}

	\put(26,20){\vector(1,0){8}}
	\put(26,-1){\vector(1,0){8}}
	\put(46,-1){\vector(1,0){7}}
	\put(46,40){\vector(1,0){7}}
	
	\put(26,40){\vector(1,0){8}}
	
	\put(46,20){\vector(1,0){7}}
	
	\put(10,50){\vector(1,-1){7}}
	
	\put(40,16){\vector(0,-1){12}}
	\put(20,-6){\vector(0,-1){7}}
	\put(40,-6){\vector(0,-1){7}}
	
	\put(36, 4){\vector(-1,1){12}}
	\put(54, 6){\vector(-1,1){10}}
	
	\put(36, 24){\vector(-1,1){12}}
	\put(54, 26){\vector(-1,1){10}}
	\put(54, -14){\vector(-1,1){10}}
	\put(34, -14){\vector(-1,1){10}}
\end{picture}

\end{center}
\vspace{1em}
\caption{
Left: the quiver $Q_{4,7}$ with vertices  labeled by rectangles contained in a $4 \times 3$ rectangle. 
Middle: the rectangles seed $\Sigma_{4,7}$, where we identify $4$-element subsets of $[7]$ with Pl\"ucker coordinates. Frozen variables are boxed. Right: the cyclically shifted  rectangles seed $\Sigma_{4,7}^1$. 
The following mutation sequence maps us from the seed in the middle, to the seed at the right:
mutate at $2567, 3567, 2367, 3467, 2347, 3457$.
}
\label{fig:G37-Le-quiver}
\end{figure}
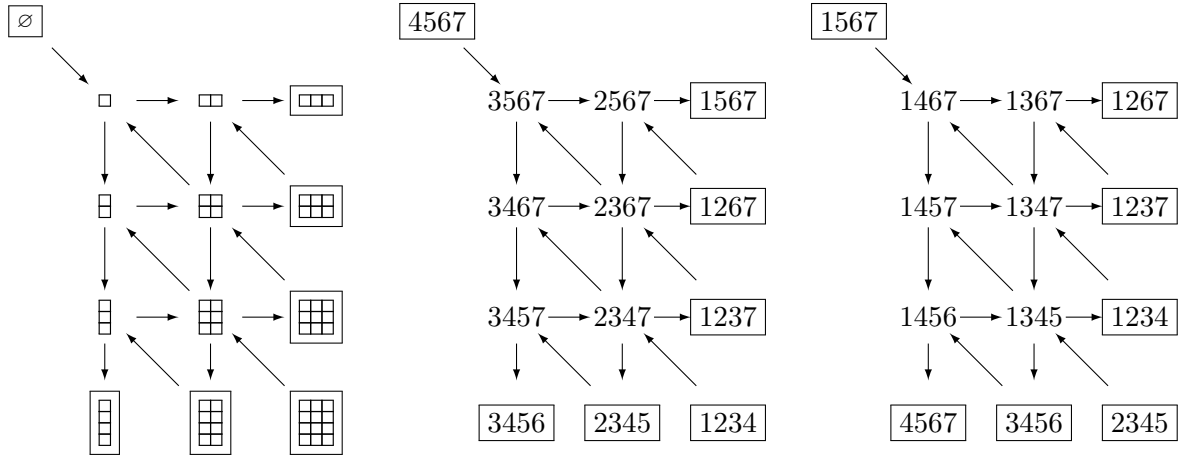

The rectangles seed gives rise to a cluster
structure on the (complex) Grassmannian.

\begin{theorem}[\cite{scott}]\label{thm:scott}
The coordinate ring $\C[\widehat{\Gr}_{k,n}]$ of the affine cone over
the Grassmannian equals the cluster algebra $\overline{\Acal}(\Sigma_{k,n})$ (see \cref{rmk:dif-ground-ring}). 
Alternatively,
if we let 
${\Gr}_{k,n}^\circ$ denote the open subset of the Grassmannian where the frozen variables
don't vanish, then 
the coordinate ring $\C[\widehat{\Gr}_{k,n}^\circ]$ is the cluster algebra
 $\A(\Sigma_{k,n})$.
\end{theorem}

Whenever we refer to ``the cluster structure for" or ``a seed for" $\widehat{\Gr}_{k,n}, \widehat{\Gr}_{k,n}^{\circ}$ or their coordinate rings, we are refering to the cluster algebra, respectively a seed in the cluster algebra, in \cref{thm:scott}. Note in this context, we are discussing the complex Grassmannian.

\subsection{Operations compatible with the cluster structure on the Grassmannian}\label{sec:operation}

\begin{definition}
We define the
\emph{cyclically shifted rectangles seed} $\Sigma_{k,n}^i$
by replacing each vertex
labeled
$J$ in $Q_{k,n}$ by $(J+i)\bmod n$.  See e.g. the right of \cref{fig:G37-Le-quiver}.
\end{definition}

\begin{lemma}[\protect{\cite[Exercise 6.7.7]{FW6}}]\label{lem:cyclic}
The seeds $\Sigma_{k,n}^i$ (for all $i$) are mutation equivalent.
More specifically, 
let 
	$\pmb{\cyc} = 
	\pmb{\cyc_{k,n}}$ 
	be the mutation sequence in which we 
start from the seed $\Sigma_{k,n}$ and mutate at each of the mutable variables
of $Q_{k,n}$ exactly once, in the following order:
mutate each row from right to left, starting from the top row and ending 
at the bottom row.  
	Then $\pmb{\cyc}(\Sigma_{k,n}) = 
	\Sigma_{k,n}^{1}$. 
Furthermore, let 
	$\pmb{\cyc^{-1}} = 
	\pmb{\cyc^{-1}_{k,n}}$ 
	be the mutation sequence in which we 
start from the seed $\Sigma_{k,n}$ and mutate at each of the mutable variables
of $Q_{k,n}$ exactly once, in the following order:
mutate each row from left to right, starting from the bottom row and ending 
at the top row.  
	Then $\pmb{\cyc^{-1}}(\Sigma_{k,n}) = 
	\Sigma_{k,n}^{-1}$. 
\end{lemma}

\begin{definition}\label{def:corectangles}
We define the 
\emph{reflected rectangles seed} or the 
\emph{corectangles seed} $\Sigma_{k,n}^{\co}$
by replacing each vertex
labeled
$J=\{j_1 < \dots < j_k\}$ in $Q_{k,n}$ by 
$\{n-j_k+1,\dots ,n-j_2+1, n-j_1+1\},$
then reversing each arrow of the quiver.
We call this the ``corectangles'' (or complement of rectangles) seed because the Pl\"ucker coordinates
are exactly those labeling the vertical steps of the Young diagrams
obtained from a $k \times (n-k)$ rectangle by removing
a rectangle from the lower right.
\end{definition}

\begin{proposition}[{\cite[Theorem 11.17]{marshscott}}]\label{lem:reflect}
The seeds $\Sigma_{k,n}$ and $\Sigma_{k,n}^{\co}$ are 
mutation equivalent via the explicit mutation sequence in
	\cite[Definition 11.4]{marshscott}.
\end{proposition}
While we do not need it in this paper, 
we note that the mutation sequence above is a \emph{maximal green sequence} \cite{green} and passes
only through Pl\"ucker coordinates.

Recall the maps $\cyc$ and $\refl$ from \cref{def:dihedral}, it is easy to see that $\refl$ is an involution and $\cyc^{-1}=\cyc^{n-1}$. Their pullbacks are automorphisms of $\CC[\widehat{\Gr}_{k,n}]$ which interact nicely with the cluster structure. 
From \cref{lem:cyclic} and \cref{lem:reflect}, we have the following.

\begin{corollary}\label{cor:cycrefl}
        If $(\txx, Q)$ is a seed for $\Gr_{k,n}$, then so is $(\txx \circ \cyc, Q)$ and $(\txx \circ \refl, Q^{\op})$, where $Q^{\op}$ is $Q$ with every arrow reversed. Thus the maps $\cyc^*, \refl^*: \CC[\widehat{\Gr}_{k,n}] \to \CC[\widehat{\Gr}_{k,n}]$ take cluster variables to cluster variables and preserve compatibility and exchange relations.
\end{corollary}

In \cref{sec:quasihom},
we will need to work with cluster structures on Grassmannians
$\Gr_{4,J}$ of $4$-planes in a vector space with basis indexed by a set $J\subset \Z$, as opposed to $[n]$.
We will therefore need the following generalization of
 \cref{def:frozen}.
\begin{definition}\label{def:frozen2}
Consider the Grassmannian $\Gr_{4,J}$, where 
$J=\{j_1< \dots < j_{\ell}\} \subset [n]$.
We let $f_{j_i}^J$ denote the Pl\"ucker coordinate associated to the 
subset 
$\{j_i,j_{i+1}, j_{i+2}, j_{i+3}\}$,
where the addition in the subscripts
is modulo $\ell$.
If we are discussing 
the cluster algebra structure on $\Gr_{4,n}$, we will refer to the 
frozen variables as either
$\{f_1^{[n]},\dots, f_n^{[n]}\}$, or $\{f_1,\dots,f_n\}$.
\end{definition}

\cref{lem:add-marker-embed-gr} gives an inclusion of Grassmannian cluster structures
that will be useful in \cref{sec:bcfw-and-products}. 
 \begin{lemma}\label{lem:add-marker-embed-gr}
         Let $J \subset [n]$. Choose $j \in J$ and set $I := J \setminus \{j\}$. Consider the natural inclusion $\iota: \C[\Gr_{k, I}] \to \C[\Gr_{k, J}]$ given by $\lr{A} \mapsto \lr{A}$. Then $\iota$ maps cluster variables to cluster variables and preserves compatibility and exchange relations.
        \end{lemma}

        \begin{proof}
        In an appropriate cyclic rotation of the rectangles seed for $\Gr_{k, J}$, the Pl\"ucker coordinates containing $j$ are exactly the frozen variables in the rightmost column. The cluster variables in the second-from-rightmost column are $f_{j-k+1}^I, f_{j-k+2}^I,\dots, f_{j-1}^I, f_{j+1}^I$. Deleting the cluster variables in the rightmost column and freezing the cluster variables in the second-from-rightmost column gives a cyclic rotation of the rectangles seed for $\Gr_{k, I}$. This implies that for every seed $(\txx, Q)$ for $\Gr_{k, I}$, there is a seed $(\txx', Q')$ for $\Gr_{k,J}$ where $\txx \subset \txx'$ and $Q$ is the induced subquiver of $Q'$ using the variables of $\txx$. This in turn implies the lemma.
        \end{proof}

\begin{remark}\label{rmk:cyc-refl-on-functionaries}
	In an abuse of notation, we use $\cyc^*, \refl^*$ to denote maps on twistor coordinates and (rational) functionaries on $\Gr_{k, k+4}$ given by identifying twistor coordinates $\llrr{Y Z_I}$ with Pl\"ucker coordinates $\lr{I}$ for $\Gr_{4, n}$ and using \eqref{eq:cyc-refl-on-plucker}. That is, $\cyc^*\llrr{Y Z_I} := \llrr{Y Z_{I-1}}$ and $\refl^*\llrr{Y Z_I} := \llrr{Y Z_{n+1-I}}$, where $I, I-1, n+1-I$ are written in increasing order. This is different from the maps $\llrr{Y Z_I} \mapsto \llrr{(\cyc Y) Z_I}$ and $\llrr{Y Z_I} \mapsto \llrr{(\refl Y) Z_I}$, which we will never use.
\end{remark}

\subsection{Quasi-homomorphisms of cluster algebras}
	
In this section we define quasi-homomorphisms of cluster
algebras, following \cite{Fraser} and \cite{Fraser2}.

	\begin{definition}[Exchange ratios] Given a cluster seed $\Sigma=((x_1,\dots,x_s), Q)$ 
		for a cluster algebra of rank $r\leq s$,
		and a mutable
		variable $x_i$ (so that $1 \leq i \leq r$), 
		the \emph{exchange ratio} of
		$x_i$ (with respect to $\Sigma$) is 
		\begin{equation}
			\hat{y}_{\Sigma}(x_i) = 
			\frac{\prod_{j: i\to j} x_j^{\# \arr(i\to j)}}
			{\prod_{j: j\to i} x_j^{\# \arr(j\to i)}},
		\end{equation}
		where $\arr(i\to j)$ denotes the number of arrows from 
		$i$ to $j$ in the quiver $Q$.
	\end{definition}

Given a cluster algebra $\A$, we let $\PP$ denote its frozen group, that is the group of Laurent monomials
in the frozen variables.  For elements $x,y\in \A$, we say that 
\emph{$x$ is proportional to $y$}, writing $x \propto y$, if $x=My$ for some Laurent monomial
$M\in \PP$. We then refer to $M$ as a \emph{frozen factor}.

\begin{definition}[Cluster quasi-homomorphism]\label{def:quasi}\cite[Definition 3.1 and Proposition 3.2]{Fraser}
Let $\A$ and $\overline{\A}$ be two cluster algebras of the same rank $r$,
and with respective frozen groups $\PP$ and $\overline{\PP}$.
	Then an algebra homomorphism $f:\A \to \overline{\A}$ that satisfies 
	$f(\PP)\subseteq \overline{\PP}$ is called a \emph{(cluster) quasi-homomorphism} from $\A$ to $\overline{\A}$
	if there are seeds
$\Sigma=((x_1,\dots,x_s),Q)$ and $\overline{\Sigma}=((\bar{x}_1,\dots,\bar{x}_{\bar{s}}), \bar{Q})$ 
	for $\A$ and $\overline{\A}$, such that 
	\begin{enumerate}
		\item $f(x_i) \propto \bar{x}_i$ for $1\leq i \leq r$
		\item $f(\hat{y}_{\Sigma}(x_i)) = \hat{y}_{\overline{\Sigma}}(\bar{x}_i)$ for $1\leq i \leq r$.
		\item the map $i \mapsto \bar{i}$ of mutable nodes in $Q$ and $\bar{Q}$ extends to 
                    an isomorphism of the corresponding induced subquivers. 
	\end{enumerate}
	If conditions (1), (2), and (3) hold, we say that 
	$\Sigma$ is \emph{similar} to $\overline{\Sigma}$, and 
	we write $f(\Sigma) \propto \overline{\Sigma}$.

Note that conditions (1) and (2) above imply condition (3), 
as they allow one to read off adjacencies of mutable nodes;
however, we choose to include condition (3) in the definition 
since it is readily 
	checkable and hence useful when looking $\Sigma$ and $\overline{\Sigma}$.

	\end{definition}

\begin{proposition}
\cite[Proposition 3.2]{Fraser} \label{prop:similar}
If a  seed
$\Sigma$ is similar to the seed $\overline{\Sigma}$, then 
each seed of the seed pattern containing $\Sigma$ is similar to the 
	corresponding seed of the seed pattern containing $\overline{\Sigma}$.
\end{proposition}

\section{Product promotion is a cluster quasihomomorphism}	\label{sec:quasihom}
		In this section we introduce the notion of \emph{product promotion}, which is 
a homomorphism from 
$\C(\widehat{\Gr}_{4,N_L})\times \C(\widehat{\Gr}_{4,N_R})$
to $\C(\widehat{\Gr}_{4,n})$,
where $N_L$ and $N_R$ are defined in 
\cref{not:LR_cluster}.
The  main result of this section is that product promotion is a 
cluster quasi-homomorphism. Therefore if we apply  product promotion to 
a cluster variable (resp. subset of a cluster),
we obtain a cluster variable (resp. subset of a cluster)  of 
$\Gr_{4,n}$, up to a Laurent monomial in 
the usual frozen variables plus three more Pl\"ucker coordinates.
We will later use product promotion to 
inductively give a description of the BCFW tiles as semi-algebraic sets in $\Gr_{k,k+4}$, see
\cref{thm:BCFW-tile-and-sign-description}. Product promotion is also a key tool in the proof of cluster adjacency (cf.~\cref{thm:clusteradjacency}).

\subsection{Product promotion}

Given an element $V$ of the Grassmannian $\Gr_{k,n}$ 
represented by a $k\times n$ matrix
with columns $v_1,\dots,v_n$, 
recall that we identify the Pl\"ucker coordinate 
$\lr{i_1,\dots,i_k}$ with $v_{i_1} \wedge \dots \wedge v_{i_k}$. In what follows, in an abuse of notation, we write $v_{i}$ as just $i$. 

\begin{notation}\label{not:LR_cluster}
Choose integers $1<a<b<c<d<n$ such that $a,b$ and $c,d,n$ are consecutive.
	Let $N_L = \{1,2,\dots,a,b,n\}$ and $N_R = \{b, \dots, c, d, n\}$.
We let  $\widehat{\Gr}^{\circ}_{4, N_L}$,
$\widehat{\Gr}^{\circ}_{4, N_R}$,
and $\widehat{\Gr}^{\circ}_{4,n}$ 
denote the (affine cones over the complex) Grassmannians in vector spaces with bases labeled by 
$N_L$, $N_R$ and 
	$\{1,2,\dots,n\}$, respectively,\footnote{
Note that e.g. we are overloading notation and letting $n$
	index a element of a vector space basis for three different
	vector spaces; however, in what follows, the meaning should
	be clear from context.}
	where we
	removed the locus where the frozen variables vanish.
\end{notation}

\begin{definition}[Product promotion]
\label{pro-twistors2}

Using \cref{not:LR_cluster}, \emph{product promotion} is the homomorphism 
	$$\Psi_{ac} = \Psi: \C(\widehat{\Gr}_{4,N_L})\times \C(\widehat{\Gr}_{4,N_R}) \to \C(\widehat{\Gr}_{4,n})$$ 
induced 
by the following substitution:
\begin{align}\label{eq:promotionvectors1}
&  \text{on $\widehat{\Gr}_{4,N_L}$:} && b \;\mapsto\; b - \frac{\lr{b\,c\,d\,n}}{\lr{a\,c\,d\,n}}\,a 
&&\;=\; \frac{(ba)\cap (cdn)}{\lr{a\,c\,d\,n}} \\
& \text{on $\widehat{\Gr}_{4,N_R}$:} && n \;\mapsto\; n - \frac{\lr{a\,b\,c\,n}}{\lr{a\,b\,c\,d}}\,d + \frac{\lr{a\,b\,d\,n}}{\lr{a\,b\,c\,d}}\,c  \;=\; \frac{\lr{a\,c\,d\,n}}{\lr{a\,b\,c\,d}}\,b - \frac{\lr{b\,c\,d\,n}}{\lr{a\,b\,c\,d}}\,a  &&\;=\; 
\frac{(ba)\cap (cdn)}{\lr{a\,b\,c\,d}} 
\label{eq:promotionvectors2}\\ 
& \text{on $\widehat{\Gr}_{4,N_R}$:} && d \;\mapsto\; d - \frac{\lr{a\,b\,d\,n}}{\lr{a\,b\,c\,n}}\,c 
&&\;=\; \frac{(dc)\cap (abn)}{\lr{a\,b\,c\,n}} \label{eq:promotionvectors3}
\end{align} 

\end{definition}

While $\Psi_{ac} = \Psi$ depends on the choice of integers $a$ and $c$,
we will usually drop the subscripts when it is clear from context.

\begin{remark}
Equivalently, product promotion is the 
	homomorphism which acts on 
$\CC(\widehat{\Gr}_{4,N_L})$ as follows. Given distinct $i,j,\ell \in \{1,2,\ldots,a-1\}$, we have
\begin{enumerate}[leftmargin=1.5cm,itemsep=0.5em,topsep=0.5em]
\item[(a1)]
	$\Psi(\lr{i\,j\,\ell\,b}) =
	\frac{ \lr{i\,j\,\ell \br b\,a \br c\,d\,n}}{\lr{a\,c\,d\,n}}$ 
\item[(a2)]	$\Psi(\lr{i\,j\,b\,n}) =
	\frac{ \lr{i\,j\,n \br a\,b \br c\,d\,n}}{\lr{a\,c\,d\,n}}$
 \end{enumerate}
and $\Psi$ acts as the identity on all other Pl\"ucker coordinates of $\widehat{\Gr}_{4,N_L}$. Product promotion acts on $\CC(\widehat{\Gr}_{4,N_R})$ as follows. 
Given distinct $i,j,\ell \in \{b,b+1,\ldots,c\}$, we have 
\begin{enumerate}[leftmargin=1.5cm,itemsep=0.5em,topsep=0.5em]
\item[(b1)]
		$\Psi(\lr{i\,j\,\ell\,d}) = 
		\frac{ \lr{i\,j\,\ell \br d\,c \br a\,b\,n}}{\lr{a\,b\,c\,n}} = 
		\frac{ \lr{c d \br i\,j\,\ell \br  a\,b\,n}}{\lr{a\,b\,c\,n}}$
\item[(b2)]
		$\Psi(\lr{i\,j\,d\,n}) = 
		\frac{ \lr{i\,j\,n \br c\,d \br a\,b\,n}}{\lr{a\,b\,c\,n}} = 
		\frac{\lr{n\,a\,b\br i\,j\br c\,d\,n}}{\lr{a\,b\,c\,n}}$
\item[(b3)]
		$\Psi(\lr{i\,j\,\ell\,n}) = 
	\frac{ \lr{i\,j\,\ell \br b\,a \br c\,d\,n}}{\lr{a\,b\,c\,d}} = 
	\frac{ \lr{a b \br i\,j\,\ell \br c\,d\,n}}{\lr{a\,b\,c\,d}}$
\end{enumerate}
and $\Psi$ acts as the identity on all other Pl\"ucker coordinates of $\widehat{\Gr}_{4,N_R}$.
Note that since Pl\"ucker coordinates are 
antisymmetric 
one can, e.g., regard Case (a2) as a special case of (a1).  
	However, we prefer to  keep Cases (a1) and (a2) separate for the reader's convenience.
\end{remark}

\begin{definition}[Upper promotion]\label{def:upper-promotion}
In many cases, we extend Definition~\ref{pro-twistors2}
	and allow $a=1$, so that $(a,b,c,d)=(1,2,n-2,n-1)$. Then $\widehat{\Gr}^{\circ}_{4, N_L}=\widehat{\Gr}^{\circ}_{4, \{1 2 n\}}$ is empty and $\C(\widehat{\Gr}_{4,N_L})$ is trivial.  The homomorphism,
	which we call \emph{upper promotion},
	reduces to $\Psi:\C(\widehat{\Gr}_{4,N_R}) \to \C(\widehat{\Gr}_{4,n})$, following
	rules (b1), (b2), (b3).
        \end{definition}

\begin{remark}
Some numerators obtained in promotion
may factor,
	see e.g. \cref{rem:factors}.
\end{remark}

\begin{remark}[Relation to Physics] \label{rk:yangianpromotion}
Product promotion is related to an operation on the building blocks of scattering amplitudes in the physics literature. Using BCFW bridges, \cite[Section 2.5]{Arkani-Hamed:2010zjl} illustrates a procedure to build a \emph{Yangian invariant} $Y_L \otimes_{BCFW} Y_R$ from products of two Yangian invariants $Y_L$ and $Y_R$. Yangian invariants depend on \emph{momentum twistors}. 
If we have $n$ particle scattering, the corresponding collection
of momentum twistors gives a point in ${\Gr}_{4,n}$, so Yangian invariants can be interpreted as functions of Pl\"ucker coordinates for $\Gr_{4,n}$. 
To connect with \cite{Arkani-Hamed:2010zjl}, we identify their indices $(j,j+1,n-1,n,1)$ for momentum twistors with our indices $(a,b,c,d,n)$ in \cref{not:LR_cluster}. Then the `left' Yangian invariant $Y_L$ depends on the momentum twistors with indices in $1,\ldots,j,I$, which can be identified with the indices $n,1,\ldots,a,b$ of $\widehat{\Gr}_{4,N_L}$. The `right' Yangian invariant $Y_R$ depends on the momentum twistors with indices $I,j+1,\ldots,n-1,n$, which can be identified with the indices $n,b,\ldots,c,d$ of $\widehat{\Gr}_{4,N_R}$. To produce $Y_L \otimes_{BCFW} Y_R$, one makes the following substitutions in $Y_L$ and $Y_R$:
\begin{align} \label{eq:promotionvectorsphy1}
	I \;\mapsto\; (j \, j+1) \cap (n-1 \, n \, 1), &\text{ or in our notation, } b \;\mapsto\; (a \, b)\cap (c \, d \, n) \text{ on }\widehat{\Gr}_{4,N_L}\\
	I \;\mapsto\; (j \, j+1) \cap (n-1 \, n \, 1), &\text{ or in our notation, } n \;\mapsto\; (a \, b)\cap (c \, d \, n) \text{ on }\widehat{\Gr}_{4,N_R} \label{eq:promotionvectorsphy2}\\
	n \;\mapsto\; \hat{n}:= (n-1 \, n) \cap (j \, j+1 \, 1), &\text{ or in our notation, } d \;\mapsto\; (c \, d)\cap (a \, b \, n) \text{ on }\widehat{\Gr}_{4,N_R}, \label{eq:promotionvectorsphy3}
\end{align}
to obtain $\hat{Y}_L$ and $\hat{Y}_R$. Then using our indices
$$Y_L \otimes_{BCFW} Y_R:= [a,b,c,d,n]\, \hat{Y}_L \hat{Y}_R$$
where $[a,b,c,d,n]$ is the unique Yangian invariant (up to multiplication by a scalar) in the indices $a,b,c,d,n$ (it is often called `R-invariant').
As indicated in the right hand sides of \eqref{eq:promotionvectorsphy1}, \eqref{eq:promotionvectorsphy2}, \eqref{eq:promotionvectorsphy3}, these operations coincide with product promotion defined in equations \eqref{eq:promotionvectors1}, \eqref{eq:promotionvectors2}, \eqref{eq:promotionvectors3}, respectively, up to a sign and a factor of the inverse of a Pl\"ucker coordinate, which do not change $Y_L \otimes_{BCFW} Y_R$.
\end{remark}

\begin{remark}\label{rem:ELTpromotion}
``Upper promotion'' was defined in \cite[Section 4]{even2021amplituhedron} as a way to understand the evolution of certain functionaries under a recursion specific to standard BCFW cells, but was not phrased as a map of rational functions on Grassmannians. It is related to the upper promotion of \cref{def:upper-promotion} by clearing denominators and replacing twistor coordinates with Pl\"ucker coordinates (see \cref{sec:Pluckertwistor}). 
The ``lower promotion'' of \cite{even2021amplituhedron} is 
essentially a variant of upper promotion applied at different indices. Such variants of upper promotion are sufficient only in the case of the standard BCFW tiling, where they can be applied sequentially to produce the functionaries needed in~\cite{even2021amplituhedron}. In this work, we produce functionaries by applying general product promotion, which naturally follows the recursive product structure of general BCFW cells to be described in the next sections (see \cref{def:BCFW_cell}). 
\end{remark}

\subsection{Product promotion is a cluster quasihomomorphism}

In this section we state and prove our main theorem about product promotion (see \cref{thm:promotion2}).
In what follows, we use the notation for frozen variables
from \cref{def:frozen2}.

Since  
$\C[\widehat{\Gr}_{4,N_L}^{\circ}]$ and $\C[\widehat{\Gr}_{4,N_R}^{\circ}]$ are cluster algebras, the product 
$\C[\widehat{\Gr}_{4,N_L}^{\circ}] \times \C[\widehat{\Gr}_{4,N_R}^{\circ}]$ is also a cluster algebra, where each  seed 
is the disjoint union of a seed for $\C[\widehat{\Gr}_{4,N_L}^{\circ}]$ and 
for $\C[\widehat{\Gr}_{4,N_R}^{\circ}]$.

\begin{theorem}[Product promotion is a cluster quasihomomorphism]\label{thm:promotion2}
Using the notation of \cref{not:LR_cluster}, let $\Sigma_0=\Sigma_0^L \sqcup \Sigma_0^R$ be the seed for the  cluster structure on 
$\C[\widehat{\Gr}_{4,N_L}^{\circ}] \times \C[\widehat{\Gr}_{4,N_R}^{\circ}]$ shown 
in
 \Cref{fig:seed_sigma_zero},
and let
	$\Fr(\Sigma_1)=\Fr(\Sigma_1^{a,c})$ be the seed for the  cluster structure on 
$\C[\widehat{\Gr}_{4,n}^{\circ}]$ shown in
 \Cref{fig:seed_sigma_one},
where we have additionally
frozen the cluster variables
$$\mathcal{T}:=\{ \lr{1\,2\,n \br a\,b \br c\,d\,n},\lr{a-1\,a\,b\,n},
\lr{1\,a\,b\,n},
\lr{a\,b\,b+1\,n},\lr{a\,b\,d\,n},
\lr{a\,b\,c\,n},
\lr{a\,b\,c\,d},
\lr{b\,c\,d\,n}, 
\lr{a\,c\,d\,n}\}.$$

Then we have that:
\begin{enumerate}
\item 	Product promotion 
$\Psi: \C(\widehat{\Gr}_{4,N_L})\times \C(\widehat{\Gr}_{4,N_R}) \to 
\C(\widehat{\Gr}_{4,n})$ 
is a quasi-homomorphism of cluster
algebras from 
$\mathcal{A}(\Sigma_0)$ to 
$\mathcal{A}(\Fr(\Sigma_1))$.

\item If we apply product promotion to a cluster of 	
$\mathcal{A}(\Sigma_0)$, we obtain a cluster of 
$\mathcal{A}(\Fr(\Sigma_1))$, up to Laurent
monomials in the elements of 
$$\mathcal{T'}:=\{ 
\lr{a\,b\,c\,n},
\lr{a\,b\,c\,d},
\lr{b\,c\,d\,n}, 
\lr{a\,c\,d\,n}\} \subset \mathcal{T}.$$
\item If we apply product promotion to a frozen variable of $\A(\Sigma_0)$, we obtain a frozen variable of $\A(\Fr(\Sigma_1))$ (which is a cluster or frozen variable for $\Gr_{4,n}$) times a Laurent monomial in $\mathcal{T'}$.
\end{enumerate}
\end{theorem}

We note that there are some special cases of
\cref{thm:promotion2}, namely
 upper promotion ($a=1$), $a=2$, $a=3$, and $c=b+1$,
 which require some clarification, 
see \cref{sec:degenerate}.

Since 
by \cref{thm:promotion2}, $\Psi$ maps
cluster variables to cluster variables times a frozen factor, it is convenient to define the following map.
\begin{definition}[Rescaled product promotion]\label{def:rPsi} Using the notation of \cref{not:LR_cluster},
	let $x$ be a cluster or frozen variable of $\C[\widehat{\Gr}_{4,N_L}^{\circ}]$ or $\C[\widehat{\Gr}_{4,N_R}^{\circ}]$. We let $\rPsi(x)=\rPsi_{ac}(x)$ denote the cluster or frozen variable of $\Gr_{4,n}$ obtained from $\Psi_{ac}(x)$ by removing the Laurent monomial in $\mathcal{T'}$ (cf.~\Cref{thm:promotion2} (2) and (3)). (If $x= \lr{bcdn}$, then $\Psi(\lr{bcdn})= \lr{bcdn}$ is a cluster variable for $\C[\widehat{\Gr}^{\circ}_{4,n}]$ and the relevant Laurent monomial in $\mathcal{T'}$ is equal to 1, so $\rPsi(\lr{bcdn})= \lr{bcdn}$.) We call $\rPsi(x)$ the \emph{rescaled product promotion} of $x$.
\end{definition}

\begin{example}
   Consider the cluster variables $\{x\}$ in the red box of \cref{fig:seed_sigma_zero}. If we apply product promotion $\Psi_{ac}$ to $\{x\}$, we obtain the cluster variables $\{\Psi_{ac}(x)\}$ in the red box of \cref{fig:seed_sigma_zero_promoted}. If we instead apply the rescaled product promotion $\rPsi_{ac}$ to $\{x\}$, we obtain the cluster variables $\{\rPsi_{ac}(x)\}$ in the red box of \cref{fig:seed_sigma_one}. Note that $\{\rPsi_{ac}(x)\}$ are obtained from $\{\Psi_{ac}(x)\}$ by removing the factors $r_n, r_d$, which are Laurent monomials in $\lr{b \, c\, d\, n}, \lr{a \, b\, c\, n}, \lr{a \, b\, c\, d}$. 
\end{example}

\begin{figure}[h]
\includegraphics[width=1.00\textwidth]{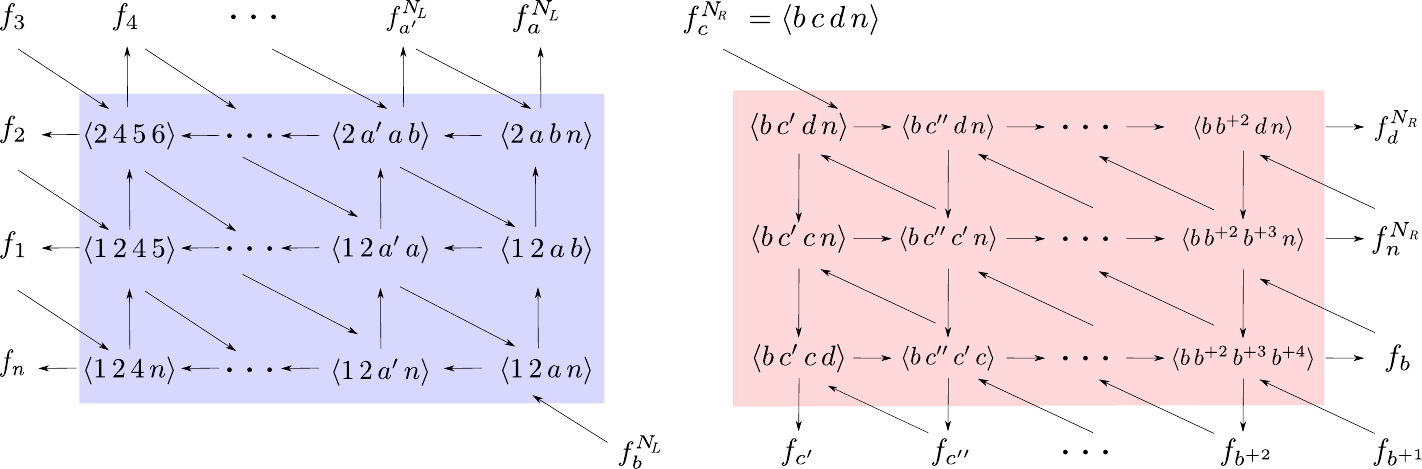}
	\caption{The initial seed $\Sigma_0$ for $\C[\widehat{\Gr}_{4,N_L}]\times \C[\widehat{\Gr}_{4,N_R}]$, which is a disjoint union of a seed $\Sigma_0^{L}$ for $\C[\widehat{\Gr}_{4,N_L}]$ (left) and a seed $\Sigma_0^R$ for $\C[\widehat{\Gr}_{4,N_R}]$ (right). We use the notation $x':=x-1,x'':=x-2$ and $x^{+i}:= x+i$. Mutable variables are in the shaded regions, the others are all frozen.}
\label{fig:seed_sigma_zero}
\end{figure}

\begin{figure}[h]
\includegraphics[width=1.00\textwidth]{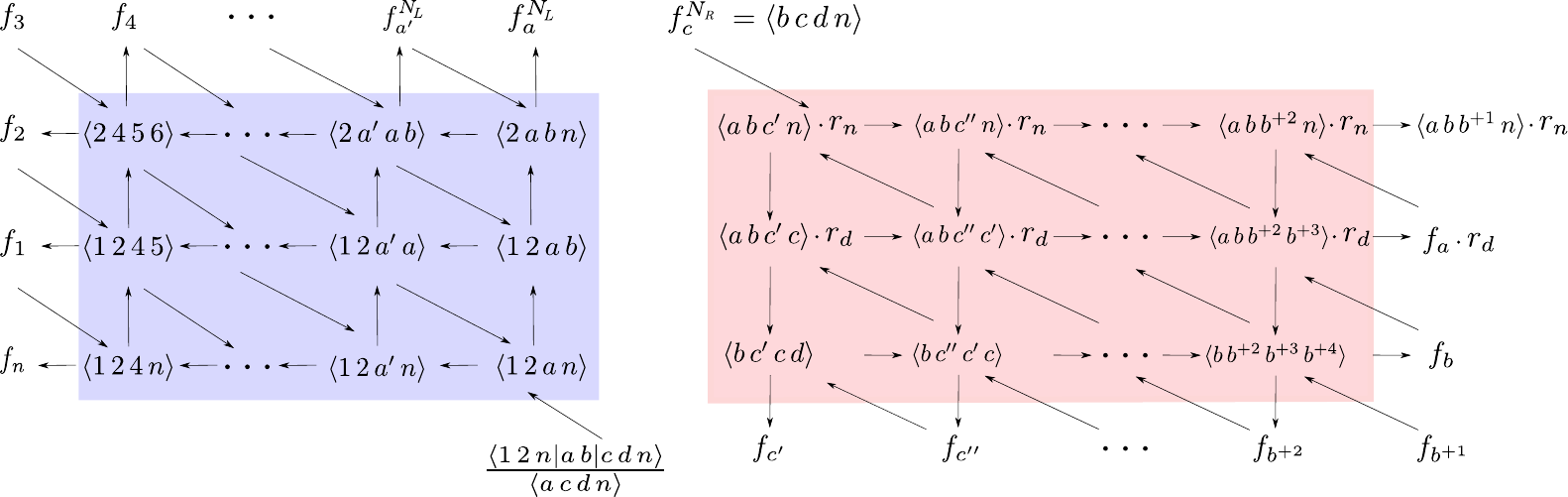}
	\caption{
	The result $\Psi(\Sigma_0)$ of applying the product promotion
		to $\Sigma_0$. 
	We use the notation $r_x:= \langle b \, c\, d\, n\rangle / \langle a \, b\, c\, x\rangle$ and $x':=x-1,x'':=x-2,x^{+i}:= x+i$.
 }
\label{fig:seed_sigma_zero_promoted}
\end{figure}

\begin{figure}[h]
\includegraphics[width=0.97\textwidth]{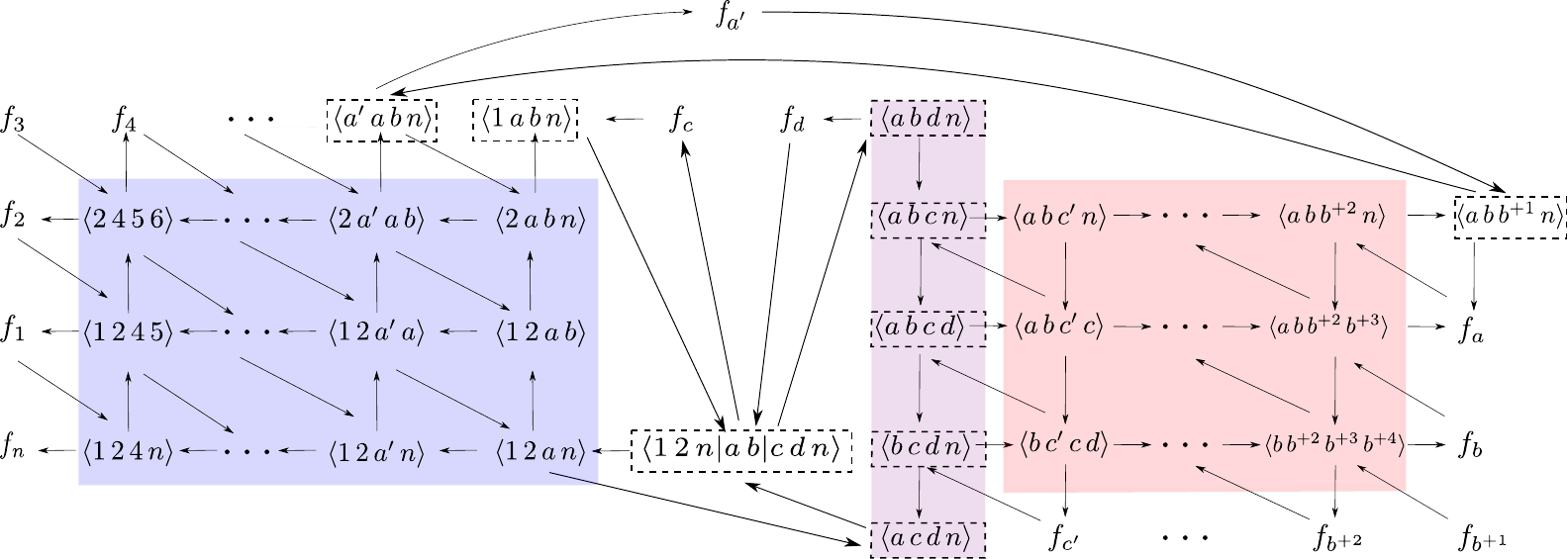}
	\caption{ 
		The
		seed $\Fr(\Sigma_1)=\Fr(\Sigma_1^{a,c})$ for $\C[\Gr_{4,n}]$. $\Fr(\Sigma_1)$ is obtained from $\Sigma_1$ by freezing the variables in the dashed boxes. We use the notation $x':=x-1$ and $x^{+i}:= x+i$.}
\label{fig:seed_sigma_one}
\end{figure}

\begin{proof}
	[Proof of \Cref{thm:promotion2}]
	Let $\Sigma_0$ denote the initial seed for $\C[\widehat{\Gr}^{\circ}_{4,N_L}]\times \C[\widehat{\Gr}^{\circ}_{4,N_R}]$ shown in
\Cref{fig:seed_sigma_zero}.  We note that $\Sigma_0$ is indeed a seed because 
each connected component is a cyclic shift of the {rectangles seed}
	for $\C[\widehat{\Gr}^{\circ}_{4,N_L}]$ and $\C[\widehat{\Gr}^{\circ}_{4,N_R}]$.
If we apply  product promotion to $\Sigma_0$, we obtain the ``promoted seed''
	$\Psi(\Sigma_0)$, shown in \Cref{fig:seed_sigma_zero_promoted}.

	Next we claim that the seed $\Sigma_1$ 
	shown in 
	\Cref{fig:seed_sigma_one} is a seed
	for $\C[\Gr_{4,n}^{\circ}]$.
	To prove this,  
	note that in $\C[\Gr_{4,n}^{\circ}]$, we have 
	\begin{equation}
		\lr{c\,d\,n \br a\,b \br n\,1\,2} 
	\lr{1\,a\,d\,n} = 
		\lr{1\,2\,d\,n} \lr{1\,a\,b\,n} 
        \lr{a\,c\,d\,n} 
		+
		\lr{1\,2\,a\,n} 
		\lr{1\,c\,d\,n}
		\lr{a\,b\,d\,n}. 
	\end{equation}
So to prove that 
$\Sigma_1$ is a seed for $\C[\Gr_{4,n}^{\circ}]$, it suffices to show
that 
the seed $\Sigma'_1$ 
obtained from $\Sigma_1$ by 
mutating 
	at $ \lr{c\,d\,n \br a\,b \br n\,1\,2}$ 
	is a seed for $\C[\Gr_{4,n}^{\circ}]$.

\begin{figure}[h]
\includegraphics[width=0.93\textwidth]{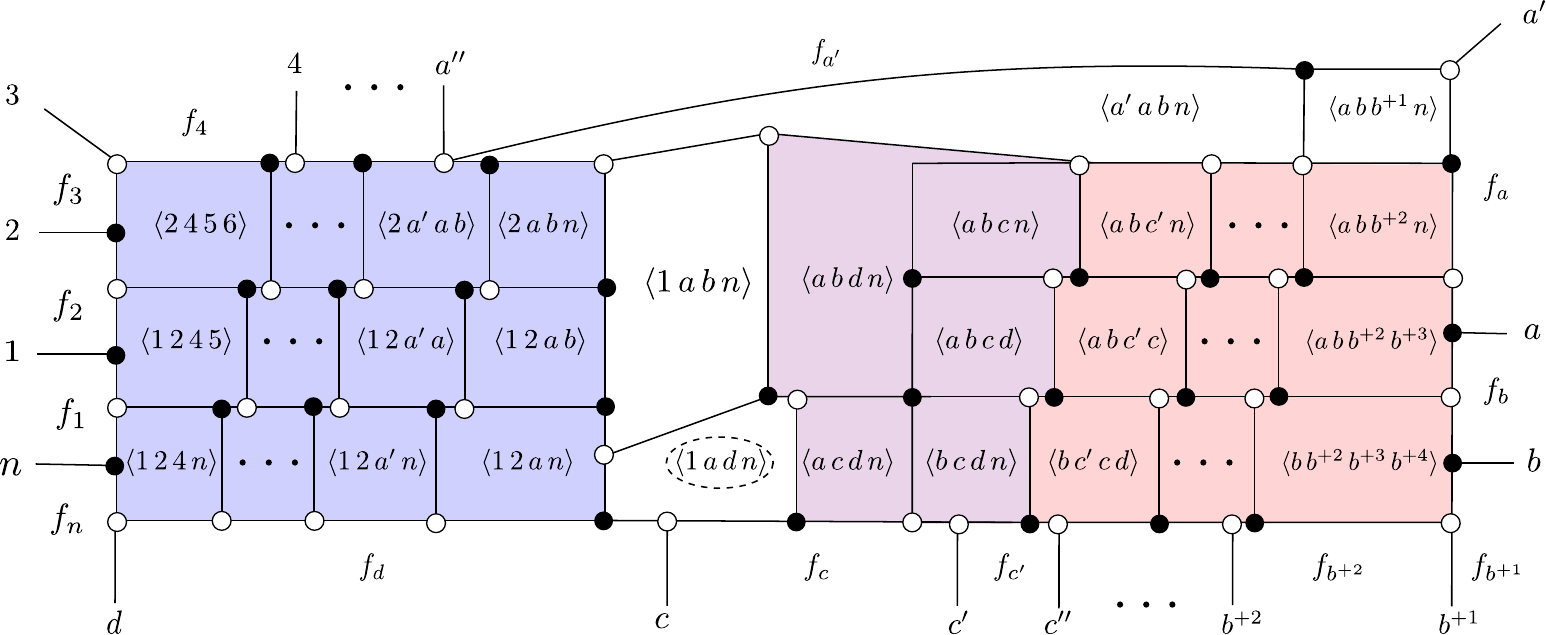}
	\caption{The plabic graph for $\Gr_{4,n}$ whose corresponding seed is $\Sigma_1'$ in the proof of \cref{thm:cluster}. We use the notation $x':=x-1,x'':=x-2$ and $x^{+i}:= x+i$.  Mutating at the face $\lr{1 \, a \, d\, n}$ (dashed circle) gives the seed $\Sigma_1$.}
\label{fig:promotion-plabic}
\end{figure}
	
	Since all cluster and frozen variables of $\Sigma'_1$ 
	are Pl\"ucker coordinates, we can verify that $\Sigma'_1$ 
	is a seed for $\C[\Gr_{4,n}^{\circ}]$
	by constructing a reduced plabic graph $G$ whose target-labeled seed is $\Sigma'_1$.
	Such a reduced plabic graph is shown in \cref{fig:promotion-plabic}.

	Now let $\Fr(\Sigma_1)$ be 	
	the seed obtained from $\Sigma_1$ 
	by freezing the cluster variables
	 $\lr{1\,2\,n \br a\,b \br c\,d\,n}$,
	$\lr{1\,a\,b\,n}$,
	$\lr{a-1\,a\,b\,n}$,
	$\lr{a\,b\,d\,n}$,
	$\lr{a\,b\,b+1\,n}$,
	$\lr{a\,b\,c\,n}$,
	$\lr{a\,b\,c\,d}$,
	$\lr{b\,c\,d\,n}$, and
	$\lr{a\,c\,d\,n}$ (see \Cref{fig:seed_sigma_one}).

 We will show that the 
product promotion, viewed as a map
$\Psi: 
\mathcal{A}(\Sigma_0) \to 
\mathcal{A}(\Fr(\Sigma_1))$,
is a quasi-homomorphism of cluster algebras.  
More specifically, we will show that the
     two seeds 
       $\Sigma_0$ and 
       $\Fr(\Sigma_1)$  
	satisfy the 
	conditions of \Cref{def:quasi}.   These two seeds are shown in \Cref{fig:seed_sigma_zero,fig:seed_sigma_one}.

Looking at 
\Cref{fig:seed_sigma_zero_promoted} and 
\Cref{fig:seed_sigma_one},
we see immediately that  conditions (1) and (3) of \cref{def:quasi} hold:
the images $\Psi(x_i)$ of mutable variables (shown in
 \Cref{fig:seed_sigma_zero_promoted}) 
differ from their counterparts 
(shown in \Cref{fig:seed_sigma_one}),
 only by the frozen variables
$\{ 
\lr{a\,b\,c\,n},
\lr{a\,b\,c\,d},
\lr{b\,c\,d\,n}\}.$
The induced subquivers obtained by restricting to the mutable cluster
variables  agree.
   
   Condition (2), that the corresponding exchange
       ratios agree, holds for the mutable variables in the left ``$N_L$'' connected component of  $\Psi(\Sigma_0)$. Indeed, this holds by inspection for $\lr{12an}$, and the neighborhoods of all other mutable variables are identical in the two seeds.

We now show condition (2) holds for the remaining exchange ratios. We first consider the leftmost column of mutable cluster variables 
from the ``$N_R$'' connected 
component of $\Psi (\Sigma_0)$ and the corresponding ones in $\Fr(\Sigma_1)$. We have:
 \begin{align*}
	 \hat{y}_{\Psi (\Sigma_0)} \left( \lr{a\, b \, c' \, n} \cdot r_n \right) &= \hat{y}_{\Fr(\Sigma_1)} \left(\lr{a\, b \, c' \, n} \right) = \frac{ \lr{a \, b \, c'' \, n} \lr{a \,b\,c'\,c}}{ \lr{a\,b\,c\,n} \lr{a\,b\,c''\,c'}}\\
	 \hat{y}_{\Psi (\Sigma_0)} \left(\lr{a\,b\,c'\,c} \cdot r_d \right) &= \hat{y}_{\Fr(\Sigma_1)} \left(\lr{a\,b\,c'\,c}\right) = \frac{\lr{a\,b\,c''\,c'}\lr{b\,c'\,c\,d}\lr{a \, b \, c \, n}}{\lr{b\,c''\,c'\,c}  \lr{a\,b\,c'\,n} \lr{a \, b \, c \, d}}\\ 
	 \hat{y}_{\Psi (\Sigma_0)} \left( \lr{b\,c'\,c\,d} \right) &= \hat{y}_{\Fr(\Sigma_1)} \left(\lr{b\,c'\,c\,d} \right) = \frac{\lr{b\,c''\,c'\,c}\lr{a \, b \, c \, d} f_{c'}}{\lr{a\,b\,c'\,c}  \lr{b \, c \, d \, n} f_{c''}}
 \end{align*}
 We now consider the exchange ratios for the other cluster variables. The cluster variables of the first and second rows (in the $N_R$ part) of $\Psi (\Sigma_0)$ differ from the cluster variables of the corresponding rows of $\Fr(\Sigma_1)$ just by common factors ($r_n$ and $r_d$ respectively). Moreover, for each cluster variable $x_i$ in $\Psi (\Sigma_0)$ not in its first column, the number of out-going arrows $i \rightarrow j$ such that $x_j$ is in row $r$ equals the number of in-going arrows $l \rightarrow i$ with $x_l$ from the same row $r$. Therefore, any common factor among the row $r$ cancels in $\hat{y}(x_i)$. Hence $\Psi (\Sigma_0)$ and $\Fr(\Sigma_1)$ have the same exchange ratios.

Therefore 
 we have a quasi-homomorphism of cluster
algebras $\Psi: 
\mathcal{A}(\Sigma_0) \to 
\mathcal{A}(\Fr(\Sigma_1))$, where 
 $\A(\Fr(\Sigma_1))$ 
is a subcluster algebra of $\C[\Gr_{4,n}^{\circ}]$, in the sense of 
\cite[Definition 4.2.6]{FWZ4}.  
But then by definition of quasi-homomorphism, 
if we promote a cluster variable (respectively, cluster)
of $\C[\widehat{\Gr}_{4,N_L}^{\circ}] \times \C[\widehat{\Gr}_{4,N_R}^{\circ}]$, we obtain a cluster variable (respectively, cluster) of 
$\A(\Fr(\Sigma_1))$, up to multiplication by Laurent monomials in 
the frozen variables of 
$\A(\Fr(\Sigma_1))$. 

The final claim we need to prove is that if we promote a cluster variable 
of $\C[\widehat{\Gr}_{4,N_L}^{\circ}] \times \C[\widehat{\Gr}^{\circ}_{4,N_R}]$, we obtain a cluster variable 
of $\C[\Gr_{4,n}^{\circ}]$, up to multiplication 
by a Laurent monomial in the 
set $\mathcal{T'}$. To show this, we consider 2 slight modifications of $\Psi(\Sigma_0)$ and $\Fr(\Sigma_1)$.

\begin{figure}
	\centering
	\includegraphics[width=\textwidth]{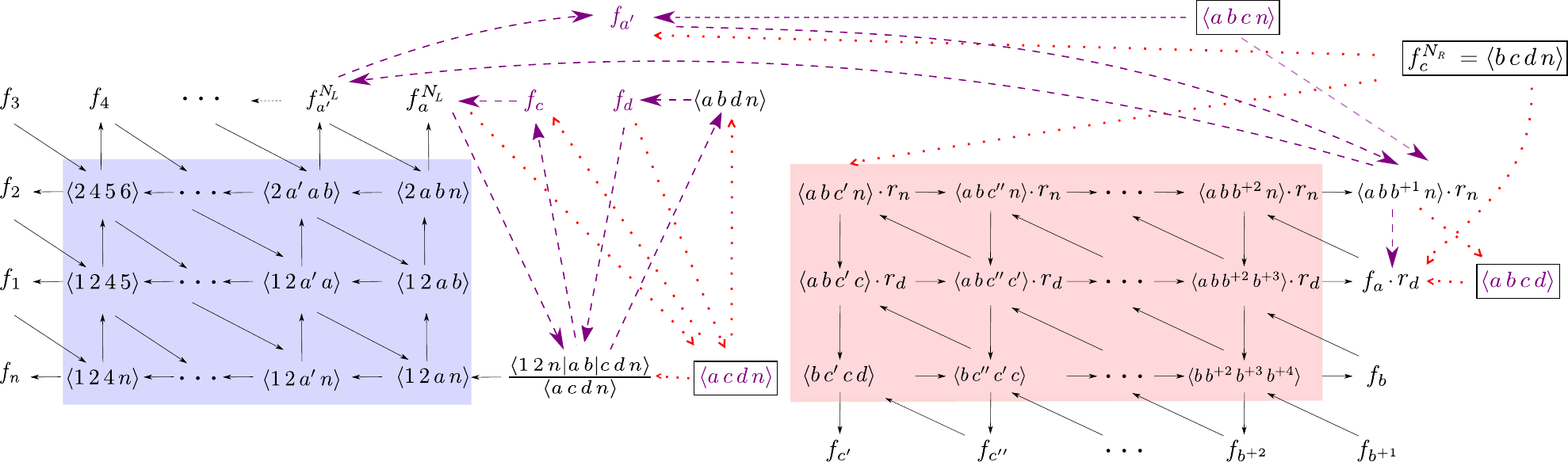}
	\caption{\label{fig:promoted-seed-extended} 
To construct the seed $\Lambda$ from $\Psi(\Sigma_0)$, 
add the variables 
$f_c=\lr{1cdn}, f_d=\lr{12dn},\lr{abdn},\lr{abcn}, \lr{abcd}, \lr{acdn}$,
then the dashed purple and dotted red arrows. 
The purple dashed arrows are all present in $\Omega$, while the red dotted arrows make the exchange ratios of the mutable variables of $\Lambda$ 
match up with those in $\Omega$. The frozen variables of $\Lambda$ are boxed. 
}
\end{figure}

Let $\Omega$ be the seed obtained from $\Fr(\Sigma_1)$ by unfreezing all its frozen variables which are not in $\mathcal{T}'$.  
Let $\Lambda$ be the seed obtained from $\Psi(\Sigma_0)$ by:
\begin{itemize}
	\item adding the variables 
$f_c=\lr{1cdn}, f_d=\lr{12dn},\lr{abdn},\lr{abcn}, \lr{abcd}, \lr{acdn}$; 
\item freezing all variables in $\mathcal{T}'$ and declaring all other variables mutable; 
\item adding the purple dashed
arrows shown in \cref{fig:promoted-seed-extended}; 
		\item and then finally 
adding the red dotted
arrows between elements of $\mathcal{T'}$ and mutable variables so that all 
the newly mutable variables have the same exchange ratios as the corresponding newly mutable
variables in $\Omega$. 
\end{itemize}
		Notice that we did not add any arrow incident to 
a mutable variable of $\Psi(\Sigma_0)$, so we did not affect any of the exchange ratios
that we previously analyzed.  The final step  of adding the red dotted arrows 
is possible because after we added the purple dashed  arrows in \cref{fig:promoted-seed-extended}, the exchange ratios for newly mutable variables (those which were not mutable in $\Psi(\Sigma_0)$)
differ from the exchange ratios in $\Omega$ by Laurent monomials in $\mathcal{T}'$.

We've now constructed two cluster algebras 
$\A(\Lambda)$ and $\A(\Omega)$ whose cluster variables are in bijection,
and whose frozen variables 
are precisely $\mathcal{T'}$.
Moreover, each initial cluster variable of $\A(\Lambda)$ is proportional to the corresponding
cluster variable of $\A(\Omega)$, and the exchange ratios around corresponding initial cluster variables 
agree.  
These properties are preserved under mutation, so each cluster variable of $\A(\Lambda)$
is proportional to the corresponding cluster variable of $\A(\Omega)$, with frozen factor  
a Laurent monomial in $\mathcal{T'}$.
The cluster variables of $\A(\Psi(\Sigma_0))$ are cluster variables of $\A(\Lambda)$ because the variables we added to $\Lambda$ are adjacent only to variables which are frozen in $\Psi(\Sigma_0)$. Similarly, cluster variables of $\A(\Fr(\Sigma_1))$ are cluster variables of $\A(\Omega)$, since the seeds only differ by freezing. So the frozen factors of variables in $\A(\Psi(\Sigma_0))$ are Laurent monomials in $\mathcal{T'}$.
\end{proof}

As a consequence of \Cref{pro-twistors2} and \Cref{thm:promotion2}, 
we obtain the following. 
\begin{corollary}\label{cor:pos}
The polynomial 
${\lr{i\,j\,n \br a\,b \br c\,d\,n}}
={\lr{a\,b\,n \br c\,d \br i\,j\,n}}={\lr{c\,d\,n \br i\,j \br a\,b\,n}}$ 
is a cluster variable and hence positive on $\Gr^{>0}_{4,n}$ whenever
 $i<j<a<b<c<d<n$. 
And the polynomial 
${\lr{i\,j\,\ell \br b\,a \br c\,d\,n}}$
is a cluster variable and hence positive on $\Gr^{>0}_{4,n}$ whenever
\begin{itemize}
\item  $i<j<\ell < a<b<c<d<n$ or 
\item $c<d<i<j<\ell<a<b<n$ or 
\item $a<b<i<j<\ell<c<d<n$.
\end{itemize}
\end{corollary}

\begin{remark}
Note that if $\{i,j,\ell,q,r,s,t\} 
\subset \{1,2,3,\dots,a-1, n\}$, we have that 
	$$\Psi(\lr{i\,j\,\ell \br r\,q \br s\,t\,b}) =
	\frac{ \lr{i\,j\,\ell \br r\,q \br s\,t \br b\,a \br c\,d\,n}}{\lr{a\,c\,d\,n}}.$$
  Therefore we can extend \Cref{cor:pos} to higher degree chain polynomials, and conclude e.g. that if 
	$i<j<\ell<q<r<s<t<b$, then 
	$\lr{i\,j\,\ell \br r\,q| s\,t \br b\,a \br c\,d\,n} = 
	\lr{i\,j\,\ell \br q\,r| s\,t \br a\,b \br c\,d\,n}$
is a cluster variable and hence positive on 
	$\Gr^{>0}_{4,n}$.
\end{remark}

\begin{figure}
\includegraphics[width=\textwidth]{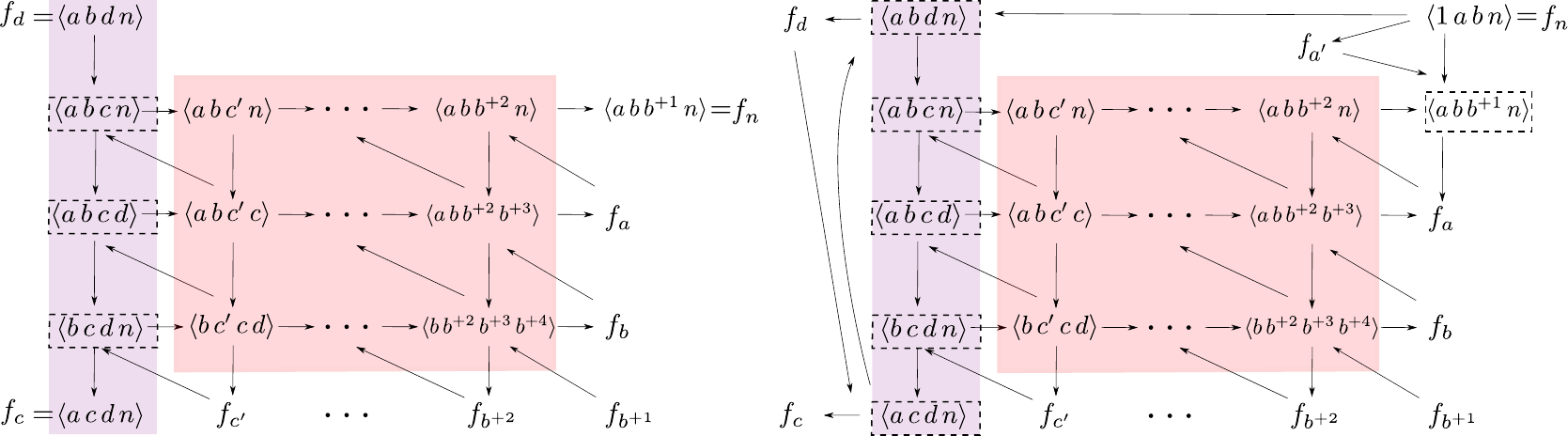}
\caption{Left: the seed $\Fr(\Sigma_1)$ in the case of upper promotion, where $a=1$ and $b=2$. Right: the seed $\Fr(\Sigma_1)$ when $a=2$ and $b=3$. We use the notation $x':= x-1$ and $x^{+i}:= x+i$.}
\label{fig:fr-sigma-one-special}
\end{figure}

\subsection{
Degenerate cases of \cref{thm:promotion2}} \label{sec:degenerate}

There are a few special cases of \cref{thm:promotion2} that warrant discussion. 
The proofs are all analogous to those of \cref{thm:promotion2} so we omit them.
\begin{itemize}
\item \cref{thm:promotion2} holds also for upper promotion (cf.~\cref{def:upper-promotion}), where\footnote{In this case, $\lr{acdn}= \lr{1cdn}=f_c$ and $\lr{abdn}= \lr{12dn}= f_d$ so are already frozen in $\CC[\Gr_{4,n}]$.} 
		$$\mathcal{T} = \mathcal{T'}=\{\lr{abcn}, \lr{abcd}, \lr{bcdn}\}.$$
		In this case, $\Sigma_0$ and $\Psi(\Sigma_0)$ consist only of the right connected component of the quiver in \cref{fig:seed_sigma_zero,fig:seed_sigma_zero_promoted}; the seed $\Fr(\Sigma_1)$ is shown on the left in \cref{fig:fr-sigma-one-special}.
		\item If $a=2$,
		then \cref{thm:promotion2} holds for $\mathcal{T} =\{\lr{abdn}, \lr{abcn}, \lr{abcd}, \lr{bcdn}, \lr{acdn}, \lr{a\,b\,b+1\,n}\}$ and $ \mathcal{T'}=\{\lr{abcn}, \lr{abcd}, \lr{bcdn}\}$. In this case, the left connected component of $\Sigma_0$ and $\Psi(\Sigma_0)$ consist of the single frozen variable $\lr{123n}$; the seed $\Fr(\Sigma_1)$ is shown on the right in \cref{fig:fr-sigma-one-special}.
		\item If $a=3$ or $c=b+1$, then some of the Pl\"ucker coordinates in \cref{fig:seed_sigma_one} coincide. If $a=3$, then $\lr{2abn}= \lr{a-1 \, a\, b\, n}$; if $c=b+1$ then $\lr{abcn}=\lr{a\,b\,b+1\,n}$. To obtain the quiver for $\Fr(\Sigma_1)$ in these cases, one should identify the vertices labeled by equal Pl\"ucker coordinates.
	\end{itemize}

\label{sec:bcfw-and-products}
    
\section{The BCFW map and the BCFW product}\label{sec:BCFWmap}

In this section
we define the \emph{BCFW map} on nonnegative Grassmannians, and define the closely related \emph{BCFW product}
on positroid cells in terms of plabic graphs. We will use these to define BCFW cells in \cref{sec:BCFWcells}. We will later see that the BCFW product is closely related to product promotion (see for example \cref{def:twistor-mtx}, \cref{thm:vanishing-for-all-ops}, and \cref{prop:vanishing_and_sign_of_functionaries_under_promotion}.)

\begin{notation}\label{not:bcfwmap}
Fix $1 \leq a< b< c< d<n$ such that $a,b$ and $c,d,n$ are consecutive, and let $N_L= \{1, \dots, a, b, n\}$ and $N_R=\{b, \dots, c, d, n\}$, as in \cref{not:LR_cluster}. Also fix $k \leq n$ and two nonnegative integers $k_L \leq |N_L|$ and $ k_R\leq |N_R|$ such that $k_L + k_R +1=k$.
\end{notation}

\begin{definition}[BCFW Map]\label{def:bcfw-map} 
Using \cref{not:bcfwmap},
the \emph{BCFW map} is the rational map
\[\mbcfw\;
:\;{\Gr}^{\scriptscriptstyle\ge0}_{k_L, N_L}\times\; \Gr_{1,5}^{\scriptscriptstyle>0}\;\times\;{\Gr}^{\scriptscriptstyle\ge0}_{k_R,N_R}\;\dashrightarrow\;\Grk\]
where $(A, [\alpha: \beta: \gamma: \delta:\varepsilon], B)$ is mapped to the (row span of the) $k \times n$ matrix $(*)$ in \Cref{fig:promotion-matrix}. Here $[\alpha: \beta: \gamma: \delta:\varepsilon]$ denote the positive homogeneous coordinates of a point in $\Gr_{1,5}^{>0} \subset \R \mathbb{P}^4$.

\begin{figure}[h]
\includegraphics[width=0.6\textwidth]{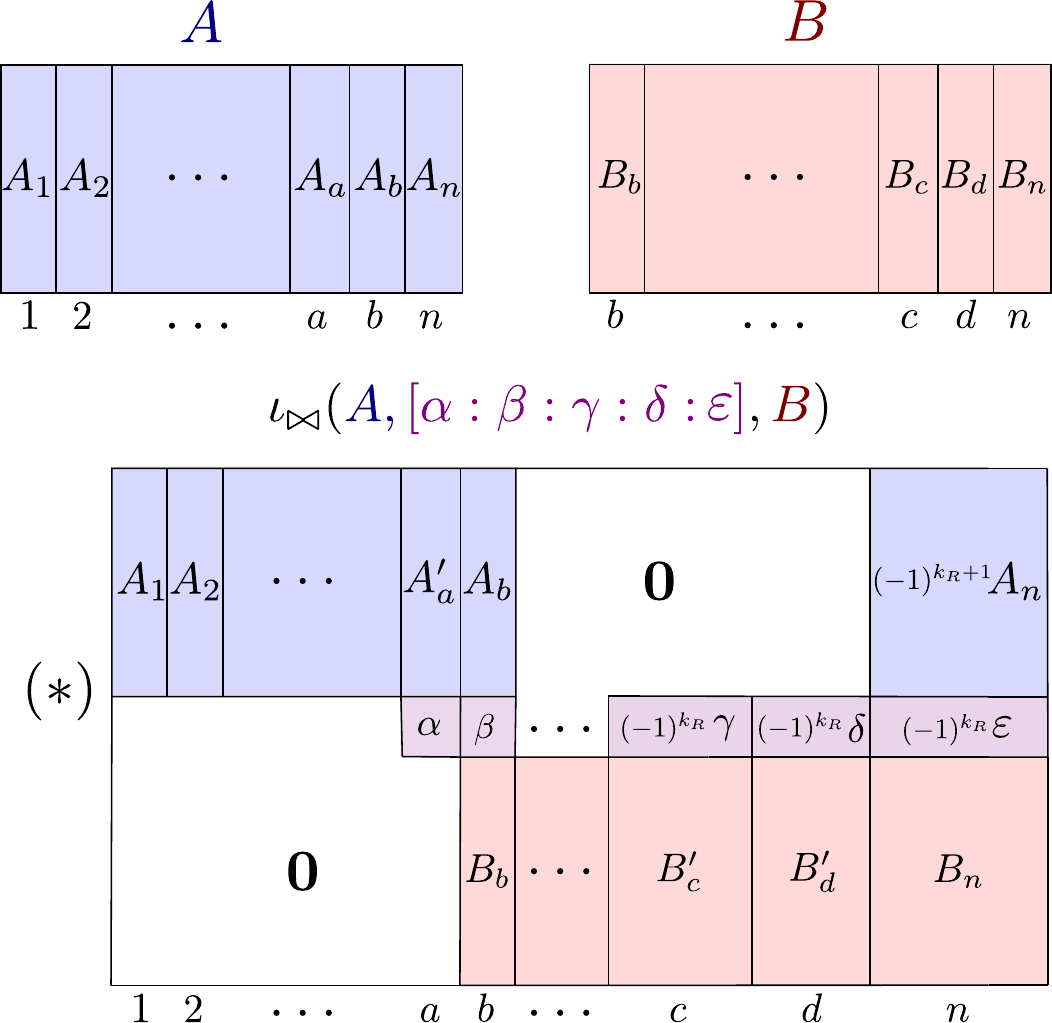}
\caption{The image of $(A, [\alpha: \beta: \gamma: \delta:\varepsilon], B)$ under the BCFW map $\mbcfw$. Here, $A'_a:= A_a + \frac{\alpha}{\beta}A_b$, $B_d':= B_d+ \frac{\delta}{\varepsilon} B_n$, $B_c':= B_c+ \frac{\gamma}{\delta} B_d'$, and in a standard abuse of notation we identify the matrix $(*)$ with its rowspan.
 }
\label{fig:promotion-matrix}
\end{figure}

\end{definition}

\begin{remark} \label{rmk:bcfw-map}
~
\begin{enumerate}
\item The definition of the BCFW map $\mbcfw$ depends on the choices of $a,c,n$, etc. in \cref{not:bcfwmap}.
Such choice will always be fixed in advance and, if it is clear, we will not mention it.
\item 
The rowspan of $(*)$ depends only on the values of $\alpha, \beta, \gamma, \delta, \varepsilon$ up to simultaneous rescaling.
Similarly, it only depends on the rowspans of $A$ and $B$ rather than the specific matrices. 
Additionally, for generic inputs the rank of $(*)$ is~$k$. Therefore the map  $\mbcfw$ is well defined.

\item The numbers $k_L$ or $k_R$ may be zero. If e.g. $k_L=0$ then $\Gr_{k_L, N_L}$ is a point and $A$ is a $0 \times (a+2)$ matrix. The matrix $(*)$ is $(k_R+1) \times n$ and the columns $1, \dots, a-1$ are zero columns.
\item
In the special case that $k_L=0$ and $a=1$, and hence $N_L=\{a,b,n\}$, we call $\mbcfw$ the \emph{upper BCFW map}. We omit the first argument and write
\[\mbcfw\;:\; \Gr_{1,5}^{\scriptscriptstyle>0}\;\times\;{\Gr}^{\scriptscriptstyle\ge0}_{k_R,N_R}\;\dashrightarrow\;\Grk.\]
\item For any set of indices $N \subset [n]$ ordered according to the usual order on integers, the definition of $\mbcfw$ naturally extends to the setting of $\Gr_{k,N}^{\scriptscriptstyle\geq 0}$ rather than $\Gr_{k,n}^{\scriptscriptstyle\geq 0}$. We replace $1$ and $n$ in the definition with the smallest and largest elements of $N$, respectively.
	\end{enumerate}
\end{remark}

We now introduce an operation on positroid cells we call \emph{BCFW product}, 
fixing
 notation as in 
	\cref{not:bcfwmap}.
We will define the operation in terms of \emph{plabic graphs}, 
see \cref{app:plabic}.

\begin{definition}[BCFW product]\label{def:butterfly}
Let $G_L, G_R$ be plabic graphs of respective ranks $k_L, k_R$ on $N_L, N_R$ as in \cref{not:bcfwmap}, let $\pos_L, \pos_R$ be their corresponding positroids and $S_L, S_R$ their corresponding positroid cells. The \emph{BCFW product} of $G_L$ and $G_R$, denoted $G_L \bcfw G_R$, is the plabic graph in the right-hand side of \Cref{fig:butterfly}. We also denote by $\pos_L \bcfw \pos_R$ and $S_L \bcfw S_R$ the positroid and positroid cell corresponding to $G_L \bcfw G_R$, and we call these the \emph{BCFW products} of $\pos_L, \pos_R$ and $S_L, S_R$.

\begin{figure}[h]
\centering
\includegraphics[width=1.0\textwidth]{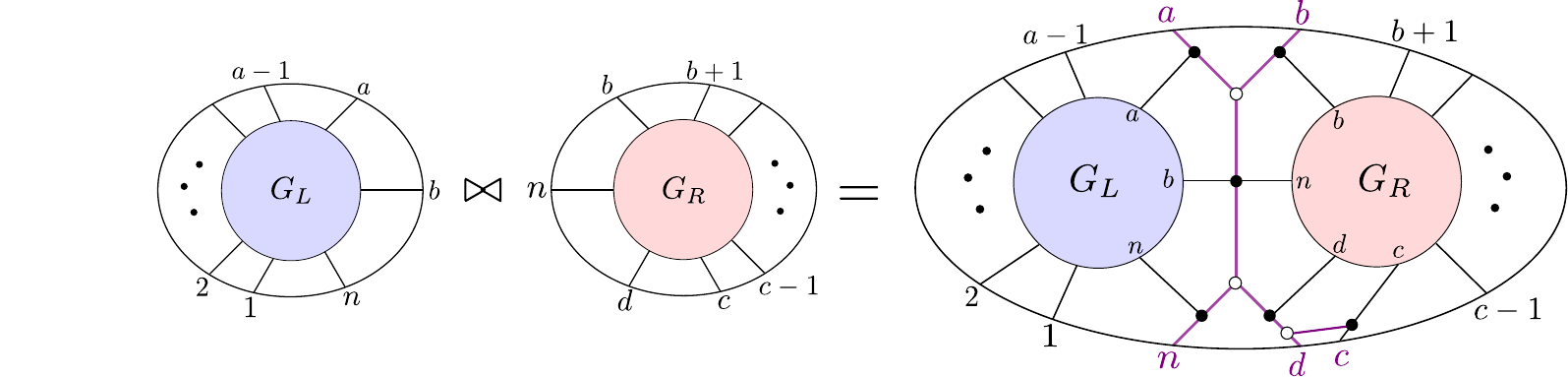}
\caption{The BCFW product $G_L \bcfw G_R$ of plabic graphs $G_L$ and $G_R$.
}
\label{fig:butterfly}
\end{figure}

 The notation $\bcfw$ is intended to echo the ``butterfly" 
 connecting $G_L, G_R$ 
 in the right-hand side of \Cref{fig:butterfly}.
 While we overload this notation, it will always be clear from context what kind of object it denotes. 
Note that whenever we apply a BCFW product, 
we will assume that the cells $S_L$ and $S_R$ use index sets $N_L$ and $N_R$ as in 
	\cref{not:bcfwmap}.\end{definition}

We will show the BCFW map on a pair of positroid cells gives their BCFW product, under appropriate assumptions. First, we need a definition.

\begin{definition}\label{def:coindependent}
Let $V \in \Gr_{k,N}$. A subset $J \subseteq N$ is \emph{coindependent\footnote{This name comes from matroid theory. The nonzero Pl\"ucker coordinates of $V$ determine a matroid $M$. The subset $J$ is coindependent for $V$ exactly when $J$ is independent in the matroid dual to $M$.} for $V$} if $V$ has a nonzero Pl\"ucker coordinate $\lr{I}_V$, such that $I \cap J = \emptyset$. By convention, if $k=0$ then all subsets are coindependent for $V$. If $S \subset  \Gr_{k,n}$, then $J$ is \emph{coindependent for} $S$ if $J$ is coindependent for all~$V \in S$. 
\end{definition}
\begin{remark}\label{rem:co}
By \cref{prop:positroid} and \cref{thm:positroidcell},
$J$ is coindependent for a cell $S$ 
if and only if
	every plabic graph $G$ for $S$ has a perfect orientation where all boundary vertices in $J$ are sinks.
\end{remark}

\begin{notation}[Coindependence]\label{not:coindipendence}
    We use \cref{not:bcfwmap} and fix positroid cells $S_L \subset {\Gr}^{\ge0}_{k_L, N_L}$ and $S_R \subset {\Gr}^{\ge0}_{k_R, N_R}$ such that $\{a, b, n\}$ is coindependent for $S_L$ and $\{b,c,d,n\}$ is coindependent for $S_R$.
\end{notation} 

We will prove 
\cref{prop:butterfly-matrix-same} (1) below, and defer the proof of (2) to 
 \cref{prop:bcfw-map-injective}. 
\begin{proposition}\label{prop:butterfly-matrix-same}
Fix $S_L, S_R$ as in \cref{not:coindipendence}.
Then we have the following:
	\begin{enumerate}
		\item The map $\mbcfw$ is well-defined on $S_L~\times~\chR~\times~S_R$ 
and its image is $S_L \bcfw S_R.$
	\item The map $\mbcfw$ is injective on $S_L~\times~\chR~\times~S_R$, and hence
	 we have an isomorphism
\[\mbcfw:S_L \times\Gr_{1,5}^{\scriptscriptstyle>0} \times S_R \;\xrightarrow{\sim}\; S_L \bcfw S_R.\]
\end{enumerate}
It follows that $\dim (S_L \bcfw S_R) = \dim S_L + \dim S_R +4$.
\end{proposition}
We note that when $k_L=0$ and $a=1$ above, $\{a,b,n\}$ is coindependent for $S_L$, so this proposition covers the case that $\mbcfw$ is the 
upper BCFW map.  

\begin{proof}
We first show that $\mbcfw$ is well-defined on $S_L \times \Gr_{1,5}^{\scriptscriptstyle>0} \times S_R$ and that the image is $S_L \bcfw S_R$.
	
	Choose reduced plabic graphs $G_L, G_R$ for $S_L, S_R$. The coindependence assumption 
plus \cref{rem:co}
implies that $G_L$, respectively $G_R$, has a perfect orientation $\O_L$, respectively $\O_R$, where $\{a, b, n\}$, respectively $\{b,c,d,n\}$, are sinks. We may assume that both $\O_L$ and $\O_R$ are acyclic using the coindependence assumption and \cref{acycliclemma}.
Now, orient the edges of $G_L \bcfw G_R$ according to $\O_L$, $\O_R$ and orient all other edges according to \cref{fig:orientation-butterfly}.

\begin{figure}[h]
\centering
\includegraphics[width=0.5\textwidth]{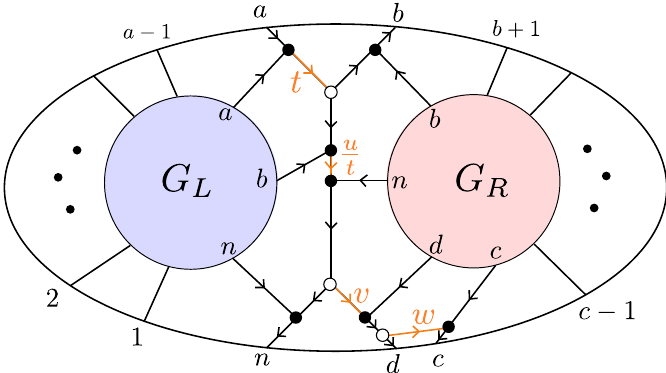}
\caption{The perfect orientation and edge weighting of $G_L \bcfw G_R$ used in the proof of \cref{prop:butterfly-matrix-same}. 
All weights  are positive, and 
edge weights of 1 are omitted. 
}
\label{fig:orientation-butterfly}
\end{figure}
 
This gives a perfect orientation of $G_L \bcfw G_R$ 
whose sources are $a$ together with 
the sources of $\O_L$ and $\O_R$. Thus $S_L \bcfw S_R$ is rank $k=1+ k_L + k_R$.

Now, we will analyze the path matrix of $G_L \bcfw G_R$ 
	(cf \cref{def:pathmatrix})
	and show that one can perform invertible row operations to obtain precisely  the matrix $(*)$ of  \cref{fig:promotion-matrix}. 
First, by using gauge transformations 
(cf \cref{lem:gauge}) 
	of the $10$ internal black/white vertices
shown in 
\cref{fig:orientation-butterfly}, 
	as well as the seven internal vertices of $G_L$ and $G_R$
	which are incident to the boundary vertices
	$\{a,b,n\}$ of $G_L$ and $\{b,c,d,n\}$ of $G_R$, we can assume that the
	$21$ oriented edges in 
\cref{fig:orientation-butterfly} are weighted as in the figure,
	where edge weights of $1$ are omitted.

Now, let $A, B$ be path matrices of $G_L, G_R$ with respect to the perfect orientations $\O_L, \O_R$. Note that $A, B$ represent elements of $S_L, S_R$. By inspection of the paths in \cref{fig:orientation-butterfly}, the path matrix of $G_L \bcfw G_R$ is 
\[ \begin{bmatrix}
	A_1& \cdots & A_{a-1} & 0& -tA_a &  0& \cdots& 0& (-1)^{k_R+1}uvw A_a'&  (-1)^{k_R+1}uv A_a'  & 	(-1)^{k_R+1}A_n'\\
	0 & \cdots & 0 & 1 & t & 0& \cdots &0 & (-1)^{k_R} uvw& (-1)^{k_R}uv& (-1)^{k_R}u\\
	0 & \cdots& \cdots&0 & B_b & B_{b+1}& \cdots&B_{c-1}&B_c' & B_d '& B_n\\
\end{bmatrix}\]
where 
$$ A'_a:= A_a + \tfrac{1}{t}A_b,\;\;\; B_d':= B_d+ vB_n,\;\;\; B_c':= B_c+ w B_d',\;\;\; A_n' = A_n + uA_a', $$
and the ``middle" row is indexed by the source $a$. By adding a suitable multiple of the row indexed by $a$ to the rows above, one obtains a full-rank matrix of the form
\begin{equation}\label{eq:path-mtx} \begin{bmatrix}
	A_1& \cdots & A_{a-1} & A_a'& A_b&  0& \cdots& 0& 0&  0  & 	(-1)^{k_R+1} A_n\\
	0 & \cdots & 0 & 1 & t & 0& \cdots &0 & (-1)^{k_R} uvw& (-1)^{k_R}uv& (-1)^{k_R}u\\
	0 & \cdots& \cdots&0 & B_b & B_{b+1}& \cdots&B_{c-1}&B_c' & B_d '& B_n\\
\end{bmatrix}.\end{equation}
Setting $t= \frac{\beta}{\alpha}, u=\frac{\varepsilon}{\alpha}, v=\frac{\delta}{\varepsilon}, w=\frac{\gamma}{\delta}$ and rescaling the row indexed by $a$ by $\alpha$ in \eqref{eq:path-mtx} gives a full-rank matrix exactly of the form 
	$(*)$ in \cref{fig:promotion-matrix}. Thus, 
$\mbcfw$ is well-defined on $(A, [\alpha: \beta: \gamma:\delta: \varepsilon], B)$ and $\mbcfw$ sends this triple to the rowspan of the matrix in \eqref{eq:path-mtx} with the variable substitutions mentioned above.
	 The rowspan of \eqref{eq:path-mtx} lies in $S_L \bcfw S_R$ for any positive parameter values. This shows that $\mbcfw$ is well-defined on $S_L \times \Gr_{1,5}^{>0} \times S_R$ and that $\mbcfw(S_L \times  \Gr_{1,5}^{>0} \times S_R) \subset S_L \bcfw S_R$. To see the reverse inclusion, note that varying $[\alpha: \beta: \delta: \gamma: \varepsilon]$ and the values of the weights in $A, B$ is the same as varying the values of the weights on $G_L \bcfw G_R$, and so the matrices in \eqref{eq:path-mtx} will vary over all points in $S_L \bcfw S_R$.

The statement that $\mbcfw$ is injective on $S_L \times \Gr_{1,5}^{>0} \times S_R$ is \cref{prop:bcfw-map-injective}. 

The dimension of $S_L \times \Gr_{1,5}^{>0} \times S_R$ is $\dim S_L + \dim S_R +4$ and $S_L \bcfw S_R$ is the image of this set under a bijective map, so has the same dimension.
\end{proof}

\begin{remark}[BCFW map and product in physics]
An analogous matrix representation of the BCFW map $\mbcfw$ (as in \cref{fig:promotion-matrix}) appeared in \cite[Equation (3.6)]{Bourjaily:2010wh}.
The first appearance of the ``butterfly'' plabic graph from 
\cref{fig:butterfly} is in \cite[Equation (3.3)]{Bai:2014cna}.
Note that the ``butterfly" plabic graph $G_L \bcfw G_R$ is related to 
the BCFW bridge recurrence in momentum space  
(see \cite[Equation (2.26)]{abcgpt}) by two applications of the
	\emph{inverse T-duality} operation on plabic graphs (see \cite[Definition 8.7, Remark 8.9]{PSW}). See \cref{fig:t-duality} for an illustration.
\end{remark}

\begin{figure}
	\includegraphics[width=\textwidth]{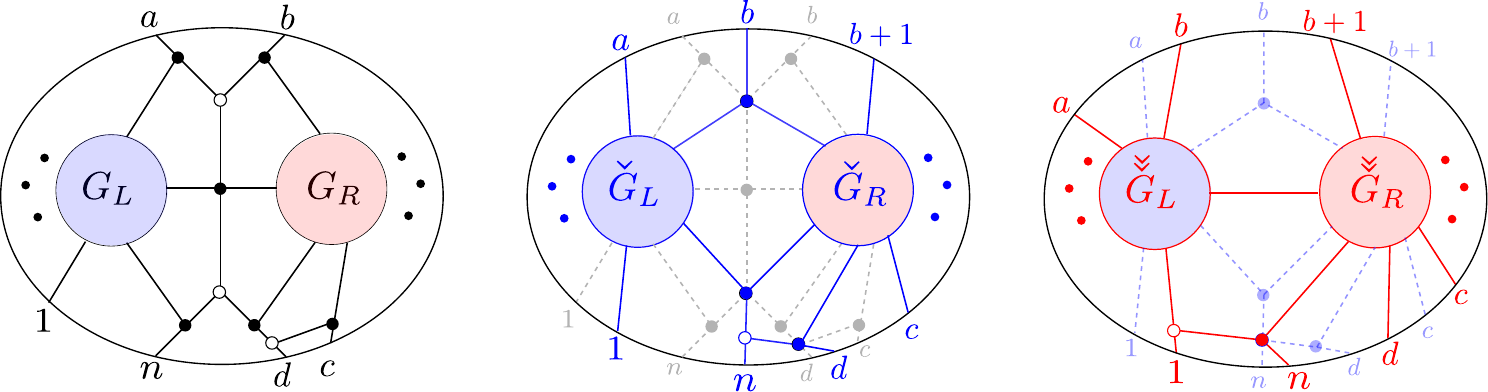}
	\caption{Applying inverse T-duality twice to $G_L \bcfw G_R$ gives the BCFW bridge recurrence in `momentum space' \cite[Equation (2.26)]{abcgpt}. We use the notation $\check{G}$ for the inverse T-dual of $G$. Left: $G_L \bcfw G_R$; center: application of inverse T duality to $G_L \bcfw G_R$; right: application of inverse T duality twice to $G_L \bcfw G_R$.}
	\label{fig:t-duality}
\end{figure}

\section{Standard and general BCFW cells}\label{def:BCFWcells}

In this section we define the standard and general BCFW cells. We then introduce \emph{chord diagrams} (cf. \cref{def:cd}) and \emph{recipes} (cf. \cref{def:recipe}), which label standard and general BCFW cells respectively. Finally, we define the \emph{BCFW matrix} for each BCFW cell, a convenient representative matrix, and use it to show that BCFW cells in $\Grk$ are $4k$-dimensional.

First, we define BCFW cells. Recall the dihedral group action on $\Grk$ from \cref{def:dihedral}. We refer to the following collection of cells as \emph{general BCFW cells}, or 
simply as \emph{BCFW cells}.

\begin{definition}[General BCFW cells]\label{def:BCFW_cell}
The set of \emph{general BCFW cells} is defined recursively:
\begin{enumerate}[align=left]
\itemsep0.25em
\item[(Base case)]
For $k=0$, the trivial cell $\Gr^{\scriptscriptstyle>0}_{0,n}$ is a general BCFW cell.
\item[(Insert zero)]
If $S$ is a general BCFW cell, then so is any cell obtained by inserting a zero column.
\item[(Cyclic shift $+$ reflect)]
If $S$ is a general BCFW cell, then so is any cyclic shift or reflection of $S$.
\item[(Product)] Fix the quantities in \cref{not:bcfwmap}. If $S_L$ and $S_R$ are general BCFW cells on $N_L$ and $N_R$, then so is their BCFW product $S_L \bcfw S_R$. 
\end{enumerate}
\end{definition}

\begin{remark}
By \cref{cor:4bidden}, general BCFW cells satisfy the coindependence assumptions of \cref{prop:butterfly-matrix-same}. Thus, taking the BCFW product of two BCFW cells is the same as applying the BCFW map to them. So the ``(Product)'' step in \cref{def:BCFW_cell} could be equivalently stated as: if $S_L$ and $S_R$ are general BCFW cells, then so is $\mbcfw(S_L \times \chR \times S_R)$.
\end{remark}

The \emph{standard} BCFW cells, 
which we define below, are 
a particularly nice subset of BCFW cells.
The images of the standard BCFW cells 
yield a tiling of the amplituhedron~\cite{even2021amplituhedron}. 

\begin{definition}[Standard BCFW cells]\label{def:std-bcfw-cells}
	\emph{Standard} BCFW cells are defined recursively as follows.
	\begin{enumerate}[align=left]
			\itemsep0.25em
			\item[(Base case)]
			For $k=0$, the trivial cell $\Gr_{0,n}$ is a standard BCFW cell.
			\item[(Insert zero')]
			If $S$ is a standard BCFW cell, then so is the cell obtained by inserting a zero column in the penultimate position.
			\item[(Product)] Fix \cref{not:bcfwmap}. If $S_L$ and $S_R$ are standard BCFW cells on $N_L$ and $N_R$, then their BCFW product $S_L \bcfw S_R$ is a standard BCFW cell. 
	\end{enumerate}
\end{definition}

\begin{remark}
The standard BCFW cells defined here are the same as those in \cite{karp2020decompositions,even2021amplituhedron}, 
though \cref{def:std-bcfw-cells} uses different terminology 
than those two papers. Indeed, \cite{karp2020decompositions} obtains standard BCFW cells  by applying
the ``BCFW bridge recurrence in momentum space"(see \cref{fig:t-duality}, 
right), then applying T-duality twice. This is the same as repeatedly applying the BCFW product.
	\cite{karp2020decompositions} also gives an explicit
construction of standard BCFW cells in terms of pairs of noncrossing 
lattice paths, while \cite{even2021amplituhedron} gives an explicit
construction of standard BCFW cells in terms of chord diagrams,
and shows that the two constructions agree \cite[Proposition 2.28]{even2021amplituhedron}.
\end{remark}

\begin{figure} 
\includegraphics[width=0.4\textwidth]{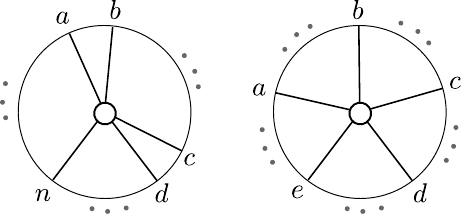} 
\caption{ \label{fig:bcfwcellsk1} Standard BCFW cells (left) and general BCFW cells (right) in $\Gr^{\scriptscriptstyle\geq 0}_{1,n}$, where the $\ldots$ denote black lollipops in the remaining indices.}
\end{figure}

\begin{example} \label{ex:bcfw_cells_k1}
For $k=1$, the BCFW cells in $\Gr_{1,n}^{\scriptscriptstyle \geq 0}$ are indexed by the elements of $\binom{\scriptscriptstyle[n]}{5}$. For $I=\{a,b,c,d,e\} \in \binom{\scriptscriptstyle[n]}{5}$, the corresponding BCFW cell consists of points with Pl\"ucker coordinates $\lr{a},\lr{b},\lr{c},\lr{d},\lr{e}>0$ and all others zero. The plabic graph has a 5-valent white vertex adjacent to boundary vertices $a,b,c,d,e$ and all other vertices are black lollipops, see \cref{fig:bcfwcellsk1}. The standard BCFW cells for $k=1$ are only those where $a,b$ and $c,d$ are consecutive and $e=n$. For $k=n-4$, the \emph{totally} positive Grassmannian $\Gr_{n-4,n}^{\scriptscriptstyle>0}$ is the only BCFW cell. It can be obtained from the point $\Gr_{0,4}^{>0}$ by repeatedly applying the upper BCFW map.
\end{example}

\begin{remark}
\label{kRkLconditions}
It follows from the above definition that if $k>0$ then $k\leq n-4$ for a BCFW cell in $\Grk$. 
Hence, when one generates a BCFW cell with the (Product) step, $k_R \leq |N_R|-4$ and $k_L \leq \max(0,|N_L|-4)$. In terms of \cref{not:bcfwmap}, $k_L \leq a-2$ and $k_R \leq c-b-1$, with the additional possibility that $k_L=0$ and $a=1$, which is the case of the upper BCFW map.
\end{remark}

\subsection{Standard BCFW cells from chord diagrams}
\label{sec:quord}

In this section we introduce \emph{chord diagrams}, and show 
how each gives an algorithm for constructing  a standard BCFW cell.
In \cref{sec:recipes} we then give a  generalization of this algorithm, called a \emph{recipe}, for 
constructing a general BCFW cell.  

\begin{definition}[Chord diagram \cite{even2021amplituhedron}]\label{def:cd} 
Let $k,n \in \mathbb{N}$. A~\emph{chord diagram} $D \in\mathcal{CD}_{n,k}$ is a set of $k$~quadruples named \emph{chords}, of integers in the set $\{1,\dots,n\}$ named \emph{markers}, of the following form:
	$$ D \;=\; \{(a_1,b_1,c_1,d_1),\dots,(a_k,b_k,c_k,d_k)\} \;\;\text{ where }\;\;
	b_i=a_i+1 \text{ and }d_i=c_i+1$$ such that every 
	chord $D_i=(a_i,b_i,c_i,d_i) \in D$ satisfies
$ 1 \;\leq\; a_i \;<\; b_i \;<\; c_i \;<\; d_i \;\leq\; n-1 $
and \emph{no} two chords $D_i,D_j \in D$ satisfy
$ a_i \;=\; a_j$ or $a_i \;<\; a_j \;<\; c_i \;<\; c_j.$
\end{definition}

It follows from this definition that the number of different chord diagrams with $n$ markers and $k$ chords is the  
Narayana number~$N(n-3,k+1)$:
$ \left|\mathcal{CD}_{n,k}\right| \;=\; \frac{1}{k+1}\binom{n-4}{k}\binom{n-3}{k}$. 
As we
explain in \cref{def:standardfromCD} below, 
each chord diagram gives rise to a standard BCFW cell by
encoding a sequence of BCFW products and penultimate zero column insertions.

\begin{figure}
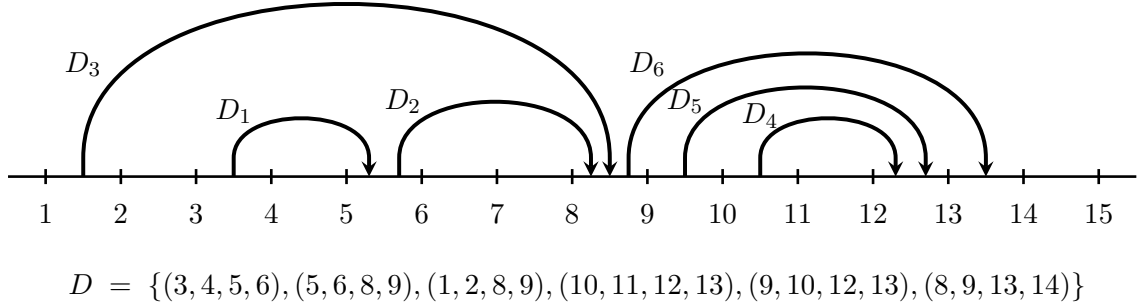

\begin{center}
\tikz[line width=1,scale=1]{
\draw (0.5,0) -- (15.5,0);
\foreach \i in {1,2,...,15}{
\def\x{\i}
\draw (\x,-0.1)--(\x,+0.1);
\node at (\x,-0.5) {\i};}
\foreach \i/\j in {1/8, 3/4.8, 5.2/7.75, 8.25/13, 9/12.2, 10/11.8}{
\def\x{\i+0.5}
\def\y{\j+0.5}
\draw[line width=1.5,-stealth] (\x,0) -- (\x,0.25) to[in=90,out=90] (\y,0.25) -- (\y,0);
}
\node at(1.5,1.5) {$D_3$};
\node at(3.5,0.875) {$D_1$};
\node at(5.75,1) {$D_2$};
\node at(9,1.5) {$D_6$};
\node at(9.5,1) {$D_5$};
\node at(10.5,0.8125) {$D_4$};
}
\vspace{0.5em}
$$ D \;=\; \{(3,4,5,6),(5,6,8,9),(1,2,8,9),(10,11,12,13),(9,10,12,13),(8,9,13,14)\} $$
\end{center}
\caption{
A chord diagram $D$ with $k=6$ chords $n=15$ markers.
} 
\label{cd-example}
\end{figure}

See Figure~\ref{cd-example}, where we visualize such a chord diagram $D$ in the plane as a horizontal line with $n$ markers labeled $\{1,\dots,n\}$ from left to right, and $k$ nonintersecting chords above it, whose \emph{start} and \emph{end} lie in the segments $(a_i,b_i)$ and $(c_i,d_i)$ respectively. The definition imposes restrictions on the chords: they cannot start before $1$, end after $n-1$, or start or end on a marker. Two chords cannot start in the same segment $(s,s+1)$, and one chord cannot start and end in the same segment, nor in adjacent segments. Two chord cannot cross. Pictorially, the \emph{forbidden} configurations are:
\begin{center}
\vspace{1em}
\tikz[line width=1,scale=1.05]{
\draw (0.6,0) -- (1.4,0);
\draw (1,-0.1)--(1,+0.1);
\node at (1,-0.5) {$i$};
\draw[line width=1.5,-stealth] (1,0.1) -- (1,0.25) to[in=180,out=90] (1.5,0.5)}
\hfill
\tikz[line width=1,scale=1.05]{
\draw (0.6,0) -- (1.4,0);
\draw (1,-0.1)--(1,+0.1);
\node at (1,-0.5) {$i$};
\draw[line width=1.5,stealth-] (1,0.1) -- (1,0.25) to[in=0,out=90] (0.5,0.5)}
\hfill
\tikz[line width=1,scale=1.05]{
\draw (0.5,0) -- (1.5,0);
\draw (1,-0.1)--(1,+0.1);
\node at (1,-0.5) {$1$};
\draw[line width=1.5,-stealth] (0.75,0) -- (0.75,0.25) to[in=180,out=90] (1.5,0.5)}
\hfill
\tikz[line width=1,scale=1.05]{
\draw (0.25,0) -- (1.25,0);
\draw (0.5,-0.1)--(0.5,+0.1);
\draw (1,-0.1)--(1,+0.1);
\node at (0.375,-0.5) {$n{-}1$};
\node at (1,-0.5) {$n$};
\draw[line width=1.5,stealth-] (0.75,0) -- (0.75,0.25) to[in=0,out=90] (0.25,0.5)}
\hfill
\tikz[line width=1,scale=1.05]{
\draw (0.25,0) -- (1.25,0);
\draw (0.5,-0.1)--(0.5,+0.1);
\draw (1,-0.1)--(1,+0.1);
\node at (0.375,-0.5) {$n{-}1$};
\node at (1,-0.5) {$n$};
\draw[line width=1.5,stealth-] (1.125,0) -- (1.125,0.25) to[in=0,out=90] (0.5,0.5)}
\hfill
\tikz[line width=1,scale=1.05]{
\draw (0.25,0) -- (1.25,0);
\draw (0.5,-0.1)--(0.5,+0.1);
\draw (1,-0.1)--(1,+0.1);
\node at (0.5,-0.5) {$i$};
\node at (1.125,-0.5) {$i{+}1$};
\draw[line width=1.5,-stealth] (0.625,0) -- (0.625,0.25) to[in=180,out=90] (1.25,0.6);
\draw[line width=1.5,-stealth] (0.875,0) -- (0.875,0.25) to[in=180,out=90] (1.35,0.4)}
\hfill
\tikz[line width=1,scale=1.05]{
\draw (0.25,0) -- (1.25,0);
\draw (0.5,-0.1)--(0.5,+0.1);
\draw (1,-0.1)--(1,+0.1);
\node at (0.5,-0.5) {$i$};
\node at (1.125,-0.5) {$i{+}1$};
\draw[line width=1.5,-stealth] (0.6,0) -- (0.6,0.25) to[in=90,out=90] (0.9,0.25) -- (0.9,0);}
\hfill
\tikz[line width=1,scale=1.05]{
\draw (-0.1,0) -- (1.1,0);
\draw (0,-0.1)--(0,+0.1);
\draw (0.5,-0.1)--(0.5,+0.1);
\draw (1,-0.1)--(1,+0.1);
\node at (-0.125,-0.5) {$i{-}1$};
\node at (0.5,-0.5) {$i$};
\node at (1.125,-0.5) {$i{+}1$};
\draw[line width=1.5,-stealth] (0.25,0) -- (0.25,0.25) to[in=90,out=90] (0.75,0.25) -- (0.75,0);}
\hfill
\tikz[line width=1,scale=1.05]{
\draw (-0.15,0) -- (0.15,0);
\draw (0.35,0) -- (0.65,0);
\draw (0.85,0) -- (1.15,0);
\draw (1.35,0) -- (1.65,0);
\node at (0.75,-0.5) {crossing};
\draw[line width=1.5,-stealth] (0.0,0) -- (0.0,0.25) to[in=90,out=90] (1.0,0.25) -- (1.0,0);
\draw[line width=1.5,-stealth] (0.5,0) -- (0.5,0.25) to[in=90,out=90] (1.5,0.25) -- (1.5,0);}
\hfill
\end{center}

We say that a chord is a \emph{top chord} if there is no chord above it, e.g. $D_3$ and $D_6$ in Figure~\ref{cd-example}. One natural way to label
the chords is by 
$D_1,\dots,D_k$ such that for all $1 \leq j \leq k$, $D_j$ is the rightmost top chord among the set of chords $\{D_1,\dots,D_j\}$ as in \cref{cd-example}. This is equivalent to sorting 
the chords according to their ends.
The visualization of chord diagrams provides us with useful terminology,
which we illustrate in \Cref{cd-example}. This terminology will primarily be used in \cref{sec:clustervariable}.

\begin{definition}[Terminology for chords]
\label{cd-terminology}
A~chord is a~\emph{top} chord 
if there is no chord above it,
 and otherwise it is a \emph{descendant} of the chords above it, called its \emph{ancestors}, and in particular a~\emph{child} of the chord immediately above it, which is called its~\emph{parent}. For example, $D_4$ has parent $D_5$ and ancestors $D_5$ and $D_6$.
Two chords are \emph{siblings} if they are either top chords or children of a common parent; for example, $D_1$ and $D_2$ are siblings, and $D_3$ and $D_6$ are siblings.
Two chords are \emph{same-end} 
if their ends occur in a common segment $(e,e+1)$, are \emph{head-to-tail} if the first ends in the segment where the second starts, and are \emph{sticky} if their 
starts lie in consecutive segments $(s,s+1)$ and~$(s+1,s+2)$.
For example, chords $D_2$ and $D_3$ are same-end, chords $D_1$ and 
	$D_2$ are head-to-tail, and chords $D_5$ and $D_6$ are sticky.
\end{definition}

\begin{remark}
\label{index-sets}
The definition of a chord diagram naturally extends to general index sets with a total order. 
	Thus, we will sometimes work with a finite set of markers 
$N \subset \{1,\dots,n\}$ rather than $\{1,\dots,n\}$, and a set $K$ of chord indices rather than $\{1,\dots,k\}$, and denote these diagrams by $\mathcal{CD}_{N,K}$. 
We will always have that the largest marker is 
$n\in N$, the starts and ends of chords will be consecutive pairs in~$N$ (and also $\mathbb{N}$) and the rightmost top chord will be denoted by $D_{\rtop} = D_{\max K}$. 
\end{remark}

\begin{definition}[Left and right subdiagrams]
\label{def:leftright}
Let $D$ be a chord diagram in $\mathcal{CD}_{N,K}$.
A \emph{subdiagram} is obtained by restricting to a subset of the chords and a 
subset of the markers which contains both these chords and the marker~$n$.
Let
$\ctop = (a,b,c,d)$ be the rightmost top chord of~$D$, where $1\leq a<b<c<d<n$, and moreover $a,b$ and $c,d$ are consecutive in $N$. 

In the case that $d,n$ are consecutive as well
we define $D_L$, the \emph{left subdiagram} of $D$, on the markers $N_L=\{1,2,\dots,a,b,n\}$
and the \emph{right subdiagram} $D_R$ on~$N_R=\{b,\dots,c,d,n\}$. The subdiagram $D_L$ contains all chords that are to the left of $D_{\rtop}$, and $D_R$ contains the descendants of~$D_{\rtop}$. 
\end{definition}

\begin{example}
\label{right-left-diagrams}
For the chord diagram $D$ in \cref{cd-example}, the rightmost top chord is $D_6 = (8,9,13,14)$, so $N_L = \{1,\dots,9,15\}$ and $D_L = \{D_1,D_2,D_3\}$, while $N_R = \{9,\dots,15\}$ and $D_R = \{D_4,D_5\}$.
\end{example}

\begin{definition}[Standard BCFW cell from a chord diagram]
\label{def:standardfromCD}
Let $D\in \mathcal{CD}_{N,K}$ be a chord diagram. 
We recursively construct from $D$ a standard BCFW cell $S_D$ in ${\Gr}^{\scriptscriptstyle\ge0}_{K, N}$ as follows:
\begin{enumerate}[align=left]
\itemsep0.125em
\item 
If $k=0$, then the BCFW cell is the trivial cell $S_D:=\Gr^{\scriptscriptstyle\ge0}_{0,N}$.
\item Otherwise, let $D_{\rt}=(a,b,c,d)$ be the rightmost top chord of~$D$ and let $p$ denote the penultimate marker in $N$.
\begin{enumerate}
\itemsep0.125em
\item 
If $d\neq p$, let $D'$ be the subdiagram on $N \setminus \{p\}$ with the same chords as $D$, and let $S_{D'}$ be the standard BCFW cell associated to $D'$.  Then, we define $S_D := \pre_{p} S_{D'}$, which denotes the standard BCFW cell obtained from $S_{D'}$ by inserting a zero column in the penultimate position~$p$. 
\item If $d=p$, let $S_L$ and $S_R$ be the standard BCFW cells on $N_L$ and $N_R$ associated to 
the left and right subdiagrams	$D_L$ and $D_R$ of $D$.  Then, we let $S_D := S_L \bcfw S_R$, the standard BCFW cell which is their BCFW product as in \cref{def:butterfly}.
\end{enumerate}
\end{enumerate}
\end{definition}

\begin{figure}
	\includegraphics[width=0.7\textwidth]{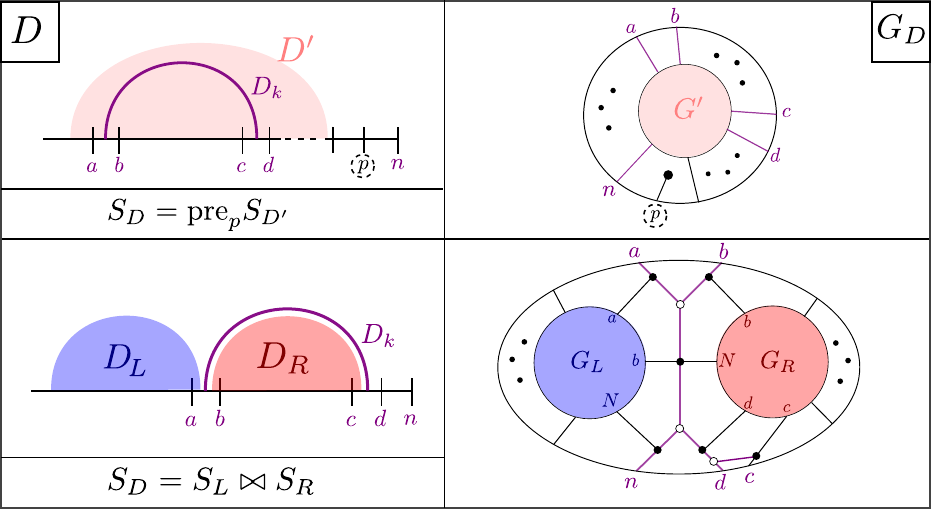}
	\caption{Recursive construction of a standard BCFW cell from a chord diagram as in \cref{def:standardfromCD}. 
 Top left (right): construction of $D$ ($G_D$) from $D'$ ($G'$) as in $(1a)$; bottom left (right) construction of $D$ ($G_D$) from $D_L,D_R$ ($G_L,G_R$) as in $(1b)$.}
	\label{fig:bcfw_chord}
\end{figure}

\begin{example}
The standard BCFW cell $S_D$ of the chord diagram $D$ in \cref{cd-example} is $S_L \bcfw S_R$ where the chord subdiagrams $D_L,D_R$ are as in \cref{right-left-diagrams}. One can keep applying the recursive definition and obtain:
\begin{align*}
S_L \;& =\;
\Gr_{0,\{1,2,15\}}
\bcfw
\left( 
\left(
\Gr_{0,\{2,3,4,15\}}
\bcfw
\Gr_{0,\{4,5,6,15\}}
\right)
\bcfw
\Gr_{0,\{6,7,8,9,15\}}
\right)
\\
S_R \;& =\;
\pre_{14} 
\left(
\Gr_{0,\{9,10,15\}}
\bcfw
\left(
\Gr_{0,\{10,11,15\}}
\bcfw
\Gr_{0,\{11,12,13,15\}}
\right)
\right)
\end{align*}
\end{example}

\begin{remark}
It is not hard to see that every standard BCFW cell arises from a chord diagram.  Moreover
two distinct chord diagrams give rise to different cells \cite{even2021amplituhedron}.
Therefore every standard BCFW cell arises from a unique chord diagram.
\end{remark}

\begin{remark}\label{rem:abuse}
In \cref{def:standardfromCD} we have taken care to work with chord diagrams on 
general index sets~$N$,  but in what follows we often slightly abuse notation and consider the index set $[n]$, and the extension to general index sets is implied.
\end{remark}

\subsection{General BCFW cells from recipes}\label{sec:recipes}

In this section, we establish conventions for labeling general BCFW cells. We also define BCFW matrices, which are distinguished representatives for the elements of general BCFW cells, obtained by repeatedly applying $\pre_I,\cyc, \refl$ and $\mbcfw$.

Each general BCFW cell may be specified 
by a list of operations from \cref{def:BCFW_cell}. 
The class of general BCFW cells includes the standard BCFW cells, but 
is additionally closed under the operations
of cyclic shift, reflection, and inserting a zero column anywhere (see \cref{def:dihedral}), at any stage of the recursive generation.  
Since any sequence of these operations can be expressed as $\pre_I$ followed by $\cyc^r$ followed by $\refl^s$ for some $I, r, s$,
we can specify in a concise form which ones take place after each BCFW product. We will record the generation of a BCFW cell using the formalism of \emph{recipe} in \cref{def:recipe}, which will be convenient for the proofs to come.

\begin{definition}[General BCFW cell from a recipe] \label{def:recipe} 
A \emph{step-tuple} on 
a finite index set $N\subset \NN$ is a $4$-tuple 
\[((a_i, b_i, c_i, d_i, n_i),\pre_{I_i}, \cyc^{r_i}, \refl^{s_i}),\]
where $I_i \subseteq N$ such that $n_i$ is the largest element in $N \setminus I_i$, 
$a_i<b_i$ and $c_i<d_i<n_i$ are both consecutive in $N\setminus I_i$,  $0 \leq r_i < |N|$, and $s_i \in \{0,1\}$. 
A step-tuple records a BCFW product of two cells using indices $(a_i, b_i, c_i, d_i, n_i)$; then zero column insertions in positions $I_i$; then
applying the cyclic shift $r_i$ times; and then applying reflection $s_i$ times. Note that some of these operations may be the identity. Each operation in a step-tuple which is not the identity is called a~\emph{step}. 

A \emph{recipe} $\rcp$ on $N$ is either the empty set (the \emph{trivial recipe} on $N$),
or a recipe $\rcp_L$ on $N_L$ followed by a recipe $\rcp_R$ on $N_R$ followed by a step-tuple
$((a_k, b_k, c_k, d_k, n_k),\pre_{I_k}, \cyc^{r_k}, \refl^{s_k})$ on $N$,
where $N_L = (N\setminus I_k) \cap \{n_k, \dots, a_k,b_k\}$ and 
$N_R = (N\setminus I_k) \cap \{b_k,\dots, c_k,d_k,n_k\}$.
We let $S_{\rcp}$ denote the general BCFW cell on $N$ obtained by applying the sequence of 
	operations specified
by $\rcp$. If $\rcp$ consists of $k$ step-tuples, then $S_{\rcp} \subset \Gr_{k, N}^{\scriptscriptstyle\geq 0}$.
\end{definition}

\begin{example} \label{ex:recipe}
Consider the recipe $\rcp$ consisting of the following sequence of $4$ step-tuples:
\begin{equation*}
    ((3,4,5,6,12),\pre_{2}), ((1,2,5,6,12), \cyc^{2}, \refl)), ((6,7,8,9,11),\pre_{10,12}), ((5,6,10,11,12), \cyc^{4}, \refl).
\end{equation*}
\cref{fig:bcfw_tile} shows the plabic graph of the general BCFW cell $S_\rcp$ obtained from $\rcp$ following \cref{def:recipe}. 
\end{example}

\begin{figure}
	\includegraphics[width=\textwidth]{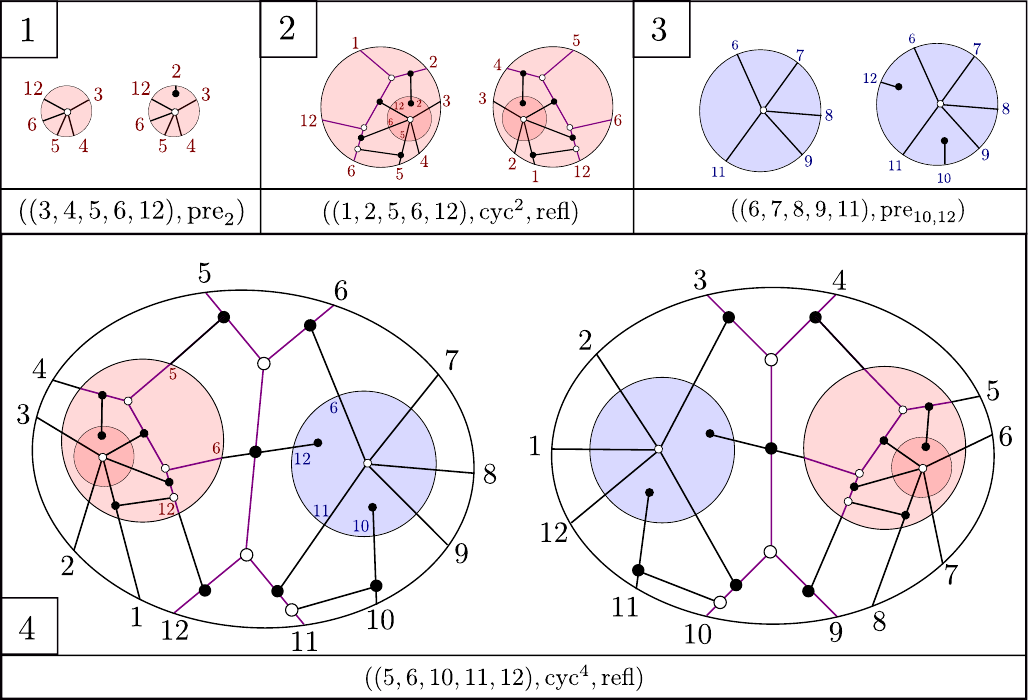}
	\caption{Illustration of building up a BCFW cell using the recipe $\rcp$ of \cref{ex:recipe}. Box $i$ shows the result after the first $i$ step-tuples. The result of the step $(a_i,b_i,c_i,d_i,n_i)$ is shown on the left in each box, and the results of the steps $\mbox{pre}_{I_i}, \mbox{cyc}^{r_i}$ and $\mbox{refl}^{s_i}$ are shown on the right.}
 
	\label{fig:bcfw_tile}
\end{figure}

Because our arguments are frequently recursive, we need some notation for the BCFW cells obtained by deleting the final step of a recipe. We use the following notation throughout.

\begin{notation}\label{not:L-and-R}
Let $\rcp$ be a recipe for a BCFW cell $S \in \Gr_{k, N}^{\scriptscriptstyle\geq 0}$. Let $\st$
denote the final step, which is either $(a_k,b_k,c_k,d_k,n_k), \pre_{I_k}, \cyc$ or $\refl$. If $\st \neq (a_k, b_k, c_k, d_k, n_k)$, then we let $\rcpp$ denote the recipe obtained by replacing $\st$ with the identity. Note that $S_{\rcpp}$ is again a BCFW cell. If $\st=(a_k, b_k, c_k, d_k, n_k)$, let $\rcp_L$ and $\rcp_R$ denote the recipes on $N_L$ and $N_R$ as in \cref{def:recipe}. Then $\rcp_L, \rcp_R$ are recipes for BCFW cells $S_L \subset \Gr_{k_L, N_L}^{\scriptscriptstyle\geq 0}$ and $S_R\subset \Gr_{k_R, N_R}^{\scriptscriptstyle\geq 0}$ and $S = S_L \bcfw S_R$. Note that to avoid clutter, we will usually use $L,R$ as subscripts rather than writing $S_{\rcp_L}, S_{\rcp_R}$.
\end{notation}

\begin{remark}
In contrast with the bijective correspondence between standard BCFW cells and chord diagrams, multiple recipes could give rise to the same general BCFW cell. 
Even the sets of 5 indices that are involved in the BCFW products are not uniquely determined by the resulting cell.
\end{remark}

We now construct a representative matrix for elements of a BCFW cell. We will use this matrix in the next section to invert the $\tZ$-map on BCFW cells (see \cref{def:twistor-mtx} and \cref{thm:BCFW-tile-and-sign-description}).

\begin{definition}[BCFW parameters and matrix] \label{def:BCFW-coords-mtx}
Let $\rcp$ be a recipe on $N$ for a BCFW cell $S_\rcp$. 
We define the \emph{BCFW parameters}
$\coord_\rcp$ and \emph{BCFW matrix} $\mtx_\rcp$ of $S_\rcp$ recursively. If $\rcp$ is the trivial recipe, then $\coord_\rcp = \varnothing$ and $\mtx_\rcp$ is a $0 \times N$ matrix. If $\st= (a_k, b_k, c_k, d_k, n_k)$, then let $\coord_L,\coord_R$ and $\mtx_L, \mtx_R$ be the BCFW parameters and matrices of $S_L, S_R$. The BCFW parameters and matrix of $S_{\rcp}$ are 
\begin{align*}
&\coord_{\rcp}\;:=\;\coord_{L }\cup \{\calpha_k, \cbeta_k, \cdelta_k, \cgamma_k, \cepsilon_k\} \cup \coord_{R} \;=\; \{\calpha_i, \cbeta_i, \cdelta_i, \cgamma_i, \cepsilon_i\}_{i=1}^k
\\
& \mtx_\rcp\;:=\;\mbcfw(\mtx_L, [\calpha_k: \cbeta_k: \cdelta_k: \cgamma_k: \cepsilon_k], \mtx_R).
\end{align*}
where in a slight abuse of notation, here $\mbcfw(\dots)$ denotes the matrix $(*)$ from \cref{fig:promotion-matrix} rather than its rowspan.
	If $\st \in \{\pre_{I_k}, \cyc, \refl\}$, then we define $\coord_\rcp:= \coord_{\rcpp}$ and $\mtx_{\rcp}:= \st(\mtx_{\rcpp})$.
	
We can think of the BCFW parameters
$\calpha_i, \cbeta_i,\cgamma_i, \cdelta_i, \cepsilon_i$ (for $1\leq i \leq k$) 
as abstract variables, but in 
\cref{cor:dim}, we will show that when they range over the positive real numbers,
the corresponding BCFW matrices 
	sweep out the BCFW cell $S_{\rcp}$.

\end{definition}

In words, $\mtx_\rcp$ is precisely the matrix obtained by following the sequence of operations given by~$\rcp$, and the BCFW parameters are exactly the $5k$ parameters used in the applications of $\mbcfw$ as in \cref{def:bcfw-map}. The entries of $\mtx_\rcp$ are rational functions of the BCFW parameters $\coord_\rcp$ and may depend on the recipe $\rcp$ chosen. Each row of the BCFW matrix is naturally indexed by a step-tuple in~$\rcp$. We will frequently use $\zeta_i \in \{\calpha_i, \cbeta_i,\cgamma_i, \cdelta_i, \cepsilon_i\}$ (for $1\leq i \leq k$) to denote one BCFW parameter.

\begin{example}[BCFW parameters and matrix] \label{ex:bcfw_matrix}
We build the BCFW matrix $M_\rcp$ of the cell $S_\rcp$ from \cref{ex:recipe} in terms of the BCFW parameters $\{\czeta_i\}$. We will denote $(10,11,12)$ as $(A,B,C)$.

\begin{itemize}
    \item (Step $1$): after last step of the first step-tuple: 
    \begin{equation*}
  M_L^{(1)}=      \begin{blockarray}{cccccc}
2 & 3 & 4 & 5 & 6 & C \\
\begin{block}{(cccccc)}
  0 & \calpha_1 & \cbeta_1 & \cgamma_1 & \cdelta_1 & \cepsilon_1 \\
\end{block}
\end{blockarray}
    \end{equation*}
  \item (Step $2$): after the first step of the second step-tuple $M_L^{(2)}=\iota_{\bowtie}(0,[\calpha_2: \cbeta_2 :\cgamma_2 : \cdelta_2:\cepsilon_2], M_L^{(1)})$:   

 \begin{equation*}
  M_L^{(2)}= \begin{blockarray}{ccccccc}
1 & 2 & 3 & 4 & 5 & 6 & C \\
\begin{block}{(ccccccc)}
\calpha_2 & \cbeta_2 & 0 & 0 & -\cgamma_2 & -\cdelta_2 & -\cepsilon_2 \\
  0 & 0 & \calpha_1 & \cbeta_1 & \cgamma'_1 & \cdelta'_1 & \cepsilon_1 \\
\end{block}
\end{blockarray},
    \end{equation*}
where: $\cdelta'_1=\cdelta_1+\frac{\cdelta_2}{\cepsilon_2}\cepsilon_1$ and $\cgamma'_1=\cgamma_1+\frac{\cgamma_2}{\cdelta_2}\cdelta'_1$. In Step $3$ we apply $\mbox{cyc}^2$ and $\mbox{refl}$ to $M_L^{(2)}$ and get $M_L^{(3)}$.
\item (Step $4$): after the last step of the third step-tuple: 
    \begin{equation*}
  M_R=      \begin{blockarray}{ccccccc}
6 & 7 & 8 & 9 & A & B & C \\
\begin{block}{(ccccccc)}
 \calpha_3 & \cbeta_3 & \cgamma_3 & \cdelta_3 & 0 & \cepsilon_3 & 0\\
\end{block}
\end{blockarray}
    \end{equation*}

\item (Step $5$):  after the first step of the fourth step-tuple $M^{(5)}=\iota_{\bowtie}(M_L^{(3)},[\calpha_4: \cbeta_4 :\cgamma_4 : \cdelta_4:\cepsilon_4], M_R):$

\begin{equation*}
  M^{(5)}= \begin{blockarray}{cccccccccccc}
1 & 2 & 3 & 4 & 5 & 6 & 7 & 8 & 9 & A & B &C \\
\begin{block}{(cccccccccccc)}
-\cgamma_2 & 0 & 0 & \cbeta_2 & \calpha'_2 & \cepsilon_2 & 0 & 0 & 0 & 0 & 0 &\cdelta_2 \\
\cgamma'_1 & \cbeta_1 & \calpha_1 & 0& -\mu_1 & -\cepsilon_1 & 0 & 0 & 0 & 0 & 0 &-\cdelta'_1 \\
0 & 0 & 0 & 0& \calpha_4 & \cbeta_4 & 0 & 0 & 0 & -\cgamma_4 & -\cdelta_4 &-\cepsilon_4 \\
0 & 0 & 0 & 0 & 0 & \calpha_3 & \cbeta_3 & \cgamma_3 & \cdelta_3 & \nu_3 & \cepsilon_3 & 0\\
\end{block}
\end{blockarray},
    \end{equation*}
where: $\calpha'_2=\calpha_2+\frac{\calpha_4}{\cbeta_4}\cepsilon_2$, $\mu_1=\frac{\calpha_4}{\cbeta_4}\cepsilon_1$, and $\nu_3=\frac{\cgamma_4}{\cdelta_4}\cepsilon_3$. Finally, in order to get $M_{\rcp}$, we apply $\mbox{cyc}^4$ and $\mbox{refl}$ to $M^{(5)}$.    
\end{itemize}  
\end{example}

We now show that allowing the BCFW parameters to vary in $\mtx_\rcp$ gives a parametrization of $S_\rcp$.

\begin{proposition}\label{cor:dim}
Let $\rcp$ be a recipe for a BCFW cell $S_\rcp$ with BCFW matrix $\mtx_\rcp$. Then the map
$$ \left(\chR\right)^k \;\to\; S_\rcp $$
$$ \left([\calpha_i: \cbeta_i:\cgamma_i: \cdelta_i: \cepsilon_i]\right)_{i=1}^k \;\mapsto\; M_\rcp
$$
\\[-1em]
sending the collection of BCFW parameters to the BCFW matrix $M_{\rcp}$ is a bijection, where as usual we identify $M_{\rcp}$ with its row span.
	In particular, $S_\rcp$ has dimension 4k.
\end{proposition}
Given a recipe $\rcp$, the BCFW parameters $\coord_{\rcp}$ of a point $V \in S_\rcp$ are given (up to rescaling) by the preimage of $V$ under the bijection above.

\begin{proof} We proceed by induction on the number of steps in $\rcp$. The base case is when $\rcp$ is trivial, and is trivially true. If $\st= \pre_{I_k}, \cyc$ or $\refl$, the proposition for $S_\rcp$ easily follows from the proposition for $S_{\rcpp}$. If $\st=(a_k,b_k,c_k,d_k,n_k)$, then this follows from the inductive hypothesis on $S_L, S_R$ and \cref{prop:butterfly-matrix-same}, which we may apply because \cref{cor:4bidden} shows that BCFW cells always satisfy the coindependence assumptions of \cref{not:coindipendence}.
\end{proof}

\section{BCFW tiles and cluster adjacency}\label{sec:BCFWtilesfacets}
In this section we prove that the (closures of the) images of BCFW cells are 
tiles, and then state and prove the cluster adjacency theorem for general
BCFW tiles. We note that our proof relies on a 
few technical ingredients that we defer to later sections.

\subsection{BCFW tiles}

Recall the definition of the amplituhedron map $\tZ$ from \cref{defn_amplituhedron},
where $Z\in \Mat_{n,k+m}^{\scriptscriptstyle>0}$, and the definition of a tile from \cref{defn_tile}. 
Our goal for this section is to use the BCFW matrix and parameters to invert the amplituhedron map 
on the image $\gto{\rcp} = \tZ(S_{\rcp})$ of each BCFW cell. Since \cref{cor:dim} shows that BCFW cells in $\Grk$ have dimension $4k$, this will show that $\gt{\rcp}$ is a tile. 

The first step is to define one rational functionary for each BCFW parameter $\czeta_i \in \coord_\rcp$
where $\czeta \in\{\calpha,\cbeta,\cgamma,\cdelta,\cepsilon\}$. The product promotion map $\Psi_{ac}$ (\cref{pro-twistors2}) plays an integral role. The signs in the following definition are introduced such that each coordinate functionary is positive on the corresponding open tile (see \cref{thm:BCFW-tile-and-sign-description}).

Recall from \cref{def:pure} the degree $\deg_i F$ of a functionary $F$ in an index~$i$, and the operations $\refl^*, \cyc^{-*}$ on functionaries (cf. \cref{def:dihedral}, \cref{rmk:cyc-refl-on-functionaries}). 

\begin{definition}[Coordinate functionaries of BCFW cells]\label{def:twistor-mtx}
Let $S_\rcp \subset \Grk$ be a BCFW cell. For each BCFW parameter $\czeta_i \in \coord_\rcp$ we define the \emph{coordinate functionary} $\czeta^\rcp_i (Y)$ to be the function on $\Gr_{k,k+4}$ given by the following recursive definition: 
\begin{itemize}
\itemsep0.25em
\item If $\st= (a_k, b_k, c_k, d_k, n_k)$, then we define
\[\calpha_k^{\rcp}(Y) \,:=\, (-1)^{k_R}\llrr{b_k\, c_k \,d_k\, n_k} \quad \quad  \cbeta_k^{\rcp}(Y)\,:=\, (-1)^{k_R+1}\llrr{a_k\, c_k\, d_k\, n_k}\]
\[ \cgamma_k^{\rcp}(Y)\,:=\,  \llrr{a_k\, b_k\, d_k\, n_k} \quad \quad \cdelta_k^{\rcp}(Y)\,:=\, -\llrr{a_k\, b_k\, c_k\, n_k} \quad \quad \cepsilon_k^{\rcp}(Y)\,:=\,\llrr{a_k\, b_k\, c_k\, d_k},\]
and for $i \neq k$ \vspace{1em}\\
$\czeta^\rcp_i (Y) \;:=\; \begin{cases}
(-1)^s	\Psi_{a_k c_k}(\czeta_i^{L})(Y)\\
\Psi_{a_k c_k}(\czeta_i^{R})(Y)
\end{cases}$ if the $i$th step-tuple is in $\begin{cases}
  \rcp_L\\
  \rcp_R\\
\end{cases},$ where $s= (k_R+1) \deg_{n_k} \czeta^{L}_i (Y)$.

\vspace{1em}
\item If $\st=$ 
$\begin{cases}
    \refl \\
    \cyc \\
    \pre_{I_k}
\end{cases}$ then $\czeta^\rcp_i (Y) := \begin{cases}
    \refl^* \czeta^{\rcpp}_i (Y) \\
    (-1)^{s'}\cyc^{-*} \czeta^{\rcpp}_i (Y) \\
    \czeta^{\rcpp}_i (Y)
\end{cases}$, where $s'= k \deg_n \czeta^{\rcpp}_i (Y)$.

\end{itemize}
We define the 
	\emph{twistor matrix} $\twmt_\rcp(Y)$ be the matrix $\mtx_\rcp$ of \cref{def:BCFW-coords-mtx} where each BCFW parameter $\czeta_i$ is set to equal the coordinate functionary $\czeta_i^\rcp(Y)$. Its entries are rational functionaries.
\end{definition}

\begin{example}[Coordinate functionaries] \label{ex:coord_func}
We will build recursively the coordinate functionaries $\{\czeta_i^\rcp(Y)\}$ of the cell $S_\rcp$ from \cref{ex:recipe}. We will follow the same steps as in \cref{ex:bcfw_matrix}. We will denote $(10,11,12)$ as $(A,B,C)$ and $\{\calpha_i(Y),\cbeta_i(Y),\cgamma_i(Y),\cdelta_i(Y),\cepsilon_i(Y)\}$ as $\{\czeta_i(Y)\}$ for brevity.

\begin{itemize}
    \item (Step $1$): The coordinate functionaries $\{\zeta^{(1)}_1(Y)\}$ are respectively: 
    \begin{equation*}
        \llrr{4 \, 5 \,6 \,C}, -\llrr{3 \,5 \,6 \,C},\llrr{3 \,4 \,6 \,C}, -\llrr{3\,4 \,5 \,C}, \llrr{3\,4\,5\,6}.
    \end{equation*}
    \item (Step $2$): The coordinate functionaries $\{\czeta^{(2)}_2(Y)\}$ are respectively:
    \begin{equation*}
      -\llrr{2\,5\,6\,C}, \llrr{1\,5\,6\,C},-\llrr{1\,2\,6\,C}, \llrr{1\,2\,5\,C}, -\llrr{1\,2\,5\,6}.
    \end{equation*}
$\{\czeta^{(2)}_1(Y)\}$ are obtained as $\czeta^{(2)}_1=\Psi_{1,5}(\czeta^{(1)}_1)$ and are respectively:
    \begin{equation*}
\llrr{4\,5\,6\,C}, -\llrr{3\,5\,6\,C}, \frac{\llrr{3\,4\,C \br 5 \,6 \br 1\,2\,C}  }{\llrr{1\,2\,5\,C}},-\frac{\llrr{3\,4\,5 \br 2 \,1 \br 5\,6\,C}  }{\llrr{1\,2\,5\,6}}, \llrr{3\,4\,5\,6}.
    \end{equation*}

\item (Step $3$): $\{\czeta^{(3)}_1(Y)\}$ and $\{\czeta^{(3)}_2(Y)\}$ are obtained by performing $\cyc^{2}$ and $\refl$ on the respective coordinate functionaries of Step $2$.
    \item (Step $4$): The coordinate functionaries $\{\czeta^{(4)}_3(Y)\}$ are respectively:
     \begin{equation*}
        \llrr{7 \, 8\,9 \,B}, -\llrr{6 \,8 \,9 \,B},\llrr{6 \,7 \,9 \,B}, -\llrr{6\,7 \,8 \,B}, \llrr{6\,7\,8\,9}.
    \end{equation*}

\item (Step $5$):  The coordinate functionaries $\{\czeta^{(5)}_4(Y)\}$ are respectively:
       \begin{equation*}
       -\llrr{6 \, A\,B \,C}, \llrr{5 \, A\,B \,C},-\llrr{5 \, 6 \,B \,C}, \llrr{5 \, 6\,A \,C}, -\llrr{5 \, 6\,A \,B}.
    \end{equation*}
$\{\czeta^{(5)}_3(Y)\}$ are obtained as $\czeta^{(5)}_3=\Psi_{5,10}(\czeta^{(4)}_3)$ and are respectively: 
    \begin{equation*}
     \frac{\llrr{7\,8\,9 \br B \,A \br 5\,6\,C}  }{\llrr{5\,6\,A\,C}},
      -\frac{\llrr{6\,8\,9 \br B \,A \br 5\,6\,C}}{\llrr{5\,6\,A\,C}},
     \frac{\llrr{6\,7\,9 \br B \,A \br 5\,6\,C}}{\llrr{5\,6\,A\,C}},
    -\frac{\llrr{6\,7\,8 \br B \,A \br 5\,6\,C}}{\llrr{5\,6\,A\,C}},
 \llrr{6\,7\,8\,9}.
    \end{equation*}

$\{\czeta^{(5)}_2(Y)\}$ are obtained as $\czeta^{(5)}_2=\Psi_{5,10}(\czeta^{(3)}_2)$ and are respectively: 
    \begin{equation*}
-\frac{\llrr{1\,4\,C \br 5 \,6 \br A\,B\,C}  }{\llrr{5\,A\,B\,C}},
      \llrr{1\,5\,6\,C},
    - \llrr{4\,5\,6\,C},
   -\llrr{1\,4\,5\,6},
 \llrr{1\,4\,5\,C}.
    \end{equation*}
$\{\czeta^{(5)}_1(Y)\}$ are obtained as $\czeta^{(5)}_1=\Psi_{5,10}(\czeta^{(3)}_1)$ and are respectively:
\begin{align*}
    \frac{\llrr{1\,2\,C \br 5 \,6 \br A\,B\,C}}{\llrr{5\,6\,B\,C}}, -\frac{\llrr{1\,3\,C \br 5\,6 \br A \,B \, C}}{\llrr{5\, A \, B \, C}}, \\ - \frac{\llrr{A\,B\,C \br 5\,6 \br 2\, 3 \br 1\,C \br 4 \,5 \, 6}}{\llrr{1\,4\,5 \,6} \llrr{5\, A \, B \, C}},
    \frac{\llrr{1\,2\,3 \br 4\,5 \br 1\, 2 \br 5\,6 \br A \,B \, C}}{\llrr{1\,4\,5 \,C} \llrr{5\, A \, B \, C}}, -\llrr{1\,2\,3\,C}.
\end{align*}

Finally, $\{\czeta^{\rcp}_i(Y)\}$ are obtained by applying  $\cyc^{4},\refl$ on the coordinate functionaries of Step $5$.     
\end{itemize}  
The twistor matrix $\twmt_\rcp(Y)$ is obtained from the BCFW matrix $M_{\rcp}$ in \cref{ex:bcfw_matrix} by setting each BCFW parameter $\czeta^{\rcp}_i$ equal to the corresponding coordinate functionary $\czeta^{\rcp}_i(Y)$ computed above.
\end{example}

It is also convenient to define the following variant of the coordinate functionaries, which will be cluster variables for $\Gr_{4,n}$. The definition is based on rescaled product promotion $\rPsi$ (\cref{def:rPsi}). 

\begin{definition}[Coordinate cluster variables of BCFW cells]\label{def:generalcluster}
Let $S_\rcp \subset \Grk$ be a BCFW cell. 
For each BCFW parameter $\czeta_i \in \coord_\rcp$
the \emph{coordinate cluster variables} $\rzeta^\rcp_i$ 
is defined as follows:
\begin{itemize}
\itemsep0.25em
\item If $\st= (a_k, b_k, c_k, d_k, n_k)$, then we define 

\[\ralpha_k^{\rcp} \,:=\, \lr{b_k\, c_k \,d_k\, n_k} \quad  \quad \rbeta_k^{\rcp} \,:=\, \lr{a_k\, c_k\, d_k\, n_k} \]
\[\rgamma_k^{\rcp} \,:=\,  \lr{a_k\, b_k\, d_k\, n_k} \quad\quad \rdelta_k^{\rcp} \,:=\, \lr{a_k\, b_k\, c_k\, n_k} \quad\quad \repsilon_k^{\rcp} \,:=\, \lr{a_k\, b_k\, c_k\, d_k},\vspace{1em}\] 
 and for $i \neq k$, \quad
$\rzeta^\rcp_i \;:=\; \begin{cases}
\rPsi_{a_k c_k}(\rzeta_i^{L})\\
\rPsi_{a_k c_k}(\rzeta_i^{R})
\end{cases}$ if the $i$th step-tuple is in $\begin{cases}
  \rcp_L\\
  \rcp_R\\
\end{cases}.$
\vspace{1em}
\item If $\st=$ 
$\begin{cases}
    \refl \\
    \cyc \\
    \pre_{I_k}
\end{cases}$ then $\rzeta^\rcp_i := \begin{cases}
    \refl^* \rzeta^{\rcpp}_i  \\
    \cyc^{-*} \rzeta^{\rcpp}_i  \\
    \rzeta^{\rcpp}_i 
\end{cases}$.
\end{itemize}
We denote the set of coordinate cluster variables for $S_\rcp$ by $\Irr(\rcp)$. Note that $\Irr(\rcp)$ depends on the recipe $\rcp$ rather than just the BCFW cell.
\end{definition}

\begin{example}[Coordinate cluster variables] \label{ex:coord_clust}
We will build recursively the coordinate cluster variables $\{\rzeta_i^\rcp\}$ of the cell $S_\rcp$ from \cref{ex:recipe}. We will follow the same steps as in \cref{ex:bcfw_matrix} and \cref{ex:coord_func}. We will denote $(10,11,12)$ as $(A,B,C)$ and $\{\ralpha_i,\rbeta_i,\rgamma_i,\rdelta_i,\repsilon_i\}$ as $\{\rzeta_i\}$ for brevity.
\begin{itemize}
    \item (Step $1$): The coordinate cluster variables $\{\rzeta^{(1)}_1\}$ are respectively:
    \begin{equation*}
        \lr{4 \, 5 \,6 \,C}, \lr{3 \,5 \,6 \,C},\lr{3 \,4 \,6 \,C}, \lr{3\,4 \,5 \,C}, \lr{3\,4\,5\,6}.
    \end{equation*}
    \item (Step $2$): The coordinate cluster variables $\{\rzeta^{(2)}_2\}$ are respectively:
    \begin{equation*}
        \lr{2\,5\,6\,C}, \lr{1\,5\,6\,C},\lr{1\,2\,6\,C}, \lr{1\,2\,5\,C}, \lr{1\,2\,5\,6}.
    \end{equation*}
$\{\rzeta^{(2)}_1\}$ are obtained as $\rzeta^{(2)}_1=\rPsi_{1,5}(\rzeta^{(1)}_1)$ and are respectively:
    \begin{equation*}
     \lr{4\,5\,6\,C}, \lr{3\,5\,6\,C}, \lr{3\,4\,C \br 5 \,6 \br 1\,2\,C},\lr{3\,4\,5 \br 2 \,1 \br 5\,6\,C}, \lr{3\,4\,5\,6}.
    \end{equation*}
\item (Step $3$): $\{\rzeta^{(3)}_1\}$ and $\{\rzeta^{(3)}_2\}$ are obtained by performing $\cyc^{2}$ and $\refl$ on the respective coordinate cluster variables of Step $2$.
    \item (Step $4$): The coordinate cluster variables $\{\rzeta^{(4)}_3\}$ are respectively:
     \begin{equation*}
        \lr{7 \, 8\,9 \,B}, \lr{6 \,8 \,9 \,B},\lr{6 \,7 \,9 \,B}, \lr{6\,7 \,8 \,B}, \lr{6\,7\,8\,9}.
    \end{equation*}
    \item (Step $5$):  The coordinate cluster variables $\{\rzeta^{(5)}_4\}$ are respectively:
       \begin{equation*}
       \lr{6 \, A\,B \,C}, \lr{5 \, A\,B \,C},\lr{5 \, 6 \,B \,C}, \lr{5 \, 6\,A \,C}, \lr{5 \, 6\,A \,B}.
    \end{equation*}
$\{\rzeta^{(5)}_3\}$ are obtained as $\rzeta^{(5)}_3=\rPsi_{5,10}(\rzeta^{(4)}_3)$ and are respectively: 
    \begin{equation*}
     \lr{7\,8\,9 \br B \,A \br 5\,6\,C},
     \lr{6\,8\,9 \br B \,A \br 5\,6\,C},
     \lr{6\,7\,9 \br B \,A \br 5\,6\,C},
    \lr{6\,7\,8 \br B \,A \br 5\,6\,C},
 \lr{6\,7\,8\,9}.
    \end{equation*}

$\{\rzeta^{(5)}_2\}$ are obtained as $\rzeta^{(5)}_2=\rPsi_{5,10}(\rzeta^{(3)}_2)$ and are respectively: 
    \begin{equation*}
     \lr{1\,4\,C \br 5 \,6 \br A\,B\,C},
      \lr{1\,5\,6\,C},
     \lr{4\,5\,6\,C},
   \lr{1\,4\,5\,6},
 \lr{1\,4\,5\,C}.
    \end{equation*}
$\{\rzeta^{(5)}_1\}$ are obtained as $\rzeta^{(5)}_1=\rPsi_{5,10}(\rzeta^{(3)}_1)$ and are respectively:
\begin{align*}
    \lr{1\,2\,C \br 5 \,6 \br A\,B\,C}, \lr{1\,3\,C \br 5\,6 \br A \,B \, C}, \\
    \lr{A\,B\,C \br 5\,6 \br 2\, 3 \br 1\,C \br 4 \,5 \, 6},
    \lr{1\,2\,3 \br 4\,5 \br 1\, C \br 5\,6 \br A \,B \, C}, \lr{1\,2\,3\,C}.
\end{align*}
Finally, $\{\rzeta^{\rcp}_i\}$ are obtained by applying  $\cyc^{4},\refl$ on the coordinate cluster variables of Step $5$.    
\end{itemize}  
\end{example}

\begin{remark}[Chain polynomials and beyond]
In \cref{prop:explicit}, we will show that coordinate cluster variables for standard BCFW cells are always chain polynomials (cf. \cref{def:chain_polynomials}). However, this is not the case for general BCFW cells. To describe a specific example $S_\rcp$ on $N = \{1,\dots,\mathrm{C}\}$, we start with the standard BCFW cell corresponding to the chord diagram 
$$ D' \;=\; ((3,7,8,9),(8,9,\mathrm{A},\mathrm{B}),(1,2,\mathrm{A},\mathrm{B})) \quad \text{with} \quad N' = \{1,2,3,7,8,9,\mathrm{A},\mathrm{B},\mathrm{C}\}. $$
Using \cref{prop:explicit}, we have 
$ \rgamma_1^{D'} \;=\; \lr{\mathrm{C}\,\mathrm{A}\,\mathrm{B} \br 2\,1 \br 3\,7 \br 8\,9 \br \mathrm{A}\,\mathrm{B}\,\mathrm{C}} \,. $ 
This variable is preserved after an application of $\pre_{\{4,6\}}$. Then, after suitable reflections and rotations, we can apply upper promotion, which in terms of the original indices of $D'$ takes place at 
$ (a,b,c,d,n) \;=\; (5,4,8,7,6). $
The above cluster variable promotes using the rule $7 \mapsto 7 - \tfrac{\lr{5\,4\,7\,6}}{\lr{5\,4\,8\,6}}8$, and yields
$$ \rgamma_1^{\rcp} \;=\; \rPsi(\rgamma_1^{D'}) \;=\; \lr{\mathrm{C}\,\mathrm{A}\,\mathrm{B} \br 2\,1 \br 3\,7 \br 8\,9 \br \mathrm{A}\,\mathrm{B}\,\mathrm{C}} \lr{4\,5\,6\,8}+\lr{\mathrm{C}\,\mathrm{A}\,\mathrm{B} \br 2\,1 \br 3\,7 \, 9} \lr{ 8\, \mathrm{A}\,\mathrm{B}\,\mathrm{C}} \lr{4\,5\,6\,7}. $$
This expression further expands to an irreducible quartic polynomial in Pl\"ucker coordinates with $6$ terms, which cannot be written as a chain polynomial.
\end{remark}

\begin{lemma}\label{lem:irr-is-compatible-cluster}
	Let $\rcp$ be a recipe for a BCFW cell. Then the coordinate cluster variables $\Irr(\rcp)$ are a collection of compatible cluster variables for $\Gr_{4,n}$.
\end{lemma}

\begin{proof}
	We prove this by induction on the number of step-tuples in $\rcp$. The base case is trivial. If the final step is $\cyc, \refl$ or $\pre_i$, the statement follows from the inductive hypothesis and \cref{cor:cycrefl,lem:add-marker-embed-gr}, which assert that $\cyc^{-*}, \refl^*$ and the inclusion map send cluster variables to cluster variables and preserve compatibility. If the final step is $(a,b,c,d,n)$, then the definition of $\rPsi$ implies that $\Irr(\rcp)$ consists of cluster variables for $\Gr_{4,n}$. To see compatibility, note that by induction, $\Irr(\rcp_L) \sqcup \Irr(\rcp_R)$ is a set of compatible cluster variables for $\CC[\widehat{\Gr}_{4, N_L}] \times \CC[\widehat{\Gr}_{4, N_R}]$. By \cref{thm:promotion2} and the definition of $\rPsi$, applying $\rPsi$ to $\Irr(\rcp_L) \sqcup \Irr(\rcp_R)$ gives a collection of compatible cluster variables for $\Gr_{4,n}$. Moreover, these cluster variables are compatible with the Pl\"ucker coordinates 
		$\lr{a b c d}, \lr{a b c n}, \lr{a b d n}, \lr{a c d n}, \lr{b c d n}$,
	which are all frozen in $\Fr(\Sigma_1)$. This completes the proof of the lemma.
\end{proof}

We now state one of our main theorems. It says that the amplituhedron map $\tZ$ is injective on each BCFW cell $S_\rcp$, giving an inverse map in terms of the coordinate functionaries. This
shows that $\gt{\rcp}$ is a tile. Moreover, it  describes each open tile $\gt{\rcp}$ as a semi-algebraic set in $\Gr_{k, k+4}$. It will be proved in \cref{subsec:inverse_problem}..

\begin{theorem}[General BCFW cells give tiles] \label{thm:BCFW-tile-and-sign-description}
	Let 
		$S_\rcp$ be a general BCFW cell with recipe
		$\rcp$.  Then
	for all $Z \in \Mat_{n,k+4}^{>0}$, $\tZ$ is injective on 
		$S_\rcp$ and thus $\gt{\rcp}$ is a tile. 
		In particular,
		given $Y \in \tZ(S_\rcp)$,  the unique preimage of $Y$ in $S_\rcp$ is given by (the rowspan of) of the twistor matrix $\twmt_\rcp(Y)$. 
		Moreover,
		$$\gto{\rcp}= \{Y \in \Gr_{k, k+4}: \czeta_i^\rcp(Y)>0 \text{ for all coordinate functionaries of }S_\rcp\}.$$
\end{theorem}

\begin{remark}
    \label{rem:ELT-method-comparison} We briefly compare our methods with those of \cite{even2021amplituhedron}.  In this 
    article we 
    obtain representative matrices (``BCFW matrices") for general BCFW cells, by recursively using the BCFW product. 
   Our proof of injectivity of $\tZ$ on these cells involves recursively inverting the map, using cluster algebras and product promotion in an essential way. On the other hand, the focus of \cite{even2021amplituhedron} is exclusively on standard BCFW cells, which are a special case of general BCFW cells.  The authors of  
   \cite{even2021amplituhedron} obtain  representative \emph{domino matrices} for standard BCFW cells, 
   and they invert the amplituhedron map on standard BCFW cells using these domino matrices, without using the recursive product structure of these cells. 
\end{remark}

\cref{thm:BCFW-tile-and-sign-description} allows us to make the following definition.

\begin{definition}[BCFW tiles] \label{def:bcfwTiles}
	A \emph{general} (respectively, \emph{standard})  \emph{BCFW tile} $Z_S$
	is the closure 
	$\overline{\tZ(S)}$ of the image of a general (respectively, standard) BCFW cell
	$S$. 
If the BCFW cell has the form 
	$S_\rcp$ (respectively, $S_D$) for a recipe $\rcp$ (respectively, chord diagram $D$), then we also denote the corresponding
	tile by
	$Z_\rcp$ (respectively, $Z_D$).
\end{definition}

\begin{example}
 Given the cell $S_{\rcp} \subset \Gr_{4,12}$ obtained by the recipe of \cref{ex:recipe}, then $Z_\rcp$ is a tile for $\mathcal{A}_{12,4,4}(Z)$. 
  Moreover, $\gto{\rcp}$ is the region in 
  $\Gr_{4,8}$ where all the coordinate functionaries $\{\czeta^{\rcp}_i(Y)\}$ in \cref{ex:coord_func} are positive. 
\end{example}

\begin{example}
For $k=1$, the amplituhedron $\mathcal{A}_{n,1,4}(Z)$ is a cyclic polytope in $\mathbb{P}^4$ whose vertices are the $n$ rows of $Z$. The $k=1$ BCFW tiles are exactly the full-dimensional simplices whose vertices are vertices of the cyclic polytope $\mathcal{A}_{n,1,4}(Z)$. In particular, the tile corresponding to the BCFW cell on the right in \cref{fig:bcfwcellsk1} is the simplex with vertices $Z_a,Z_b,Z_c,Z_d,Z_e$.
\end{example}

\begin{table}[h] 
\centering 
\begin{tabular}{|l||c|c|c|c|c|c|c|c|c|c|} 
\hline 
\diagbox{$n$}{$k$} & $0$ & $1$ & $2$ & $3$ & $4$ & $5$  & $ 6$ & $ 7$ & $ 8$ & $ 9$ \\
\hline\hline 
$5$ & $1$ & $1$ & $ $ & $ $ & $ $ & $ $ & $ $ & $ $ & $ $ & $ $ \\
\hline
 $6$ & $1$ & $ 6$ & $ 1$ & $ $ & $ $ & $ $ & $ $ & $ $ & $ $ & $ $ \\
 \hline 
 $7$ & $1$ & $21$ & $ 21$ & $ 1$ & $ $ & $ $ & $ $ & $ $ & $ $ & $ $\\
\hline 
 $8$ & $1$ & $56$ & $176$ & $56 $ & $ 1$ & $ $ & $ $ & $ $ & $ $ & $ $\\
\hline 
 $9$ & $1$ & $126$ & $939$ & $939$ & $126$ & $ 1$ & $ $ & $ $ & $ $ & $ $ \\
\hline 
$10$ & $1$ & $252$ & $3785$ & $8690 $ & $ 3785$ & $ 252$ & $ 1$ & $ $ & $ $ & $ $ \\
\hline  $11$ & $1$ & $462$ & $12562$ & $55847 $ & $ 55847$ & $ 12562$ & $ 462$ & $ 1$ & $ $ & $ $ \\
\hline  $12$ & $1$ & $792$ & $36102$ & $279148 $ & $ 534258$ & $ 279148$ & $ 36102$ & $ 792$ & $ 1$ & $ $ \\
\hline $13$ & $1$ & $1287$ & $92807$ & $1158664$ & $3795064$ & $3795064$ & $1158664$ & $92807$ & $1287$ & $1$ 
\\
\hline
\end{tabular}
\vspace{1em}
\caption{Number of BCFW tiles of $\mathcal{A}_{n,k,4}(Z)$. For comparison, the standard BCFW tiles of $\mathcal{A}_{13,4,4}(Z)$ consists of 
5292 out of the 3795064 general BCFW tiles.} 
\label{tab:Schroeder}
\end{table}

\begin{notation} \label{not:pluckfunc}
Given a cluster variable $x$ in $\Gr_{4,n}$, we will denote as $x(Y)$ the functionary on $\Gr_{k,k+4}$ obtained by identifying Pl\"ucker coordinates $\lr{I}$ on $\Gr_{4,n}$ with twistor coordinates $\llrr{I}$ on $\Gr_{k,k+4}$, see \cref{sec:Pluckertwistor}. When it is clear from the context, we will talk about cluster variables and the corresponding functionaries interchangeably. 
\end{notation}

\cref{thm:BCFW-tile-and-sign-description} gives a description of each open BCFW tile as a semi-algebraic set in $\Gr_{k,k+4}$, cut out by the positivity of the coordinate functionaries. Each coordinate functionary is a (signed) Laurent monomial in the coordinate cluster variables, and vice versa. Using \cref{thm:BCFW-tile-and-sign-description} and the ingredients in its proof, we can obtain a description of the open tile using cluster variables.

\begin{corollary}[Sign description for general BCFW tiles]\label{cor:cluster-sign-description} 
	Let $\gt{\rcp}$ be a general BCFW tile. For each element $x$ of $\Irr(\rcp)$, the functionary $x(Y)$ has a definite sign $s_x$ on $\gto{\rcp}$ and 
	\[\gto{\rcp}= \{Y \in \Gr_{k,k+4}: s_x \, x(Y) >0 \text{ for all } x \in \Irr(\rcp) \}.\]
\end{corollary}

\begin{proof}
	\cref{cor:clust-var-strong-sign} shows that for each element $x$ of $\Irr(\rcp)$, the functionary $x(Y)$ has a definite sign $s_x$ on $\gto{\rcp}$. This implies that 
		$\gto{\rcp}\subset \{Y \in \Gr_{k,k+4}: s_x \, x(Y) >0 \text{ for all } x \in \Irr(\rcp) \},$
		so all that remains is the reverse inclusion. By \cref{thm:BCFW-tile-and-sign-description}, $\gto{\rcp}$ is the subset of $\Gr_{k,k+4}$ where the coordinate functionaries are positive, so it suffices to show that the coordinate functionaries are positive on the points in the right hand set.
		
		Suppose $Y \in \Gr_{k,k+4}$ satisfies $s_x x(Y)>0$ for all coordinate cluster variables.
 Comparing the definition of the coordinate functionaries in \cref{def:twistor-mtx} to the definition of the coordinate cluster variables in \cref{def:generalcluster}, we see that each coordinate functionary is a signed Laurent monomial in the coordinate cluster variables. Thus, the signs $s_x$ of the coordinate cluster variables imply that each coordinate functionary $\zeta$ also has a particular sign on $Y$. This sign must be positive, since $x(Y')$ has sign $s_x$ on any $Y'$ in the tile $\gto{\rcp}$ and every coordinate functionary is positive on $Y'$.
\end{proof}

One may compute the signs $s_x$ in \cref{cor:cluster-sign-description}
 recursively using \cref{thm:signs-under-cyc-rot-pre,prop:vanishing_and_sign_of_functionaries_under_promotion,rem:extend}. 

\begin{remark}[Relevance to physics] \label{rk:CZ_eq}
A connection between certain positroid cells in $\Gr^{\geq 0}_{k,n}$ and cluster variables in $\Gr_{4,n}$ was explored in \cite{Mago:2020kmp, He:2020uhb}. The authors consider a positive parametrisation of positroid cells using edge or face variables of an associated plabic graph. Then by solving the `$C \cdot z$ equations', they express these variables as ratios of elements of $\CC[\Gr_{4,n}]$. 
They observe that most of the polynomials appearing are cluster variables for $\Gr_{4,n}$ and are relevant for the \emph{symbol alphabet} of loop amplitudes in $\mathcal{N}=4$ SYM. However, they also note that there are polynomials which are \emph{not} cluster variables.
We observe that for a positroid cell $S$ which gives a tile $Z_S$, solving the `$C \cdot z$ equations' is equivalent to inverting the amplituhedron map $\tilde{Z}$ on $Z_S$. In this section (see \cref{def:generalcluster}, \cref{thm:BCFW-tile-and-sign-description} and \cref{cor:cluster-sign-description}) we showed that, parametrising a BCFW cell $S_\rcp$ using BCFW parameters $\zeta_i \in \mathcal{C}_\rcp$, one may express all $\zeta_i$ as ratios of cluster variables in $\Gr_{4,n}$ (thought as functionaries in $\mbox{Gr}_{k,k+4}$). Morover, changing the recipe $\rcp$ for a BCFW cell may give a different set of cluster variables.
It would be interested to explore the relevance of all such cluster variables for the symbol alphabet and loop amplitudes.    
\end{remark}

\begin{example}[Sign description of BCFW tiles]
 Consider the tile $Z_\rcp$ in $\mathcal{A}_{12,4,4}(Z)$, where $\rcp$ is the recipe from \cref{ex:recipe}.
  Then, $\gto{\rcp}$ is the region in 
  $\Gr_{4,8}$ where all the coordinate cluster variables $\{\rzeta^{\rcp}_i\}$ in \cref{ex:coord_clust} -- thought as functionaries -- have a certain definite sign. Note that their signs is determined by the positivity of the coordinate functionaries in \cref{ex:coord_func} (and vice-versa).
\end{example}

\subsection{Cluster adjacency for BCFW tiles}
Now we turn to cluster adjacency for general BCFW tiles. Recall from \cref{def:facet2} the notion of a facet of $\gt{\rcp}$. We will use the following theorem (a special case of \cref{thm:signs-under-cyc-rot-pre,prop:vanishing_and_sign_of_functionaries_under_promotion}) to find functionaries which vanish on facets of $\gt{\rcp}$. 
Recall \cref{rem:Z} and the definition of product promotion $\Psi$ from \cref{pro-twistors2}.

\begin{theorem}\label{thm:vanishing-for-all-ops}
	\begin{enumerate}
		\item 		Let $S \subset \Grk$ be a positroid cell. Suppose that $F$ is a pure rational functionary which vanishes on $\gt{S}$ for all $Z$. Then
		\[
		\begin{cases}
			F\\
			\cyc^{-*}F\\
			\refl^* F
		\end{cases} \text{ vanishes on }
		\begin{cases}
			\gt{\pre_I S}\\
			\gt{\cyc S}\\
			\gt{\refl S}
		\end{cases}, \quad \mbox{for all } Z.
		\]
		\item 	Let $S_L \subset \Gr_{k_L, N_L}$ and $S_R \subset \Gr_{k_R, N_R}$ be positroid cells as in \cref{not:coindipendence}. Let $F$ be a pure functionary with indices contained in $N_L$ (resp. $N_R$) which vanishes on $\gt{S_L}$ (resp. $\gt{S_R}$) for all $Z$. Then $\Psi (F)$ vanishes on $\gt{S_L \bcfw S_R}$ for all $Z$.
	\end{enumerate}
\end{theorem}
 
 Using \cref{thm:vanishing-for-all-ops} and \cref{lem:boundaries_before_amp_map}, 
 we can prove the following theorem.

\begin{theorem}[Cluster adjacency for general BCFW tiles]
	\label{thm:clusteradjacency}
	Let $\gt{\rcp}$ be a general BCFW tile of 
	$\Ank$. Each facet $\gt{S}$ of $\gt{\rcp}$ lies on a hypersurface cut out by a functionary $F_S(\llrr{I})$ such that $F_S(\lr{I}) \in \Irr(\rcp)$. Thus $\{F_S(\lr{I}): \gt{S} \text{ a facet of }\gt{\rcp}\}$ consists  of compatible cluster variables of $\Gr_{4,n}$.
\end{theorem}

\begin{proof}
	Note that by \cref{lem:irr-is-compatible-cluster}, $\Irr(\rcp)$ consists of compatible cluster variables for $\Gr_{4,n}$ (under the usual identification of Pl\"ucker coordinates and twistor coordinates). So the final sentence of the theorem statement follows from the second sentence. To prove the second sentence, we will show the following claim.
 
 \noindent \textbf{Claim:} Let $S \subset \overline{S_\rcp}$ be a cell in the boundary of $S_\rcp$. Then for all $Z$, $\gt{S}$ lies in the vanishing locus of an element of $\Irr(\rcp)$.
 
 We prove the claim by induction on the number of steps in $\rcp$. We will use the notation of \cref{not:L-and-R}. For the base case ($k=0$ and $\rcp=\emptyset$), there is nothing to prove. If $\st=\pre_I$, then the map $\pre_I$ is a stratification-preserving homeomorphism from $\overline{S_\rcpp}$ to $\overline{S_\rcp}$. That is, any cell in the boundary of $S_\rcp$ is of the form $\pre_I (S)$ for some $S$ in the boundary of $S_\rcpp$. By the inductive hypothesis, there is some functionary $F \in \Irr(\rcpp)$ which vanishes on $\gt{S}$ for all $Z$. By \cref{thm:vanishing-for-all-ops}, $F$ also vanishes on $\gt{\pre_I S}$ for all $Z$. By \cref{def:generalcluster}, $F$ is also an element of $\Irr(\rcp)$, so the claim holds in this case. The arguments for $\st=\cyc$ or $\st=\refl$ are analogous.
 
 If $\st=(a,b,c,d,n)$, then by \cref{lem:boundaries_before_amp_map} 
	the boundary of $S_\rcp$ consists of (1) cells of the form $S_L' \bcfw S_R'$ where $S_L' \subseteq \overline{S_L}, S_R' \subseteq \overline{S_R}$, $(S_L', S_R') \neq (S_L, S_R)$ and $\{a,b,n\}, \{b,c,d,n\}$ are coindependent for $S_L', S_R'$, respectively; and 
 (2) cells for which some element of $\binom{\{a,b,c,d,n\}}{4}$  fails to be coindependent.
 
 Consider a cell $S_L' \bcfw S_R'$ appearing in (1) and suppose $S_L' \neq S_L$. By the inductive hypothesis on $S_L$, there is a functionary $F \in \Irr(L)$ which vanishes on $\gt{S_L'}$ for all $Z$. By \cref{thm:vanishing-for-all-ops}, the promotion $\Psi(F)$ vanishes on $\gt{S_L' \bcfw S_R'}$ for all $Z$. By \cref{lem:sign_of_chord_twistors}, the twistor coordinates $\llrr{bcdn}$, $\llrr{acdn}$, $\llrr{abdn}$, $\llrr{abcn}$, $\llrr{abcd}$ are nonvanishing on $\gt{S_L' \bcfw S_R'}$. So in fact the rescaled product promotion $\rPsi(F)$ vanishes on $\gt{S_L' \bcfw S_R'}$ for all $Z$. By \cref{def:generalcluster}, $\rPsi(F)$ is also an element of $\Irr(\rcp)$. This proves the claim for boundary cells of the form $S_L' \bcfw S_R'$ where $S_L' \neq S_L$. An analogous argument 
	shows the claim for the boundary cells of the form $S_L \bcfw S_R'$. 
 
Now, let $S$ be a cell in (2), and say $I$ is an element of $\binom{\{a,b,c,d,n\}}{4}$ which is not coindependent for $S$. This means that every basis of $S$ has nonempty intersection with $I$. Consider the twistor coordinate $\llrr{I}$. for all $Z$, every term in the sum \eqref{eq:cauchy-binet} is equal to zero. So $\llrr{I}$ vanishes on $\gt{S}$ for all $Z$. By \cref{def:generalcluster}, $\llrr{I}$ is in $\Irr(\rcp)$ (it is one of the five twistor coordinates added). This concludes the proof of the claim.

	Now, the claim implies the theorem. Indeed, any facet of $\gt{\rcp}$ is by definition $\gt{S}$ for some cell $S \subset \overline{S_\rcp}$. By the claim, $\gt{S}$ lies in the vanishing locus of an element of $\Irr(\rcp)$ for any choice of $Z$.
\end{proof}

We believe that one can generalize
\cref{thm:clusteradjacency} 
and \cref{cor:cluster-sign-description}  as follows.

\begin{conjecture}[Cluster adjacency and positivity tests for tiles] \label{conj:clusteradjmain}
		Let $Z_{S}$ be a tile of the amplituhedron 
		$\mathcal{A}_{n,k,m}(Z)$.  
  \begin{enumerate}[label=(\roman*)]
      \item (Cluster adjacency): Then for each facet
		$Z_{S'}$ of $Z_{S}$, there is a functionary
		$F_{S'}(\llrr{I})$ which vanishes on $Z_{S'}$, such that the collection
		$$\mathcal{F} = \{F_{S'}(\lr{I}): Z_{S'} \text{ a facet of }Z_{S}\}$$
		is a collection of compatible cluster variables for 
		$\Gr_{m,n}$. 
  \item (Positivity test): Moreover for even $m$, given any extended cluster $\txx$ 
		for $\Gr_{m,n}$ containing
		$\mathcal{F}$, there are signs
		$s_x\in \{1,-1\}$ for each $x \in \txx$, such that 
		\begin{equation*}
			\gto{S}= \{Y \in \Gr_{k, k+m}: s_x \cdot x(Y)>0 \text{ for all }x \in \txx \},
		\end{equation*}
		where as in \cref{not:pluckfunc}, $x(Y)$ is the functionary corresponding to the cluster variable $x$. 
  \end{enumerate}
		\end{conjecture}
We will prove \cref{conj:clusteradjmain} for standard BCFW tiles in \cref{sec:quiver} (see \cref{thm:sign-definite}), 
and for a non-BCFW tile (the `spurion') in our 
companion paper \cite{companion} for the $m=4$ amplituhedron.

\section{Cluster variables for standard BCFW tiles}
\label{sec:clustervariable}
   
In this section, we focus on standard BCFW tiles $\gt{D}$. We give explicit formulas for the coordinate cluster variables $\Irr(D)$ (cf.~\cref{def:generalcluster}) in \cref{prop:explicit}. 
We give explicit formulas for the sign of each element of $\Irr(D)$ on the open tile $\gto{D}$ in \cref{prop:domino-var-signs-on-tile}.
\subsection{Domino Cluster Variables}

We first rewrite \cref{def:generalcluster} in the standard BCFW case using the combinatorics of chord diagrams. For a standard BCFW tile $\gt{D}$, we call the coordinate cluster variables \emph{domino cluster variables} or simply \emph{domino variables.}
The terminology comes from 
\cite{karp2020decompositions} and \cite{even2021amplituhedron}, 
where some related variables were used for
 entries of matrices parametrizing BCFW cells.
 The term ``domino'' was used for such variables, since the 
 vectors and matrices tended to have adjacent pairs of nonzero same-sign entries, such as $\alpha,\beta$ and $\gamma,\delta$ in Figure~\ref{fig:promotion-matrix}.

\begin{definition}[Domino cluster variables]
\label{domino-var} 
Let $D$ be a chord diagram in $\CD$, and let 
$\ctop = (a,b,c,d)$ denote the rightmost top chord, where $1\leq a<b<c<d<n$,
and moreover $a,b$ and $c,d$ are consecutive.
We define the \emph{domino (cluster) variables }
	$\xx(D)=\{\ralpha_i, \rbeta_i, \rgamma_i, \rdelta_i, \repsilon_i\ \vert \ 1 \leq i \leq k\}$
recursively as follows, using
the rescaled product promotion $\rPsi$ from \Cref{def:rPsi}.

\begin{itemize}
\item[1.] (Removing the penultimate marker) If $d<n-1$, let $D'$ be the chord diagram on $[n] \setminus \{n-1\}$ 
obtained from $D$ by removing the marker $n-1$.
For any chord $D_j$ of $D$, and for 
$\rzeta\in \{\ralpha, \rbeta, \rgamma, \rdelta, \repsilon\}$:
$$ \rzeta_j^D \;=\; \rzeta_j^{D'} $$

\item[2.] (Removing the rightmost top chord) 
Otherwise, if $d=n-1$, 

the domino variables
of $D$ are obtained from those of the left and right subdiagrams $D_L, D_R$ (cf \cref{def:leftright}) as follows.

\begin{itemize}
\item[$\bullet$] 
The domino variables associated to the top chord $\ctop$ are given by the five Pl\"ucker coordinates:
$$ \ralpha_{\rtop}^D = \lr{b\,c\,d\,n} \quad
\rbeta_{\rtop}^D = \lr{a\,c\,d\,n} \quad 
\rgamma_{\rtop}^D = \lr{a\,b\,d\,n}  \quad
\rdelta_{\rtop}^D = \lr{a\,b\,c\,n} \quad 
\repsilon_{\rtop}^D = \lr{a\,b\,c\,d} $$
\item[$\bullet$]
For $\rzeta\in \{\ralpha, \rbeta, \rgamma, \rdelta, \repsilon\}$, and for $D'=D_L$ or $D_R$, with $D_j\in D'$, we define:
\begin{equation*}
\rzeta_j^D \;=\; 
\rPsi(\rzeta_j^{D'}).
\end{equation*}

\end{itemize}
\end{itemize}
We note that when the chord diagram $D$ is understood, we will often use 
$\rzeta_j$ to denote $\rzeta_j^D$.
\end{definition}

\begin{remark}\label{rmk:beta=alpha}
Suppose 
that the chord diagram $D$ contains a chord $D_j = (b,b+1,c,d)$, i.e., a~sticky same-end child of the top chord $\ctop$ according to the terminology in Definition~\ref{cd-terminology}. Then by \cref{def:rPsi}, $\rbeta_j^{D'} =\lr{bcdn}=\rbeta_j^D= \ralpha_{\rtop}^D$. This implies that in any chord diagram $D$, if $D_p$ has a sticky same-end child $D_j$, then $\rbeta_j^D= \ralpha_p^D$.
\end{remark}

\cref{domino-var} gives rise to  explicit formulas for the domino variables,
see \cref{prop:explicit}.
These formulas have different cases depending on whether certain chords are head-to-tail siblings, same-end parent and child, or sticky parent and child, using the terms introduced in Definition~\ref{cd-terminology}. We use the notation $x':= x-1$ and use the adjective ``sticky" to mean ``is a sticky child of some chord".

\begin{notation}\label{notation:arrow}
Let $D \in\CD$ be a chord diagram with chords $(a_1,b_1,c_1,d_1),\dots,(a_k,b_k,c_k,d_k)$,
and consider any chord $D_i = (a_i,b_i,c_i,d_i)$. 
We use the following notation for pieces of chain polynomials:
\begin{align*}
|\ldots \,x\,y \nearrow_i n \rangle \;&=\;
 |\ldots\,x\,y\br b_{(1)}\,a_{(1)} \br c_{(1)}\,d_{(1)} \br b_{(2)}\,a_{(2)} \br c_{(2)}\,d_{(2)} \br \cdots \br b_{(m)}\,a_{(m)} \br c_{(m)}\,d_{(m)} \,n \rangle
\\
\langle \,n\nwarrow_i x\,y\, \ldots | \;&=\; \langle n \, c_{(m)}\,d_{(m)} \br b_{(m)}\,a_{(m)} \br \cdots \br c_{(2)}\,d_{(2)} \br b_{(2)}\,a_{(2)}  \br c_{(1)}\,d_{(1)} \br b_{(1)}\,a_{(1)} \br x\,y\, \ldots |,
\end{align*}
where the chords $D_{(1)}=(a_{(1)},b_{(1)},c_{(1)},d_{(1)})$, $\dots$, $D_{(m)}=(a_{(m)},b_{(m)},c_{(m)},d_{(m)})$ are the following possibly-empty chain of ancestors of $D_i$, ordered bottom to top: $D_{(1)}$ is the lowest ancestor of $D_i$ that does not end in $(x,y)$, and $D_{(r+1)}$ is the lowest ancestor of $D_{(r)}$ that does not end at $(c_{(r)}, d_{(r)})$, i.e. is not same-end with $D_{(r)}$. 

We also set
\begin{equation*}
\rightarrowp_\ell \, n \;=\; \begin{cases}
\nearrow_\ell n & \text{ if } D_\ell \text{ is not sticky}\\
-a'_{\ell} & \text{ otherwise  }

\end{cases}, \qquad \qquad n \, \leftarrowp_\ell \;=\; \begin{cases}
n\,\nwarrow_\ell & \text{ if } D_\ell \text{ is not sticky}\\
-a'_{\ell} & \text{ otherwise  }

\end{cases}.
\end{equation*}
\end{notation}

\begin{theorem}[Domino cluster variables as chain polynomials]
	\label{prop:explicit}
Let $D \in\CD$ be a chord diagram 
and use \cref{notation:arrow}.
The domino cluster variables $\xx(D)$ are given by the following chain polynomials:

\begin{align*}
\ralpha_i \;&=\;
\lr{b_i\,c_i\,d_i \nearrow_i n}
\\
\rbeta_i \;&=\;
\begin{cases}
\ralpha_p & \text{if $D_i$ has sticky same-end parent $D_p$}
\\[0.2em]
\lr{a_i\,c_i\,d_i \rightarrowp_i n} \;\;\;\;\;\;\;\;\;\;\;\;\;\;\;\;\;\; & \text{otherwise}
\end{cases}\\
\rgamma_i \;&=\;
\begin{cases}
\lr{n \leftarrowp_i a_i\,b_i \br c_i\,d_i \br c_j\,d_j \nearrow_j n} & \text{if $D_i$ has head-to-tail sibling $D_j=(c_i, d_i, c_j, d_j)$}
\\[0.2em]
\lr{n \leftarrowp_p a_p\,b_p \br a_i\,b_i \br c_i\,d_i \nearrow_i n}
& \text{if  $D_i$ has same-end parent $D_p$ and $D_i$ not sticky}
\\[0.2em]
\lr{n \leftarrowp_p \,b_i \,a_i  \, a_i'}
& \text{if $D_i$ has same-end parent $D_p$ and $D_i$ sticky}
\\[0.2em]
\lr{a_i\,b_i\,d_i \rightarrowp_i n} \;\;\;\;\;\;\;\;\;\;\;\;\;\;\;\;\;\; & \text{otherwise}
\end{cases}
\\
\rdelta_i \;&=\; 
\lr{a_i\,b_i\,c_i \rightarrowp_i n}.
\\
\repsilon_i \;&=\; \lr{a_i\,b_i\,c_i\,d_i}
\end{align*}
\end{theorem}

\begin{example}[Domino cluster variables] 
\label{domino-variables-formulas}
Using \cref{prop:explicit}, the domino cluster variables $\xx(D)$ for the chord diagram $D$ in \cref{cd-example} are as follows. We will denote $(10,11,12,13,14,15)$ as $(A,B,C,D,E,F)$. For the rightmost top chord $D_6 = (8,9,D,E)$:
\begin{align*}
& \ralpha_6 = \lr{9\,D\,E\,F} 
&& \rbeta_6 = \lr{8\,D\,E\,F}
&& \rgamma_6 = \lr{8\,9\,E\,F}
&& \rdelta_6 = \lr{8\,9\,D\,F}
&& \repsilon_6 = \lr{8\,9\,D\,E}
\end{align*}
which is also immediate from Definition~\ref{domino-var}. The other top chord $D_3 = (1,2,8,9)$ follows the same cases, except for $\rgamma_3$ which is affected by its head-to-tail sibling $D_6$ on the right:
\begin{align*}
& \ralpha_3 = \lr{2\,8\,9\,F} 
&& \rbeta_3 = \lr{1\,8\,9\,F}
&& \rgamma_3 = \lr{F\,1\,2\,|\,8\,9\,|\,D\,E\,F}
&& \rdelta_3 = \lr{1\,2\,8\,F}
&& \repsilon_3 = \lr{1\,2\,8\,9}
\end{align*}
The chord $D_5 = (9,A,C,D)$ is sticky. The variable $\ralpha_5$ demonstrates the $\nearrow_5$ arrow notation, while $\rbeta_5$, $\rgamma_5$, $\rdelta_5$ are affected by $D_5$ being sticky, e.g.
$\rbeta_5=\lr{9\,C\,D \rightarrowp_5 F}=
\lr{9\,C\,D\,(-8)}=\lr{8\,9\,C\,D}.$
\begin{align*}
& \ralpha_5 = \lr{A\,C\,D\,|\,9\,8\,|\,D\,E\,F} 
&& \rbeta_5 = \lr{8\,9\,C\,D}
&& \rgamma_5 = \lr{8\,9\,A\,D}
&& \rdelta_5 = \lr{8\,9\,A\,C}
&& \repsilon_5 = \lr{9\,A\,C\,D}
\end{align*}
Then, $D_4 = (A,B,C,D)$ is a sticky child of its same-end parent $D_5$, which shows the third case for $\rgamma_4=\lr{n \, \leftarrowp_5 \, b_4 \, a_4 \, a'_4} =-\lr{8,B,A,9}$, where $n \,\leftarrowp_5=-a'_5$ in this case as $D_5$ is sticky:
\begin{align*}
& \ralpha_4 = \lr{B\,C\,D\,|\,9\,8\,|\,D\,E\,F} 
&& \rbeta_4 = \ralpha_5
&& \rgamma_4 = \lr{8\,9\,A\,B}
&& \rdelta_4 = \lr{9\,A\,B\,C}
&& \repsilon_4 = \lr{A\,B\,C\,D}
\end{align*}
Note that $D_5$ is omitted from the chain denoted by $\nearrow_4$ in $\ralpha_4$ since it ends in $(C,D)$. 

The next chord $D_2 = (5,6,8,9)$ is a nonsticky child of its same-end parent~$D_3$, and hence illustrates the second case for $\rgamma_2$:
\begin{align*}
& \ralpha_2 = \lr{6\,8\,9\,F}
&& \rbeta_2 = \lr{5\,8\,9\,F}
&& \rgamma_2 = \lr{F\,1\,2\,|\,5\,6\,|\,8\,9\,F}
&& \rdelta_2 = \lr{5\,6\,8\,|\,2\,1\,|\,8\,9\,F}
&& \repsilon_2 = \lr{5\,6\,8\,9}
\end{align*}
Finally, for the chord $D_1 = (3,4,5,6)$, the variable $\rgamma_1$ is given by a cubic chain polynomial, as $D_1$ is head-to-tail with $D_2$ and below $D_3$; moreover, $\langle n \, \leftarrowp_1|= \langle n \nwarrow_1|$ as $D_1$ is not sticky:
$$ \rgamma_1 \,=\, \lr{F\,\nwarrow_1\,3\,4\,|\,5\,6\,|\,8\,9\,\nearrow_1\,F} \,=\, \lr{F\,8\,9\,|\,2\,1\,|\,3\,4\,|\,5\,6\,|\,8\,9\,F} $$
and the remaining domino cluster variables are given by:
\begin{align*}
& \ralpha_1 = \lr{4\,5\,6\,|\,2\,1\,|\,8\,9\,F}
&& \rbeta_1 = \lr{3\,5\,6\,|\,2\,1\,|\,8\,9\,F}
&& \rdelta_1 = \lr{3\,4\,5\,|\,2\,1\,|\,8\,9\,F}
&& \repsilon_1 = \lr{3\,4\,5\,6}
\end{align*}
\end{example}





From \cref{prop:explicit}, we can conclude the following corollary, which will be useful in proofs.

\begin{corollary}\label{prop:invariant} 
	Let $D$ be a chord diagram on $\{1,\dots,c,d,n\}$ as in 
	\Cref{domino-var}. Then each domino {cluster} variable $\rzeta$  can be represented as 
	a chain polynomial which satisfies the following properties:
	\begin{enumerate}
		\item If $\rzeta$ has degree $1$, then 
		$d$ cannot appear without at least one of $n$ and $c$.
		\item If $\rzeta$ has degree $ \geq 2$, then
		\begin{enumerate}
			\item $n$ can only occur at the two end clauses:
			$\lr{n\ast \ast \br \dots }$ or 
			$\lr{\dots \br \ast \ast\ n}$.
			\item If $d$ occurs in a clause so do both $c$ and $n$. Thus $d$ occurs in an end clause.
		\end{enumerate}
	\end{enumerate}
\end{corollary}

We will prove \cref{prop:explicit} by induction on the number of chords, using the following key lemma.

\begin{lemma} \label{lem:when-promote-var-factor}
	Let $D \in \CD$ be a chord diagram,
	with rightmost top chord $\ctop=(a,b,c,d)$ and subdiagrams $D_L, D_R$ 
	 as in \cref{def:leftright}.  Suppose the formulas in \Cref{prop:explicit} 
	 and \cref{prop:invariant} 
	hold for the domino variables $\rzeta^{D'}$ for $D' \in \{D_L, D_R\}$. Then the numerator of $\Psi_{ac}(\rzeta^{D'})$ is irreducible and thus
	equals $\rzeta^D$ in all cases except
	\begin{enumerate}
		\item if $\ctop$ has a same-end sticky child $D_i$, then 
		\[\Psi_{ac}(\rgamma_i^{D'})= \dfrac{\ralpha_{\rtop}\rgamma_i^{D}}{\rdelta_{\rtop}} \qquad \text{and} \qquad \Psi_{ac}(\rdelta_i^{D'})= \dfrac{\ralpha_{\rtop}\rdelta_i^{D}}{\repsilon_{\rtop}},\]
		\item if $\ctop$ has a sticky child $D_i$ which is not same end, then
		\[\Psi_{ac}(\rbeta_i^{D'})=\dfrac{\ralpha_{\rtop}\rbeta_i^{D}}{\repsilon_{\rtop}}, \qquad \Psi_{ac}(\rgamma_i^{D'})=\dfrac{\ralpha_{\rtop}\rgamma_i^{D}}{(\repsilon_{\rtop})^s}, \qquad \Psi_{ac}(\rdelta_i^{D'})=\dfrac{\ralpha_{\rtop}\rdelta_i^{D}}{\repsilon_{\rtop}},\]
\item if $\ctop$ has a sticky child $D_i$ and $D_i$ has a same-end 
	
	child $D_j$, then
\[\Psi_{ac}(\rgamma_j^{D'})
	=\dfrac{\ralpha_{\rtop}\rgamma_j^{D}}{(\repsilon_{\rtop})^s}.\]
	\end{enumerate}
where $s \in \{1,2\}$ is the number of clauses of $\rzeta^{D'}$ where $n$ appears without $d$.
\end{lemma}
\begin{remark}
In \cref{lem:when-promote-var-factor}, 
we have that $s=2$ in (2) if $D_i$ has a head-to-tail sibling which is not same-end to $\ctop$; 
	and $s=2$ in (3) if $D_i, D_j$ are not same-end to $\ctop$.
\end{remark}

\begin{proof} [Proof of \cref{lem:when-promote-var-factor}]
	Recall from \cref{def:rPsi} that the 
	numerator $\nu$ of $\Psi_{ac}(\rzeta^{D'})$ 
	is equal to $\rzeta^D$ (a cluster variable for $\CC[\Gr_{4,n}]$) times a monomial in $\mathcal{T}'= \{\ralpha_{\rtop},\rbeta_{\rtop}, \rdelta_{\rtop}, \repsilon_{\rtop}  \}= \{\lr{bcdn}, \lr{acdn}, \lr{abcn}, \lr{abcd}\}$. If $\Psi(\rzeta^{D'})=\rzeta^{D'}$, then $\Psi(\rzeta^{D'})$ has no denominator and is equal to $\nu$. Because $\rzeta^{D'}$ is a cluster variable for $\CC[\Gr_{4,N_L}]$ or $\CC[\Gr_{4,N_R}]$, by \cref{lem:add-marker-embed-gr} it is also a cluster variable for $\CC[\Gr_{4,n}]$. Thus we have $\nu= \rzeta^{D}$.
	
So we restrict our attention to domino variables on which promotion is not the identity. 
We will characterize when $\nu$ factors, while matching up these
	cases with certain configurations of chords.

	First, suppose $D'=D_L$, a chord diagram on $\{1, 2, \dots, a, b, n\}$. We may assume that  $b$ appears in some clause of $\rzeta^{D'}$ without $a$, since otherwise promotion (in particular \eqref{eq:promotionvectors1}) sends $\rzeta^{D'}$ to itself. By \cref{prop:invariant}, $\rzeta^{D'}= \lr{ijbn}$ for some $i<j<a$. Applying (a2), we have 
	\[\Psi(\rzeta^{D'})= \frac{ \lr{i\,j\,n \br a\,b \br c\,d\,n}}{\lr{a\,c\,d\,n}} \quad \text{so} \quad \nu = \lr{i\,j\,n \br a\,b \br c\,d\,n}.\]
	By \cref{cor:pos}, $\nu$ is a cluster variable for $\Gr_{4,n}$ and thus is irreducible.
	
	Now, suppose that $D'=D_R$, a chord diagram on $\{b, b+1, \dots, c, d, n\}$. Notice that the substitution \eqref{eq:promotionvectors3} has an effect on $\rzeta^{D'}$ if and only if in some clause, $d$ appears without $c$, which by \cref{prop:invariant}, occurs if and only if $\rzeta^{D'}=\lr{ijdn}$ for $b \leq i<j <c$. Applying (b2) we obtain
	\[\Psi(\lr{i\,j\,d\,n}) = 
	\frac{\lr{n\, a\,b\br i\,j\br c\,d\,n}}{\lr{a\,b\,c\,n}} \quad \text{and} \quad \nu= \lr{n\, a\,b\br i\,j\br c\,d\,n}.\]
	If $i=b$, $\nu=-\lr{nabj}\lr{bcdn}= \ralpha_{\rtop} \lr{abjn}$. Otherwise, by \cref{cor:pos}, $\nu$ is a cluster variable for $\Gr_{4,n}$. Using \cref{prop:explicit}, $\rzeta^{D'}= \lr{bjdn}$ for $j \neq c$ precisely when $\rzeta=\rgamma_i$ for $D_i$ a sticky and same-end child of $\ctop$. So the preceding paragraph gives the factorization for $\rgamma_i$ in (1).
	
	For all other $\rzeta^{D'}$, promotion acts only by \eqref{eq:promotionvectors2}. We first consider the case of degree $1$, when $\rzeta^{D'}= \lr{ij \ell n}$ for $b\leq i<j<\ell<d$. Applying (b3)
	yields $\nu= \lr{ij\ell \br ba \br cdn}$. If $i=b$, then $\nu =-\lr{b j \ell a} \lr{bcdn}= \lr{abj\ell} \ralpha_{\rtop}$. Otherwise, $\nu$ does not factor by \cref{cor:pos}. 
	
	Using \cref{prop:explicit}, $\rzeta^{D'}= \lr{bj \ell n}$ for $j<\ell <d$ precisely when $\rzeta= \rdelta_i$ for $D_i$ a sticky child of $\ctop$; $\rzeta= \rbeta_i$ for $D_i$ a sticky, not same-end child of $\ctop$; $\rzeta= \rgamma_i$ for $D_i$ a sticky, not same-end child of $\ctop$ with no head-to-tail sibling; and $\rzeta= \rgamma_i$ where $D_i$ has a sticky same-end parent $D_p$ which is a sticky child of $\ctop$. So the preceding paragraph gives the factorization for $\rdelta_i$ in (1) and (2), the factorization for $\rbeta_i$ in (2), and some cases of the factorizations for $\rgamma_i$ in (2) and (3). The remaining cases for $\rgamma_i$ in (2) is when $D_i$ is a sticky, not same-end child of $\ctop$ with a head-to-tail sibling; for (3), it is when $D_i$ is a non-sticky same-end child of $D_p$, which is a sticky child of $\ctop$.
	
	Now we assume $\rzeta^{D'}$ has degree at least 2,
	and that $n$ appears in at least one clause without $d$ 
	(otherwise by \cref{prop:invariant}, promotion fixes $\rzeta^{D'}$). Note that by \cref{prop:invariant}, the index $n$ only appears in an end clause of $\rzeta^{D'}$, say $\br ijn \rangle$ or $\langle nij\br$. In this case, when we apply 
	\eqref{eq:promotionvectors2} 
	to obtain the numerator $\nu$ of $\Psi(\rzeta^{D'})$ only, 
	we replace $\br ijn \rangle$ by $\br ij \br ba \br cdn \rangle$ and replace 
	$\langle nij\br$ by $\langle ncd \br ba \br ij \br$.  Note that 
	if $n$ appears in 
	an end clause of $\nu$, $d$ appears in the same clause.
	
	We first show that $\lr{acdn}, \lr{abcn}$ do not divide $\nu$. Consider $C \in \Gr_{4,n}$ such that columns $C_1, \dots, C_{a-1}$ are zero vectors, columns $C_b, \dots, C_{n}$ give an element of $\Gr_{4,N_R}^{>0}$ and $C_a = C_n$. Clearly $\lr{acdn}, \lr{abcn}$ vanish on $C$, so it suffices to show that $\nu$ does not vanish on $C$. Notice that when evaluated on $C$, a chain polynomial $\lr{\cdots \br ij \br ba \br cdn}$ simplifies to $-\lr{\cdots \br ija} \lr{bcdn}= - \lr{bcdn} \lr{\cdots \br ijn}$
	and similarly $\lr{ncd\br ba \br ij \br \cdots}$ simplifies to $- \lr{bcdn} \lr{
	nij \br \cdots} $. Applying this observation to $\nu$, we obtain
	\[\nu(C)= - \lr{bcdn}_C \cdot \rzeta^{D'}(C)  \quad \text{or} \quad \nu(C)= \lr{bcdn}_C^2 \cdot \rzeta^{D'}(C)
\] 

	depending on whether one or both end clauses of $\rzeta^{D'}$ are affected by \eqref{eq:promotionvectors2}.
	In either case, $\nu(C)$ does not vanish on $C$, since both factors are cluster variables
	for $\Gr_{4,N_R}$ and thus do not vanish when evaluated on $C_b, \dots, C_n$.

	Now we show $\lr{abcd}$ does not divide $\nu$. Let $C\in\Gr_{4,n}$ be as in the last paragraph and let $B$ be the matrix obtained by replacing $C_d$ with $C_n$ and $C_n$ with $-C_d$. Since $B_a$ and $B_d=C_n$ are equal, $\lr{abcd}_B=0$. On the other hand, as we noted earlier,
	each end clause of $\nu$ containing $n$ also contains $d$, so 
	by \cref{prop:invariant}, 
	$\nu(B)=\nu(C)\neq 0$.  
	
	\begin{claim}\label{claim:divisible}
	Suppose that $\rzeta^{D'}$ as above has degree at least 2.  Then 
 $\lr{bcdn}$ divides 
	the numerator $\nu$ of $\Psi_{ac}(\rzeta^{D'})$ 
	 if and only if $b$ and $n$ appear together
	in the same clause of $\rzeta^{D'}$. 
	\end{claim} 
We start with the forward direction of the claim.
	We assume that $b$ and $n$ do not appear together in the same clause of $\rzeta^{D'}$.
	 We will again construct an element of the Grassmannian on which $\lr{bcdn}$ vanishes but $\nu$ does not (which will imply that $\lr{bcdn}$ does not divide $\nu$). Choose $C \in \Gr_{4,n}^{\ge 0}$ such that columns $C_1, \dots, C_{a-1}$ are zero vectors and columns $C_{a}, \dots, C_{n}$ give an element of $\Gr_{4, \{a, \dots, n\}}^{>0}$. Let $B$ be the matrix obtained from $C$ by replacing $C_b$ with $C_n$. Then 
	$\lr{bcdn}$ vanishes on $B$, and 
	 we have 
	\[ \nu(B)= \lr{acdn} \rzeta^{D'}(B) \quad \text{or} \quad \nu(B)= \lr{acdn}^2 \rzeta^{D'}(B). \]

	If $b$ does not occur in $\rzeta^{D'}$, then $\rzeta^{D'}(B)= \rzeta^{D'}(C)$, which is positive,
hence $\nu(B)$ is non-vanishing. 
	
If $b$ does occur in $\rzeta^{D'}$, it must come from 
the first marker of a top chord in $D_R$. Recall we are assuming $b$ does not appear in a clause with $n$ and $\rzeta^{D'}$ has degree at least 2. Using \cref{prop:explicit} and its notation, we may have $b=a_{(m)}$ in a string $$yz \rchn_i n = \cdots \br b_{(m)} a_{(m)} \br c_{(m)} d_{(m)} n \rangle \quad \text{or} \quad n \lchn_i yz= \langle n c_{(m)} d_{(m)} \br b_{(m)} a_{(m)} \br \cdots$$ or we may have $b$ as one of $a_i, a_i', a_p$ or $a_p'$. If $b=a_i$ or $b=a_p$ then $\rzeta^{D'}$ is a Pl\"ucker coordinate, so we need only consider the case 
	where 
	$b=a_{(m)}$ as above, or $b=a_i'$ or $b=a_p'$ and $\rzeta=\rgamma_i$ for some chord.
	
	\textbf{Case 1.} Suppose $b$ appears in $\rzeta^{D'}= \rgamma_i^{D'}$ as $a_i'$, in the notation of \cref{prop:explicit}. Then we must be in the second of the eight cases of $\rgamma$:
	$D_i$ has a head-to-tail sibling $D_j= (u,v,w,x)$, and $D_i$ is the sticky not same-end child 
	of parent $D_p=(b,s,y,z)$, 
	which is a top chord. So $\rgamma_i^{D'}= \lr{a_i b b_i \br c_i d_i \br c_j d_j \br a_i b \br c_p d_p n}$. Then
	$\nu=\lr{a_i b b_i \br c_i d_i \br c_j d_j \br a_i b \br c_p d_p \br ba \br cd n}$ and on $B$ it simplifies to 
	\[\nu(B)=-\lr{a_i b b_i \br c_i d_i \br c_j d_j b}\lr{a_i c_p d_p b} \lr{acdn}= \lr{n a_i b_i \br c_i d_i \br c_j d_j n}_C\lr{a_i c_p d_p n}_C \lr{acdn}_C. \]

	The first factor is $\rgamma_i^{D''}$ where $D''$ consists of the chords beneath $D_p$. 
	In particular the first term is positive on $C$, as are the other terms, so 
	$\nu(B) \neq 0$.
	
	\textbf{Case 2.} Suppose $b$ appears in $\rzeta^{D'}= \rgamma_i^{D'}$ as $a_p'$. Then $D_i$ is a non-sticky same-end child of $D_p$, and $D_p$ has a sticky parent $D_q=(b, a_p', c_q, d_q)$ which may or may not be same end. Then
	\[\nu=\lr{ a_p b b_p \br a_i b_i \br c_i d_i \br ba \br cdn} \quad \text{ or } \quad \nu =\lr{a_p b b_p \br a_i b_i \br c_i d_i \br a_p' b \br c_q d_q \br ba \br cdn}\]
	depending on if $D_p$ is a same-end child or not. Evaluating on $B$ we obtain
	\[\nu(B) 
	=- \lr{na_p b_p \br a_i b_i \br c_i d_i n}_C\lr{acdn}_C \quad \text{ or } \quad \nu(B)
	= \lr{n a_p b_p \br a_i b_i \br c_i d_in}_C \lr{a_p'  c_q d_q n}_C \lr{acdn}_C.\]
	The first factor is $\rgamma_i^{D''}$, where $D''$ consists of the chords beneath $D_q$, so is positive on $C$.  Therefore $\nu(B) \neq 0$.
	
	\textbf{Case 3}. Suppose $b$ appears in $\rzeta^{D'}$ only as $a_{(m)}$ in a string $``yz \rchn_j n"$ or $ ``n \lchn_j yz"$. This means that $D_R$ has a top chord $(b, f, g, h)$ which is an ancestor of the chord corresponding to $\rzeta^{D'}$. After applying \eqref{eq:promotionvectors2}, we see e.g.
	$$\nu = \lr{\cdots \br f b \br gh \br ba \br cdn} \quad \text{or} \quad \nu= \lr{n cd \br ba \br gh \br f b\br \cdots}.$$
	Evaluating on $B$, we see factorizations
	\[\nu(B) =\lr{ \cdots n }_C \lr{fgh n}_C \lr{acdn} \quad \text{or} \quad \nu(B)= \lr{ncda} \lr{n gh f}_C \lr{ n \cdots}_C.\]
	or possibly both. The first and last factors above (or the ``middle factor" if both factorizations occur) are precisely the domino variable $\rzeta^{D''}$ where $D''$ is the diagram of chords beneath $(b,f,g,h)$. Thus $\nu(B)$ is nonzero.
	
	Now, we show the backwards direction of \cref{claim:divisible}. Suppose 
	that $b$ and $n$ appear in the same clause of $\rzeta^{D'}$, 
	which has degree at least $2$. Then using \cref{prop:explicit} and its notation, 
	we must have that $\rzeta=\rgamma_i$, where $D_i$ is either a sticky child of $\ctop$ with a head-to-tail sibling to its right; 
	or $D_i$ is a non-sticky, same-end child of $D_p=(b,*,*,*)$, which is a sticky child of $\ctop$. 
	It is immediate that $\nu$ has a factor of $\ralpha_{\rtop}$, and the polynomial $\nu/\ralpha_{\rtop}$ agrees with the formula in \cref{prop:explicit} for $\rgamma_i^D$. We have already shown that polynomials of this form are not divisible by $\lr{bcdn}$ in Cases 1 and 2 above. This completes the proof of the factorization of $\rPsi(\rgamma_i^{D'})$ in the remaining cases in (2), (3).
	
\end{proof}

\begin{proof}[Proof of \cref{prop:explicit}]
	We proceed by induction on the size of the marker set. The base case, with three markers, is trivial as there are no chords.
	
Let $D_{\rtop}=(a,b,c,d)$ be the rightmost top chord of $D$. First, suppose $d< n-1$. As in \cref{domino-var}, let $D'$ be the chord diagram on $N' = \{1,\dots,n-2,n\}$ obtained from $D$ by erasing the marker $n-1$. In this case, by definition $\rzeta_j^D = \rzeta_j^{D'}$ for all domino variables.
Further, all the formulas for $\rzeta_i^D$ in the statement of \cref{prop:explicit} are identical to those for $D'$, which hold by the inductive hypothesis. Hence, the theorem holds for $D$.

Now suppose $d= n-1$. It is immediate from the definition that \cref{prop:explicit} holds for the domino variables $\ralpha_{\rtop},\rbeta_{\rtop},\rgamma_{\rtop},\rdelta_{\rtop},\repsilon_{\rtop}$ indexed by the top chord.  So we turn to the remaining domino variables $\rzeta^D$. Let $D_L, D_R$ be as in \cref{def:leftright}. By the inductive hypothesis, \cref{prop:explicit} and \cref{prop:invariant} hold for $D_L$ and $D_R$, and so in particular, we may apply \cref{lem:when-promote-var-factor} to $D$.

Consider a domino variable $\rzeta_j^D$ where the chord $D_j$ is in $D_L$. By definition, $\rzeta_j^D = \rPsi_{ac}(\rzeta_j^{D_L})$. By \cref{prop:invariant}, $\Psi_{ac}$ fixes $\rzeta_j^{D_L}$, and thus $\rzeta_j^D= \rzeta_j^{D_L}$, unless $\rzeta_j^{D_L}= \lr{xybn}$ for $x<y<a$. Using \cref{prop:explicit} for $D_L$, the latter equality occurs only if $D_j= (x,y,a,b)$ is a top (rightmost) chord of $D_L$ and $\rzeta=\rgamma$; in this case, $\rgamma_j^D= \lr{n xy \br ab \br cdn}$ by direct calculation (\cref{lem:when-promote-var-factor} ensures there is no factorization). 

On the other hand, the statement of \cref{prop:explicit} gives the same formula for $\rzeta_j^D$ and $ \rzeta_j^{D_L}$ for all variables except $\rgamma_j$ where $D_j= (x,y,a,b)$ is a top chord. This is because the ancestors and descendants of $D_j$ are the same in $D_L$ and in $D$ for all $D_j$ in $D_L$, and the siblings of $D_j$ are the same in both chord diagrams unless $D_j$ is a sibling of $D_k$. Now,if $ D_j = (x,y,a,b)$ is a top chord, then the formula in \cref{prop:explicit} for $\rgamma_j^{D}$ is $\lr{xyn \br ab \br cdn}$, which agrees with the formula in the above paragraph. So by induction (and direct calculation), \cref{prop:explicit} holds for $\rzeta_j^D$.

We now turn to domino variables $\rzeta_j^D$ where $D_j$ is in $D_R$. Suppose first that $\Psi_{ac}$ fixes $\rzeta_j^{D_R}$ and thus $\rzeta_j^D= \rzeta_j^{D_R}$. Using \cref{prop:invariant}, this means that either (1) the clauses of $\rzeta_j^{D_R}$ containing $n$ also contain $c,d$; or (2) $\rzeta_j^{D_R}$ is a Pl\"ucker coordinate which avoids $d,n$; or (3) $\rzeta_j^{D_R}= \lr{x~y~c~d}$ with $x, y<c$. In case (1), by \cref{prop:explicit}, the clause $\br cdn\rangle$ in $\rzeta_j^{D_R}$ is there because $D_j$ ends at $c,d$ or an ancestor of $D_j$ ends at $c,d$ and contributes to $\rchn_j$ or $\lchn_j$. This implies that $D_k$ will not contribute to $\rchn_j$. So \cref{prop:explicit} gives the same formulas for  $\rzeta_j^D$ and $ \rzeta_j^{D_R}$ in the former case and thus \cref{prop:explicit} holds by induction for these domino variables. In cases (2) and (3), $\rzeta_j^{D_R}$ is equal to $\repsilon_j$ or $D_j$ is a sticky child of some chord, and the formulas of \cref{prop:explicit} are again the same for $\rzeta_j^{D_R}$ and $\rzeta_j^{D}$.  

If $\Psi_{ac}$ does not fix $\rzeta_j^{D_R}$, then either $\rzeta_j^{D_R}= \lr{xydn}$ or an end clause of $\rzeta_j^{D_R}$ contains $n$ but not $c$ or $d$. In the first case, the fourth paragraph of \cref{lem:when-promote-var-factor}'s proof gives $\rzeta_j^{D}$ in various cases, which direct comparison shows is equal to the formula given in \cref{prop:explicit}. In the second case, the numerator of $\Psi_{ac}(\rzeta_j^{D_R})$ is obtained from $\rzeta_j^{D_R}$ by (11), which changes $\br xyn \rangle$ to $\br xy \br ba \br cdn \rangle$. If this numerator does not factor, then it is equal to $\rzeta_j^{D}$. The change $\br xyn \rangle \mapsto \br xy \br ba \br cdn \rangle$ is exactly the effect of adding $D_k$ to the chain of ancestors contributing to $\rchn_j$ and $\lchn_j$, which is the only difference in the formulas in \cref{prop:explicit} for $\rzeta_j^{D_R}$ and $\rzeta_j^D$. If the numerator of $\Psi_{ac}(\rzeta_j^{D_R})$ does factor, then $\rzeta_j^{D}$ is computed in the proof of \cref{lem:when-promote-var-factor}, and direct comparison shows that this is the same as the formula in \cref{prop:explicit}.

\end{proof}

\subsection{Signs of domino variables on tiles }
\label{sec:signs}

\cref{cor:cluster-sign-description} asserts that for each coordinate cluster variable $x \in \Irr(\rcp)$, the corresponding functionary $x(Y)$ has a fixed sign on the general BCFW tile $\gto{\rcp}$. For standard BCFW tiles, we give here an explicit combinatorial rule for this sign.

\begin{proposition}
\label{prop:domino-var-signs-on-tile} 
Let $D \in \CD$ be a chord diagram. The domino variables $\xx(D)$ have the following fixed signs on $\gto{D}$: 
\begin{align*}
\sgn(\ralpha_i(Y)) \;&=\; (-1)^{\after(D_i)+1} \\
\sgn(\rbeta_i(Y)) \;&=\; \begin{cases}
(-1)^{\after(D_i) \cdot \nonsticky(D_i)} & \text{if $D_i$ is not a sticky same-end child} \\
\sgn(\ralpha_p(Y)) & \text{if $D_i$ is a sticky same-end child of $D_p$}
\end{cases} \\
\sgn(\rgamma_i(Y)) \;&=\; \begin{cases}
(-1)^{\after(D_i) \cdot \sticky(D_i)} & \text{if $D_i$ has a right head-to-tail sibling} \\
(-1)^{\after(D_i) \cdot \sticky(D_i) + \after(D_p) \cdot \sticky(D_p) + 1} & \text{if $D_i$ is a same-end child of $D_p$} \\
(-1)^{\after(D_i) \cdot \nonsticky(D_i) + \below(D_i) + 1} & \text{otherwise}
\end{cases} \\
\sgn(\rdelta_i(Y)) \;&=\; (-1)^{\after(D_i) \cdot \nonsticky(D_i) + \below(D_i)} \\
\sgn(\repsilon_i(Y)) \;&=\; +1,
\end{align*}
where   
\begin{align*}
\after(D_i) \;&=\; \left|\left\{j : a_i \leq a_j \right\}\right| 
\text{ is the number of chords starting to the right of $D_i$ including $D_i$}\\
\below(D_i) \;&=\; \left|\left\{j : a_i < a_j < c_j \leq c_i \right\}\right| \text{ is the number of chords below $D_i$} \\
\sticky(D_i) \;&=\; \left|\left\{j : a_j = a_i - 1 \right\}\right|
\text{ is $1$ if 
$D_i$ is a sticky child of another chord and $0$ otherwise}\\
\nonsticky(D_i)&= 1 -\sticky(D_i).
\end{align*}
\end{proposition}

\begin{example}[Signs of domino variables] \label{ex:stdBCFW_signs}
Let $\gt{D}$ be the tile corresponding to the chord diagram~$D$ in Figure~\ref{cd-example}. The following $A$ domino variables are negative on $\gto{D}$:
$$ \ralpha_2,\; \ralpha_3,\;  
\ralpha_5 = 
\rbeta_4,\;  \rbeta_1,\;  \rbeta_6,\;  \rgamma_2,\;  \rdelta_1,\;  \rdelta_5,\;  \rdelta_6. $$
The remaining domino variables are positive on $\gto{D}$.
\end{example}

\cref{prop:domino-var-signs-on-tile} can be deduced from the recursive definition of the domino variables (cf. \cref{domino-var}), the factorizations in \cref{lem:when-promote-var-factor}, and the following lemma detailing how signs of domino variables on tiles change after a BCFW step.

\begin{lemma}\label{lem:signs}
	Using \cref{not:bcfwmap}, let $S_L \subset \Gr_{k_L, N_L}^{\ge 0}$ and 
	$S_R \subset \Gr_{k_R, N_R}^{\ge 0}$ be BCFW cells with chord diagrams $D_L, D_R$.
	Let $F \in \xx(D_L)$ (resp. $F \in \xx(D_R)$).
	If $F$ has sign $s \in \{\pm 1\}$ on $\gto{L}$ (resp. $\gto{R}$) for all $Z$, then $$
	(-1)^{r} \cdot s \cdot \Psi_{a c}(F)  
	> 0 \qquad \text{(resp. }
	s \cdot \Psi_{a c}(F) > 0)$$ on $Z_{S_L \bcfw S_R}^{\circ}$ for all $Z$, where $r=(k_R+1) \deg_n F$. Moreover, the twistors
		\[(-1)^{k_R}\llrr{bcdn}, \quad (-1)^{k_R+1}\llrr{acdn}, \quad \llrr{abdn},\quad -\llrr{abcn}, \quad \llrr{abcd} \quad \mbox{are positive on } \gto{S} \mbox{ for all } Z.\]
\end{lemma}
\cref{lem:signs} follows from
\cref{lem:sign_of_chord_twistors}, \cref{prop:vanishing_and_sign_of_functionaries_under_promotion} and \cref{cor:clust-var-strong-sign}, applied to standard BCFW cells.

\section{The cluster algebra of a standard BCFW tile}\label{sec:quiver}

In this section we explicitly describe a cluster algebra 
$\Acal(\Sigma_D)$ associated to a standard
BCFW tile $Z_D$.  
We show that each cluster variable 
is a regular function on the Grassmannian and has a fixed sign on the tile.
Moreover, we give \emph{many} descriptions of each open BCFW tile $\gto{D}$ as a semialgebraic set in $\Gr_{k,k+4}$ using cluster variables in \emph{any} fixed extended cluster of $\Acal(\Sigma_D)$.

Finally, we give an algorithm for constructing the quiver associated to 
each tile $Z_D$.
A useful tool in our proof is the notion of a \emph{signed seed} of a cluster
algebra, see \Cref{def:signedseed}, together with the fact that 
the property of being a signed seed
is preserved under mutation.

\subsection{The seed of a standard BCFW tile}

We will inductively construct a cluster from a chord diagram $D$, 
by repeatedly
 removing either the penultimate marker
or the rightmost top chord $\ctop$.  In the latter case, we build our cluster 
 based on those of the left and right subdiagrams of $D$.  

Recall that in \cref{domino-var} we associated a set $\xx(D)$ of domino cluster variables 
to each chord diagram $D$.  The following result, which 
is a consequence of \Cref{thm:promotion2}, enlarges the set $\xx(D)$  to an extended cluster $\txx(D)$ for $\Gr_{4,n}$. By \cref{thm:GSV}, there is a unique seed $\widetilde{\Sigma}_D$ for $\Gr_{4,n}$ whose extended cluster is $\txx(D)$.

\begin{theorem}[Extended cluster for a standard BCFW tile]\label{thm:cluster}
Let $D$ be a chord diagram in $\CD$, and let
$D_k = (a_k,b_k,c_k,d_k)$ denote the rightmost top chord, where $1\leq a_k<b_k<c_k<d_k<n$,
where $a_k,b_k$ and $c_k,d_k$ are consecutive.

We inductively construct an extended cluster ${\txx}(D)$ for $\Gr_{4,n}$, 
with frozen variables $f_1,\dots,f_n$, as:
\begin{itemize}
\item[1.] (Base case) If $n=4$ then $\txx(D)$ consists of the single Pl\"ucker coordinate
	$f_1 = \lr{1\,2\,3\,4}$.
\item[2.] (Remove penultimate marker) If $d<n-1$, then let $D'$ be the chord
  diagram on $\{1,2,\dots,n-2,n\}$ obtained from $D$ by removing the marker $n-1$.
	We let $\txx(D)$ be $\txx(D')$ together with the four frozen variables in 
		$\Gr_{4,n}$ 
	involving $n-1$,  
		that is,
	$$\txx(D)= \txx(D') \cup \{f_{n-4},f_{n-3},f_{n-2},f_{n-1}\}.$$
\item[3.] (Remove rightmost top chord) Otherwise, if $d=n-1$,
let $D_L$ and $D_R$ be the left and right subdiagrams as in \cref{def:leftright}.
	
   Recall that 
		$\overline{\Psi}=\overline{\Psi}_{ac}$ denotes rescaled product promotion, as in \Cref{def:rPsi}.
		We let $\txx(D)$ be 
		$$\txx(D) = 
\{\rbeta_k, \rgamma_k, \rdelta_k,\repsilon_k\} \cup 
		\OPsi(\txx(D_L) \cup \txx(D_R)) \cup 
\{f_c, f_d, f_{a-1}\},$$ the union of the new domino variables,
the promotions of the old variables, 
		and three frozen variables in $\Gr_{4,n}$.
		(We note that $\ralpha_k$ will already be an element of $\OPsi(\txx(D_R))$.)
\end{itemize}

Then $\txx(D)$ is the extended cluster 
of a unique seed $\widetilde{\Sigma}_D= (\txx(D), \tQ_D)$ 
for the cluster algebra structure on 
$\Gr_{4,n}$.
\end{theorem}
\begin{remark}
While \cref{thm:cluster}  is stated for markers $[n]$, 
it naturally extends to general index sets, e.g. as the ones in \Cref{index-sets}. It can also naturally be extended to general BCFW cells $\gt{\rcp}$ by applying $\cyc^{-*}$ or $\refl^*$ for each $\cyc$ or $\refl$ step in the recipe $\rcp$.
\end{remark}

\begin{remark}
We note that in Case (3) of \Cref{thm:cluster}, if $a=1$ and $b=2$, then 
$D_L$ is the chord diagram on 3 markers with no chords and 
 $$\txx(D) = 
\{\rbeta_k, \rgamma_k, \rdelta_k,\repsilon_k\} \cup 
	\OPsi(\txx(D_L) \cup \txx(D_R)) \cup 
\{f_c, f_d, f_{a-1}\} = 
\{\rbeta_k, \rgamma_k, \rdelta_k,\repsilon_k\} \cup 
\OPsi(\txx(D_R)),$$
because in this case, 
$f_c = \rbeta_k$, $f_d=\rgamma_k$, and $f_{a-1} = f_n = \rPsi(\lr{2,3,d,n})$.
\end{remark}

We now focus on the domino variables; our goal will be to explicitly describe the arrows
connecting them in the quiver $\widetilde{Q}_D$ of $\widetilde{\Sigma}_D$. 
We start by partitioning them into mutable and frozen domino variables.
Geometrically,
the frozen domino variables correspond
to facets of $Z_D$.

\begin{definition}[Mutable and frozen domino variables]\label{def:frozmut}
Let $D \in \mathcal{CD}_{n,k}$ be a chord diagram, corresponding to a
standard BCFW tile $Z_D$ in $\A_{n,k,4}$.
Let $\Mut(Z_D)$ denote the following collection of domino cluster variables,
        where $1 \leq i \leq k$:
        \begin{itemize}
                \itemsep0.25em
                \item
                $\ralpha_i$ where $D_i$ has a sticky child
                \item
                $\rbeta_i$ where $D_i$ starts where another chord ends or $D_i$ has a same-end sticky parent. 
                \item
                $\rdelta_i$ where $D_i$ has a same-end child.
                \item
                $\repsilon_i$ where $D_i$ has a same-end child.
        \end{itemize}
Let $\AFacet(Z_D)$ denote
the complementary collection of domino cluster variables, that is:
        \begin{itemize}
                \itemsep0.25em
                \item
                $\ralpha_i$ unless $D_i$ has a sticky child
                \item
                 $\rbeta_i$ unless $D_i$ starts where another chord ends or $D_i$ has a same-end sticky parent. 
\item   $\rgamma_i$ in all cases. 
   \item
                $\rdelta_i$ unless $D_i$ has a same-end child.
                \item
                $\repsilon_i$ unless $D_i$ has a same-end child.
        \end{itemize}
\end{definition}

\begin{remark}
Recall from \cref{rmk:beta=alpha} that if $D_i$ has a same-end sticky parent $D_p$, then $\rbeta_i=\ralpha_p$.    
\end{remark}

\begin{example}[Mutable and frozen domino variables]
Let $Z_D$ be the tile with the chord diagram~$D$ from 
\cref{cd-example} and domino variables as in \cref{domino-variables-formulas}. Among those, the mutable variables
are:
$$ \ralpha_5,\, \ralpha_6,\, \rbeta_2,\, \rbeta_4,\, \rbeta_6,\, \rdelta_3,\, \rdelta_5,\, \repsilon_3,\, \repsilon_5 \;\in\; \Mut(Z_D). $$
Hence
$\AFacet(Z_D)$ consists of the
remaining $21$ domino variables. Note that $\ralpha_5 = \rbeta_4$ by \cref{rmk:beta=alpha}.
\end{example}

\begin{theorem}[Frozen variables as facets]\label{prop:standard_facets}
Let $D \in \mathcal{CD}_{n,k}$ be a chord diagram, corresponding to a
standard BCFW tile $Z_D$ in $\Ank$.
Then for each cluster variable $\rzeta_i\in \AFacet(Z_D)$,
there is a unique facet of $\gt{D}$
which lies in the zero locus of the functionary $\rzeta_i(Y)$.
Moreover, for 
any $Z$, there are no other facets of $\gt{D}$.
\end{theorem}
\cref{prop:standard_facets} is proved in our 
companion paper \cite{companion}.

We next define arrows between the domino variables, then show that these are 
 arrows in $\widetilde{Q}_D$.

\begin{definition}[The seed $\Sigma_D$ of a BCFW tile $Z_D$]
\label{def:seed}
Let $D \in \CD$ be a chord diagram,
and $\gt{D}$ the corresponding BCFW tile.
Recall the definition of
$\Mut(Z_D)$ from \cref{def:frozmut}.
We define a seed $\Sigma_D=(\bxx(D), Q_D)$ (which will be a 
subseed of $\widetilde{\Sigma}_D$ cf. \cref{thm:quiver}) 
as follows.
The extended cluster $\bxx(D)$ is obtained from $\txx(D)$  by 
freezing some variables; more specifically, the mutable cluster variables
	of $\bxx(D)$ will be precisely $\Mut(Z_D)$, and all other 
	elements of $\txx(D)$ are declared to be frozen. 
To obtain the quiver $Q_D$, we consider each chord $D_i$ in turn, 
check if it satisfies any 
of the conditions in the table below, and if so, we draw the corresponding arrows.
\vspace{0.5em}
\begin{center}
\begin{tabular}{|c|c|c|c|}
\hline
\rotatebox{90}{Condition \quad\;}&
\tikz[line width=1,scale=1]{
\def\r{1}
\draw (1*\r,0) -- (1.5*\r,0);
\draw[dashed] (1.5*\r,0) -- (2.75*\r,0);
\draw (2.75*\r,0) -- (3.5*\r,0);
\draw[dashed] (3.5*\r,0) -- (5*\r,0);
\draw (5*\r,0) -- (5.5*\r,0);
\foreach \i in {6,7}{
\def\x{\i/2*\r}
\draw (\x,-0.1)--(\x,+0.1);
}				
\foreach \i/\j in {2/5.833,6.166/10}{
\def\x{\i/2*\r+0.25*\r}
\def\y{\j/2*\r+0.25*\r}
\draw[line width=1,-stealth] (\x,0) -- (\x,0.15) to[in=90,out=90] (\y,0.15) -- (\y,0);}
\node at(2.2*\r,1*\r) {$j$};
\node at(4.3*\r,1*\r) {$i$};
}&
\tikz[line width=1,scale=1]{
\draw (0.75,0) -- (1.25,0);
\draw[dashed] (1.25,0) -- (2.25,0);
\draw (2.25,0) -- (2.75,0);
\draw[dashed] (2.75,0) -- (4,0);
\draw (4,0) -- (5,0);
\foreach \i in {8.5,9.5}{
\def\x{\i/2}
\draw (\x,-0.1)--(\x,+0.1);
}
\foreach \i/\j in {1/4.5+0.066, 2.5/4.5-0.066}{
\def\x{\i}
\def\y{\j}
\draw[line width=1,-stealth] (\x,0) -- (\x,0.15) to[in=90,out=90] (\y,0.15) -- (\y,0);}
\node at(2.4,0.6) {$j$};
\node at(2.7,1.5) {$i$};
}&
\tikz[line width=1,scale=1]{
\draw (0.5,0) -- (1.75,0);
\draw[dashed] (1.75,0) -- (3.25,0);
\draw (3.25,0) -- (3.75,0);
\draw[dashed] (3.75,0) -- (4.75,0);
\draw (4.75,0) -- (5.25,0);
\foreach \i in {1.5,2.5,3.5}{
\def\x{\i/2}
\draw (\x,-0.1)--(\x,+0.1);
}
\foreach \i/\j in {1/5, 1.5/3.5}{
\def\x{\i}
\def\y{\j}
\draw[line width=1,-stealth] (\x,0) -- (\x,0.15) to[in=90,out=90] (\y,0.15) -- (\y,0);}
\node at(3,1) {$j$};
\node at(2,1.5) {$i$};
}\\
& head-to-tail left sibling $D_j$
& same-end child $D_j$
& sticky child $D_j$
\\[0.5em]
\hline
\rotatebox{90}{\quad Arrows \quad\;\;}&
\tikz[line width=0.75,scale=1,minimum size=18pt,inner sep=0pt, outer sep=0pt,fill=lightgray!25]{
\node        (c1) at (5,5) {$\rgamma_j$};
\node        (d1) at (6,5) {$\rdelta_j$};
\node        (a2) at (5,4) {$\ralpha_i$};
\node[fill,draw,circle] (b2) at (6,4) {$\rbeta_i$};
\path[very thick,->] (b2) edge (d1);
\path[very thick,->] (c1) edge (b2);
\path[very thick,->] (b2) edge (a2);
}
&
\tikz[line width=0.75,scale=1,minimum size=18pt,inner sep=0pt, outer sep=0pt,fill=lightgray!25]{
\node        (c5) at (12,2) {$\rgamma_i$};
\node[fill,draw,circle] (d5) at (13,2) {$\rdelta_i$};
\node[fill,draw,circle] (e5) at (15,2) {$\repsilon_i$};
\node        (c4) at (12,1) {$\rgamma_j$};
\node        (d4) at (13,1) {$\rdelta_j$};
\node        (e4) at (15,1) {$\repsilon_j$};
\path[very thick,->] (e5) edge (d5);
\path[very thick,->] (d5) edge (c5);
\path[very thick,->] (c4) edge (d5);
\path[very thick,->] (d4) edge (e5);
\path[very thick,->] (d5) edge (d4);
\path[very thick,->] (e5) edge (e4);
}
&
\tikz[line width=0.75,scale=1,minimum size=18pt,inner sep=0pt, outer sep=0pt,fill=lightgray!25]{
\node[fill,draw,circle] (a5) at (9,2) {$\ralpha_i$};
\node (b5) at (10,2) {$\rbeta_i$};
\node (e5) at (6.5,2) {$\repsilon_i$};
\node (a4) at (10,1) {$\ralpha_j$};
\node (e4) at (6.5,1) {$\repsilon_j$};
\node[gray] (text) at (7.5,2.5) {if same-end};
\path[very thick,->,dotted] (a5) edge (e5);
\path[very thick,->] (e4) edge (a5);
\path[very thick,->] (b5) edge (a5);
\path[very thick,->] (a5) edge (a4);
}
\\[0.5em]
\hline
\end{tabular}
\end{center}
\vspace{0.5em}
If $D_i$ has sticky same-end child $D_j$ then the dotted arrow from $\ralpha_i$ to $\repsilon_i$ appears, along with the usual arrows of the ``sticky'' and ``same-end'' cases. In view of \cref{rmk:beta=alpha}, in this case $\ralpha_i$ stands also for $\rbeta_j$ as they are equal.   
\end{definition}

\begin{figure}[h]
\begin{center}
\tikz[line width=1,scale=1]{
\draw (0.5,0) -- (15.5,0);
\foreach \i in {1,2,...,15}{
\def\x{\i}
\draw (\x,-0.1)--(\x,+0.1);
\node at (\x,-0.5) {\i};}
\foreach \i/\j in {1/8, 3/4.8, 5.2/7.75, 8.25/13, 9/12.2, 10/11.8}{
\def\x{\i+0.5}
\def\y{\j+0.5}
\draw[line width=1.5,-stealth] (\x,0) -- (\x,0.25) to[in=90,out=90] (\y,0.25) -- (\y,0);
}
\node at(1.5,1.5) {$D_3$};
\node at(3.5,0.875) {$D_1$};
\node at(5.75,1) {$D_2$};
\node at(9,1.5) {$D_6$};
\node at(9.5,1) {$D_5$};
\node at(10.5,0.8125) {$D_4$};
}
\\
\vspace{1em}
\includegraphics[width=0.6\textwidth]{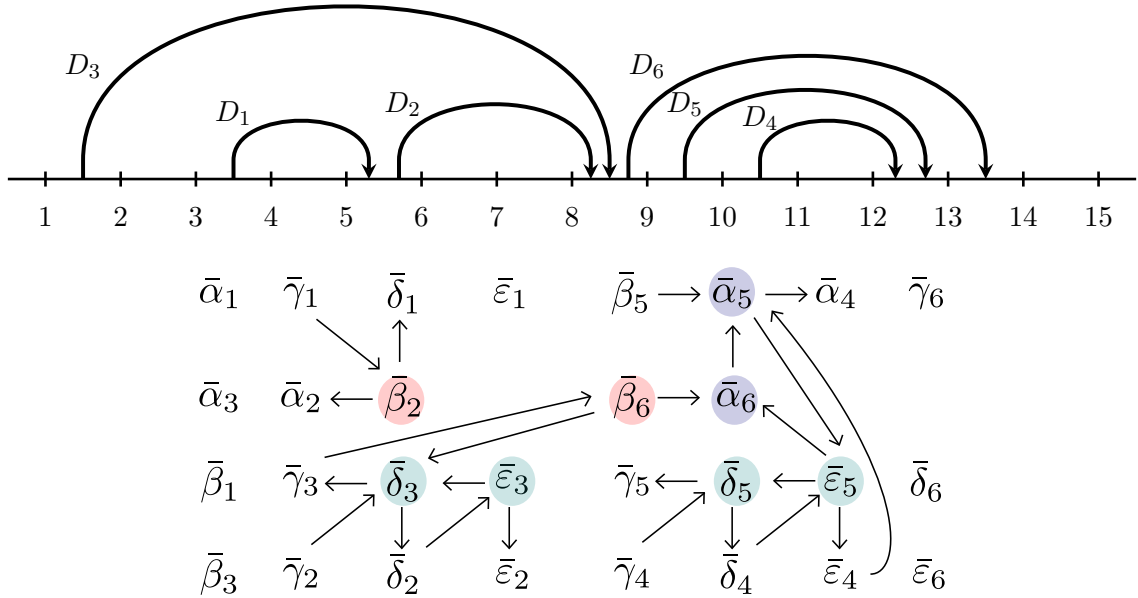}
\end{center}

\caption{The seed $\Sigma_D$ associated to the chord diagram $D$ above (also in \cref{cd-example}). The variables $\xx(D)$
	 are as in \cref{domino-variables-formulas}. 
	 The mutable variables $\Mut(Z_D)$ are circled; 
	 the other variables are the frozen variables $\AFacet(Z_D)$.
	 The colors (red, green, blue) indicate the different cases of \cref{def:seed}. 
\label{cd-example-again}}
\vspace{1em}
\end{figure}

\begin{example}[Seed of a standard BCFW tile]
The seed $\Sigma_D$ from \cref{cd-example-again} is built from \cref{def:seed} by applying the rules for the following conditions. Head-to-tail left siblings: $(i,j) \in \{(2,1),(6,3)\}$; same-end child: $(i,j) \in \{(3,2),(5,4)\}$; sticky child: $(i,j) \in \{(6,5),(5,4)\}$.
\end{example}

\cref{thm:quiver}
and \cref{thm:sign-definite} are the main results of this section.
They will be proved in 
\cref{sec:proofquiver} and 
\cref{sec:thmsign}.
	\begin{theorem}[The seed of a standard BCFW tile  is a subseed of a $\Gr_{4,n}$ seed]
	\label{thm:quiver}
Let $D \in \CD$. The seed $\Sigma_D=(\bxx(D), Q_D)$ is obtained from the
seed $\widetilde{\Sigma}_D= (\txx(D), \tQ_D)$ for $\Gr_{4,n}$ by 
freezing the variables not in the set $\Mut(Z_D)$ of
\cref{def:seed}.
Hence every cluster
variable (respectively, exchange relation) of 
$\Acal(\Sigma_D)$ is a cluster variable (resp., exchange relation) for
$\Gr_{4,n}$.

\end{theorem}

The following theorem characterizes the open BCFW tile
$\gto{D}$ in terms of \emph{any} extended cluster of $\Acal(\Sigma_D)$. It generalizes \cref{cor:cluster-sign-description} for standard BCFW tiles.
\begin{theorem}
	[Positivity tests for standard BCFW tiles]
	\label{thm:sign-definite}
Let $D \in \CD$. 
Using \cref{not:pluckfunc}, every cluster and frozen variable $x$ in 
 $\A(\Sigma_D)$ is such that $x(Y)$ has a definite sign $s_x \in \{1, -1\}$ on the open BCFW tile $\gto{D}$,
	where the signs of the domino variables are 
	given by \cref{prop:domino-var-signs-on-tile}. We have
 \begin{align*}
\gto{D} &= \{Y \in \Gr_{k,k+4}: s_x \cdot x(Y)>0 \text{ for all }x \in \xx(D)\}\\
&= \{Y \in \Gr_{k,k+4}: s_x \cdot x(Y)>0 \text{ for all }x \in \txx(D)\}\\
&= \{Y \in \Gr_{k, k+4}: s_x \cdot x(Y)>0 \text{ for all }x \text{ in any fixed extended cluster of } \A(\Sigma_D)\}.
\end{align*}
\end{theorem}

\begin{example}[Positivity test for the standard BCFW tiles] \label{ex:stdBCFW_postest}
    For the tile $Z_D$ with chord diagram $D$ in \cref{cd-example-again}, we have:
    \begin{equation*}
        \gto{D} = \{Y \in \Gr_{6,10}: s_x \cdot x(Y)>0 \text{ for all }x \in \xx(D)\},
    \end{equation*}
    with $\xx(D)$ as in \cref{domino-variables-formulas} and the respective signs $s_x$ as in \cref{ex:stdBCFW_signs}.
\end{example}

The following lemma provides a description of exchange relations which is 
equivalent to the local description given in \cref{def:seed}.  \cref{lem:explicit-reln}
will be useful in the proof of \cref{thm:quiver}.

\begin{lemma} \label{lem:explicit-reln}
	Let $D \in \CD$. The exchange relations of $\Sigma_D$ are as follows.
\\[0.5em]
\noindent
\textbf{R1.} (Sticky) 
Suppose that $D_i$ is a chord with a sticky child $D_j$. Then 
	
\begin{equation*}
		\ralpha_i \ralpha_i' = 
		(\rbeta_i) \repsilon_j (\ralpha_p) 
		+ \ralpha_j (\repsilon_i),
\end{equation*}
where $\rbeta_i$ appears unless $D_i$ has a sticky same-end parent, $\ralpha_p$ is present only if $D_i$ has a sticky parent $D_p$, and 	$\repsilon_i$ is present only if 	$D_j$ is same-end with $D_i$. 
	\begin{center}
		\begin{tabular}{cccc}
            &
			\tikz[line width=1,scale=0.6]{
				
				\draw[dashed] (0,0) -- (0.5,0);
				\draw (0.5,0) -- (2,0);
				\draw[dashed] (2,0) -- (3,0);
				\draw (3,0) -- (3.6,0);
				\draw[dashed] (3.6,0) -- (4,0);
				\draw (4,0) -- (4.6,0);
				\draw[dashed] (4.6,0) -- (5,0);
				
				\foreach \i/\j in {1.5/,2.5/,3.5/,
					6.2/,7.2/, 8.2/,9.2/}{
					\def\x{\i/2}
					\draw (\x,-0.1)--(\x,+0.1);
					\node at (\x,-0.25) {$\j$};
				}
				
				\foreach \i/\j in {1/4.35, 1.5/3.35}{
					\def\x{\i}
					\def\y{\j}
					\draw[line width=1,-stealth] (\x,0) -- (\x,0.15) to[in=90,out=90] (\y,0.15) -- (\y,0);}
				
				\node at(2.45,0.899) {\footnotesize$j$};
				\node at(2.7,1.4) {\footnotesize$i$};

    \node at (1,2) [shape=rectangle,draw,inner sep=2.5pt]{i.};
				
			}&
			
			\tikz[line width=1,scale=0.6]{
				
				\draw[dashed] (0,0) -- (0.5,0);
				\draw (0.5,0) -- (2,0);
				\draw[dashed] (2,0) -- (3,0);
				\draw (3,0) -- (4,0);
				\draw[dashed] (4,0) -- (4.5,0);
				
				\foreach \i/\j in {1.5/,2.5/,3.5/,6.5/,7.5/}{
					\def\x{\i/2}
					\draw (\x,-0.1)--(\x,+0.1);
					\node at (\x,-0.25) {$\j$};
				}
				
				\foreach \i/\j in {1/3.5+0.066, 1.5/3.5-0.066}{
					\def\x{\i}
					\def\y{\j}
					\draw[line width=1,-stealth] (\x,0) -- (\x,0.15) to[in=90,out=90] (\y,0.15) -- (\y,0);}
				\node at(2.45,0.5) {\footnotesize$j$};
				\node at(2.2,1.1) {\footnotesize$i$};

    \node at (1,2) [shape=rectangle,draw,inner sep=2.5pt]{ii.};
				
			} &
			\\
			&
			$\ralpha_i \ralpha'_i = \rbeta_i \repsilon_j + \ralpha_j$ &
			$\ralpha_i \ralpha'_i = \rbeta_i \repsilon_j + \ralpha_j \repsilon_i$&\\[1em]

			\tikz[line width=1,scale=0.6]{
				
				\draw[dashed] (0,0) -- (0.5,0);
				\draw (0.5,0) -- (2.5,0);
				\draw[dashed] (2.5,0) -- (3.5,0);
				\draw (3.5,0) -- (4.1,0);
				\draw[dashed] (4.1,0) -- (4.5,0);
				\draw (4.5,0) -- (5.1,0);
				\draw[dashed] (5.1,0) -- (5.5,0);
				\draw (5.5,0) -- (6.1,0);
				\draw[dashed] (6.1,0) -- (6.5,0);
				
				\foreach \i/\j in {1.5/,2.5/,3.5/,4.5/,
					7.2/,8.2/, 9.2/,10.2/, 11.2/,12.2/}{
					\def\x{\i/2}
					\draw (\x,-0.1)--(\x,+0.1);
					\node at (\x,-0.25) {$\j$};
				}
				
				\foreach \i/\j in {1/5.85, 1.5/4.85, 2/3.85}{
					\def\x{\i}
					\def\y{\j}
					\draw[line width=1,-stealth] (\x,0) -- (\x,0.15) to[in=90,out=90] (\y,0.15) -- (\y,0);}
				
				\node at(2.9,0.888) {\footnotesize$j$};
				\node at(3.2,1.33) {\footnotesize$i$};
				\node at(3.5,1.8) {\footnotesize$p$};

    \node at (1,2) [shape=rectangle,draw,inner sep=2.5pt]{iii.};
				
			}&
			
			\tikz[line width=1,scale=0.6]{
				
				\draw[dashed] (0,0) -- (0.5,0);
				\draw (0.5,0) -- (2.5,0);
				\draw[dashed] (2.5,0) -- (3.5,0);
				\draw (3.5,0) -- (4.5,0);
				\draw[dashed] (4.5,0) -- (5,0);
				\draw (5,0) -- (5.5,0);
				\draw[dashed] (5.5,0) -- (6,0);
				
				\foreach \i/\j in {1.5/,2.5/,3.5/,4.5/,7.5/,8.5/,10/,11/}{
					\def\x{\i/2}
					\draw (\x,-0.1)--(\x,+0.1);
					\node at (\x,-0.25) {$\j$};
				}
				
				\foreach \i/\j in {1/5.25 ,1.5/4+0.066, 2/4-0.066}{
					\def\x{\i}
					\def\y{\j}
					\draw[line width=1,-stealth] (\x,0) -- (\x,0.15) to[in=90,out=90] (\y,0.15) -- (\y,0);}
				
				\node at(2.95,0.48) {\footnotesize$j$};
				\node at(2.75,1.1) {\footnotesize$i$};
				\node at(3.1,1.6) {\footnotesize$p$};

    \node at (1,2) [shape=rectangle,draw,inner sep=2.5pt]{iv.};
			}&

			\tikz[line width=1,scale=0.6]{
				
				\draw[dashed] (0,0) -- (0.5,0);
				\draw (0.5,0) -- (2.5,0);
				\draw[dashed] (2.5,0) -- (3.5,0);
				\draw (3.5,0) -- (4.1,0);
				\draw[dashed] (4.1,0) -- (4.5,0);
				\draw (4.5,0) -- (5.1,0);
				\draw[dashed] (5.1,0) -- (5.5,0);
				
				\foreach \i/\j in {1.5/,2.5/,3.5/,4.5/,
					7.2/,8.2/, 9.2/,10.2/}{
					\def\x{\i/2}
					\draw (\x,-0.1)--(\x,+0.1);
					\node at (\x,-0.25) {$\j$};
				}
				
				\foreach \i/\j/\z in {1/4.95/0.6, 1.5/4.75/0.25, 2/3.85/0.1}{
					\def\x{\i}
					\def\y{\j}
                    \draw[line width=1,-stealth] (\x,0) -- (\x,\z) to[in=90,out=90] (\y,\z) -- (\y,0);}
				
				\node at(3,0.9) {\footnotesize$j$};
				\node at(3.1,1.5) {\footnotesize$i$};
				\node at(3,2) {\footnotesize$p$};

    \node at (1,2) [shape=rectangle,draw,inner sep=2.5pt]{v.};
				
			}&
			
			\tikz[line width=1,scale=0.6]{
				
				\draw[dashed] (0,0) -- (0.5,0);
				\draw (0.5,0) -- (2.5,0);
				\draw[dashed] (2.5,0) -- (3.5,0);
				\draw (3.5,0) -- (4.5,0);
				\draw[dashed] (4.5,0) -- (5,0);
				
				\foreach \i/\j in {1.5/,2.5/,3.5/,4.5/,7.5/,8.5/}{
					\def\x{\i/2}
					\draw (\x,-0.1)--(\x,+0.1);
					\node at (\x,-0.25) {$\j$};
				}
				
				\foreach \i/\j/\z in {1/4+0.1/0.6 ,1.5/4/0.35, 2/4-0.1/0.05}{
					\def\x{\i}
					\def\y{\j}
                    \draw[line width=1,-stealth] (\x,0) -- (\x,\z) to[in=90,out=90] (\y,\z) -- (\y,0);}
				
				\node at(2.8,0.8) {\footnotesize$j$};
				\node at(2.7,1.3) {\footnotesize$i$};
				\node at(2.7,1.65) {\footnotesize$p$};

                \node at (1,2) [shape=rectangle,draw,inner sep=2.5pt]{vi.};
			} \\
            $\ralpha_i \ralpha'_i = \rbeta_i \repsilon_j \ralpha_p + \ralpha_j$ &
			$\ralpha_i \ralpha'_i = \rbeta_i \repsilon_j \ralpha_p + \ralpha_j \repsilon_i$ & 
            $\ralpha_i \ralpha_i' = \repsilon_j \ralpha_p + \ralpha_j$ & 
            $\ralpha_i \ralpha_i' = \repsilon_j \ralpha_p + \ralpha_j \repsilon_i$\\
		\end{tabular}
	\end{center}

\textbf{R2.} (Head-to-tail) Suppose that $D_i$ has a head-to-tail sibling $D_j$ on its left.
		Then 
	\begin{center}
		\begin{tabular}{c}
			\tikz[line width=1,scale=0.7]{
				\def\r{1.3}
				
				\draw[dashed] (0.25*\r,0) -- (0.75*\r,0);
				\draw (0.75*\r,0) -- (1.75*\r,0);
				\draw[dashed] (1.75*\r,0) -- (2.75*\r,0);
				\draw (2.75*\r,0) -- (3.75*\r,0);
				\draw[dashed] (3.75*\r,0) -- (4.75*\r,0);
				\draw (4.75*\r,0) -- (5.75*\r,0);
				\draw[dashed] (5.75*\r,0) -- (6.25*\r,0);
				
				\foreach \i/\j in {2/{}, 3/{}, 6/{}, 7/{}, 10/{}, 11/{}}{
					\def\x{\i/2*\r}
					\draw (\x,-0.1)--(\x,+0.1);
					\node at (\x,-0.25) {$\j$};
				}
				
				\foreach \i/\j in {2/5.833,6.166/10}{
					\def\x{\i/2*\r+0.25*\r}
					\def\y{\j/2*\r+0.25*\r}
					\draw[line width=1,-stealth] (\x,0) -- (\x,0.15) to[in=90,out=90] (\y,0.15) -- (\y,0);}
				
				\node at(2.2*\r,0.9*\r) {\footnotesize $ j$};
				\node at(4.3*\r,0.9*\r) {\footnotesize $i$};
				
			}
		\end{tabular}
  \\
			$\rbeta_i \rbeta_i' = \rgamma_j+\ralpha_i \rdelta_j$
	\end{center}
\textbf{R3a.} (Same end)
Suppose that $D_i$ has a same-end child $D_j$.  Then 
	\begin{equation*}
		\rdelta_i \rdelta_i' = 
		\rgamma_j \repsilon_i (\rdelta_p) (\rbeta_{\ell}) +
		\rgamma_i \rdelta_j (\repsilon_p),
	\end{equation*}
	where $\rdelta_p$ and $\repsilon_p$ are present only if 
	$D_i$ has a same-end parent $D_p$ and
	$\rbeta_{\ell}$ is present only if $D_i$
			has a head-to-tail sibling $D_{\ell}$ to its right. 
		
	\begin{center}
		\begin{tabular}{ccc}
			\tikz[line width=1,scale=0.8]{
				\def\r{1.5}
				\draw[dashed] (0.5*\r,0) -- (0.75*\r,0);
				\draw (0.75*\r,0) -- (1.35*\r,0);
				\draw[dashed] (1.35*\r,0) -- (1.8*\r,0);
				\draw (1.8*\r,0) -- (2.35*\r,0);
				\draw[dashed] (2.35*\r,0) -- (2.7*\r,0);
				\draw (2.7*\r,0) -- (3.25*\r,0);
				\draw[dashed] (3.25*\r,0) -- (3.5*\r,0);
				
				\foreach \i/\j in {1.75/,4.75/ ,5.5/,6.5/}{
					\def\x{\i/2*\r}
					\draw (\x,-0.1)--(\x,+0.1);
					\node at (\x,-0.25) {$\j$};
				}
				
				\foreach \i/\j/\z in {1.5/5.666/0.3, 3.75/5.333/0.15}{
					\def\x{\i/2*\r+0.25*\r}
					\def\y{\j/2*\r+0.25*\r}
					\draw[line width=1,-stealth] (\x,0) -- (\x,\z*\r) to[in=90,out=90] (\y,\z*\r) -- (\y,0);}
				
				\node at(2.5*\r,0.55*\r) {\footnotesize $j$};
				\node at(2*\r,1.05*\r) {\footnotesize $i$};

    \node at (1,2) [shape=rectangle,draw,inner sep=2.5pt]{i.};
			}
			&
			
			\tikz[line width=1,scale=0.8]{
				\def\r{1.5}
				
				\draw[dashed] (0.5*\r,0) -- (0.75*\r,0);
				\draw (0.75*\r,0) -- (1.35*\r,0);
				\draw[dashed] (1.35*\r,0) -- (1.8*\r,0);
				\draw (1.8*\r,0) -- (2.35*\r,0);
				\draw[dashed] (2.35*\r,0) -- (2.7*\r,0);
				\draw (2.7*\r,0) -- (3.25*\r,0);
				\draw[dashed] (3.25*\r,0) -- (3.75*\r,0);
				\draw (3.75*\r,0) -- (4.25*\r,0);
				\draw[dashed] (4.25*\r,0) -- (4.5*\r,0);
				
				\foreach \i/\j in {1.75/,4.75/ ,5.5/,6.5/, 7.5/,8.5/}{
					\def\x{\i/2*\r}
					\draw (\x,-0.1)--(\x,+0.1);
					\node at (\x,-0.25) {$\j$};
				}
				
				\foreach \i/\j/\z in {1.5/5.5/0.3, 3.75/5.25/0.15, 5.75/7.5/0.15}{
					\def\x{\i/2*\r+0.25*\r}
					\def\y{\j/2*\r+0.25*\r}
					\draw[line width=1,-stealth] (\x,0) -- (\x,\z*\r) to[in=90,out=90] (\y,\z*\r) -- (\y,0);}
				
				\node at(2.5*\r,0.55*\r) {\footnotesize $j$};
				\node at(2*\r,1.05*\r) {\footnotesize $i$};
				\node at(3.6*\r,0.55*\r) {\footnotesize $\ell$};

    \node at (1,2) [shape=rectangle,draw,inner sep=2.5pt]{ii.};
			}
			& 
			\tikz[line width=1,scale=0.8]{
				\def\r{1.5}
				
				\draw[dashed] (0.5*\r,0) -- (0.75*\r,0);
				\draw (0.75*\r,0) -- (1.1*\r,0);
				\draw[dashed] (1.1*\r,0) -- (1.5*\r,0);
				\draw (1.5*\r,0) -- (1.75*\r,0);
				
				\draw[dashed] (1.75*\r,0) -- (2.15*\r,0);
				\draw (2.15*\r,0) -- (2.5*\r,0);
				
				\draw[dashed] (2.5*\r,0) -- (2.9*\r,0);
				\draw (2.9*\r,0) -- (3.65*\r,0);
				\draw[dashed] (3.65*\r,0) -- (3.9*\r,0);
				
				\foreach \i/\j in {1.75/,5/,6/,7/}{
					\def\x{\i/2*\r}
					\draw (\x,-0.1)--(\x,+0.1);
					\node at (\x,-0.25) {$\j$};
				}
				\foreach \i/\j/\z in {1.5/6.25/.45, 2.75/6./.3, 4/5.75/.15}{
					\def\x{\i/2*\r+0.25*\r}
					\def\y{\j/2*\r+0.25*\r}
					\draw[line width=1,-stealth] (\x,0) -- (\x,\z*\r) to[in=90,out=90] (\y,\z*\r) -- (\y,0);
				}
				\node at(2.7*\r,0.55*\r) {\footnotesize $j$};
				\node at(2.4*\r,0.95*\r) {\footnotesize $i$};
				\node at(2.2*\r,1.3*\r) {\footnotesize $p$};

    \node at (1,2) [shape=rectangle,draw,inner sep=2.5pt]{iii.};
			}
			
			\\
			$\rdelta_i \rdelta'_i=\rgamma_j \repsilon_i + \rgamma_i \rdelta_j$ & 
			$\rdelta_i \rdelta'_i=\rgamma_j \repsilon_i \rbeta_{\ell} + \rgamma_i \rdelta_j$ & 			
			$\rdelta_i \rdelta'_i=\rgamma_j \repsilon_i \rdelta_p + \rgamma_i \rdelta_j \repsilon_p$ \\

		\end{tabular}
		
	\end{center}
	
\textbf{R3b.} (Same end)
	Suppose that $D_i$ has a same-end child $D_j$.  Then 
	\begin{equation*}
		\repsilon_i \repsilon_i' = 
		\rdelta_j (\ralpha_i) (\repsilon_p) +
		\rdelta_i \repsilon_j (\ralpha_p),
	\end{equation*}
	where $\ralpha_i$ is present only if $D_j$ is sticky to
	$D_i$, $\repsilon_p$ is present only if 
	$D_i$ has a same-end parent $D_p$, and $\ralpha_p$ is present
	only if $D_i$ has a sticky parent $D_p$.
	\begin{center}
		\begin{tabular}{cccc}
			\tikz[line width=1,scale=0.6]{
				
				\draw (0.5,0) -- (1.5,0);
				\draw[dashed] (1.5,0) -- (2,0);
				\draw (2,0) -- (3,0);
				\draw[dashed] (3,0) -- (4,0);
				\draw (4,0) -- (5,0);
				
				\foreach \i/\j in {1.5/,2.5/,4.5/,5.5/,8.5/,9.5/}{
					\def\x{\i/2}
					\draw (\x,-0.1)--(\x,+0.1);
					\node at (\x,-0.25) {$\j$};
				}
				
				\foreach \i/\j in {1/4.5+0.066, 2.5/4.5-0.066}{
					\def\x{\i}
					\def\y{\j}
					\draw[line width=1,-stealth] (\x,0) -- (\x,0.15) to[in=90,out=90] (\y,0.15) -- (\y,0);}
				
				\node at(2.4,0.6) {\footnotesize $j$};
				\node at(2.7,1.5) {\footnotesize $i$};

    \node at (1,2) [shape=rectangle,draw,inner sep=2.5pt]{i.};
				
			}
			&
			\tikz[line width=1,scale=0.6]{
				
				\draw (0.5,0) -- (2,0);
				\draw[dashed] (2,0) -- (3,0);
				\draw (3,0) -- (4,0);
				
				\foreach \i/\j in {1.5/,2.5/,3.5/,6.5/,7.5/}{
					\def\x{\i/2}
					\draw (\x,-0.1)--(\x,+0.1);
					\node at (\x,-0.25) {$\j$};
				}
				
				\foreach \i/\j in {1/3.5+0.066, 1.5/3.5-0.066}{
					\def\x{\i}
					\def\y{\j}
					\draw[line width=1,-stealth] (\x,0) -- (\x,0.15) to[in=90,out=90] (\y,0.15) -- (\y,0);}
				
				\node at(2.45,0.4) {\footnotesize$j$};
				\node at(2.3,1.2) {\footnotesize$i$};

    \node at (1,2) [shape=rectangle,draw,inner sep=2.5pt]{ii.};
			}
			&
			
			\tikz[line width=1,scale=0.6]{
				
				\draw (0.5,0) -- (1.5,0);
				\draw[dashed] (1.5,0) -- (2,0);
				\draw (2,0) -- (3,0);
				\draw[dashed] (3,0) -- (3.5,0);
				\draw (3.5,0) -- (4.5,0);
				\draw[dashed] (4.5,0) -- (5.5,0);
				\draw (5.5,0) -- (6.5,0);
				
				\foreach \i/\j in {1.5/,2.5/, 4.5/,5.5/, 7.5/,8.5/, 11.5/,12.5/}{
					\def\x{\i/2}
					\draw (\x,-0.1)--(\x,+0.1);
					\node at (\x,-0.25) {$\j$};
				}
				
				\foreach \i/\j in {1/6+0.1, 2.5/6, 4/6-0.1}{
					\def\x{\i}
					\def\y{\j}
					\draw[line width=1,-stealth] (\x,0) -- (\x,0.15) to[in=90,out=90] (\y,0.15) -- (\y,0);}
				
				\node at(5,0.3) {\footnotesize$j$};
				\node at(3.25,1.3) {\footnotesize$i$};
				\node at(3.5,1.9) {\footnotesize$p$};

    \node at (1,2) [shape=rectangle,draw,inner sep=2.5pt]{iii.};
				
			}&
			\tikz[line width=1,scale=0.6]{
				
				\draw (0.5,0) -- (1.5,0);
				\draw[dashed] (1.5,0) -- (2,0);
				\draw (2,0) -- (3.5,0);
				\draw[dashed] (3.5,0) -- (4.5,0);
				\draw (4.5,0) -- (5.5,0);
				\foreach \i/\j in {1.5/,2.5/, 4.5/,5.5/,6.5/,9.5/,10.5/}{
					\def\x{\i/2}
					\draw (\x,-0.1)--(\x,+0.1);
					\node at (\x,-0.25) {$\j$};
				}
				
				\foreach \i/\j in {1/5+0.1, 2.5/5, 3/5-0.1}{
					\def\x{\i}
					\def\y{\j}
					\draw[line width=1,-stealth] (\x,0) -- (\x,0.15) to[in=90,out=90] (\y,0.15) -- (\y,0);}
				
				\node at(4,0.3) {\footnotesize$j$};
				\node at(3,1.0) {\footnotesize$i$};
				\node at(3,1.55) {\footnotesize$p$};

    \node at (1,2) [shape=rectangle,draw,inner sep=2.5pt]{iv.};
			}  \\
			$\repsilon_i \repsilon'_i = \rdelta_j + \rdelta_i \repsilon_j$ &
			$\repsilon_i \repsilon'_i = \rdelta_j \ralpha_i + \rdelta_i \repsilon_j$&
			$\repsilon_i \repsilon'_i = \rdelta_j \repsilon_p + \rdelta_i \repsilon_j$ &
			$\repsilon_i \repsilon'_i = \rdelta_j \ralpha_i \repsilon_p + \rdelta_i \repsilon_j$ \\[1em]
			
			\tikz[line width=1,scale=0.6]{
				
				\draw (0.5,0) -- (2,0);
				\draw[dashed] (2,0) -- (2.5,0);
				\draw (2.5,0) -- (3.5,0);
				\draw[dashed] (3.5,0) -- (4.5,0);
				\draw (4.5,0) -- (5.5,0);
				
				\foreach \i/\j in {1.5/,2.5/,3.5/,5.5/,6/,9.5/,10.5/}{
					\def\x{\i/2}
					\draw (\x,-0.1)--(\x,+0.1);
					\node at (\x,-0.25) {$\j$};
				}
				
				\foreach \i/\j in {1/5+0.1, 1.5/5, 3/5-0.1}{
					\def\x{\i}
					\def\y{\j}
					\draw[line width=1,-stealth] (\x,0) -- (\x,0.15) to[in=90,out=90] (\y,0.15) -- (\y,0);}
				
				\node at(4,0.4) {\footnotesize$j$};
				\node at(3.3,0.9) {\footnotesize$i$};
				\node at(3.1,1.6) {\footnotesize$p$};

    \node at (1,2) [shape=rectangle,draw,inner sep=2.5pt]{v.};
				
			}&
			\tikz[line width=1,scale=0.6]{
				
				\draw (0.5,0) -- (2.5,0);
				\draw[dashed] (2.5,0) -- (3.5,0);
				\draw (3.5,0) -- (4.5,0);
				
				\foreach \i/\j in {1.5/,2.5/,3.5/,4.5/,7.5/,8.5/}{
					\def\x{\i/2}
					\draw (\x,-0.1)--(\x,+0.1);
					\node at (\x,-0.25) {$\j$};
				}
				
				\foreach \i/\j/\z in {1/4+0.1/0.8, 1.5/4/0.45, 2/4-0.1/0.05}{
					\def\x{\i}
					\def\y{\j}
					\draw[line width=1,-stealth] (\x,0) -- (\x,\z) to[in=90,out=90] (\y,\z) -- (\y,0);}
				
				\node at(2.9,0.85) {\footnotesize$j$};
				\node at(2.75,1.4) {\footnotesize$i$};
				\node at(2.55,2) {\footnotesize$p$};

    \node at (1,2) [shape=rectangle,draw,inner sep=2.5pt]{vi.};
				
			}   &

   \tikz[line width=1,scale=0.6]{
				
				\draw (0.5,0) -- (2,0);
				\draw[dashed] (2,0) -- (2.5,0);
				\draw (2.5,0) -- (3.5,0);
				\draw[dashed] (3.5,0) -- (4.5,0);
				\draw (4.5,0) -- (5.5,0);
				\draw[dashed] (5.5,0) -- (6.5,0);
                \draw (6.5,0) -- (7,0);
				
				\foreach \i/\j in {1.5/,2.5/,3.5/,5.5/,6/,9.5/,10.5/,13/,14/}{
					\def\x{\i/2}
					\draw (\x,-0.1)--(\x,+0.1);
					\node at (\x,-0.25) {$\j$};
				}
				
				\foreach \i/\j in {1/6.75, 1.5/5+0.0833, 3/5-0.10833}{
					\def\x{\i}
					\def\y{\j}
					\draw[line width=1,-stealth] (\x,0) -- (\x,0.15) to[in=90,out=90] (\y,0.15) -- (\y,0);}
				
				\node at(3.4,0.9) {\footnotesize$j$};
				\node at(3.3,1.5) {\footnotesize$i$};
				\node at(4,2.1) {\footnotesize$p$};

    \node at (1,2) [shape=rectangle,draw,inner sep=2.5pt]{vii.};
				
			}&
			\tikz[line width=1,scale=0.6]{
				
				\draw (0.5,0) -- (2.5,0);
				\draw[dashed] (2.5,0) -- (3.5,0);
				\draw (3.5,0) -- (4.5,0);
				\draw[dashed] (4.5,0) -- (5.5,0);
                \draw (5.5,0) -- (6.25,0);
				
				\foreach \i/\j in {1.5/,2.5/,3.5/,4.5/,7.5/,8.5/,11.5/,12.5/}{
					\def\x{\i/2}
					\draw (\x,-0.1)--(\x,+0.1);
					\node at (\x,-0.25) {$\j$};
				}
				
				\foreach \i/\j/\z in {1/6/0.35, 1.5/4+0.0833/0.5, 2/4-0.0833/0.05}{
					\def\x{\i}
					\def\y{\j}
					\draw[line width=1,-stealth] (\x,0) -- (\x,\z) to[in=90,out=90] (\y,\z) -- (\y,0);}
				
				\node at(2.9,0.9) {\footnotesize$j$};
				\node at(2.85,1.5) {\footnotesize$i$};
				\node at(3.5,2.1) {\footnotesize$p$};

    \node at (1,2) [shape=rectangle,draw,inner sep=2.5pt]{viii.};
				
			}  
			\\
			$\repsilon_i \repsilon'_i = \rdelta_j \repsilon_p + \rdelta_i \repsilon_j \ralpha_p$ &
			$\repsilon_i \repsilon'_i = \rdelta_j \ralpha_i \repsilon_p + \rdelta_i \repsilon_j \ralpha_p$  & 
   
			$\repsilon_i \repsilon'_i = \rdelta_j + \rdelta_i \repsilon_j \ralpha_p$ &
			$\repsilon_i \repsilon'_i = \rdelta_j \ralpha_i + \rdelta_i \repsilon_j \ralpha_p$  \\
		\end{tabular}
	\end{center}

\end{lemma}

\begin{example} \label{ex:exrels}
The exchange relations for $\Sigma_D$ in \cref{cd-example-again} can be obtained from \cref{lem:explicit-reln}
as:
\begin{align*}
 &\mbox{R1. i), } (i,j)=(6,5):  \ralpha'_6 \ralpha_6= \rbeta_6 \repsilon_5+\ralpha_5, \quad  \mbox{R1. iv), } (p,i,j)=(6,5,4): 
  \ralpha'_5 \ralpha_5= \rbeta_5 \repsilon_4\ralpha_6 +\ralpha_5 \repsilon_5.&\\
& \mbox{R2, } (i,j)=(2,1):  \rbeta'_2 \rbeta_2= \rgamma_1+\ralpha_2 \rdelta_1, \quad  \mbox{R2, } (i,j)=(6,3): 
  \rbeta'_6 \rbeta_6= \rgamma_3+\ralpha_6 \rdelta_3.&\\
& \mbox{R3a ii), } (i,j,\ell)=(3,2,6):   \rdelta'_3 \rdelta_3= \rgamma_2 \repsilon_3 \rbeta_6+\rgamma_3 \rdelta_2, \quad   \mbox{R3a i), } (i,j)=(5,4): 
  \rdelta'_5 \rdelta_5= \rgamma_4 \repsilon_3+\rgamma_3 \rdelta_2.&\\
&\mbox{R3b i), } (i,j)=(3,2):  \repsilon'_3 \repsilon_3= \rdelta_2+\rdelta_3 \repsilon_2,  \quad  \mbox{R3b viii), } (p,i,j)=(6,5,4): 
 \repsilon'_5 \repsilon_5= \rdelta_4 \ralpha_5+\rdelta_5 \repsilon_4 \ralpha_6.&
\end{align*} 
See \cref{ex:primedvars} for expressions of the primed variables and \cref{domino-variables-formulas} for the non-primed ones.
\end{example}

\subsection{Cluster variables obtained from 
$\Sigma_D$ by a single mutation}

There are several ways to show that 
the cluster variables in $\A(\Sigma(D))$ are regular on $\Gr_{4,n}$.  
One way is to show that the cluster variables obtained by any single mutation
of the initial seed 
are regular functions on 
$\Gr_{4,n}$ (then use the \emph{Starfish lemma} 
\cite[Proposition 6.4.1]{FW6}).
We do this in  \cref{primedvars_prop},
which gives explicit formulas for these cluster variables as polynomials in Pl\"ucker coordinates.
We  give a second way in the proof of \cref{thm:quiver},
 which proves the stronger statement that 
each exchange relation in $\A(\Sigma(D))$ is in fact an exchange relation 
in the cluster structure for 
$\Gr_{4,n}$.

\begin{proposition}\label{primedvars_prop}
Let $D \in \mathcal{CD}_{n,k}$, and 
let $\bar{\alpha}'_i, \bar{\beta}'_i, \bar{\delta}'_i, \bar{\varepsilon}'_i$ be as in \cref{lem:explicit-reln}. Then:

\vspace{0.5em}
 \textbf{R1.} $\bar{\alpha}_i'=\lr{a_i,b_i+1,c_j,d_j}$.

\textbf{R2.}
	$\bar{\beta}_i'=\lr{a_j,b_j,b_i \rightarrowp_j n}$.	\vspace{-1em}
 \begin{itemize} 
			 \item[] \textbf{R3a.} i)
$\bar{\beta}_\ell,\bar{\delta}_p,\bar{\epsilon}_p$ are not present: $\bar{\delta}_i'= \begin{cases}
\lr{a_i \, b_i \, d_i|a_j \, b_j| c_i \, d_i  \nearrow_j n} & \text{if $D_j$ not sticky} \\
\lr{a'_j,a_j,b_j,d_i} & \text{if $D_j$ sticky}
\end{cases}$

\item[] \hspace{2.3em} ii) $\bar{\beta}_\ell$ is present: $\bar{\delta}_i'= \lr{ n \leftarrowp_j \, a_j \, b_j|a_\ell \, b_\ell| c_\ell \, d_\ell \nearrow_\ell n}$.
\item[] \hspace{2.3em} iii)
$\bar{\delta}_p,\bar{\varepsilon}_p$ are present: $\bar{\delta}_i'= \lr{ n \leftarrowp_p \, a_p \, b_p|c_i \, d_i| b_j \, a_j \rightarrowp_j n}$.
		\end{itemize}

\textbf{R3b.}
$\bar{\varepsilon}'_i=\lr{a_j,b_j,c_i,\nearrow_i n}$.

\end{proposition}

\begin{example}\label{ex:primedvars}
The variables $\bar{\alpha}'_i, \bar{\beta}'_i, \bar{\delta}'_i, \bar{\varepsilon}'_i$ in \cref{primedvars_prop} for \cref{cd-example} are as follows. See also \cref{ex:exrels} for the corresponding exchange relations.
We will denote $(10,\ldots,15)$ as $(A,\ldots,F)$.
\begin{equation*}
 \mbox{R1. i), } (i,j)=(6,5):  \ralpha'_6=\lr{8\,A\,C\,D}, \quad \mbox{R1. iv), } (p,i,j)=(6,5,4): 
  \ralpha'_5=\lr{9\,B\,C\,D}.
\end{equation*}
For rule R2, we have the pairs $(i,j)\in \{(2,1),(6,3)\}$, which gives respectively:
\begin{equation*}
 \rbeta'_2=\lr{3\,4\,6 \rightarrowp_1 F}=\lr{3 \, 4 \, 6 |2 \, 1|8 \, 9 \, F}, \quad \rbeta'_6=\lr{1\,2\,9 \rightarrowp_3 F}=\lr{1\,2\,9\,F}.
\end{equation*} 
\begin{equation*}
\mbox{R3a ii), } (i,j,\ell)=(3,2,6):  \rdelta'_3=\lr{F \leftarrowp_2 5\, 6|8\, 9|D\,E \nearrow_6 F}=\lr{F\,8\,9|6\, 5|8\,9|D\,E\,F}.
\end{equation*} 
\begin{equation*}
\mbox{R3a i), } (i,j)=(5,4):  \rdelta'_5=\lr{9\,A\,B\,D}.
\end{equation*} 
Finally, for rule R3b, we have the pairs $(i,j)\in \{(3,2),(5,4)\}$, which gives respectively:
\begin{equation*}
  \repsilon'_3=\lr{5\,6\,8 \nearrow_3 F}=\lr{5\,6\,8\,F}, \quad \repsilon'_5=\lr{A\,B\,C \nearrow_5 F}=\lr{A\,B\,C|9\,8|D\,E\,F}.
\end{equation*} 
\end{example}

In the following proof, we will use the following identities:
\begin{lemma} \label{lemma_identities} Let $a,b,c,d,e,x,y,z \in [n]$, then
 \begin{itemize}
     \item[i)] $\lr{n \leftarrowp_a b\,c\,d}=-\lr{b\,c\,d \rightarrowp_a n}$;
     \item[ii)] $\lr{x \, y \,z \,a}\lr{b \,c \,d \,e}=\lr{x \,y\,z\,b}\lr{a\,c\,d\,e}-\lr{x\,y\,z\,c}\lr{a\,b\,d\,e}+\lr{x\,y\,z\,d}\lr{a\,b\,c\,e}-\lr{x\,y\,z\,e}\lr{a\,b\,c\,d}$, and we will say we used the Pl\"ucker identity for the five marked vectors in $\lr{x\,y \,z \,\underline{a}}\lr{\underline{b} \,\underline{c} \,\underline{d} \,\underline{e}}$.
 \end{itemize}   
\end{lemma}

\begin{proof}[Proof of \cref{primedvars_prop}]
	We will provide detailed proofs of Cases (R1) and (R2), below.
	Case (R3) can be proved in a similar fashion.

\noindent{\bf Case R1. i)} $D_i$ has a sticky not same-end child $D_j$. $D_i$ is not a sticky child. 

In this case, the exchange relation is 
\begin{equation*}
\ralpha_i \ralpha_i' = 
		\rbeta_i \repsilon_j +
		 \ralpha_j,   
\end{equation*}
moreover $(a_i,b_i,b_i+1)=(a'_j,a_j,b_j)$. Then
\begin{equation*}
  \ralpha_j=\lr{b_j, c_j, d_j \nearrow_j n}=\lr{b_i+1, c_j, d_j \nearrow_j n}=\lr{b_i+1, c_j, d_j |b_i \, a_i| c_i \, d_i \nearrow_i n}, 
\end{equation*}
as $|c_j, d_j \nearrow_j n \rangle=| c_j, d_j |b_i \, a_i| c_i \, d_i \nearrow_i n \rangle$, in this case. From expanding the chain polynomial above:
\begin{equation*}
    \ralpha_j=\underbrace{\lr{b_i+1, c_j, d_j \, b_i}}_{-\repsilon_j} \underbrace{\lr{a_i \, c_i \, d_i \nearrow_i n}}_{\rbeta_i}-\underbrace{\lr{b_i+1, c_j, d_j \, a_i}}_{-\ralpha'_i} \underbrace{\lr{b_i \, c_i \, d_i \nearrow_i n}}_{\ralpha_i}.
\end{equation*}

\noindent{\bf Case R1. ii)} $D_i$ is not a sticky child and it has a sticky same-end child $D_j$. In this case, the exchange relation is 
\begin{equation*}
 \ralpha_i \ralpha_i' = 
		\rbeta_i \repsilon_j +
		 \ralpha_j \repsilon_i,   
\end{equation*}
moreover $(a_i,b_i,b_i+1,c_i,d_i)=(a'_j,a_j,b_j,c_j,d_j)$. Then
\begin{equation*}
 \underbrace{\lr{a_i \, c_i \, d_i \nearrow_i n}}_{\rbeta_i} \underbrace{\lr{b_i, b_i+1, c_i, d_i}}_{\repsilon_j}+\underbrace{\lr{b_i+1, c_i, d_i \nearrow_j n}}_{\ralpha_j} \underbrace{\lr{a_i, b_i, c_i, d_i}}_{\repsilon_i}= \underbrace{\lr{b_i \, c_i \, d_i \nearrow_j n}}_{\ralpha_i} \underbrace{\lr{a_i, b_i+1, c_i, d_i}}_{\ralpha'_i},
\end{equation*}
where we used the Pl\"ucker identity for the $5$ highlighted vectors in $\lr{\underline{a_i} \, c_i \, d_i \nearrow_i n} \lr{\underline{b_i}, \underline{b_i+1}, \underline{c_i}, \underline{{d_i}}}$, and $|c_i, d_i \nearrow_j n \rangle=|c_i, d_i \nearrow_i n \rangle$, in this case.

\noindent{\bf Case R1. iii)} $D_i$ has a sticky non same-end child $D_j$ and a sticky non same-end parent $D_p$. In this case, the exchange relation is 
\begin{equation*}
    \ralpha_i \ralpha_i' = 
		\rbeta_i \repsilon_j \ralpha_p +
		 \ralpha_j,
\end{equation*}
moreover $(a_p,b_p)=(a'_i,a_i)$ and $(a_i,b_i,b_i+1)=(a'_j,a_j,b_j)$. Then
\begin{equation*}
 \ralpha_j = \lr{b_j, c_j, d_j \nearrow_j n}=\lr{b_i+1, c_j, d_j \nearrow_j n} =
  \lr{b_i+1, c_j, d_j|b_i \, a_i| c_i \, d_i | a_i \, a'_i | c_p \, d_p \nearrow_p n},   
\end{equation*}
as $|c_j, d_j \nearrow_j n \rangle=|c_j, d_j|b_i \, a_i| c_i \, d_i | b_p \, a_p | c_p \, d_p \nearrow_p n \rangle$, in this case. From expanding the chain polynomial  above:
\begin{equation*}
 -\underbrace{\lr{b_i+1, c_j, d_j,a_i}}_{-\ralpha'_i} \underbrace{\lr{b_i \, c_i \, d_i | a_i \, a'_i | c_p \, d_p \nearrow_p n}}_{\ralpha_i}-\underbrace{\lr{b_i+1, c_j, d_j,b_i}}_{-\repsilon_j} \underbrace{ \lr{a_i \, c_i \, d_i \,a'_i}}_{-\rbeta_i} \underbrace{\lr{a_i \, c_p \, d_p \nearrow_p n}}_{\ralpha_p},   
\end{equation*}
as $|c_i, d_i|a'_i \, a_i| c_p \, d_p  \nearrow_p n \rangle=|c_i \, d_i \nearrow_i n \rangle$, in this case.

\noindent{\bf Case R1. iv)} $D_i$ has a sticky same-end child $D_j$ and a sticky non same-end parent $D_p$. In this case, the exchange relation is 
\begin{equation*}
 \ralpha_i \ralpha_i' = 
		\rbeta_i \repsilon_j \ralpha_p +
		 \ralpha_j \repsilon_i,   
\end{equation*}
moreover $(a_p,b_p)=(a'_i,a_i),(a_i,b_i,b_i+1)=(a'_j,a_j,b_j)$ and $(c_j,d_j)=(c_i,d_i)$. Then
\begin{align*}
&& \underbrace{\lr{b_i+1, c_i, d_i|a_i \, a'_i|c_p \, d_p \nearrow_p n}}_{\ralpha_j}\underbrace{\lr{a_i, b_i, c_i, d_i}}_{\repsilon_i} = \\ &=& \lr{a_i, c_i, d_i|a_i \, a'_i|c_p \, d_p \nearrow_p n} \underbrace{\lr{b_i+1, b_i, c_i, d_i}}_{-\repsilon_j}-\underbrace{\lr{b_i, c_i, d_i|a_i \, a'_i|c_p \, d_p \nearrow_p n}}_{\ralpha_i} \underbrace{\lr{b_i+1,a_i, c_i, d_i}}_{-\ralpha'_i},
\end{align*}
where we used the Pl\"ucker identity for the $5$ marked vectors in $\lr{\underline{b_i+1}, c_i, d_i|a_i \, a'_i|c_p \, d_p \nearrow_p n} \lr{\underline{a_i}, \underline{b_i}, \underline{c_i},\underline{d_i}}$ and 
$|c_i, d_i|a_i \, a'_i| c_p \, d_p \nearrow_p n \rangle=|c_i, d_i|b_p \, a_p| c_p \, d_p \nearrow_p n \rangle=|c_i, d_i \nearrow_i n \rangle$. Finally, we have:
\begin{equation*}
  \lr{a_i, c_i, d_i|a_i \, a'_i|c_p \, d_p \nearrow_p n}=\lr{a'_i \, a_i \, d_i|a_i \, c_i|c_i \, d_i \nearrow_p n}=-\underbrace{\lr{a'_i \, a_i \, d_i \, c_i}}_{\rbeta_i} \underbrace{\lr{a_i \, c_i \, d_i \nearrow_p n}}_{\ralpha_p}.
\end{equation*}

\noindent{\bf Case R1. v)} $D_i$ has a sticky non same-end child $D_j$ and a sticky same-end parent $D_p$. In this case, the exchange relation is
\begin{equation*}
\ralpha_i \ralpha_i' = 
		 \repsilon_j \ralpha_p +
		 \ralpha_j,    
\end{equation*}
moreover $(a_p,b_p)=(a'_i,a_i)$ and $(a_i,b_i,b_i+1,c_i,d_i)=(a'_j,a_j,b_j,c_j,d_j)$. Then
\begin{equation*}
 \ralpha_j = \lr{b_j, c_j,  d_j \nearrow_j n}  =   \lr{b_i+1, c_j, d_j|b_i \, a_i|c_i \, d_i \nearrow_p}, 
\end{equation*}
as $|c_j, d_j \nearrow_j n \rangle=|c_j, d_j|b_i \, a_i|c_i \, d_i \nearrow_p n \rangle$, in this case. From expanding the chain polynomial above:
\begin{equation*}
\underbrace{\lr{b_i+1, c_j, d_j \, b_i}}_{-\repsilon_j} \underbrace{\lr{a_i \, c_i \, d_i \nearrow_p n}}_{\ralpha_p}-\underbrace{\lr{b_i+1, c_j, d_j \, a_i}}_{-\ralpha'_i} \underbrace{\lr{b_i \, c_i \, d_i \nearrow_p n}}_{\ralpha_i},   
\end{equation*}
as $\nearrow_p n$ equals $\nearrow_i n$, in this case.

\noindent{\bf Case R1. vi)} $D_i$ has a sticky same-end child $D_j$ and a sticky same-end parent $D_p$.  In this case, the exchange relation is
\begin{equation*}
\ralpha_i \ralpha_i' = 
		 \repsilon_j \ralpha_p +
		 \ralpha_j \repsilon_i,    
\end{equation*}
moreover $(a_p,b_p)=(a'_i,a_i),(a_i,b_i,b_i+1)=(a'_j,a_j,b_j)$ and $(c_p,d_p)=(c_j,d_j)=(c_i,d_i)$. Then
\begin{equation*}
\underbrace{\lr{b_i+1, c_i, d_i \nearrow_p n}}_{\ralpha_j} \underbrace{\lr{a_i \, b_i \, c_i \,d_i}}_{\repsilon_i} = \underbrace{\lr{a_i, c_i, d_i \nearrow_p n}}_{\ralpha_p} \underbrace{\lr{b_i+1, b_i, c_i,d_i}}_{-\repsilon_j} - \underbrace{\lr{b_i, c_i, d_i \nearrow_p n}}_{\ralpha_i} \underbrace{\lr{b_i+1, a_i, c_i,d_i}}_{-\ralpha'_i},    
\end{equation*}
where we used the Pl\"ucker identity on the $5$ highlighted vectors in $\lr{\underline{b_i+1}, c_i, d_i \nearrow_p n} \lr{\underline{a_i} \, \underline{b_i} \, \underline{c_i} \,\underline{d_i}}$ and used the fact that $|\ldots \nearrow_p n \rangle, |\ldots \nearrow_i n \rangle, |\ldots \nearrow_j n \rangle$ are all equal, in this case.

\noindent {\bf Case R2.} $D_i$ has a head-to-tail sibling $D_j$ on its left. In this case, the exchange relation is 
 \begin{equation*}
 \rbeta_i \rbeta_i' = \rgamma_j+\ralpha_i \rdelta_j,    
 \end{equation*}
moreover $(c_j,d_j)=(a_i,b_i)$. Then, expanding the following chain polynomial
\begin{equation*}
\underbrace{\lr{ n \leftarrowp_j \, a_j \, b_j| a_i \, b_i| c_i \, d_i \nearrow_i n}}_{\rgamma_j}= \underbrace{\lr{ n \leftarrowp_j \, a_j \, b_j \, a_i}}_{-\rdelta_j} \underbrace{\lr{b_i \, c_i \, d_i \nearrow_i n}}_{\ralpha_i} - \underbrace{\lr{ n \leftarrowp_j \, a_j \, b_j \, b_i}}_{-\rbeta'_i} \underbrace{\lr{a_i \,  c_i \, d_i \nearrow_i n}}_{\rbeta_i},  
\end{equation*}
as $\rightarrowp_i$ equals $\nearrow_i$ in this case, and $\lr{ n \leftarrowp_j \, a_j \, b_j \, x}=-\lr{a_j \, b_j \, x \rightarrowp_j n}$. 

\end{proof}

\subsection{The proof of \cref{thm:quiver}}\label{sec:proofquiver}
We next turn to the proof of \cref{thm:quiver}.
In what follows, we say that a relation $xx'=M+M'$ is \emph{an exchange relation for $\Gr_{4,n}$} if there is a seed $\Sigma$ for $\Gr_{4,n}$ such that $xx'=M+M'$ is the exchange relation of $x$ in $\Sigma$. 

As a first step towards proving \Cref{thm:quiver}, we 
prove the following lemma.

\begin{lemma}\label{lem:quiver_top_chord} 
Suppose $D_i=(a_i,b_i,c_i,d_i)
$ is the rightmost 
	top chord of $D$.  
Then every exchange relation in 
\Cref{lem:explicit-reln} is an exchange relation for $\Gr_{4,n}$.
In particular, in this case, the relations become the following:
\begin{enumerate}
\item[R1.] (Sticky, not same-end) 
		Suppose $D_j =(b_i, b_j, c_j, d_j) 
		$ where $d_j < c_i$ is a sticky child of $D_i$.
		Then 
		\begin{equation*}
			\underbrace{\lr{b_i c_i d_i n}}_{\ralpha_i}\underbrace{\lr{a_i b_j c_j d_j}}_{\ralpha_i'} = 
			\underbrace{\lr{a_i c_i d_i n}}_{\rbeta_i} \underbrace{\lr{b_i b_j c_j d_j}}_{\repsilon_j} + 
			\underbrace{\lr{b_j c_j d_j \br b_i a_i  \br c_i d_i n}}_{\ralpha_j}.
		\end{equation*}
	\item[R1'.] (Sticky same-end) 
		  Suppose $D_j =(b_i, b_j, c_i, d_i)
		  $ is a sticky same-end child of $D_i$. Then 
		\begin{equation*}
		\underbrace{\lr{b_i c_i d_i n}}_{\ralpha_i} \underbrace{\lr{a_i b_j c_i d_i}}_{\ralpha_i'} = 
		\underbrace{\lr{a_i c_i d_i n}}_{\rbeta_i} \underbrace{\lr{b_i b_j c_i d_i}}_{\repsilon_j}+ 
		\underbrace{\lr{b_j c_i d_i n}}_{\ralpha_j} \underbrace{\lr{a_i b_i c_i d_i}}_{\repsilon_i}.
		\end{equation*}
    \item[R2.] (Head-to-tail) Suppose
		$D_j=(a_j, b_j,a_i,b_i)
		$ is a head-to-tail sibling of $D_i$ to its left.
	Then \begin{equation*} 	
		\underbrace{\lr{a_i c_i d_i n}}_{\rbeta_i} \underbrace{\lr{a_j b_j b_i n}}_{\rbeta_i'} = 
		\underbrace{\lr{a_j b_j n \br a_i b_i \br c_i d_i n}}_{\rgamma_j} 
		+\underbrace{ \lr{b_i c_i d_i n}}_{\ralpha_i} \underbrace{\lr{a_j b_j a_i n}}_{\rdelta_j}.
           \end{equation*}
   \item[R3a.] (Same end, not sticky)
	   Suppose that $D_j=(a_j, b_j,c_i,d_i)
	   $ is a same-end nonsticky child of $D_i$ (so $b_i < a_j$).
		Then 
	\begin{equation*}
	\underbrace{\lr{a_i b_i c_i n}}_{\rdelta_i} \underbrace{\lr{a_j b_j d_i \br a_i b_i \br c_i d_i n}}_{\rdelta_i'}
	= \underbrace{\lr{a_j b_jn \br c_i d_i \br a_i b_i n} }_{\rgamma_j}\underbrace{\lr{a_i b_i c_i d_i}}_{\repsilon_i}
	+ \underbrace{\lr{a_i b_i d_i n}}_{\rgamma_i} \underbrace{\lr{a_j b_j c_i \br b_i a_i \br c_i d_i n}}_{\rdelta_j}.
        \end{equation*}
   \item[R3a'.] (Same end sticky)
	   Suppose $D_j =(b_i, b_j, c_i, d_i)
	   $ is a sticky same-end child of $D_i$.
		Then 
	\begin{equation*}
		\underbrace{\lr{a_i b_i c_i n}}_{\rdelta_i}\underbrace{\lr{a_i b_i b_j d_i}}_{\rdelta_i'} = 
		\underbrace{\lr{a_i b_i b_j n}}_{\rgamma_j} \underbrace{ \lr{a_i b_i c_i d_i}}_{\repsilon_i} + 
		\underbrace{\lr{a_i b_i d_i n}}_{\rgamma_i} \underbrace{\lr{a_i b_i b_j c_j}}_{\rdelta_j}.
        \end{equation*}
	\item[R3b.] (Same end, not sticky)
	   Suppose $D_j=(a_j, b_j,c_i,d_i)
	   $ is a same-end nonsticky child of $D_i$. Then
	\begin{equation*}
		\underbrace{\lr{a_i b_i c_i d_i}}_{\repsilon_i} \underbrace{\lr{a_j b_j c_i n}}_{\repsilon_i'} = 
		\underbrace{\lr{a_j b_j c_i \br b_i a_i \br c_i d_i n}}_{\rdelta_j} + 
		\underbrace{\lr{a_i b_i c_i n}}_{\rdelta_i} \underbrace{\lr{a_j b_j c_i d_i}}_{\repsilon_j}.
        \end{equation*}
	\item[R3b'.] (Same end sticky)
	     Suppose $D_j =(b_i, b_j, c_i, d_i)
	     $ is a sticky same-end child of $D_i$. Then
	\begin{equation*}
\underbrace{\lr{a_i b_i c_i d_i}}_{\repsilon_i} \underbrace{\lr{b_i b_j c_i n}}_{\repsilon_i'} = 
	\underbrace{\lr{a_i b_i b_jc_i}}_{\rdelta_j} \underbrace{\lr{b_i c_i d_i n}}_{\ralpha_i} +
\underbrace{	\lr{a_i b_i c_i n}}_{\rdelta_i} \underbrace{\lr{b_i b_j c_i d_i}}_{\repsilon_j}.
        \end{equation*}
\end{enumerate}
\end{lemma}

\begin{proof} Let $N$ be the set of 7 indices appearing in the domino variables of $D_i, D_j$. One can check computationally that the relations listed are exchange relations for $\Gr_{4, N} \cong \Gr_{4,7}$, for example by searching through all seeds. Using \cref{lem:add-marker-embed-gr}, this shows that the relations are exchange relations for $\Gr_{4, n}$ for any $[n] \supset N$.
\end{proof}

\begin{definition}\label{def:stable-under-promotion}
Let $D \in \CD$ be a chord diagram with rightmost top chord $\ctop=(a,b,c,d)$ and let $D_L, D_R$ be as in \cref{domino-var}.

Let $R$ be a relation from \cref{def:seed} for $\tilde{D}$, where $\tilde{D}=D_L$ or $\tilde{D}=D_R$ and suppose it is an exchange relation for $\Gr_{4,n}$. We say that 
$R$ is 
\emph{stable under promotion (induced by $\ctop$)} if 
relation $R$ with the corresponding domino variables of 
$D$ is also an exchange relation for $\Gr_{4,n}$.
\end{definition}
For example, if
 \begin{equation} \label{sample}
 \rbeta_i^{\tilde{D}} (\rbeta_i^{\tilde{D}})' = 
\rgamma_j^{\tilde{D}}+\ralpha_i^{\tilde{D}} \rdelta_j^{\tilde{D}}
 \end{equation} 
is an exchange relation for $\Gr_{4,n}$
and 
 \begin{equation*}
	\rbeta_i^{{D}} (\rbeta_i^{{D}})' = 
	\rgamma_j^{{D}}+\ralpha_i^{{D}} \rdelta_j^{{D}}
\end{equation*} 
is as well, then we say \eqref{sample} is stable under promotion (induced by $\ctop$).

\begin{proposition}\label{prop:exch-rel-true-on-Gr}
Let $D$ be a chord diagram on $[n]$. The relations (R1), (R2), (R3a) and (R3b) in 
\Cref{lem:explicit-reln} are exchange relations for $\Gr_{4,n}$. 
\end{proposition}
\begin{proof}

We will verify this using promotion. Because promotion is a quasi-cluster homomorphism (\cref{thm:promotion2}), we can obtain exchange relations for $\Gr_{4,n}$ by promoting exchange relations for $\Gr_{4, N_L}$ and $\Gr_{4, N_R}$. Using the notation of \cref{thm:promotion2}, let $x$ be a mutable variable of $\A(\Sigma_0)$ in a seed $\Sigma$ and $u$ the corresponding cluster variable of $\A(\Fr(\Sigma_1))$ in a seed $\Sigma'$ (which is also a seed for $\Gr_{4,n}$). If $xx'=M+ M$ is the exchange relation for $x$ in $\Sigma$ and $u u'=N+N$ is the exchange relation in $\Sigma'$, then $\Psi_{ac}(M + M')$ differs from $N+N'$ by a Laurent monomial in $\mathcal{T'}$. To clear this monomial, we write $\Psi_{ac}(M+M')$ over a common denominator, then take the numerator and remove all common factors in $\mathcal{T'}$. Call the resulting binomial $f_{ac}(\Psi_{ac}(M+M'))$; it is equal to $N + N'$.

Consider a chord $D_i$ in $D$.
By repeatedly removing penultimate markers which are not in any chord or removing the rightmost top chord and taking either $D_L$ or $D_R$ (defined as in \cref{thm:cluster}), we obtain a chord diagram $D'$ where $D_i$ is the rightmost top chord. Suppose the rightmost top chords we removed in this process were, in order, $(a_1,b_1,c_1,d_1),\dots,
(a_{p},b_{p},c_{p},d_p)$.
	Using \cref{lem:when-promote-var-factor},
we will show that each relation
in \Cref{def:seed} is obtained from an exchange relation in 
\Cref{lem:quiver_top_chord} by the sequence $f_{a_1 c_1} \circ \Psi_{a_1 c_1} \circ \dots \circ f_{a_p c_p} \circ \Psi_{a_p c_p}$. By the previous paragraph, this will imply each relation in \cref{def:seed} is an exchange relation for $\Gr_{4,n}$.

The following observation will be useful. It follows immediately from the definition of promotion.

\begin{observation}\label{obs:same-denominator}
	Suppose $D$ is a chord diagram on $N$ with largest markers $x,y:=x+1,n$ and choose $a, c$ such that $a<c-1, c<n-1$ and either $x \leq a$ or $a+1 \leq x \leq c$. Let $M, M'$ be any monomials in the domino variables of $D$ with no clauses containing $y$ but not $x$, the same number of clauses containing both $y$ and $n$, and the same number of clauses with $n$ but not $y$. Then the denominators of $\Psi_{ac}(M)$ and $\Psi_{ac}(M')$ are equal. 
\end{observation}

	{\bf [Case (R2.)]}
 Suppose that $D_i$ and $D_j$ are 
chords which are head-to-tail siblings, with $D_j$ to the left of $D_i$.
Since $D_i$ is a rightmost top chord in $D'$, \Cref{lem:quiver_top_chord} implies that 
\begin{equation}\label{R2}
	\rbeta_i^{D'} (\rbeta_i^{D'})' = 
	\rgamma_j^{D'}+\ralpha_i^{D'} 
 \rdelta_j^{D'} \quad \text{or} \quad \lr{a_i c_i d_i n} (\rbeta_i^{D'})' = \lr{n a_j b_j \br a_i b_i \br c_id_i n} + \lr{b_i c_i d_i n} \lr{a_j b_j a_i n}
 \end{equation}
is an exchange relation for $\Gr_{4,n}$.
Now we show inductively that \eqref{R2} is stable under the promotions $\Psi_{a_p c_p}, \dots, \Psi_{a_1 c_1}$. Suppose it is stable up until $\Psi=\Psi_{a_q c_q}$, corresponding to adding the rightmost top chord $D_q$. Let $\tilde{D}$ be the chord diagram obtained from $D'$ by adding chords $D_p, D_{p-1}, \dots, D_{q+1}$ (and the appropriate chord diagrams beneath or to the left). By induction, \eqref{R2} holds for $\tilde{D}$. Using \cref{prop:explicit}, the terms on the right hand side satisfy the conditions of \cref{obs:same-denominator}. More specifically, $\ralpha_i^{\tilde{D}}= \lr{b_i c_i d_i \rchn_i n}$,
\begin{equation*}
	\rgamma_j^{\tilde{D}}= \begin{cases}
		 \lr{n \lchn_j a_j b_j \br a_i b_i \br c_id_i \rchn_j n} & \text{$D_j$ not sticky}\\
		  \lr{ a_j a_j' b_j \br a_i b_i \br c_id_i \rchn_j n}& \text{$D_j$ sticky}
	\end{cases} \quad \text{and} \quad
\rdelta_j^{\tilde{D}}= \begin{cases}
	\lr{a_j b_j a_i \rchn_j n} & \text{$D_j$ not sticky}\\
	\lr{a_j' a_j b_j a_i}& \text{$D_j$ sticky}\\
\end{cases}.
\end{equation*}
There are no clauses containing the penultimate marker $d$ of $\tilde{D}$ but not $d-1$. In $\rgamma_j^{\tilde{D}}$, the clause $\br c_id_i \rchn_j n \rangle$ is equal to $\br c_id_i \rchn_i n \rangle$ because $D_j, D_i$ are siblings. Also, $b_j a_i \rchn_j n$ in $\rdelta_j^{\tilde{D}}$ involves the same chain of ancestors as $n \lchn_j a_j b_j$ in $\rgamma_j^{\tilde{D}}$, so if this chain is nontrivial, the end clause is the same in each. This verifies the conditions of \cref{obs:same-denominator}.

So after promotion $\Psi_{a_q c_q}$, both terms have the same denominator. By \cref{lem:when-promote-var-factor}, nontrivial factorization occurs in the numerator only if $D_j$ is a sticky child of $D_q$. In this case, the numerator of $\Psi(\rgamma_j^{\tilde{D}})$ is $\ralpha_q \rgamma_j$ and the numerator of $\Psi(\rdelta_j^{\tilde{D}})$ is $\ralpha_q \rdelta_j$, so $\ralpha_q$ is a common factor of the numerator. So applying $f_{a_q c_q}$ gives \eqref{R2} for the chord diagram with $D_q$ added. Thus \eqref{R2} is stable under promotion by $D_q$.

{\bf [Case (R3a.)]}
 Suppose that $D_i$ is a same-end parent of $D_j$.  
\Cref{lem:quiver_top_chord} implies that 
\begin{equation}\label{R3} 
	\rdelta_i^{D'} (\rdelta_i^{D'})' = 
	\rgamma_j^{D'} \repsilon_i^{D'}  +
	\rgamma_i^{D'} \rdelta_j^{D'}
 \end{equation}
is an exchange relation for $\Gr_{4,n}$.
It is straightforward to verify that \eqref{R3} is stable under any sequence of promotions corresponding to chords which are not same-end parents or head-to-tail siblings with $D_i$. (Note that this implies none of the chords in the sequence are same-end or head-to-tail with $D_i$.) Indeed, using \cref{prop:explicit}, the assumptions of \cref{obs:same-denominator} hold for each promotion in the sequence, so both terms of the right-hand side have the same denominator. There are no factorizations in the numerator unless the chord $D_q$ added is sticky to $D_i$. Then by \cref{lem:when-promote-var-factor}, both $\Psi(\rgamma_j)$ and $\Psi(\rgamma_i)$ have a factor of $\ralpha_p$ and there are no other factorizations. So we see \eqref{R3} is stable.

Now suppose one of the chords $D_1, \dots, D_p$ is a same-end parent of $D_i$. This chord is necessarily $D_p$, so we are
applying product promotion $\Psi=\Psi_{a_p c_p}$ to \eqref{R3}.
The denominator of
$\Psi(\rgamma_j^{D'})$ is
$\epsilon_p$, 
and the denominator of $\Psi(\rgamma_i^{D'})$ is $\delta_p$; $\Psi$ fixes both $\rdelta_j^{D'}$ and $\repsilon_i^{D'}$. 
If $D_p$ is not a sticky parent, then by 
 \cref{lem:when-promote-var-factor}, 
 when we apply promotion 
to the right-hand side of \eqref{R3}, no extra factors will 
appear in the numerators. If $D_p$ is a sticky parent, then by 
\cref{lem:when-promote-var-factor} (1), the numerator of $\Psi(\rgamma_i^{D'})$ is $\ralpha_{p} \rgamma_i^{D''}$; by \cref{lem:when-promote-var-factor} (3), the numerator of $\Psi(\rgamma_j^{D'})$ is $\ralpha_{p} \rgamma_j^{D''}$; and no other factorizations occur.
In either case, after applying $f_{a_p c_p}$, we obtain the exchange relation 
\begin{equation}\label{R3next}
	\rdelta_i^{D''} (\rdelta_i^{D''})' = 
		\rgamma_j^{D''} \repsilon_i^{D''} \rdelta_p^{D''} +
		\rgamma_i^{D''} \rdelta_j^{D''} \repsilon_p^{D''},
 \end{equation}
where $D''$ is the diagram obtained from $D'$ by adding $D_p$ and chords to its left. Further, \eqref{R3next} is stable under all subsequent promotions, again using \cref{prop:explicit} to check \cref{obs:same-denominator} and \cref{lem:when-promote-var-factor} to check that any factorization in numerators contributes the same factor to each term. 

Now suppose instead one of the chords $D_1, \dots, D_p$ is a head-to-tail sibling with $D_i$ to its right. Again, this chord must be $D_p$. Note that $\Psi_{a_p c_p}$ acts nontrivially
on a domino variable of \eqref{R3} only if that domino variable has a clause
containing $b_r=d_i$ but not $a_r=c_i$. 
Using \Cref{prop:explicit}, the only such domino variable of \eqref{R3} is
$\rgamma_i^{D'}=\lr{a_i, b_i, b_p,n}$, 
and hence when we apply promotion to 
$\rgamma_i^{D'}$ we introduce a 
new denominator of $\beta_p$.  There is no factorization in any numerators, so applying $f_{a_p c_p}$, 
we obtain \begin{equation}\label{R3nextnext} 
	\rdelta_i^{D''} (\rdelta_i^{D''})' = 
	\rgamma_j^{D''} \repsilon_i^{D''} \rbeta_p^{D''} +
	\rgamma_i^{D''} \rdelta_j^{D''},
 \end{equation}
where $D''$ is the chord diagram obtained from $D'$ by adding $D_p$
and chords below it. The relation \eqref{R3nextnext} is stable under all subsequent promotions, by a similar argument as previously. 

The proofs of relations (R1) and (R3b) are similar to the proofs
above, so we omit them. 
\end{proof}

\begin{proof}[Proof of \cref{thm:quiver}]

	Let $D \in \CD$ and $\Sigma_D$ be the seed defined in \cref{def:seed}. By \cref{thm:cluster}, $\txx(D)$ is an extended cluster for $\Gr_{4,n}$ and  $\widetilde{\Sigma}_D= (\txx(D), \tQ_D)$ is the unique seed of $\Gr_{4,n}$ with that cluster. We need to show that for any $x \in \Mut(Z_D)$, the exchange relation of $x$ is the same in $\Sigma_D$ and $\widetilde{\Sigma}_D$.
	
	Recall that the mutable variables $\Mut(Z_D)$ are precisely those appearing on the left hand side of relations (R1), (R2), (R3a), (R3b). Fix $x \in\Mut(Z_D)$, let $R$ be its exchange relation in $\Sigma_D$, and let $U$ be all neighbors of $x$ in $\Sigma_D$. By \cref{prop:exch-rel-true-on-Gr}, there is a seed $\Sigma'$ of $\Gr_{4,n}$ containing $\{x\} \cup U$ and $R$ is the exchange relation of $x$ in $\Sigma'$. Now, by \cite[Theorem 10]{CL}, any seed of $\Gr_{4,n}$ containing $\{x\} \cup U$ may be obtained from $\Sigma'$ by a sequence of mutations avoiding $\{x\} \cup U$, and thus $x$ has exchange relation $R$ in any seed containing $\{x\} \cup U$. Since $\widetilde{\Sigma}_D$ contains $\{x\} \cup U$, $x$ also has exchange relation $R$ in $\widetilde{\Sigma}_D$. We have shown the first sentence of the theorem statement. The second sentence follows immediately.
\end{proof}

\subsection{Signed seeds and the proof of \cref{thm:sign-definite}}\label{sec:thmsign}

The following notion of \emph{signed seed} will
be useful for proving \Cref{thm:sign-definite}.

\begin{definition}[Signed seed]\label{def:signedseed}
	Given a seed $((x_1,\dots,x_s),Q)$, 
	let $(\sigma_1,\dots,\sigma_s)\in \{1,-1\}^s$
	be an assignment of signs to each cluster and frozen variable.
	We denote the vertices of  $Q$ by
	$v_1,\dots,v_s$ (preserving this labeling after mutation), and 
	we write $\sigma(x_i)$ and $\sigma(v_i)$ to denote 
	$\sigma_i$.  For $v_i$ a mutable vertex of $Q$, let
	$$\inn_{\sigma}^Q(v_i) = \prod_{v_j \to v_i} \sigma_j \text{ and }
	\out_{\sigma}^Q(v_i) = \prod_{v_i \to v_j} \sigma_j.$$
	We say that 
	$\tilde{\Sigma} = ((x_1,\dots,x_s),(\sigma_1,\dots,\sigma_s),Q)$
	is a \emph{signed seed} if 
	for each mutable vertex $v_i$, 
	\begin{equation}\label{eq:signed}
		\inn_{\sigma}^Q(v_i) = \out_{\sigma}^Q(v_i). 
	\end{equation}
\end{definition}

\begin{definition}
	We define mutation of signed seeds by
	$$\mu_k(\tilde{\Sigma}) = 
	(\mu_k(x_1,\dots,x_s), 
	(\sigma'_1,\dots,\sigma'_s), \mu_k(Q)),$$
	where $\sigma'_k 
	= \sigma_k \cdot \inn_{\sigma}^Q(v_k),$ 
	and $\sigma'_i = \sigma_i$ for $i\neq k$.
\end{definition}

\begin{remark} The definition of signed seeds and their mutation is motivated by the following. Say $x_1, \dots, x_s$ are real nonzero numbers with signs $\sigma_1, \dots, \sigma_s$. When \eqref{eq:signed} holds, then for any $i$, the number
 $$x'_i = \frac{\prod_{i\to j} x_j + \prod_{j \to i} x_j}{x_i}$$
has sign
	$\sigma'_i = 
	\sigma_i \cdot \inn_{\sigma}^Q(v_i) 
	= \sigma_i \cdot \out_{\sigma}^Q(v_i).$ 
	So if $\sigma_i$ is the sign of a cluster variable $x_i$ when evaluated on some point $p$, signed seed mutation determines the sign of an arbitrary cluster variable on $p$.
\end{remark}

\begin{remark}
There is a convenient reformulation of the notion of signed seed
in terms of the \emph{exchange matrix} $B(Q)$ associated to a quiver
$Q$ with $s$ vertices, of which $r\leq s$ are mutable.  
Recall that the \emph{exchange matrix}
$B:=(b_{ij})$ of $Q$ is the $s\times r$ 
	matrix 
	defined by 
$$b_{ij} = \begin{cases}
\# (\text{arrows }i\to j) &\text{ if there is an arrow }i\to j\\
  - \# (\text{ arrows }j\to i) &\text{ otherwise. }
\end{cases}$$
Thus we can write 
	$$\inn_{\sigma}^Q(v_i) = \prod_{b_{ji}>0} (\sigma_j)^{b_{ji}}
	\text{ and }
	\out_{\sigma}^Q(v_i) = \prod_{b_{ji}<0} (\sigma_j)^{b_{ji}}.$$
	It follows that 
	\eqref{eq:signed} is equivalent to the statement that 
	\begin{equation}\label{eq2:signed}
		\prod_{j} (\sigma_j)^{b_{ji}} = 1.
	\end{equation}
\end{remark}

\begin{proposition}
If $\tilde{\Sigma}$ is a signed seed, then 
$\mu_k(\tilde{\Sigma})$ is a signed seed
for
any mutable direction $k$.
\end{proposition}
\begin{proof}
We will show that after we mutate
in direction $k$, \eqref{eq:signed} still holds for
$\mu_k(\tilde{\Sigma})$.  Let $Q'$ denote $\mu_k(Q)$.  
In $Q'$, \eqref{eq:signed} clearly holds at vertex $v_k$ itself,
and also for any vertex $v_i$ which is not
adjacent to $v_k$.
So let $v_i$ be a mutable vertex which is adjacent to $v_k$, i.e. 
$b_{ki} \neq 0$: we 
need to verify \eqref{eq:signed}.

Without loss of generality we suppose that $b_{ki}>0$.
(The argument
when $b_{ki}<0$  is similar,
so we omit it.)  
Let $B'=(b'_{ij})$
be the exchange matrix of $Q'$.
We will verify \eqref{eq2:signed}, that is, 
we will show that $\prod_j (\sigma')^{b'_{ji}} = 1$.
	Note that since $b_{ki}>0$, when we mutate at $k$, there are three cases:
	$$b'_{ji} = \begin{cases}
		-b_{ji} & \text{ if }j=k\\
		b_{ji} + b_{jk}b_{ki} & \text{ if }b_{jk}>0\\
		b_{ji} & \text{ otherwise}.
	\end{cases} $$
	Now we have that 

\begin{align*}
\prod_j (\sigma')^{b'_{ji}} & = 
		(\sigma'_k)^{b'_{ki}} \cdot \prod_{j\neq k, b_{jk}>0} (\sigma_j)^{b'_{ji}} \cdot \prod_{j\neq k, b_{jk}\leq 0} (\sigma_j)^{b_{ji}}\\
		&= (\sigma_k \inn_{\sigma}^Q(v_k))^{-b_{ki}} \cdot
		\prod_{j\neq k, b_{jk}>0} (\sigma_j)^{b_{ji}+b_{jk} b_{ki}} \cdot
		\prod_{j\neq k, b_{jk} \leq 0} (\sigma_j)^{b_{ji}} \\
		&= \left( \sigma_k^{b_{ki}} \cdot \prod_{j\neq k, b_{jk}>0} 
		\sigma_j^{b_{ji}} \cdot \prod_{j\neq k, b_{jk}\leq 0} \sigma_j^{b_{ji}} \right) \cdot 
		\inn_{\sigma}^Q(v_k)^{b_{ki}} \cdot
		\prod_{j \neq k, b_{jk}>0} (\sigma_j)^{b_{jk} b_{ki}} \\
		&= \prod_j (\sigma_j)^{b_{ji}} \cdot 
		\inn_{\sigma}^Q(v_k)^{b_{ki}} \cdot
		 \inn_{\sigma}^Q(v_k)^{b_{ki}}=1.
\end{align*}
\end{proof}

\begin{proposition}[$\Sigma_D$ is a signed seed]\label{prop:signed-seed}
Use \cref{prop:domino-var-signs-on-tile}
to assign a sign in $\{+,-\}$ to each domino variable of 
$\Sigma_D$.
Then with these signs, the seed
$\Sigma_D$ in \Cref{def:seed} is a signed 
seed.\footnote{Each mutable variable of 
$\Sigma_D$ is incident only to domino variables,
so for the purpose of checking the signed seed property,
	it doesn't matter what sign we associate to 
the other frozen variables of $\Sigma_D$.}
\end{proposition}

\begin{example} 
	The seed $\Sigma_D$ in \cref{cd-example-again}, with the signs given in \cref{ex:stdBCFW_signs} is a signed seed. It is easy to verify that $\mbox{in}(\rzeta_i)=\mbox{out}(\rzeta_i)$, for each mutable $\rzeta_i \in \Mut(Z_D)$. For example:
\begin{align*}
&\mbox{in}(\ralpha_5)=\mbox{sgn}(\ralpha_6) \, \mbox{sgn}(\rbeta_5) \, \mbox{sgn}(\repsilon_4)= (+1)(+1)(+1)=+1= (+1)(+1)=\mbox{sgn}(\ralpha_4) \, \mbox{sgn}(\repsilon_5)=\mbox{out}(\ralpha_5).\\
&\mbox{in}(\ralpha_6)=\mbox{sgn}(\rbeta_6) \, \mbox{sgn}(\repsilon_5)= (-1)(+1)=-1=\mbox{sgn}(\ralpha_5)=\mbox{out}(\ralpha_6).\\
\end{align*}
\end{example}

\begin{proof}[Proof of \cref{prop:signed-seed}] 
Recall that \cref{prop:domino-var-signs-on-tile} recursively follows from \cref{lem:signs}. We will use \cref{lem:signs} recursively in this proof. Because \cref{lem:signs} and \cref{prop:domino-var-signs-on-tile} are consistent, the signs in this proof obtained by using \cref{lem:signs} are the same as those given in \cref{prop:domino-var-signs-on-tile}.

We start by showing that for each relation in 
\Cref{lem:quiver_top_chord}, which concerns the case where 
chord $D_i$ is the rightmost top chord of our chord diagram,
each of the 
two monomials on the right-hand side
has the same sign, viewed as functions
on the associated tile of the amplituhedron. We then show using \cref{lem:signs} that this property is preserved under promotion.

In what follows, suppose that $D_i=(a,b,c,d)$ is the rightmost top chord, and let $\ell:= k_R+1$ be the number of chords in the subdiagram
consisting of $D_i$ and all chords below it.

In Case (R2), we have that $\sgn(\ralpha_i)=(-1)^{\ell-1}$ by \cref{lem:signs}.
Before we added $D_i$ to the chord diagram, we had that 
$\rgamma_j=\lr{a_j,b_j,b,n}$,
with sign $+1$, and 
$\rdelta_j=\lr{a_j,b_j,a,n}$, with sign $-1$, again by \cref{lem:signs}.
Since both $\rgamma_j$ and $\rdelta_j$ contain an odd number of $n$'s,
and come from $D_L$, their promotions pick up a sign of $(-1)^{\ell}$ by \cref{lem:signs}. Note that $\rdelta_j$ is fixed by promotion, while $\rgamma_j$ is affected by rule (a2). By \cref{lem:when-promote-var-factor}, $\Psi(\rgamma_j)= \rgamma_j/\lr{acdn}$. Since $\lr{acdn}= \rbeta_i$ has sign $(-1)^\ell$, we will have 
$\sgn(\rgamma_j) = (-1)^{2\ell}=+1$,
and $\sgn(\rdelta_j)=(-1)(-1)^{\ell}.$
It follows that $\rgamma_j$ will have sign $+1$
and $\ralpha_i \rdelta_j$ will have sign $+1$,
 so both monomials
on the right-hand side are positive. 

Similar arguments yield the signs of the monomials in the exchange
	relations for all other cases.

In Case (R1), we have that 
	$\sgn(\rbeta_i)=(-1)^{\ell}$,
	$\sgn(\repsilon_j)=1$,
	and $\sgn(\ralpha_j)=(-1)^{\ell-2}$.

	In Case (R1'), we have that 
	$\sgn(\rbeta_i)=(-1)^{\ell}$,
	$\sgn(\repsilon_j)=
	\sgn(\repsilon_i)=1$,
	and $\sgn(\ralpha_j)=(-1)^{\ell-2}$.

 In Case (R3a), we have that $\sgn(\rgamma_j)=\sgn(\rdelta_j)=-1$
	and $\sgn(\repsilon_i)=\sgn(\rgamma_i)=1$. 
	
	In Case (R3a'), we have that 
	$\sgn(\rgamma_j)=\sgn(\repsilon_i)=\sgn(\rgamma_i)=\sgn(\rdelta_j)=1$.

	In Case (R3b), we have that 
	$\sgn(\rdelta_j)=\sgn(\rdelta_i)=-1$ and 
	$\sgn(\repsilon_j)=1$.

	In Case (R3b'), we have that 
	$\sgn(\rdelta_j)=(-1)^{\ell}$, $\sgn(\ralpha_i)=(-1)^{\ell-1}$,
	$\sgn(\repsilon_j)=1$ and 
	$\sgn(\rdelta_i)=-1$.
To verify the signed seed 
property, it suffices to show that when we apply promotion
to an exchange relation where we know that both monomials on the right-hand side
have the same sign, we obtain a relation in which both monomials on the right-hand
side still have the same sign. But this follows from 
\cref{lem:signs}, and the fact that 
cluster exchange relations for the Grassmannian are
pure.
\end{proof}

\begin{proof}[Proof of \Cref{thm:sign-definite}]
The first statement just follows from \cref{cor:cluster-sign-description}. 

We turn to the second statement.	
 \cref{prop:domino-var-signs-on-tile} shows that each domino variable $x \in \txx(D)$ has a fixed sign $s_x$ on $\gto{D}$. All other variables in $\txx(D)$ are obtained by repeatedly applying rescaled product promotion to a frozen variable of some $\Gr_{4, N}$. 
	By \cref{lem:sign_of_bdry_twistors,rmk:lots-strong-sign}, each such variable $x \in \txx(D)$ has a fixed sign $s_x$ on $\gto{D}$. The signed seed property (\cref{prop:signed-seed}) shows that all cluster variables $s$ in $\A(\Sigma_D)$ have a fixed sign $s_x$ on $\gto{D}$. This completes the proof of the second statement of the theorem.

We turn to the third statement. The fact that the two subsets of $\Gr_{k,k+4}$ 
on the right-hand side of \cref{thm:sign-definite} coincide
is a consequence of the signed seed property 
	(\cref{prop:signed-seed}). The above paragraph asserts that 
	$\gto{D} \subseteq \{Y \in \Gr_{k,k+4}: s_x \cdot x(Y)>0 \text{ for all }x \in \txx(D)\}.$ 
	So all that remains is to show the reverse inclusion. This argument is identical to the proof of \cref{cor:cluster-sign-description}. 
\end{proof}

\subsection{Algorithm for determining 
the quiver of a BCFW tile}

In this section we describe an algorithm for 
constructing the quiver of a standard BCFW tile.
As we explain in \cref{rem:extend}, 
it is possible to generalize this algorithm
to the case of general BCFW tiles, but for simplicity
we focus on standard tiles here.
This algorithm relies on a few computational ingredients.

Recall the definitions of the mutation
sequences $\pmb{\cyc_{k,n}}$ and $\pmb{\cyc_{k,n}^{-1}}$ from 
\cref{lem:cyclic}.  In what follows, 
we apply $\pmb{\cyc_{4,n-1}}$ to $\Sigma_{4,n}^{-2}$ by 
regarding the rightmost \emph{two} columns of variables as frozen.
\begin{lemma}[Adding a new marker $d=n-1$]\label{cor:n-1} 
Let $c=n-2$ and $d=n-1$.
Suppose that $\pmb{\mu}$ is a mutation sequence from the rectangles seed
for $\Gr_{4,\{1,\dots,c,\widehat{d},n\}}$ 
(see \cref{fig:G37-Le-quiver})
	to the seed $\tSigma_{D'}$ (from \cref{thm:cluster})
associated to $D'$,
where $D'$ is a chord diagram
on markers $\{1,\dots,c,\widehat{d},n\}$. 
Then 
$\pmb{\mu} \circ 
	\pmb{\cyc}_{4,n-1} \circ 
	\pmb{\cyc}_{4,n}^{-2}(\Sigma_{4,n})=
 \tSigma_{D}$, i.e. 
$\pmb{\mu} \circ \pmb{\cyc}_{4,n-1} \circ \pmb{\cyc}_{4,n}^{-2}$
is a mutation sequence 
from the rectangles seed $\Sigma_{4,n}$
to $\tSigma_{D}$, where $D$
is the chord diagram obtained from $D'$ by adding a new marker $d$. 
\end{lemma}
\begin{proof}
	By \cref{lem:cyclic},
	$\pmb{\cyc}_{4,n}^{-2}(\Sigma_{4,n}) = \Sigma_{4,n}^{-2}$,
	and hence only the rightmost column of frozens contain the 
	index $d=n-1$.  If we ignore those four frozen variables,
	the remaining quiver is a cyclic shift of the rectangles seed
	for $\Gr_{4,\{1,2,\dots,\hat{d},n\}}$.  If 
	we apply
	$\pmb{\cyc}_{4,n-1}$ to that remaining quiver we get exactly
	the rectangles seed for $\Gr_{4,\{1,2,\dots,\hat{d},n\}}$.  
	Therefore we can now apply $\pmb{\mu}$ to 
	obtain 
	$\tSigma_{D'}$ together with the four frozen variables containing $d=n-1$
	in the rightmost column.  
	This is the seed $\tSigma_{D}$.
\end{proof}

\begin{lemma}[A mutation sequence from the rectangles seed $\Sigma_{4,n}$ to $\Sigma_1=\Sigma_1^{a,c}$]
\label{lem:Fr1}
Let $a,b,c,d,n$ be as in 
 \cref{not:LR_cluster}, and let
 $n_R = n-a$.
The following is a 
mutation sequence from the rectangles seed
$\Sigma_{4,n}$ to the seed $\Sigma_1 = \Sigma^{a,c}_1$ 
shown in \cref{fig:seed_sigma_one}: 

First mutate down each column, going from right to left, skipping the vertices labeled by rectangles 
\[                 3 \times (n_R-2),1 \times (n_R-3), 2 \times (n_R-3), 3 \times (n_R-3), 1 \times (n_R-4).\]
Then again mutate down each column, going from right to left, skipping the vertices labeled by rectangles 
\[2 \times (n_R-2),3 \times (n_R-2),1 \times (n_R-3), 2 \times (n_R-3), 3 \times (n_R-3)\]
and all rectangles which   are contained in a $1 \times (n_R-4)$ or a $3 \times 1$ rectangle. 
 
After this sequence, the variables of $\Sigma_1^{a,c}$ label the vertices of the quiver as shown in 
    \cref{fig:box-to-fr-sigma-one}.
\end{lemma}

\begin{figure}  
		\includegraphics[width=\textwidth]{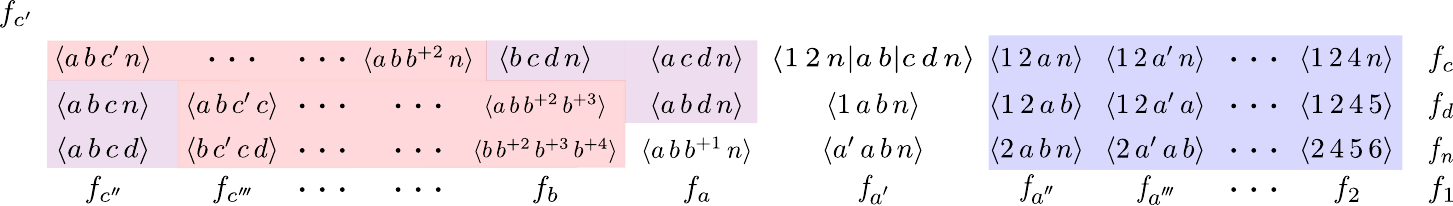}

   \caption{ The elements of the extended cluster of $\Sigma_1$ of \cref{fig:seed_sigma_one}
 arranged according to the vertices of the rectangle quiver of $\Sigma_{4,n}$, from which they are obtained by applying the mutation sequence of \cref{lem:Fr1}. \label{fig:box-to-fr-sigma-one} }
	\end{figure}

We note that while 
\cref{lem:Fr1} 
was tested extensively via computer checks, it 
is experimental.  The reader who does not want to rely on experimental 
statements can use 
the (less-efficient) procedure for 
going from $\Sigma_{4,n}$ to $\Sigma_1$ explained in \cref{rem:notexperimental}.

\begin{remark}\label{rem:notexperimental}
By \cite{OhSpeyer}, given any 
two Pl\"ucker seeds $\Sigma$ and $\Sigma'$
for $\Gr_{4,n}$, one can find a sequence
of mutations at $4$-valent vertices that does not
involve mutating at any Pl\"ucker coordinate that 
$\Sigma$ and $\Sigma'$ have in common.  Thus 
a (less-efficient) alternative to 
 to \cref{lem:Fr1}
is to start at the rectangles seed and randomly 
mutate at $4$-valent vertices labeled by 
Pl\"ucker coordinates which are not in $\Sigma_1$,
until one arrives at the seed $\Sigma_1$.
\end{remark}

\begin{proposition}[A mutation sequence from $\Fr(\Sigma_1)$ to $\tilde{\Sigma}_D$]
\label{prop:combinesequence} Let $D\in \CD$ be a chord diagram and $D_L,D_R$ its left and right subdiagrams as in \cref{def:leftright}.
If $\pmb{\mu}^{(L)}$ 
and $\pmb{\mu}^{(R)}$ are mutation sequences from
$\Sigma_0^L$ and
	$\Sigma_0^R$  
	(shown in \cref{fig:seed_sigma_zero_promoted})
	to $\tilde{\Sigma}_{D_L}$ and 
	 $\tilde{\Sigma}_{D_R}$, respectively,
then 
$\pmb{\mu}^{(R)} 
\pmb{\mu}^{(L)}$ 
is a mutation sequence from 
$\Fr(\Sigma_1)$ to $\tilde{\Sigma}_D$.
\end{proposition}
\begin{proof}
	By \cref{thm:promotion2} and \cref{prop:similar},
	the seed 
$\Sigma_0^L \sqcup \Sigma_0^R$
	 is similar to 
	$\Fr(\Sigma_1)$, and hence
	the seed $\pmb{\mu}^{(R)} 
	\pmb{\mu}^{(L)}(\Sigma_0^L \sqcup \Sigma_0^R) = 
	\tilde{\Sigma}_{D_L} \sqcup \tilde{\Sigma}_{D_R}$
	 is similar to 
	$\pmb{\mu}^{(R)} 
	\pmb{\mu}^{(L)}(\Fr(\Sigma_1))$, that is,
	$\Psi(\tilde{\Sigma}_{D_L} \sqcup \tilde{\Sigma}_{D_R})
	\propto 
	\pmb{\mu}^{(R)} 
	\pmb{\mu}^{(L)}(\Fr(\Sigma_1))$.
	Therefore when we apply the rescaled product promotion map, we obtain
	$\overline{\Psi}(\tilde{\Sigma}_{D_L} \sqcup \tilde{\Sigma}_{D_R})
	= 
	\pmb{\mu}^{(R)} 
	\pmb{\mu}^{(L)}(\Fr(\Sigma_1))$.
It follows that 
$\pmb{\mu}^{(R)} 
\pmb{\mu}^{(L)}(\Fr(\Sigma_1))$ contains the 
 cluster variables
$\OPsi(\txx(D_L) \cup \txx(D_R))$.
The seed
$\pmb{\mu}^{(R)} 
\pmb{\mu}^{(L)}(\Fr(\Sigma_1))$ also contains the 
variables
$\{\rbeta_k, \rgamma_k, \rdelta_k,\repsilon_k\} \cup 
\{f_c, f_d, f_{a-1}\}$ 
	because these are frozen
	in $\Fr(\Sigma_1)$.  Therefore
$\pmb{\mu}^{(R)} 
\pmb{\mu}^{(L)}(\Fr(\Sigma_1))$ contains the entire
	cluster $\txx(D)$, defined in 
	\cref{thm:cluster}, and hence coincides
	with $\tSigma_D$.
\end{proof}

By combining the previous results, we obtain a recursive construction of the 
mutation sequence from 
the rectangles seed $\Sigma_{4,n}$ to $\tilde{\Sigma}_D$.

\begin{theorem}[Recursive algorithm for $\tSigma_D$]\label{thm:algo}
Let $\tSigma_D$ be the seed associated to a chord diagram $D\in \CD$.
As before, let $c=n-2$ and $d=n-1$.
The following is an inductive algorithm for constructing the mutation sequence from
	the rectangles seed $\Sigma_{4,n}$ to $\tilde{\Sigma}_D$.
\begin{itemize}
\item Suppose that $D$ has no chord ending 
 at $(c,d)$.  Let $D'$ be the 
chord diagram on $\{1,2,\dots,c,\widehat{d},n\}$
obtained from $D$ by removing the penultimate marker $d$.
By induction we know the mutation sequence from the 
rectangles seed for $\Gr_{4,\{1,\dots,c, \widehat{d},n\}}$ to 
$\tSigma_{D'}$,
so by \cref{cor:n-1}, we can  obtain the mutation sequence 
from the rectangles seed $\Sigma_{4,n}$ to $\tSigma_D$.
\item  
Suppose that the rightmost top chord 
ends at $(c,d)$, that is,
		$D_{\rtop}=(a,b,c,d)$.
		Let $N_L = \{1,2,\dots,a,b,n\}$
and $N_R = \{b,\dots,c,d,n\}$, with $D_L$ and $D_R$ the 
left and right subdiagrams of $D$.
Then \cref{prop:combinesequence} gives a mutation sequence from 
	$\Fr(\Sigma_1)$ to 
	$\tilde{\Sigma}_D$, and 
	\cref{lem:Fr1} gives a mutation sequence from 
	the rectangles seed $\Sigma_{4,n}$ to $\Sigma_1$, so by concatenating them,
	we obtain a mutation sequence from 
	$\Sigma_{4,n}$ to 
	$\tilde{\Sigma}_D$.
\end{itemize}
\end{theorem}

\begin{remark}\label{rem:extend}
It is possible to generalize \cref{thm:algo}
to the case of general BCFW tiles.
Recall from \cref{def:recipe}
that general BCFW cells are built by a sequence
of operations consisting of the BCFW product,
inserting a marker, performing a cyclic shift,
and performing a reflection.  Since 
we have mutation sequences corresponding to the 
	cyclic shift (cf. \cref{lem:cyclic}) 
	and reflection (cf. \cref{lem:reflect}),  
	it is possible to extend \cref{thm:algo}
	to produce a mutation sequence from 
	the rectangles seed to the seed associated
	to any general BCFW tile.
\end{remark}

\section{Proofs about the BCFW product}\label{sec:BCFWcells}

Recall the BCFW map and BCFW product from \Cref{def:bcfw-map,def:butterfly}. In this section, we show the effect of the BCFW product at the level of matroids under appropriate hypotheses (cf. \cref{not:coindipendence}). Under the same hypotheses, we show that the BCFW map is injective, and describe the closure and boundary of $S_L \bcfw S_R$. We also verify every BCFW cell satisfies these hypotheses.

\begin{lemma}[Positroids and BCFW product]\label{lem:matroid_under_bcfw} 
	Let $S_L \subset {\Gr}^{\ge0}_{k_L, N_L}$ and $S_R \subset {\Gr}^{\ge0}_{k_R, N_R}$ be positroid cells as in \cref{not:coindipendence}, with associated
	positroids
	$\pos_L$ and $\pos_R$, and plabic graphs $G_L$ and $G_R$. 
Let $\pos_L'$ and $\pos_R'$ be the positroids corresponding to $G'_L$ and $G'_R$
	(shown at the right of 
	 \cref{fig:bcfw-bases-graph}); note that $G'_L$ and $G'_R$ can be regarded as
	 subgraphs of $G_L \bcfw G_R$.
	
Then the bases of the positroid $\pos_L \bcfw \pos_R$ of $S_L \bcfw S_R$ are exactly the sets $I_L \sqcup \{f\} \sqcup I_R$ where $I_L, f, I_R$ are disjoint and satisfy one of the following:
	\begin{center}
	\begin{tabular}
 {|c | c|c|c|}
		\hline
		&$I_L$ & $f$ & $I_R$\\
		\hline
		(1)&$I_L \in \pos_L, ~ b \notin I_L$& $a$ &$I_R \in \pos_R'$\\
		\hline
		(2)&$I_L \in \pos_L$& $b$ &$I_R \in \pos_R'$\\
				\hline
		(3)&$I_L \in \pos_L'$& $c$ &$I_R \in \pos_R,~d \notin I_R$\\
		\hline
		(4)&$I_L \in \pos_L'$& $d$ &$I_R \in \pos_R$\\
		\hline
		(5)&$I_L \in \pos_L'$& $n$ &$I_R \in \pos_R$\\
		\hline
		(6)&$I_L \in \pos_L'$& $n$ &$(I_R \setminus \{c\}) \cup \{d\} \in \pos_R$\\
		\hline
	\end{tabular}
\end{center}

\end{lemma}

\begin{proof}
 
	We prove the lemma by analyzing the possible source sets of perfect orientations of $G_L \bcfw G_R$, which by \cref{prop:positroid} and \cref{thm:positroidcell},
	give rise to the bases of 
	$\pos_L \bcfw \pos_R$.
	
	\begin{figure}
		\includegraphics[width=\textwidth]{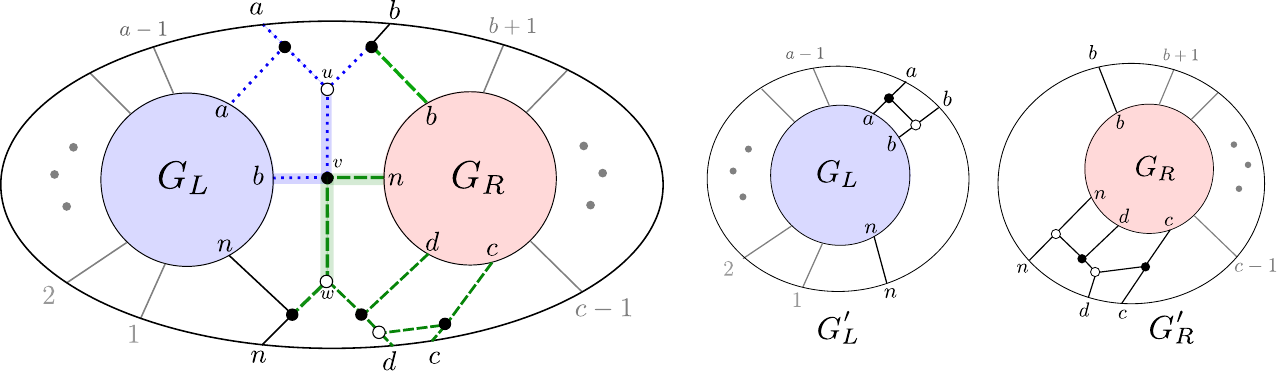}
		
		\caption{\label{fig:bcfw-bases-graph}
On the left, $G_L \bcfw G_R$ with edge colorings used in the proof of \cref{lem:matroid_under_bcfw}. Center and right, the graphs $G_L', G_R'$ mentioned in the statement of the same result.}
	\end{figure}
	In any perfect orientation $\O$ of $G_L \bcfw G_R$, 
	either the dashed green highlighted edges or the dotted blue 
	 highlighted edges must form a directed path and the other highlighted edges of \cref{fig:bcfw-bases-graph} must be oriented towards the vertex $v$. If the blue highlighted edges form a directed path, then $\O$ restricted to the blue edges is a perfect orientation of $G_L'$; similarly with the green edges and $G_R'$.

Consider a perfect orientation $\O$ of $G_L \bcfw G_R$.

\noindent \textbf{Case 1}: Suppose the green highlighted edges are a directed path in $\O$. Then, depending on the orientation of the edges incident to $u$, we are in one of the cases in \cref{fig:bcfw-bases-case1} where edges $e,e'$ form a directed path. 
	\begin{figure}
	\includegraphics[width=\textwidth]{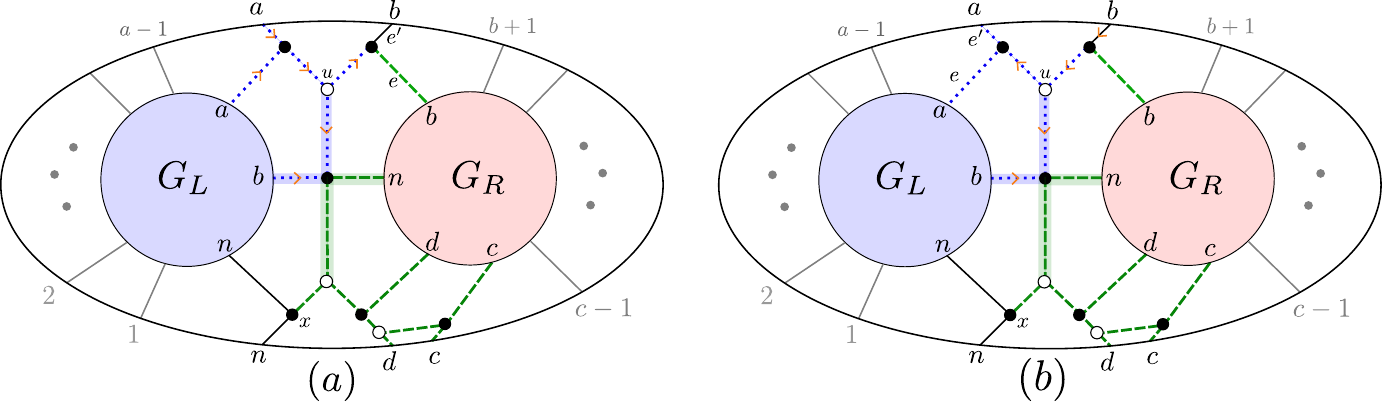}
	\caption{\label{fig:bcfw-bases-case1}}
\end{figure}
In (a), the restriction $\O_L$ of $\O$ to $G_L$ is a perfect orientation where $a,b$ are sinks. The restriction $\O_R'$ of $\O$ to $G_R'$ is a perfect orientation whose source set must be disjoint from that of $\O_L$ (because only one edge may be oriented away from vertex $x$, $n$ may not be a source in both $\O_L$ and $\O_R$). Any such perfect orientations of $G_L, G_R'$ may arise as $\O_L, \O_R'$. The sources of $\O$ are precisely $a$, the sources of $\O_L$ and the sources of $\O_R'$. So this case shows that all sets in (1) are bases of $\pos_L \bcfw \pos_R$. Using similiar logic, case (b) shows that all sets in (2) are bases of $\pos_L \bcfw \pos_R$.

\noindent \textbf{Case 2}: Suppose the blue highlighted edges are a directed path in $\O$. Then, depending on the orientation of the edges incident to $w$, we are in one of the cases in \cref{fig:bcfw-bases-case2}.
\begin{figure}
	\includegraphics[width=\textwidth]{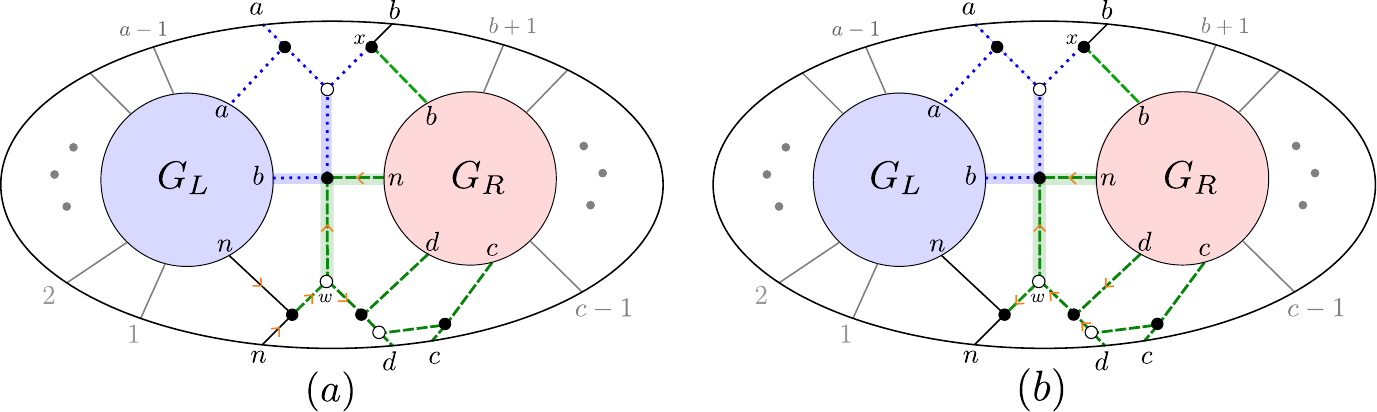}
	\caption{\label{fig:bcfw-bases-case2}}
\end{figure}

In (a), the restriction $\O_L'$ of $\O$ to $G_L'$ has $n$ as a sink. The restriction $\O_R$ of $\O$ to $G_R$ also has $n$ as a sink, and has source set disjoint from that of $\O_L'$, again because only one edge may be oriented away from $x$. These are the only conditions on $\O_L'$ and $\O_R$. Fixing $\O_L'$ and $\O_R$ with sources $I_L, I_R$, one can find $\O$ with sources $I_L \cup \{n\} \cup I_R$. If $d \in I_R$ and $c$ is not, there is another choice of $\O$, with sources $I_L \cup \{n\} \cup (I_R \setminus \{d\} \cup \{c\} )$. This covers cases (5) and (6), and there are no further possibilities for $\O$. In (b), $\O_R$ has sinks $d,n$ and $\O_L', \O_R$ have disjoint source sets. The sources of $\O$ are the sources of $\O_L', \O_R$ and either $c$ or $d$. This covers cases (3) and (4).

\end{proof}

\begin{proposition}\label{prop:bcfw-map-injective} 
Let $S_L \subset {\Gr}^{\ge0}_{k_L, N_L}$ and $S_R \subset {\Gr}^{\ge0}_{k_R, N_R}$ be positroid cells as in \cref{not:coindipendence}.
	Then $\mbcfw$ is injective on $S_L \times \chR \times S_R$.
\end{proposition}
\begin{proof} 
	
	By \cref{prop:butterfly-matrix-same} (1), the image $\mbcfw(S_L \times \Gr_{1,5}^{>0} \times S_R)$ is equal to the cell $S_L \bcfw S_R$.  
	Pick a point $V \in S_L \bcfw S_R$ and choose any preimage $(A, [\alpha: \beta:\gamma:\delta:\epsilon], B)$. Let $M$ be the representative matrix of $V$ resulting from this choice; that is $M$ is the matrix $(*)$ in \cref{fig:promotion-matrix}.
	
	We first express $[\alpha: \beta:\gamma:\delta:\epsilon]$ in terms of Pl\"ucker coordinates of $V$. Let $I_L$ be a basis for $\pos_L$ avoiding $\{a, b, n\}$ and let $I_R$ be a basis for $\pos_R$ avoiding $\{b, c, d, n\}$. Because adding bridges does not remove bases (cf. \cref{thm:bridges-and-lollipops}), \cref{lem:matroid_under_bcfw} implies that $I_f:=I_L \sqcup\{f\}\sqcup I_R$ is a basis of $\pos_L \bcfw \pos_R$ for $f \in \{a, b, c, d, n\}$. Now, the submatrix of $M$ on columns $I_f$ has a single non-zero entry in the $(k_L+1)$th row (which is one of $\alpha, \beta, \pm\gamma, \pm\delta, \pm\epsilon$),
	which is in column $f$. Deleting row $a$ and column $f$ from the submatrix gives a block-diagonal matrix with determinant $\pm Q \neq 0$. So we have
	\begin{equation*}
		\lr{I_a}_M= \alpha \cdot Q, \qquad \lr{I_b}_M= \beta \cdot Q, \qquad \lr{I_c}_M= \gamma \cdot Q, \qquad \lr{I_d}_M= \delta \cdot Q, \qquad \lr{I_n}_M= \epsilon \cdot Q,
	\end{equation*}
	or in other words that
	\begin{equation}\label{eq:param-from-image}
		[\alpha: \beta: \gamma:\delta:\epsilon]=[\lr{I_a}_V: \lr{I_b}_V: \lr{I_c}_V:\lr{I_d}_V:\lr{I_n}_V].
	\end{equation}
	
	Now, we recover $A,B$. Let 
	\begin{equation*}
		V_L' := V \cap \Span(e_1, \dots, e_a, e_b, e_n) \text{ and }V_R':= V \cap \Span(e_b, \dots, e_c, e_d, e_n).
	\end{equation*}
	We claim that $V_L'=A \cdot y_a(\frac{\alpha}{\beta})$ and $V_R'= B \cdot y_d(\frac{\delta}{\epsilon}) \cdot y_c(\frac{\gamma}{\delta})$ (that is, $V_L'$ and $V_R'$ are elements of $S_{\pos_L'}$ and $S_{\pos_R'}$, using the notation of \cref{lem:matroid_under_bcfw}). The containment $\supset$ is clear in both cases. For the containment $\subset$, it suffices to show $\dim V_L' = k_L$  and $\dim V_R '= k_R$. 
	
	Any vector in $V_L' \cap V_R'$ would be in $V \cap \Span(e_b, e_n)$. However, if there existed $v \in V \cap \Span(e_b, e_n)$, then all nonzero Pl\"ucker coordinates of $V$ would contain either $b$ or $n$. The basis $I_a$ constructed above contains neither $b$ nor $n$, so such a $v$ cannot exist. Thus $V_L' \cap V_R' $ is trivial. Let $w$ denote the span of the $(k_L+1)$st row of $M$. Note that $w$ has trivial intersection with $V_L'$ and with $V_R'$ since $\alpha, \beta, \gamma, \delta, \epsilon$ are all nonzero. So we have
	\[ V=w + V_L' + V_R'= w \oplus V_L' \oplus V_R'\]
	and $\dim V_L '+V_R'= k-1=k_L + k_R$. Since $V_L'$ and $V_R'$ have dimensions at least $k_L$ and $k_R$, respectively, we conclude that $\dim V_L' = k_L$  and $\dim V_R' = k_R$ as desired. 
	
	Since $V_L'=A \cdot y_a(\frac{\alpha}{\beta})$ and $V_R'= B \cdot y_d(\frac{\delta}{\epsilon}) \cdot y_c(\frac{\gamma}{\delta})$, we have 
	\begin{equation}\label{eq:A,B-from-image}
		A=V_L' \cdot y_a\left(-\frac{\alpha}{\beta}\right) \qquad \text{and} \qquad B=  V_R' \cdot y_c\left(-\frac{\gamma}{\delta}\right)  \cdot y_d\left(-\frac{\delta}{\epsilon}\right).
	\end{equation}
	
	Combining \eqref{eq:param-from-image} and \eqref{eq:A,B-from-image}, we see that the preimage of $V$ under $\mbcfw$ is uniquely determined by $V$; in other words, $\mbcfw$ is injective.
\end{proof}

\begin{remark}\label{rmk:def-of-w}
	The line $w$ in the proof of \cref{prop:bcfw-map-injective} can also be defined as $w:=V \cap \Span(e_a, e_b, e_c, e_d, e_n)$. Indeed, the containment $\subset$ is clear. If the right hand side had dimension at least 2, this would imply the existence of a vector in $V_L'\cap \Span(e_a, e_b, e_n)$ or $V_R' \cap \Span(e_b,e_c,e_d,e_n)$. Then all nonzero Pl\"ucker coordinates of $V$ contain either $a,b,$ or $n$ in the first case, or $b,c,d$ or $n$ in the second case. But in the first case, the Pl\"ucker coordinate $\lr{I_c}$ defined in the proof---or in the second case $\lr{I_a}$ defined in the proof---does not contain any of these indices.
\end{remark}

Now we turn to the closure and boundary of $S_L \bcfw S_R$, which almost have a product structure. This will be used when analyzing the facets of BCFW tiles. 
We need the following  lemma.

\begin{lemma}\label{lem:chord-coind-or-no}
Let $S_L \subset {\Gr}^{\ge0}_{k_L, N_L}$ and $S_R \subset {\Gr}^{\ge0}_{k_R, N_R}$ be positroid cells as in \cref{not:coindipendence}, then $P:=\mbcfw(S_L \times \chR \times S_R)=S_L \bcfw S_R$ and each element of $\binom{\{a,b,c,d,n\}}{4}$ is coindependent for $P$. 
\end{lemma}

\begin{proof}
	For the first statement, take $J\in\binom{\{a,b,c,d,n\}}{4}$ and write $\{f\}=\{a,b,c,d,n\}\setminus J.$ Choose a basis $I_R$ for $S_R,$ avoiding $b,c,d,n,$ and a basis $I_L$ for $S_L$ avoiding $a,b,n.$ Then by \cref{lem:matroid_under_bcfw} $I=I_L\cup I_R\cup\{f\}$ is a basis for $P.$ 

\end{proof}

For a subset $S \subset \Grk$, we use $\overline{S}$ to denote its closure in the Euclidean topology on $\Grk$.

\begin{lemma}[Boundary of $S_L \bcfw S_R$]\label{lem:boundaries_before_amp_map} 
Let $C_{abcdn} \subset \Grk$ be the union of cells 
for which some element of $\binom{\{a,b,c,d,n\}}{4}$ is not coindependent. 
Let $S:=S_L \bcfw S_R$, with $S_L,S_R$ as in \cref{not:coindipendence}. Then
	\begin{equation}\label{eq:bdries-before-amp}
		\overline{S}= (\overline{S} \cap C_{abcdn}) \sqcup S \sqcup \bigsqcup_{(S_L', S_R') \in \mathcal{C}}  S_L' \bcfw S'_R,
	\end{equation}
where the union is over the collection $\mathcal{C}$ 
of pairs $(S_L',S_R') \neq (S_L, S_R)$ such that $S_L' \subset \overline{S_L} ,S_R' \subset \overline{S_R}$ and $\{a,b,n\}$ and $\{b,c,d,n\}$ are coindependent for $S_L'$ and $S_R'$ respectively.
Thus, \[\partial\overline{S}=(\partial \overline{S} \cap C_{abcdn}) \sqcup \bigsqcup_{(S_L', S_R') \in \mathcal{C}} S_L' \bcfw S'_R.\]
\end{lemma}

\begin{proof}
By \cref{thm:closure},
$\overline{S}$ is the union of 
$S$  
together with any cells we get by deleting one or more of the edges
of $G_L \bcfw G_R$. 

Let $G$ be the plabic graph $G_L \bcfw G_R$.  
We now show that if we delete some edges of $G_L$ (respectively $G_R$), 
obtaining a graph $G'_L$ (respectively, $G'_R$) corresponding to a 
cell $S'_L$  in which $\{a,b,n\}$ (respectively,
$\{b,c,d,n\}$) fails to be coindependent, then the cell $S'$ 
corresponding to the resulting
graph $G'$ will lie in $C_{abcdn}$.
Suppose that $\{a,b,n\}$ fails to be coindependent for 
$G'_L$.   Then any perfect orientation $\O$ of $G'_L$ must have
a source at $a$, $b$, or $n$.  Suppose $\O$ has a source at $a$.
Then the edge of $G$ emanating from the vertex of $G_L$ labeled $a$ 
must be directed towards $a$.  But now any extension of $\O$ to 
a perfect orientation of $G'$ must have a source at the external
vertex labeled $a$.  
Similarly, if the perfect orientation $\O$ of $G'_L$ has a source 
	at $b$, then any extension of $\O$ to a perfect orientation of 
	$G'$ must have a source 
	at either the external vertex $a$ or $b$ of $G'$.
Finally, if $\O$ has a source at the vertex $n$ of $G'_L$,
then any extension of $\O$ must have a source at the external vertex
	$n$ of $G'$. This shows that $\{a,b,n\}$ fails to be coindependent
	for $S'$ so $S'$ lies in $C_{abcdn}.$  The proof
for $G_R$ is analogous.

Finally  we can check by hand that if we delete any edge of the butterfly,
some element of $ {\{a,b,c,d,n\} \choose 4}$ will fail to 
be coindependent, where we use the characterization of 
\cref{rem:co}.  For example, if we delete the $u-v$ edge in 
\cref{fig:bcfw-bases-graph}, then it is impossible
to find a perfect orientation in which $a$ and $b$ are sinks.
If we delete the $v-w$ edge, it is impossible to find a perfect orientation
in which $c,d,n$ are sinks.  The other cases are similar.
\end{proof}

Finally, we verify that BCFW cells always fulfill the assumptions of \cref{prop:butterfly-matrix-same,prop:bcfw-map-injective,lem:boundaries_before_amp_map}.

\begin{definition}
	We say that $P\subset\Grk$ is \emph{4-coindependent} if every 4-element subset of $[n]$ is co-independent for $P$. 
\end{definition}

The BCFW product preserves 4-coindependence, which means that taking the BCFW product of two BCFW cells is the same as applying $\mbcfw$ to them.
\begin{corollary}\label{cor:4bidden}
	If $S_L,S_R$ are 4-coindependent, then $S_L \bcfw S_R$ is as well.
	Therefore all BCFW cells are 4-coindependent.
\end{corollary}
\begin{proof}
	We first note that 4-coindependence is preserved under cyclic rotation, reflection, and adding a zero column, and all positroid cells for $k=0$ are 4-coindependent. So the second statement follows from the first.
	
	Let $\pos_L, \pos_R, \pos_L', \pos_R'$ be as in \cref{lem:matroid_under_bcfw}. 
	Using \cref{prop:positroid}, we 
	note that if $J \in \pos_L$, then $J \in \pos_L'$ as well, and similarly for $\pos_R, \pos_R'$. Thus our assumptions imply $\pos_L'$ and $\pos_R'$ are 4-coindependent.
	
	Let $B\in\binom{[n]}{4}$ be a quadruple of indices. We will show that there exists $I\in\binom{[n]\setminus B}{k}$ which is a basis of $\pos_L \bcfw \pos_R$. Let $B_L = B \cap N_L$ and $B_R = B \cap N_R$. 
	\begin{enumerate}
		\item\label{it:1_for_4bidden} If $B \subset N_R$, then by 4-coindependence for $\pos_R'$, there is a basis $I_R$ of $\pos_R'$ avoiding $B$. By assumption, $\pos_L$ has a basis $I_L$ which avoids $a, b, n$ and thus is disjoint from $I_R$ and from $B$. Set $I= I_L \sqcup \{a\} \sqcup I_R$. By \cref{lem:matroid_under_bcfw} (1), $I$ is a basis of $\pos_L \bcfw \pos_R$, and it avoids $B$ by construction.

		\item\label{it:2_for_4bidden} If $B \subset N_L$, a similar argument as above (now using \cref{lem:matroid_under_bcfw}(4)) gives a basis $I=I_L \sqcup \{d\} \sqcup I_R$ which avoids $B$, where $I_L \in \pos_L'$ avoids $B$ and $I_R \in \pos_R$ avoids $b,c,d,n$.
		
		\item Suppose $B$ is not contained in $N_R$ or $N_L$ and $b \notin B$. This implies that $|B_L|+|B_R| \leq 5$, so one of $B_L, B_R$ has size at most 2. Add $n$ to the smaller of the two; denote the new sets by $B_L'$ and $B_R'$. Let $I_R$ be a basis of $\pos_R'$ avoiding $B_R' \cup \{b\}$ and let $I_L$ be a basis of $\pos_L$ avoiding $B_L' \cup \{b\}$. Then $I = I_R\sqcup \{b\} \sqcup I_L $ is a basis by \cref{lem:matroid_under_bcfw} (2) and avoids $B$.
		
		\item \label{it:3_for_4bidden} Suppose $B$ is not contained in $N_R$ or $N_L$ and $b \in B$. Then for some $f \in \{c,d,n\}$, we have $f\notin B$. If possible, choose $f=d$ or $f=n$; otherwise, let $f=c$ (that is, only set $f=c$ if $d, n \in B$). Then let $I_L$ be a basis of $\pos_L'$ avoiding $B_L \cup \{n\}$ and let $I_R$ be a basis of $\pos_R$ avoiding $B_R \cup \{f\}$. The sets $I_L, I_R$ are disjoint, since both avoid $b$ and $I_L$ avoids $n$. They are both disjoint from $\{f\}$ by construction, and if $f=c$, then $d \notin I_R$. So by \cref{lem:matroid_under_bcfw} (3), (4) or (5), $I = I_R\sqcup \{f\} \sqcup I_L$ is a basis.
	\end{enumerate}
\end{proof}

\section{The proof that BCFW cells give tiles}\label{sec:imagesBCFWcells}

The main goal of this section is to 
show how to invert the amplituhedron map on the 
image of a BCFW cell, thus proving 
 \cref{thm:BCFW-tile-and-sign-description}, which shows 
that BCFW cells map injectively into the amplituhedron. 
In 
\cref{sec:evolve}, \cref{sec:proof1}, and \cref{sec:proof2}
we will explain how signs of functionaries 
on tiles evolve when we apply the operations of cyclic shift, reflection,
and promotion.
In \cref{subsub:canonical_BCFW_form} we will 
 explain a useful way of writing the rows of the 
BCFW matrix.  
Finally in 
\cref{subsec:inverse_problem}
we will prove 
 \cref{thm:BCFW-tile-and-sign-description}.

\subsection{How signs of functionaries evolve}\label{sec:evolve}
We need the following definitions, which strengthen the notion of the \emph{sign} of a functionary on a tile (c.f. \cref{def:signs}).

\begin{definition}\label{def:strongly_positive}
A rational functionary $F$ has \emph{strong sign} $+1$ or is \emph{strongly positive} on the image of $S$ if for all $C\in S$, $F(CZ)$ can be written as a ratio of polynomials in the Pl\"ucker coordinates of $C$ and maximal minors of $Z$ with all coefficients of sign $+1$.
A rational functionary has \emph{strong sign} $-1$ or is \emph{strongly negative} on the image of $S$ if it is the negative of a functionary with strong sign $+1$ on the image of $S$.
\end{definition}

Note that if $F$ has strong sign $s$ on the image of $S$, it also has sign $s$ on the image of $S$. The converse may not be true. 

Say $S$ is a cell and $S'$ is obtained from $S$ by applying an operation from \cref{def:BCFW_cell}. The following two theorems 
(which will be proved in \cref{sec:proof1} and \cref{sec:proof2})
show that if a functionary $F$ has fixed sign on the image of $S$, then there is a related functionary $F'$ with a fixed sign on the image of $S'$.

\begin{theorem}[Signs under dihedral group and $\mbox{pre}$]\label{thm:signs-under-cyc-rot-pre}
	Let $S \subset \Grk$ be a positroid cell. Suppose that $F$ is a pure rational functionary with (strong) sign $s$ on $\gto{S}$. Then
	\[
	\begin{cases}
		F\\
		\cyc^{-*}F\\
		\refl^* F
	\end{cases} \text{ has (strong) sign }
	\begin{cases}
	s\\
	(-1)^{k \deg_n F}s\\
	s
\end{cases} \text{ on the image of }
	\begin{cases}
	\pre_I S\\
	{\cyc S}\\
	{\refl S}
\end{cases}.
\]
\end{theorem}

If $B \subset \{-1,0, 1\}^r$, we define
$\cl(B) \subset \{-1,0,1\}^r$ to be minimal set of tuples which contains $B$ and is closed under the operation of changing an entry from $\pm 1$ to $0$. In the following, with a slight abuse of notation, we will often denote the promotion $\Psi(F)$ of a functionary $F$ as $\Psi F$.

\begin{theorem}[Signs under promotion]\label{prop:vanishing_and_sign_of_functionaries_under_promotion}
	Let $S_L \subset \Gr_{k_L, N_L}^{\ge 0}, S_R \subset \Gr_{k_R, N_R}^{\ge 0}$ be positroid cells as in \cref{not:coindipendence}.
	Let $F$ be a pure rational functionary with indices contained in $N_L$ (resp. $N_R$). 
	\begin{enumerate}
		\item If $F$ vanishes on the image of $S_L$ (resp. $S_R$), then $\Psi_{a c} F$ vanishes on the image of $S_L \bcfw S_R$.
\item\label{it:strong_sign} If 
$F$ has strong sign $s \in \{\pm 1\}$ on the image of $S_L$ (resp. $S_R$), then $\Psi_{ac} F$ has strong sign $(-1)^{(k_R+1) \deg_n F}s$ (resp. $s$) on the image of $S_L \bcfw S_R$. 
\item\label{it:for_separation_theorem}
Let $(F_1,\ldots,F_r)$ be an $r$-tuple of pure rational functionaries with indices in $N_L$ (resp. $N_R$). Denote by $A\subseteq\{-1,0,1\}^r$ the collection of sign vectors
$\sgn (F_1(Y,Z), \dots, F_r(Y,Z))$,
where $Z$ varies over $\Mat_{N_L, k_L +4}^{>0}$ and $Y$ varies over $\gto{S_L}$ (resp. $Z \in \Mat_{N_R, k_R +4}^{>0}, Y \in \gto{S_R}$). Let 
$$A' =\{((-1)^{(k_R+1) \deg_n F_1}s_1,\ldots,(-1)^{(k_R+1) \deg_n F_r}s_r):(s_1,\ldots,s_r)\in A\} \quad (\mbox{resp. } A'=A).$$

Then for all $Z$ and all $Y \in \gto{S_L \bcfw S_R}$, $\sgn (\Psi_{ac}F_1,\ldots,\Psi_{ac}F_r) \in \cl(A').$ 
\end{enumerate}
\end{theorem}

\begin{lemma}\label{lem:sign_of_chord_twistors}
Let $S_L \bcfw S_R \subset \Grk$, with $S_L,S_R$ as in \cref{not:coindipendence}.  Then the twistor coordinates 
	\[(-1)^{k_R}\llrr{bcdn}, \quad (-1)^{k_R+1}\llrr{acdn}, \quad \llrr{abdn},\quad -\llrr{abcn}, \quad \llrr{abcd}\]
	have strong sign $+1$ on the image of $S_L \bcfw S_R$.
\end{lemma}

\begin{proof}
	Expand each twistor $\llrr{ijlh}$ as in \eqref{eq:cauchy-binet}. We obtain a sum of products of Pl\"ucker coordinate of the matrix $C \in S_L \bcfw S_R$, and the matrix $Z$. The sum is over all Pl\"ucker coordinates of $C$ which are non zero and intersect the set $\{a,b,c,d,n\}$ precisely in the singleton $\{a,b,c,d,n\}\setminus\{i,j,l,h\}.$ It is straightforward to see that all these terms come with the sign in the lemma using \cref{lem:matroid_under_bcfw}. \cref{lem:chord-coind-or-no} shows that the summation is non-empty.
\end{proof}
\begin{lemma}
\label{lem:sign_of_bdry_twistors}
Let $S$ be a BCFW cell of $\Grk$. For $i, j \in [n-1]$ with $i +1<j$, the twistor $\llrr{i,i+1,j,j+1}$ is strongly positive on $\gto{S}$. For $i \in \{2, \dots, n\}$, the twistor $\llrr{i, i+1, n, 1}$ has strong sign $(-1)^{k-1}$ on $\gto{S}$.
\end{lemma}
\begin{proof}
The proof is almost identical to the proof of the previous lemma. Again we expand the twistor coordinates using \eqref{eq:cauchy-binet}, and observe that all possible nonzero terms terms come with the sign indicated in the lemma statement. As long as there is at least one Pl\"ucker coordinate of $C$ which does not use the indices $i,i+1,j,j+1$ and is nonzero, the summation is non zero. But such a coordinate exists, as $S$ is $4$-coindependent, by \cref{cor:4bidden}.
\end{proof}

\subsection{The proof of \cref{thm:signs-under-cyc-rot-pre}}\label{sec:proof1}
In this section and the next, if $F$ is a functionary, we will sometimes write $F(Y, Z)$ where $Y \in \Gr_{k, k+4}$ and $Z \in \Mat_{n, k+4}^{>0}$. We also sometimes write $\llrr{Y Z_I}$ for the twistor coordinate $\llrr{I}$. Recall that in $\llrr{I}= \llrr{Y Z_I}$, the elements of $I$ are listed in \emph{increasing} order.

The following lemma relates twistor coordinates before and after applying $\pre_i, \cyc, \refl$.

\begin{lemma}\label{lem:twistor-after-pre-cyc-rot}
Let $N$ denote the index set of $S$, and let $S'$ denote one of $\pre_i S, \cyc S, \refl S$. Choose $C' \in S'$, and $Z'$ a positive matrix of the appropriate size. Then there exists $C \in S$ and $Z$ positive so that for all $I \in \binom{N}{4}$,
	\[\llrr{CZ~Z_I}= \begin{cases}
		\llrr{C'Z'~Z'_I} & \text{ if }S'=\pre_i S\\
		(-1)^{k \cdot |I \cap \{n\}|} \cyc^{-*}\llrr{C'Z'~Z'_I} & \text{ if } S'=\cyc S\\
		\refl^* \llrr{C'Z'~Z'_I} & \text{ if }S'=\refl S.\\
	\end{cases}\]
\end{lemma}
\begin{proof}
	If $S'= \pre_i S$, then $Z' \in \Mat_{N \cup \{i\}, k+4}^{>0}$. The matrices $C$ and $Z$ are respectively defined by deleting the $i$th column of $C'$ and the $i$th row of $Z'$. It is clear that $CZ=C'Z'$ and for all $x\in N$, $Z_x=Z_x'$. The equality of twistor coordinates follows.
	
	For the remaining cases, we assume $N=[n]$ to simplify notation. If $S'= \cyc S$, then we define $C \in S$ to be the unique element such that $\cyc C=C'$. Recall from the definition of cyclic shift (cf. \cref{def:dihedral}) that $\cyc C= C \cdot [\cyc_{k,n}]$ where $ [\cyc_{k,n}]$ is an $n \times n$ matrix. We define $Z:= [\cyc_{k,n}]Z'$, so that $C'Z'= C [\cyc_{k,n}]Z= C Z $. Note that $Z$ is again a matrix with positive maximal minors.
	
	So we have $\llrr{CZ~Z_I}= \llrr{C'Z'~Z_I}$. We now rewrite the right hand twistor entirely in terms of $Z'$. From the definition of $[\cyc_{k,n}]$, we have that $Z_i=Z'_{i+1}$ for $i<n$ and $Z_n= (-1)^{k-1} Z'_1$. If $I=\{u<v<x<y\}$ does not contain $n$, then 
	$$\llrr{C'Z'~Z_I}= \llrr{C'Z'~Z'_{u+1}Z'_{v+1}Z'_{x+1}Z'_{y+1}}=\cyc^{-*} \llrr{C'Z' ~ Z'_I}$$ where the last equality holds because the indices are in increasing order. If $I=\{u<v<x<n\}$, then 
	$$\llrr{C'Z'~Z_I}=(-1)^{k-1} \llrr{C'Z'~Z'_{u+1}Z'_{v+1}Z'_{x+1}Z'_{1}}= (-1)^{k} \llrr{C'Z'~Z'_1Z'_{u+1}Z'_{v+1}Z'_{x+1}}$$
	and the final twistor is $(-1)^k \cyc^{-*} \llrr{C'Z' ~ Z'_I}$.
	
	If $S'=\refl S$, the proof is quite similar. Recall that $\refl$ is an involution. Define $C:= \refl C'$, so that $C'= \refl C= P_{ k,1,\binom{k}{2} } C [\refl_n]$ (cf. \cref{def:dihedral} for the definitions of $P_{k, i, s}$ and $[\refl_n]$). We choose $Z= [\refl_n] Z' P_{k+4, 1, \binom{k+4}{2}}$, which again has positive maximal minors. With these choices, $C'Z'$ is equal to $CZ$ with the first row multiplied by $(-1)^{\binom{k}{2}}$ and the first column multiplied by $(-1)^{\binom{k+4}{2}}$. Note that $Z_{i}$ is $Z'_{n-i+1}$ with first entry multiplied by $(-1)^{\binom{k+4}{2}}$. This implies that for $I=\{u<v<x<y\}$, 
	\[\llrr{CZ Z_u Z_v Z_x Z_y}= \llrr{C'Z' Z'_{n+1-u} Z'_{n+1-v} Z'_{n+1-x} Z'_{n+1-y}}= \llrr{C'Z' Z'_{n+1-y} Z'_{n+1-x} Z'_{n+1-v} Z'_{n+1-u}}\]
	where the first equality uses the fact that the parity of $\binom{k}{2}$ depends only on $k \mod 4$. Since the indices of the rightmost twistor are in increasing order, it is equal to $\refl^* \llrr{C'Z' Z'_I}$, as desired.
\end{proof}

\begin{proof}[Proof of \cref{thm:signs-under-cyc-rot-pre}]
	We first consider the case when $F$ has sign $s \in \{0, +1, -1\}$ on the image of $S$. We will give the argument for the sign of $\cyc^{-*}F$ on the image of $\cyc S$; the others are very similar. Note that by definition, $\cyc^{-*}F$ has an expression in terms of the cyclic shifts of twistor coordinates $\cyc^{-*}\llrr{I}$. Pick $C' \in \cyc S$ and $Z'$ positive, and let $C \in S$ and $Z$ be the matrices from \cref{lem:twistor-after-pre-cyc-rot}. Applying \cref{lem:twistor-after-pre-cyc-rot} separately to each twistor coordinate in $F$, each term gets a sign of $(-1)^{k \deg_n F}$ so we see that $F(CZ, Z)= (-1)^{k \deg_n F} \cyc^* F(C'Z', Z')$. By assumption, $F(CZ, Z)$ has sign $s$, and so $\cyc^* F(C'Z', Z')$ has sign $s$ as well. Since $C'$ and $Z'$ were chosen arbitrarily, this shows $\cyc^* F$ has sign $s$ on the image of $\cyc S$.
	
	We now consider the case when $F$ has strong sign $s \in \{+1, -1\}$ on the image of $S$. That is, as a function from $S \times \Mat_{n, k+4}^{>0}$, $F=f(C,Z)/g(C,Z)$ where $f,g$ are polynomials in the Pl\"ucker coordinates of $C$ and the maximal minors of $Z$, all coefficients of $f$ are sign $s$ and all coefficients of $g$ are positive. We again will only give the argument for the strong sign of $\cyc^{-*}F$ on the image of $\cyc S$; the others are very similar. \cref{lem:twistor-after-pre-cyc-rot} implies that for any $C' \in \cyc S$ and $Z'$ positive, $(-1)^{k \deg_n F} \cyc^{-*}F(C'Z', Z')$ is equal to $F(CZ, Z)$, where $C,Z$ are as constructed in the lemma. Thus we have $$(-1)^{k \deg_n F} \cyc^{-*}F(C'Z', Z')=f(C,Z)/g(C,Z).$$ 
	What remains is to rewrite $f(C,Z)$ and $g(C,Z)$ in terms of the Pl\"ucker coordinates and maximal minors of $C', Z'$ without changing the signs of any coefficients. This is straightforward: since $C'= \cyc C$, $\lr{I}_C = \lr{I+1}_{C'}$, and, similarly, since $Z= \cyc Z'$, $\lr{I}_Z= \lr{I-1}_{Z'}$.
\end{proof}

\subsection{The proof of \cref{prop:vanishing_and_sign_of_functionaries_under_promotion}}\label{sec:proof2}
We will treat the case when the indices of $F$ are contained in $N_R$. 
The proof for the other case is similar; we remark on the  differences at the end of this section.

\begin{notation} \label{not:bullet}
For a positroid $S_R$ in $\Gr^{\ge0}_{k_R,N_R}$, we set 
$S_R^{\bulR}=\Gr^{>0}_{k_L,N_L} \bcfw S_R$, and for a positroid $S_L$ of $\Gr^{\ge0}_{k_L,N_L}$, we set $S_L^{\bulL}=S_L\bcfw \Gr^{>0}_{k_R,N_R}.$
\end{notation} 
 
Using \cref{not:bullet}, \cref{lem:matroid_under_bcfw} implies that $S_L \bcfw S_R \subset \overline{S_R^\bulR}$. Indeed, let $\pos_L,\pos_L'$ be as in \cref{lem:matroid_under_bcfw}. Every element of $\pos_L, \pos_L'$ is a basis for $\Gr^{>0}_{k_L,N_L}$, so by \cref{lem:matroid_under_bcfw}, every basis for $S_L \bcfw S_R$ is a basis of $S_R^\bulR$.
	
	The proof will utilize the following lemmas. 
	
		\begin{lemma}\label{lem:9_6_elt}
		$S_R^\bulR$ can be constructed in the following way.
		Start with $S_R.$
		\begin{itemize}
			\item Apply the upper BCFW map to $S_R.$
			\item Perform $\inc_1,\ldots,\inc_{k_L},$ and then $\pre_{k_L+1},\ldots,\pre_{a-1}.$
			\item Apply $y_{n}(t_1), y_{1}(t_2), \dots y_{k_L-1}(t_{k_L-1})$ in that order.
			\item For each $h=1,\ldots,b-k_L,$ apply 
			the operations $x_{k_L+h-1}(t_{k_L;h}),x_{k_L+h-2}(t_{k_L-1;h}), \dots, x_{h}(t_{1;h})$ in that order.
		\end{itemize}
	\end{lemma}
	The proof is completely analogous to the proof of \cite[Lemma 9.6]{even2021amplituhedron}, which goes by comparing the collections of non zero Pl\"ucker coordinates for the two descriptions. See \cref{fig:S-bullet-plabic} for a plabic graph proof of the lemma.
	
	\begin{figure}
		\centering
		\includegraphics[width=\textwidth]{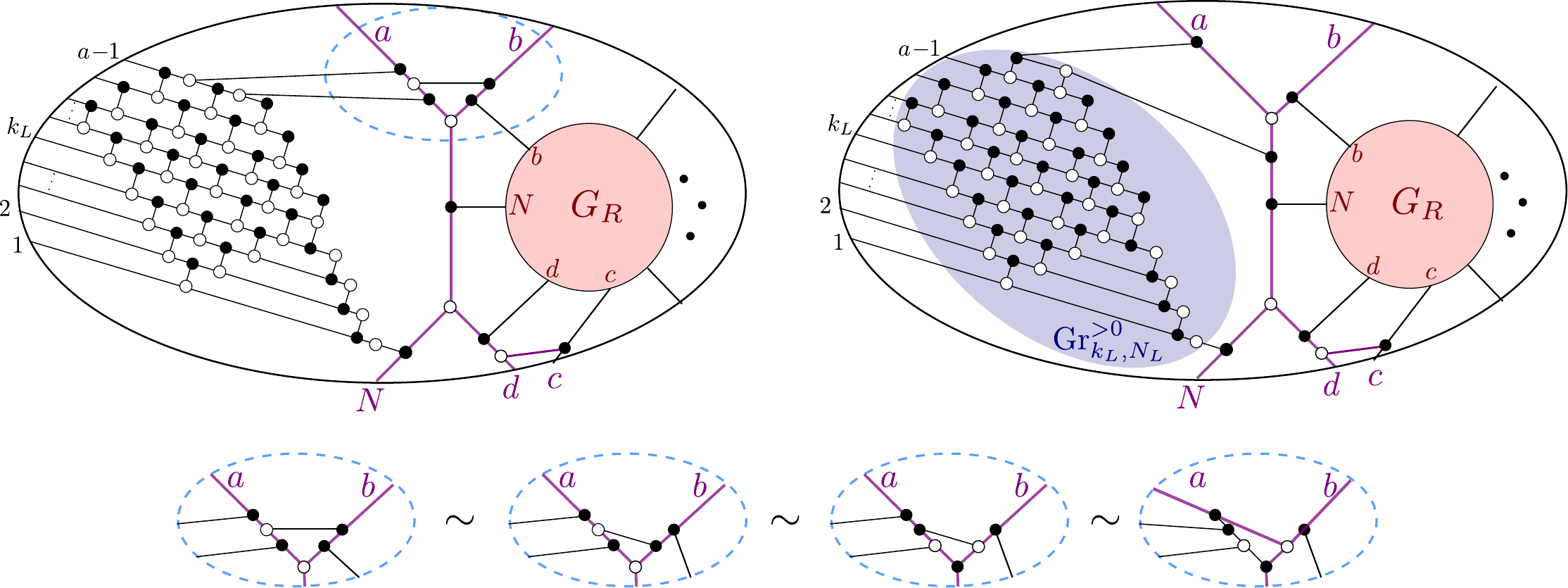}
		\caption{On the upper left, the plabic graph constructed by the sequence of operations in \cref{lem:9_6_elt}. Performing the moves in the bottom row produces the graph on the upper right, which is of the form $G \bcfw G_R$ where $G$ a plabic graph for $\Gr_{k_L,N_L}^{>0}$. The graph on the left is thus a plabic graph for $S_R^\bulR$.}
		\label{fig:S-bullet-plabic}
	\end{figure}
	
	\begin{lemma}\label{lem:evaluation-after-bullet} 
		Let $C \in S_R^\bulR$ and $Z \in \Mat_{n, k+4}^{>0}$. Then there exists $C' \in S_R$ and $Z' \in  \Mat_{N_R, k_R+4}^{>0}$ such that for any functionary $F$ with indices in $N_R$, $\Psi_{a c}F(CZ, Z)= F(C'Z', Z').$
	\end{lemma}
This follows from \cite[Section 4]{even2021amplituhedron}. For the reader's convenience, we outline a proof.
\begin{proof}
	We will show the lemma holds for each intermediate cell obtained in the process in \cref{lem:9_6_elt}. The first step in \cref{lem:9_6_elt} is the upper BCFW map. Let $T:= \mbcfw(\chR \times S_R)$, choose $C \in T$ and choose a positive matrix $Z$. The proof of \cite[Lemma 4.28]{even2021amplituhedron} together with \cite[Lemma 4.22]{even2021amplituhedron} define an element $C' \in S_R$ and positive matrix $Z'$ with the following property: for any functionary $F$ with indices in $N_R$, there is a functionary $G$ such that $F(C'Z', Z')= G(CZ,Z)$. \cite[Lemma 4.35]{even2021amplituhedron} implies that $G = \Psi_{a c}F$; this is spelled out more explicitly in \cite[Example 4.40]{even2021amplituhedron}. Note that by \cref{lem:sign_of_chord_twistors}, the twistor coordinates in \cite[Lemma 4.35]{even2021amplituhedron} are all nonzero on the image of $\mbcfw(\chR \times S_R)$, so one may divide in \cite[Lemma 4.35]{even2021amplituhedron} to obtain precisely the formulas defining promotion.
	
	Now, let $T$ denote a cell obtained partway through the process described in \cref{lem:9_6_elt}, and $T'$ the cell obtained in the previous step. Fix $C \in T$ and $Z$ a positive matrix of the appropriate dimensions. \cite[Lemmas 4.18, 4.21 and 4.24]{even2021amplituhedron} construct $C' \in T'$ and $Z'$ a positive matrix of the appropriate dimensions with the following property: 
	
	\noindent $(**):$ if $I \in \binom{N_R}{4}$ or $I=\{a,b,x,y\}$ where $x,y \in N_R$, then the twistor coordinates $\llrr{C'Z' ~Z'_I}$ and $\llrr{CZ~ Z_I}$ are equal.
	
	Consider any functionary $F$ with indices contained in $N_R$. All twistor coordinates in $\Psi_{ac} F$ are of the sort appearing in $(**)$, so we have that  $\Psi_{a c} F(C'Z', Z')= \Psi_{a c} F(CZ, Z)$. Taking $T=S_R^\bulR$ and iterating, we obtain the lemma statement.
\end{proof}

The final lemma involves functionaries with strong sign on the image of a cell, and tracks how the strong sign evolves under various operations.

		\begin{lemma}\label{lem:strong_pos} We adopt the notation and assumptions of \cref{prop:vanishing_and_sign_of_functionaries_under_promotion} and suppose $F$ has indices contained in $N_L$ (resp. $N_R$). 
		\begin{enumerate}
			\item\label{it:strong_pos_under_pre_inc}
			Suppose $F$ has strong sign $s$ on the image of $S_L$ (resp. $S_R$). 
			Then for $i \notin N_L$ (resp. $i \notin N_R$), $F$ has strong sign $s$ on the image of $\pre_i(S_L)$ (resp. $\pre_i(S_R)$). Additionally, $F$ has strong sign $(-1)^{\deg_n F} s$ (resp. $s$) on the image of $\inc_i (S_L)$ (resp. $\inc_i (S_R)$). 
			\item\label{it:strong_pos_under_x_y} Suppose $F$ has strong sign $s$ on the image of $S_L$ (resp. $S_R$). Choose $i \notin N_L$ (resp. $i \notin N_R$). Then $F$ has strong sign $s$ on the image of $x_i(\R_+).S_L$ and $y_{i-1}(\R_+).S_L$ (resp. $x_i(\R_+).S_R$ and $y_{i-1}(\R_+).S_R$).
			\item \label{it:strong_pos_under_upper_emb} If $F$ has strong sign $s$ on the image of $S_L$, then $\Psi_{a c} F$ has strong sign $(-1)^{\deg_n F} s$ on the image of $\mbcfw(S_L \times \chR \times \Gr_{0, \{b,c,d,n\}}^{>0})$.  If $F$ has strong sign $s$ on the image of $S_R$, then $\Psi_{a c} F$ has strong sign $s$ on the image of $\mbcfw(\chR\times S_R)$. 
		\end{enumerate}
	\end{lemma}
Again, the proof of this lemma relies on various results from \cite[Section 4]{even2021amplituhedron}.
\begin{proof}

	Throughout, we will show the results only for $S:=S_R$, except in a few places where there is a difference between the arguments.  So we assume that on the image of $S_R$, $F= f(C,Z)/g(C,Z)$ where $f,g$ are polynomials in the Pl\"ucker coordinates of $C,Z$, the coefficients of $f$ are all of sign $s$ and the coefficients of $g$ are all positive. We will sometimes write $F(Y,Z)$ to make the dependence of $F$ on $Z$ clear. We also write $\llrr{Y~Z_I}$ instead of $\llrr{uvxy}$.
	
	\noindent \textbf{Proof of (1):}
	The case of $\pre_i S$ is shown in \cref{thm:signs-under-cyc-rot-pre}.
	
	The argument for $\inc_i S$ uses \cite[Lemma 4.21]{even2021amplituhedron}, which shows that for any $Z'$ and $C':=\inc_{i}(C)$, there is a $Z$ such that for each twistor $\llrr{I}$ appearing in $F$, $\llrr{C'Z'~Z'_I}= \llrr{CZ~Z_I}$. (There is no sign because $[\ell] \cap I$ is empty, since $I \subset N_R$.) We have $\lr{I}_C= \lr{I \cup \{i\}}_{C'}$ and by the proof of \cite[Lemma 4.21]{even2021amplituhedron}, the same statement holds for $Z, Z'$. So adding $i$ to each Pl\"ucker coordinate and minor appearing in $f$ and $g$ gives a formula for $F$ on the image of $\inc_{i} (S)$.
	
	The argument for $\inc_i (S_L)$ (assuming $F$ has strong sign $s$ on $S_L$) is exactly the same, except that if $n \in I$, then $\llrr{C'Z'~Z'_I}= -\llrr{CZ~Z_I}$. So $F(C'Z', Z')= (-1)^{\deg_n F}F(CZ, Z)$, and the strong sign of $F$ on the image of $\inc_\ell(S_L)$ is $(-1)^{\deg_n F}s$.
	
	\noindent \textbf{Proof of (2):}
	The arguments for $x_i$ and $y_{i-1}$ are very similar, so we will only give the $x_i$ argument here. Let $S'= x_i(\R_+).S_R$. 
	
	First, if $S=S'$, then the statement is clearly true. So we will assume $S \neq S'$, and thus $S \subset \overline{S'}$. 
	
	Every $C'\in S'$ can be written uniquely as $Cx_i(t)$ for $t>0,~C\in S$. The parameter $t$ is the ratio $\lr{I\cup\{i+1\}}_{C'}/\lr{I\cup\{i\}}_{C'}$ where $I\subset [n]\setminus\{i,i+1\}$ is any subset such that $I\cup\{i\}$ is a basis of $S$ but $I\cup\{i+1\}$ is not. Such $I$ exists since the matroid of $S'$ strictly contains that of $S.$
	
	Fix $Z'$. Let $Z= x_i(t)Z'$, so that $CZ= Cx_i(t)Z'= C'Z'$. By \cite[Lemma 4.24]{even2021amplituhedron}, $\llrr{C'Z'~Z_I'}= \llrr{CZ~Z_I}$ for all twistors appearing in $F$, because $i \notin I$. This implies that $F(C'Z', Z')=F(CZ, Z)= f(C, Z)/ g(C,Z)$. What remains is to re-express the right hand side in terms of Pl\"uckers of $C'$ and $Z'$, without changing the signs of any coefficients.
	
	It follows from the proof of \cite[Lemma 7.11]{lam2015totally} that the Pl\"ucker coordinates of $C$ are Laurent polynomials in the Pl\"ucker coordinates of $C'$ with positive coefficients. Maximal minors of $Z$ are positive polynomials in the maximal minors of $Z'$ and in $t$, which is the ratio of two Pl\"ucker coordinates of $C'$. Thus we may rewrite $f(C, Z)/ g(C,Z)$ in terms of maximal minors of $Z'$ and Pl\"ucker coordinates of $C'$ without changing the sign of any coefficient.
	
	\noindent \textbf{Proof of (3):}
	Note that $\mbcfw(\chR\times S)= \mbcfw(\Gr_{0, \{a,b,n\}}^{>0} \times \chR\times S)$ can also be described as
	\[S_3=y_c(\R_+).y_d(\R_+).y_n(\R_+).x_a(\R_+).\inc_a(S).\]
	Write \[S_1=y_n(\R_+).x_a(\R_+).\inc_a(S),~~S_2=y_d(\R_+).S_1,~~S_3=y_c(\R_+).S_2.\]
	
	Note that $F$ has strong sign $s$ on the image of $S_1$ by (1) and (2).

	For $C\in S_3$, $C=y_c(t).C'$ for unique $t>0$ and $C'\in S_2$. Similarly, $C'=y_d(s).C''$ for unique $s>0$ and $C''\in S_1.$ Note that
	$t,s$ can be computed as $\lr{I\cup\{c\}}_C/\lr{I\cup\{d\}}_C$, and $\lr{I\cup\{d\}}_C/\lr{I\cup\{n\}}_C$ respectively for any $I\in\binom{b+1,\ldots,c-1}{k_L} $ such that $I$ is a basis for $S$. Such an $I$ is guaranteed to exist because $\{b,c,d,n\}$ is coindependent for $S$.
	
	Fix $Z$. Let $Z':= y_c(t)Z$ and $Z'':= y_d(s)Z'$, so that $C''Z''=CZ$. Note that $\llrr{C''Z''~Z_I''}=\llrr{CZ~Z_I''}= \Psi_{a c}\llrr{CZ~Z_I}$. This implies that $\Psi_{a c}F (CZ, Z)= F(C''Z'', Z'')$. Since $F$ has strong sign $s$ on the image of $S_1$, the right hand side is equal to $p(C'', Z'')/q(C'', Z'')$ where $p,q$ are polynomials in the Pl\"uckers of $C'', Z''$, all coefficients of $p$ have sign $s$ and all coefficients of $q$ are positive. Again, what remains is to re-express this in terms of the Pl\"uckers of $C,Z$ without changing the signs of coefficients. Again, the proof of \cite[Lemma 7.11]{lam2015totally} shows that the Pl\"ucker coordinates of $C''$ are positive Laurent polynomials in the Pl\"ucker coordinates of $C'$, which are themselves positive Laurent polynomials in the Pl\"ucker coordinates of $C$. Since additionally $s,t$ are positive ratios of Pl\"ucker coordinates of $C$, we are done.
	
	The argument for the strong sign of $\Psi_{ac}F$ on the image of 
	\[\mbcfw(S_L \times \chR \times \Gr_{0, \{b,c,d,n\}}^{>0})= y_a(\R_+).y_{b}(\R_+).y_c(\R_+).x_d(\R_+)\circ  \inc_d \circ \pre_{c}(S_L)\]
	is very similar. The only difference is that (1), (2) imply $F$ has strong sign $(-1)^{\deg_n F}s$ on the image of $S_1:=y_{b}(\R_+).y_c(\R_+).x_d(\R_+)\circ  \inc_d \circ \pre_{c}(S_L).$ Then essentially the same argument as above shows that $\Psi_{a c}F$ also has strong sign $(-1)^{\deg_n F}s$ on the image of $y_a(\R_+). S_1= \mbcfw(S_L \times \chR \times \Gr_{0, \{b,c,d,n\}}^{>0}).$
	
\end{proof}
	
	With these lemmas in hand, we proceed to the proof of \cref{prop:vanishing_and_sign_of_functionaries_under_promotion}.
	
	\begin{proof}[Proof of \cref{prop:vanishing_and_sign_of_functionaries_under_promotion}]
We assume that the indices in $F$ are contained in $N_R$, as the other case is similar. At the end of the proof 
we will discuss the differences in the other case.
		
			\noindent \textbf{Item (1):} Suppose $F$ vanishes on the image of $S_R$. By \cref{lem:evaluation-after-bullet}, $\Psi_{ac}F$ evaluated on a point of $S_R^\bulR$ is equal to $F$ evaluated on a point of $S_R$ (using a different matrix $Z$). Thus, $\Psi_{ac}F$ vanishes on the image of $S_R^\bulR$, and also on the image of $\overline{S_R^\bulR}$. Since $S_L \bcfw S_R \subset \overline{S_R^\bulR}$, $\Psi_{ac}F$ vanishes on $S_L \bcfw S_R$, as desired.
			
			\noindent \textbf{Item (2):} Suppose $F$ has strong sign $s$ on the image of $S_R$. Since $\{a, b, n\}$ is co-independent for $S_L$, we may find $1\leq i_1<i_2<\ldots i_{k_L}<a$ such that $B=\{i_1, i_2, \dots, i_{k_L}\}$ is a basis for $S_L$. Let $S' \subset \Gr_{k_L, N_L}^{\geq 0}$ be the positroid cell whose only basis is $B$.
			Define
			\begin{equation}\label{eq:S_bullet}
				S_\bulR :=S' \bcfw S_R = \pre_{N_L \setminus (I \cup \{a, b, n\})} \circ \inc_{i_{k_L}} \circ \cdots \circ \inc_{i_2} \circ \inc_{i_1}(\mbcfw(\chR \times S_R)).\end{equation}
			
			Since the bases of $S'$ are a subset of the bases in $S_L$, it is easy to see using \cref{lem:matroid_under_bcfw} that every basis of $S_\bulR $ is a basis of ${S_L \bcfw S_R}$. So we have
			\begin{equation}\label{eq:sandwich}
				S_\bulR\subseteq \overline{S_L \bcfw S_R} \subseteq \overline{S_R^\bulR}.
			\end{equation}
Using \cref{lem:strong_pos} on $N_R, S_R$, we will show that $\Psi_{ac}F$ has strong sign $s$ on the images of both $S_R^\bulR$ and $S_\bulR.$ 

We deal with the sign of $\Psi_{ac} F$ on the image of $S_R^\bulR$ first. In \cref{lem:9_6_elt}, we apply the upper BCFW map, $\inc_\ell, \pre_\ell$ for $\ell \notin N_R$, and $x_i, y_{i-1}$ for $ i \notin N_R$. So by \cref{lem:strong_pos}, $\Psi_{a c}F$ has strong sign $s$ on the image of $S_R^\bulR$.

Since the operations applied in \eqref{eq:S_bullet} are a subset of those applied in \cref{lem:9_6_elt}, an identical argument shows that $\Psi_{a c}F$ has strong sign $s$ on the image of $S_\bulR$.

By \eqref{eq:sandwich}, $\Psi_{ac}F$ having strong sign $s$ on the image of $S_R^\bulR$ implies that either $\Psi_{ac}F$ has strong sign $s$ or it is identically zero on the image of $S_L \bcfw S_R$. The second possibility is ruled out by $\Psi_{ac}$ having strong sign $s$ on the image of $S_\bulR$. This completes the argument.

\noindent \textbf{Item (3):} Again we prove the statement for $S_R$. We first study the possible sign profiles of $(\Psi F_1,\ldots,\Psi F_r)$ on points of $\gto{S_R^\bulR}.$ \cref{lem:evaluation-after-bullet} implies that for every $Y\in\gto{S_R^\bulR}$ there exists $C'\in S_R$ and a positive matrix $Z'\in \Mat_{N_R,k_R+4}^{>0}$ such that 
\[\sgn (\Psi F_1(Y,Z),\ldots,\Psi F_r(Y,Z))=
\sgn (F_1(C'Z',Z'),\ldots, F_r(C'Z',Z')\in A.\]
Thus, the possible sign profiles for points of $\gto{S_R^\bulR},$ and all possible positive $Z$ also lie in the set $A.$

Since $\gto{S}\subseteq\gt{S_R^\bulR}=\overline{\gto{S_R^\bulR}},$ every $Y\in\gto{S}$ is the limit of a sequence $(Y_n)_{n=1}^\infty\subset\gto{S_R^\bulR}.$ By passing to a subsequence we may assume that the sequence of sign profiles
\[\sgn (\Psi F_1(Y_n),\ldots,\Psi F_r(Y_n))_{n=1}^\infty\subseteq A\] 
is the constant sequence. It is immediate to see that the limit of point with a sign profile in $A$ has a sign profile in $\cl(A).$ 

\noindent \textbf{Changes for $S_L$:} 
There is a version of  
\cref{lem:9_6_elt} for $S^\bulL_L$, where the first step is the ``lower embedding" $\mbcfw(S_L \times \chR \times \Gr_{0, \{b,c,d,n\}}^{>0})$ of \cite[Definition 3.7]{even2021amplituhedron}. There is also a version of \cref{lem:evaluation-after-bullet}, which has an added sign of $(-1)^{(k_R +1) \deg_n F}$ on the right hand side. 
This is because for any functionary $F$, the evaluation of $F$ on a point in the image of $S_L$ is the same as the evaluation of $(-1)^{\deg_n F} s \Psi_{ac}F$ on a point in the image of $\mbcfw(S_L \times \chR)$ (see \cite[Example 4.38]{even2021amplituhedron}). Then each of the subsequent $k_R$ applications of $\inc$ change the sign by $(-1)^{\deg_n \Psi F} = (-1)^{\deg_n F}$. From here, the arguments for (1),(2),(3) are identical.
	\end{proof}

\subsection{The rows of the BCFW matrix}\label{subsub:canonical_BCFW_form} 
In this section, we set up a particular way of writing the rows 
$v_i^\rcp$ (for $1\leq i \leq k$) 
of the BCFW matrix $\mtx_\rcp$. This will be very useful when we invert $\tZ$ on $\gto{\rcp}$. Recall from \cref{def:dihedral} the matrices $[\cyc_{k,n}], [\refl_n], P_{k, i, s}$.

\begin{definition}\label{def:vs}
	Let $S_\rcp \subset \Grk$ be the BCFW cell associated to recipe $\rcp$.  Recall the 
	notation $\st$ and $\rcpp$ from \cref{not:L-and-R}. For each coordinate $\czeta_i \in \coord_\rcp$
	(cf \cref{def:BCFW-coords-mtx}), where $1\leq i \leq k$ and $\czeta\in \{\alpha,\beta,\gamma,\delta,\epsilon\}$, we recursively define a row vector $v_{\czeta_i}^\rcp \in \R^n$: 
	\begin{enumerate}
		\item If $\st=(a_k, b_k, c_k, d_k, n_k)$, then for $i=k$, we set
		\[v_{\calpha_k}^\rcp= e_{a_k}, \quad v_{\cbeta_k}^\rcp= e_{b_k}, \quad v_{\cgamma_k}^\rcp= (-1)^{k_R}e_{c_k}, \quad v_{\cdelta_k}^\rcp= (-1)^{k_R}e_{d_k}, \quad v_{\cepsilon_k}^\rcp= (-1)^{k_R}e_{n_k}, \quad \]
		and for $i<k$, we set
		\begin{equation*}
			v_{\czeta_i}^\rcp= \begin{cases}
				v_{\czeta_i}^{L} \cdot y_{a_k}(\frac{\alpha_k}{\beta_k}) P_{n,n, k_R+1} & \text{ if the $i$th step tuple is in }\rcp_L\\
				v_{\czeta_i}^{R} \cdot y_{d_k}(\frac{\delta_k}{\epsilon_k}) y_{c_k}(\frac{\gamma_k}{\delta_k}) &\text{ if the $i$th step tuple is in }\rcp_R.
		\end{cases}\end{equation*}
		\item If $\st=\cyc$, then $v_{\czeta_i}^\rcp = v_{\czeta_i}^\rcpp \cdot [\cyc_{k,n}]$.
		\item If $\st=\pre_{I_k}$, then $v_{\czeta_i}^\rcp = v_{\czeta_i}^\rcpp$.
		\item If $\st = \refl$, then 
		\[v_{\czeta_i}^\rcp = 
		\begin{cases}
			(-1)^{\binom{k}{2}}	v_{\czeta_i}^\rcpp \cdot [\refl_n] &\text{ if the $i$th step tuple indexes the top row of }\mtx_\rcp\\
			v_{\czeta_i}^\rcpp \cdot [\refl_n] & \text{ else}.
		\end{cases}
		\]
	\end{enumerate}
	Note that we implicitly embed $\R^{N_L} \hookrightarrow \R^n$ and $\R^{N_R} \hookrightarrow \R^n$ in (1) above, so $y_i(t)$ is $n \times n$. 
	
	Finally for $1 \leq i \leq k$ we  define 
	$v_i^\rcp
	:= \calpha_i v_{\calpha_i}^\rcp + \cbeta_i v_{\cbeta_i}^\rcp + \cgamma_i v_{\cgamma_i}^\rcp + \cdelta_i v_{\cdelta_i}^\rcp + \cepsilon_i v_{\cepsilon_i}^\rcp$. If the recipe is clear from context, we will drop the superscript
	$\rcp$.  \end{definition}
From the definition of $\mtx_\rcp$ (\cref{def:BCFW-coords-mtx}) and \cref{cor:dim}, we have the following lemma.

\begin{lemma}\label{lem:vs-unique}
	Let $S_\rcp \subset \Grk$ be a BCFW cell. For $i=1, \dots, k$, the vector $v_i^\rcp $ is precisely the row of the BCFW matrix $\mtx_\rcp$ indexed by the $i$th step tuple. In particular, for each $V \in S_\rcp$, there is a unique choice of $([\calpha_i: \cbeta_i: \cgamma_i: \cdelta_i: \cepsilon_i]) \in (\chR)^k$ such that $V = \Span(v_1^\rcp, \dots, v_k^\rcp)$.
\end{lemma}

\subsection{The inverse problem for BCFW cells: 
the proof of \cref{thm:BCFW-tile-and-sign-description}}
\label{subsec:inverse_problem}

In this subsection we will show how to invert the amplituhedron map on the image of a BCFW cell, thus proving that BCFW cells map injectively into the amplituhedron. We will use the following key lemma.

\begin{lemma}
	\label{lem:solve_params_5}
	Let $k \geq 1$ and let $Y \in \Mat_{k \times (k+4)}$ and $Z \in \Mat_{5 \times (k+4)}$,
	with row vectors $Y_1,\dots,Y_k$, and $Z_1,\dots,Z_5$, respectively. Define 
	$$v:= \llrr{2345} Z_1 - \llrr{1345} Z_2 + \llrr{1245} Z_3 - \llrr{1235} Z_4 + \llrr{1234} Z_5.$$
	Suppose at least one of the $5$ twistor coordinates
	$\llrr{2345}, \llrr{1345}, \llrr{1245}, \llrr{1235}, \llrr{1234} $ is nonzero. Then $\Span(Y_1,\ldots, Y_k) \cap \Span(Z_1,\ldots, Z_{5}) = \Span(v)$, and in particular is the trivial vector space if and only if $v=0$.
\end{lemma}
\begin{proof}
	The assumption on twistor coordinates
	implies that $Y$ has rank $k$, $Z$ has rank either $4$ or $5$ and that $\Span(Y_1,\ldots, Y_k, Z_1,\ldots, Z_{5}) \subset \R^{k+4}$ has dimension exactly $k+4$. If $Z$ has rank 5, the lemma now follows from \cite[Lemma 4.34]{even2021amplituhedron}.
	
	If $Z$ has rank 4, then $\Span(Y_1,\ldots, Y_k) \cap \Span(Z_1,\ldots, Z_{5})$ is trivial for dimension reasons. We will show $v=0$. Since $Z$ is rank 4, we can find a nontrivial linear combination
	\[\sum_{j=1}^5 x_j Z_j =0.\]
	Now, fix $i \in [5]$ so that $x_i \neq 0$. We claim $\llrr{Y Z_{[5] \setminus \{i\}}} \neq 0$. Indeed, if $\llrr{Y Z_{[5] \setminus \{i\}}}=0$, then we would have
	$$ \sum_{i=1}^{k} a_i Y_i = \sum_{[5]\setminus \{i\}} x_j'Z_j.$$
	Since the intersection of $\Span(Y_1, \dots, Y_k)$ and $\Span(Z_1, \dots, Z_5)$ is trivial, each side of the above linear combination must be equal to 0. But this shows the left nullspace of $Z$ has dimension at least 2, a contradiction. So $\llrr{Y Z_{[5] \setminus \{i\}}} \neq 0$.
	
	Choose $j \neq i$ and apply the linear functional $\llrr{Y Z_{[5] \setminus \{i,j\}}u}$ to the linear relation above. Then
	\[x_j \llrr{Y Z_{[5] \setminus \{i,j\}} Z_j} + x_i \llrr{Y Z_{[5] \setminus \{i,j\} } Z_j} =0\]
	or, rearranging, 
	\[\frac{x_j}{x_i}= (-1)^{j-i} \frac{\llrr{Y Z_{[5] \setminus \{j\}}}}{\llrr{Y Z_{[5] \setminus \{i\}}}}.\]
	This shows $v=0$.
\end{proof}

Before using \cref{lem:solve_params_5} to invert $\tZ$, we need an alternate formulation of the twistor matrix $\twmt_\rcp(Y)$ for $S_\rcp$. Recall from \cref{def:twistor-mtx} that the entries of $\twmt_\rcp(Y)$ are rational functions in the recursively defined coordinate functionaries $\czeta_i(Y)$. We show that these functionaries can be expressed differently using the vectors $v_{\czeta_i}^\rcp$.

 \begin{proposition}\label{lem:coord-functionaries-as-twistors}
 	Let $S_\rcp \in \Grk$ be a BCFW cell and $\{\czeta_i(Y)\}$ its collection of coordinate functionaries. Let $v_{\zeta_i}:=v_{\zeta_i(Y)}^\rcp$ be the vectors from \cref{def:vs} with $\czeta_j$ set to equal $ \czeta^\rcp_j(Y)$. Then for $i=1, \dots, k$, we have 
 	 \begin{align}\label{eq:coords-as-twistors}
 		\galp_i(Y)=&\llrr{v_{\beta_i}Z,v_{\gamma_i}Z,v_{\delta_i}Z,v_{\epsilon_i}Z} \quad \gbet_i(Y)=-\llrr{v_{\alpha_i}Z,v_{\gamma_i}Z,v_{\delta_i}Z,v_{\epsilon_i}Z} \quad
 		\ggam_i(Y)=\llrr{v_{\alpha_i}Z,v_{\beta_i}Z,v_{\delta_i}Z,v_{\epsilon_i}Z}\\
 		\notag&\gdel_i(Y)=-\llrr{v_{\alpha_i}Z,v_{\beta_i}Z,v_{\gamma_i}Z,v_{\epsilon_i}Z} \quad
 		\geps_i(Y)= \llrr{v_{\alpha_i}Z,v_{\beta_i}Z,v_{\gamma_i}Z,v_{\delta_i}Z}.
 	\end{align}
 \end{proposition}
\begin{proof}
	We induct on the number of steps in $\rcp$. The base case of 0 steps is trivially true. We now need to show that if \eqref{eq:coords-as-twistors} holds for $\rcpp$ (or $\rcp_L, \rcp_R$), then it holds for $\rcp$.
	
	If $\st= \pre_{I_k}$, then the coordinate functionaries are the same for $\rcp$ and $\rcpp$. This shows that the vectors for $\rcp, \rcpp$ are the same as well, since they are defined using the same formulas and are evaluated on the same functionaries. So \eqref{eq:coords-as-twistors} holds for $\rcp$ if it holds for $\rcpp$. 
	
	The arguments for $\st=\refl$ and $\st=\cyc$ are similar to each other, so we give only the latter case here. By \cref{def:twistor-mtx}, when we go from $\rcpp$ to $\rcp$, the left hand sides of \eqref{eq:coords-as-twistors} change by 
	\[F \mapsto \tilde{F} := (-1)^{k \deg_n F} \cyc^{-*} F.\]
 We will show the right hand sides also change by this operation. Notice that for a pure functionary $F$, $\tilde{F}$ is obtained by substituting
 \[Z_n \mapsto (-1)^{k-1}Z_1 \quad \text{and} \quad Z_j \mapsto Z_{j+1} \text{ for } j<n.\]
 This can be seen by writing $F$ in terms of twistors $\llrr{J}$ with elements of $J$ written in increasing order, as is our convention.
	
	Note that $v_{\zeta_i}^\rcp$ is obtained from $v_{\zeta_i}^\rcpp$ by right multiplying by $[\cyc_{k,n}]$ and replacing the coordinate functionaries $\czeta_j^\rcpp(Y)$ with the coordinate functionaries $\czeta_j^\rcp(Y)$. 
	That is, if
	\[v_{\zeta_i}^\rcpp = F_1 e_1 + \dots + F_{n-1}e_{n-1}+F_n e_n\]
	where $F_i$ are rational functionaries, then 
		\[v_{\zeta_i}^\rcp = \tilde{F}_1e_2 + \dots+  \tilde{F}_{n-1}e_n + (-1)^{k-1} \tilde{F}_n e_1.\]
		Right multiplying by $Z$, we see that replacing $v_{\zeta_i}^\rcpp Z$ by $v_{\zeta_i}^\rcp Z$ in the right hand sides of \eqref{eq:coords-as-twistors} has the effect substituting $Z_n$ with $(-1)^{k-1}Z_1$ and $Z_j$ with $Z_{j+1}$ for all other $j$, as desired.
	
	Now, if $\st=(a, b, c, d, n)$, \eqref{eq:coords-as-twistors} holds for $i=k$ by definition, since 
	\[v_{\calpha_k}Z= Z_{a} \quad v_{\cbeta_k}Z= Z_{b} \quad v_{\cgamma_k}Z= (-1)^{k_R}Z_{c} \quad v_{\cdelta_k}Z= (-1)^{k_R}Z_{d}\quad v_{\cepsilon_k}Z= (-1)^{k_R}Z_{n}.\]
	If the $i$th step tuple of $\rcp$ is in $\rcp_L$, then the coordinate functionaries change by 
	\[F \mapsto \tilde{F}:= (-1)^{(k_R+1)\deg_n F}\Psi_{a c} F \]
	 when we pass from $\rcp_L$ to $\rcp$. We will show this is also how the right hand sides of  \eqref{eq:coords-as-twistors} change. The argument if the $i$th 4-step tuple of $\rcp$ is in $\rcp_R$ is very similar, and is left to the reader.
	
	Note that $v_{\zeta_i}^\rcp$ is obtained from $v_{\zeta_i}^L$ by right multiplication by $y_{a}(\frac{\alpha_k(Y)}{\beta_k(Y)}) P_{n,n, k_R+1}$ and then replacing the coordinate functionaries $\czeta_j^L(Y)$ with $\czeta_j^\rcp(Y)$. That is, if 
		\[v_{\zeta_i}^L = F_1 e_1 + \dots + F_{a}e_{a}+ F_{b}e_{b}+F_{n} e_{n}\]
	where $F_i$ are rational functionaries, then 
		\[v_{\zeta_i}^\rcpp = \tilde{F}_1 e_1 + \dots + \tilde{F}_{a} e_a + \tilde{F}_{b} \left(e_b+  \frac{\alpha_k(Y)}{\beta_k(Y)} e_a\right)+(-1)^{k_R +1}\tilde{F}_{n} e_{n}.\]
		Recalling that $\frac{\alpha_k(Y)}{\beta_k(Y)} = -\frac{\llrr{bcdn}}{\llrr{acdn}}$, we see that replacing $v_{\zeta_i}^L Z$ by $v_{\zeta_i}^\rcp Z$ in the right hand sides of \eqref{eq:coords-as-twistors} has the effect of substituting
		\[ Z_b \mapsto Z_{b}-\frac{\llrr{bcdn}}{\llrr{acdn}} Z_a \quad \text{and} \quad Z_n \mapsto (-1)^{k_R+1}Z_n \]
		which is precisely the map $F \mapsto \tilde{F}$ since all functionaries appearing are pure.
\end{proof}

We now partially invert the $\tZ$-map on a BCFW cell. That is, we show there is a subset of $\gto{\rcp}$ on which $\tZ:S_\rcp \to \gto{\rcp}$ is invertible. Later we show that this subset is the entire open BCFW tile $\gto{\rcp}$.

\begin{proposition}\label{prop:preimage-if-coord-fcn-pos}
	Let $S_\rcp \in \Grk$ be a BCFW cell. Let $Y \in \Gr_{k, k+4}$ be a point such that $\czeta_i^\rcp(Y)>0$ for all coordinate functionaries. Then $\twmt_\rcp(Y) \in S_\rcp$ and is the unique element of $\tZ^{-1}(Y) \cap S_\rcp$. In particular, 
	\begin{equation}\label{eq:loci-is-subset} \gto{\rcp} \supseteq \{Y \in \Gr_{k, k+4}: \czeta_i^\rcp(Y)>0 \text{ for all coordinate functionaries}\}.\end{equation}
\end{proposition}
\begin{proof}
	First, since all coordinate functionaries are positive on $Y$, $[\calpha_i(Y): \cbeta_i(Y): \cgamma_i(Y):\cdelta_i(Y):\cepsilon_i(Y)]_{i=1}^k$ is an element of $(\chR)^k$. The matrix $C:=\twmt_\rcp(Y)$ is precisely the image of this point under the map in \cref{cor:dim}, so is in $S_\rcp$. Writing the rows of $\twmt_\rcp(Y)$ using the vectors $v_{\czeta_i}$, we see the vectors $v_{\czeta_i}$ are all evaluated on the coordinate functionaries of $Y$, so are precisely the vectors appearing in \cref{lem:coord-functionaries-as-twistors}.
	
	 Now, by \cref{lem:vs-unique}, row $i$ of $C$ is equal to 
	 \[v_i = \calpha_i(Y) v_{\alpha_i}+\cbeta_i(Y) v_{\beta_i}+\cgamma_i(Y)v_{\gamma_i}+\cdelta_i(Y)v_{\delta_i}+\cepsilon_i(Y)v_{\epsilon_i}\]
	 so by \cref{lem:coord-functionaries-as-twistors}, row $i$ of $CZ$ is equal to 
	 \begin{align*}
	 	v_i Z= &\llrr{v_{\beta_i}Z,v_{\gamma_i}Z,v_{\delta_i}Z,v_{\epsilon_i}Z} v_{\alpha_i}Z
	 -\llrr{v_{\alpha_i}Z,v_{\gamma_i}Z,v_{\delta_i}Z,v_{\epsilon_i}Z} v_{\beta_i}Z
	 +\llrr{v_{\alpha_i}Z,v_{\beta_i}Z,v_{\delta_i}Z,v_{\epsilon_i}Z}v_{\gamma_i}Z\\
	\notag &-\llrr{v_{\alpha_i}Z,v_{\beta_i}Z,v_{\gamma_i}Z,v_{\epsilon_i}Z}v_{\delta_i}Z
	 + \llrr{v_{\alpha_i}Z,v_{\beta_i}Z,v_{\gamma_i}Z,v_{\delta_i}Z}v_{\epsilon_i}Z.
 \end{align*}
Note that $v_iZ$ is a nonzero vector and is also exactly the vector $v$ from \cref{lem:solve_params_5} applied to $Y$ and the matrix with rows $v_{\alpha_i}Z, v_{\beta_i}Z, v_{\gamma_i}Z,v_{\delta_i}Z,v_{\epsilon_i}Z$. All of the twistor coordinates appearing as coefficients are nonzero by assumption, thus by \cref{lem:solve_params_5}, the vector $v_iZ$ spans the subspace $Y \cap \Span(v_{\alpha_i}Z, v_{\beta_i}Z, v_{\gamma_i}Z,v_{\delta_i}Z,v_{\epsilon_i}Z)$ and in particular is a nonzero vector in $Y$. This implies that the rowspan of $CZ$ is contained in $Y$. Since both subspaces are $k$-dimensional, they must be equal. Thus $C$ is a preimage of $Y$.

Now, suppose $C' \in S_D$ is another preimage of $Y$, with BCFW parameters $\{\czeta_i'\}$ and vectors $v_{\czeta_i'}$. We will show that in fact $C'$ has the same BCFW parameters as $C$, starting with $i=k$. For $i=k$, we have $v_{\czeta_k} = v_{\czeta_k'}$ because these vectors are the same for all elements of $S_D$. 

The $k$th row of $C'Z$ is
 \[v_k'Z= \calpha_k' v_{\alpha_k'}Z+\cbeta_k' v_{\beta_k'}Z+\cgamma_k' v_{\gamma_k'}Z+\cdelta_k' v_{\delta_k'}Z+\cepsilon_k' v_{\epsilon_k'}Z= \calpha_k' v_{\alpha_k}Z+\cbeta_k' v_{\beta_k}Z+\cgamma_k' v_{\gamma_k}Z+\cdelta_k' v_{\delta_k}Z+\cepsilon_k' v_{\epsilon_k}Z\]
 and is a nonzero element of $Y$. Also, $v_k'Z \in \Span(v_{\alpha_k}Z, v_{\beta_k}Z, v_{\gamma_k}Z,v_{\delta_k}Z,v_{\epsilon_k}Z)$, so must be proportional to $v_k Z$, which spans $Y \cap \Span(v_{\alpha_k}Z, v_{\beta_k}Z, v_{\gamma_k}Z,v_{\delta_k}Z,v_{\epsilon_k}Z)$. Thus 
 \[[\calpha_k: \cbeta_k: \cgamma_k:\cdelta_k:\cepsilon_k] = [\calpha_k': \cbeta_k': \cgamma_k':\cdelta_k':\cepsilon_k'].\]
 
 Now, suppose $\czeta_j=\czeta_j'$ for all $j>i$. This implies that $v_{\czeta_i}=v_{\czeta_i'}$ since $v_{\czeta_i}$ depends only on the BCFW parameters indexed by $j>i$. Repeating the argument above with the $i$th row of $CZ'$ shows that 
  \[[\calpha_i: \cbeta_i: \cgamma_i:\cdelta_i:\cepsilon_i] = [\calpha_i': \cbeta_i': \cgamma_i':\cdelta_i':\cepsilon_i'].\]
\end{proof}

To prove
\cref{thm:BCFW-tile-and-sign-description}, 
we just need to show the reverse inclusion in \eqref{eq:loci-is-subset}. That is, for every point $C \in S_\rcp$, we need to show that $CZ$ has positive coordinate functionaries. To do so, we analyze the signs of functionaries on $\gto{\rcp}$ before and after promotion, cyclic shift, and reflection, using results from 
\cref{sec:evolve}.

\begin{proof}[Proof of \cref{thm:BCFW-tile-and-sign-description}]
	It suffices to show the reverse inclusion in \eqref{eq:loci-is-subset}, since the desired statement is already known for the set on the right hand side. That is, we need to show that for all $Z \in \Mat_{n,k+4}^{>0}$,
	\begin{equation}\label{eq:loci-is-supset}
		\gto{\rcp} \subseteq \{Y \in \Gr_{k, k+4}: \czeta_i^\rcp(Y)>0 \text{ for all coordinate functionaries}\}
	\end{equation}
We will show that all coordinate functionaries for $\rcp$ have strong sign $+1$ on the image of $S_\rcp$, and so in particular are positive on $\gto{\rcp}$.

We proceed by induction. The base case is $k=0$ or $\rcp = \emptyset$, which is trivially true since there are no coordinate functionaries for $\rcp$.

Now, suppose $\st\in \{\pre_{I_k}, \refl\}$. By induction, the coordinate functionaries of $\rcpp$ are strongly positive on the image of $S_\rcpp$. By \cref{def:twistor-mtx}, the coordinate functionaries for $\rcp$ are obtained from those for $\rcpp$ by doing nothing (if $\st=\pre_{I_k}$) or applying $\refl^*$ (if $\st=\refl$). In either case, \cref{thm:signs-under-cyc-rot-pre} show that the coordinate functionaries for $\rcp$ have strong sign $+1$ on the image of $S_\rcp$.

If $\st = \cyc$, then $\czeta_i^\rcp (Y)= (-1)^{k \deg_n \czeta_i^\rcp(Y)} \cyc^{-*} \czeta_i^\rcp(Y)$. Again, the inductive hypothesis and \cref{thm:signs-under-cyc-rot-pre} show that this functionary has strong sign $+1$ on the image of $S_\rcp$.

If $\st = (a,b,c,d,n)$, then \cref{lem:sign_of_chord_twistors} shows that each of the five twistors $\czeta_k^{\rcp}(Y)$ have strong sign $+1$ on the image of $S_\rcp$. \cref{prop:vanishing_and_sign_of_functionaries_under_promotion},~\cref{it:strong_sign} and the inductive hypothesis shows that the remaining coordinate functionaries have strong sign $+1$ on the image of $S_\rcp$.
\end{proof}

Since the map $Y \mapsto \twmt_\rcp(Y)$ of \cref{thm:BCFW-tile-and-sign-description} amounts to a continuous inverse of $\tZ$ on $\gto{\rcp}$, we have the following corollary.
\begin{corollary}\label{cor:open_map}
	The amplituhedron map, restricted to a BCFW cell, is open. 
\end{corollary}

\begin{corollary}\label{cor:bdry-covered-by-bdry-image}
	Let $S_\rcp$ be a BCFW cell. Then $\partial \gt{\rcp} \subset \tZ(\partial \overline{S}_\rcp).$
\end{corollary}
\begin{proof}
	Since $\tZ$ restricts to a homeomorphism on $S_\rcp$, $\gto{\rcp}$ is open and thus is contained in the interior of $\gt{\rcp}$. Since $\gt{\rcp}$ is the image of $\overline{S}_\rcp$, the boundary of $\gt{\rcp}$ is contained in the image of  $\partial \overline{S}_\rcp$.
\end{proof}

We also show that each coordinate cluster variable of $\rcp$ (c.f. \cref{def:generalcluster}) has a strong sign on the image of $S_{\rcp}$.

\begin{corollary}\label{cor:clust-var-strong-sign}
	Let $S_\rcp$ be a BCFW cell and let $\rzeta_i \in \Irr(\rcp)$ be a coordinate cluster variable. Then for some $s \in \{\pm 1\}$, $\rzeta_i$ has strong sign $s$ on the image of $S_\rcp$.
\end{corollary}
\begin{proof}
	We proceed by induction. Recall that $\rzeta_i^\rcp$ is either one of the twistor coordinates in \cref{lem:sign_of_chord_twistors}; or it is equal to one of $\rzeta_i^\rcpp, \cyc^{-*} \rzeta_i^\rcpp, \refl^{*} \rzeta_i^\rcpp$; or it is the rescaled product promotion of $\rzeta_i^L$ or $\rzeta_i^R$. In the first case, the corollary is true by \cref{lem:sign_of_chord_twistors}. In the second, it is true by \cref{thm:signs-under-cyc-rot-pre}. In the third case, \cref{prop:vanishing_and_sign_of_functionaries_under_promotion} tells us that $\Psi (\rzeta_i^L)$ (or $\Psi( \rzeta_i^R)$, as appropriate) has a strong sign on the image of $S_\rcp$. By definition, $\rzeta_i= \rPsi( \rzeta_i^L)$ differs from $\Psi (\rzeta_i^L)$ by a Laurent monomial in the twistors appearing in \cref{lem:sign_of_chord_twistors}. Since each of these twistors has a strong sign on the image of $S_\rcp$, this implies that $\rzeta_i$ also has a strong sign on the image of $S_\rcp$.
\end{proof}

\begin{remark}\label{rmk:lots-strong-sign}
	While \cref{cor:clust-var-strong-sign} is stated for the coordinate cluster variables of $\Irr(\rcp)$, the proof holds for any functionary which can be obtained by repeatedly applying $\rPsi$, $\cyc^{-*}$, and $\refl$ to a cluster variable with strong sign on the image of the appropriate BCFW cell. In particular, twistor coordinates whose indices are cyclically consecutive in $N$ have a strong sign on the image of every BCFW cell with marker set $N$. Repeatedly applying $\rPsi$, $\cyc^{-*}$, and $\refl$ to such a twistor coordinate will produce a functionary with strong sign on the image of the appropriate BCFW cell.
\end{remark}

\section{BCFW tilings}\label{sec:BCFWtilings}
Recall the notion of \emph{tiling} from \cref{def:tiling}.
In this section we will prove the \emph{BCFW tiling conjecture}
(see \cref{thm:BCFWtiling}), which was implicit in \cite{arkani-hamed_trnka}, 
and which gives  a large class of tilings of the amplituhedron
$\Ank$ using BCFW tiles.

\begin{notation}\label{not:bcfw_tiling}
 Throughout this section we use \cref{rem:Z}, \cref{not:bcfwmap} and fix $k\geq0,n\geq k+4$. Moreover, we define $b_{min}:=2$ if $k_L=0$ and otherwise $b_{min}:=k_L+3$. 
\end{notation}

\begin{definition}[BCFW collections]\label{def:BCFWoutput}

We say that a collection $\mathcal{T}$ of $4k$-dimensional BCFW cells in 
$\Gr_{k,n}^{\geq 0}$  
	is a \emph{BCFW collection of cells} for 
$\mathcal{A}_{n,k,4}$ if it has the following recursive form:
\begin{itemize}
\item If $n=k+4$, $\mathcal{T}$ is the single BCFW cell 
		$\Gr_{k,n}^{>0}$.
\item If $k=0$, $\mathcal{T}$ is the single trivial BCFW cell
		$\Gr^{>0}_{0,n}$.
\item If $\mathcal{T}=\{S_\rcp\}$ is a BCFW collection of cells, 
	so is $\refl \mathcal{T}:=\{\refl S_\rcp\}$ and $\cyc^r \mathcal{T}:=\{\cyc^r S_\rcp\}$.
\item Otherwise $$\mathcal{T} = 
	\mathcal{T}_{pre} \sqcup \bigsqcup_{b,k_L,k_R} \mathcal{T}_{k_L,k_R,b,n},$$
	where
		\begin{itemize}
\item $b$ ranges from 
	$b_{min}$ to $n-3-k_R$, and $b_{min}, k_L,k_R$ as in \cref{not:bcfw_tiling}, 
\item 
$\mathcal{T}_{pre}=\{\pre_{d}(S^i)|i\in\mathcal{F}\},$ where 
			$\{S^i|i\in\mathcal{F}\}$ is a BCFW collection of cells for 
				$\mathcal{A}_{[n]\setminus\{d\},k, 4}$,
\item  $\mathcal{T}_{k_L,k_R,b,n}$ has the form
				$\{S_L^i \bcfw S_R^j \ \vert \ i\in \mathcal{D}, j\in \mathcal{E}\}$
				where 
				$\{S_L^i 
				\ \vert \ i\in \mathcal{D}\}$ and 
				$\{S_R^j \ \vert \ 
				j\in \mathcal{E}\}$ are BCFW collections of cells for 
				$\mathcal{A}_{N_L,k_L,4}$ 
				and $\mathcal{A}_{N_R,k_R,4}$.
		\end{itemize}
		\end{itemize}
	See \cref{fig:bcfw-tiling-ex} for examples of BCFW collections.
	\end{definition}

\cref{thm:BCFWtiling} is the second main result of our paper.
The statement was conjectured
in \cite{arkani-hamed_trnka}, and in the case of 
the standard
BCFW collection, was proven in 
\cite{even2021amplituhedron}. 

\begin{figure}
	\includegraphics[width=0.8\textwidth]{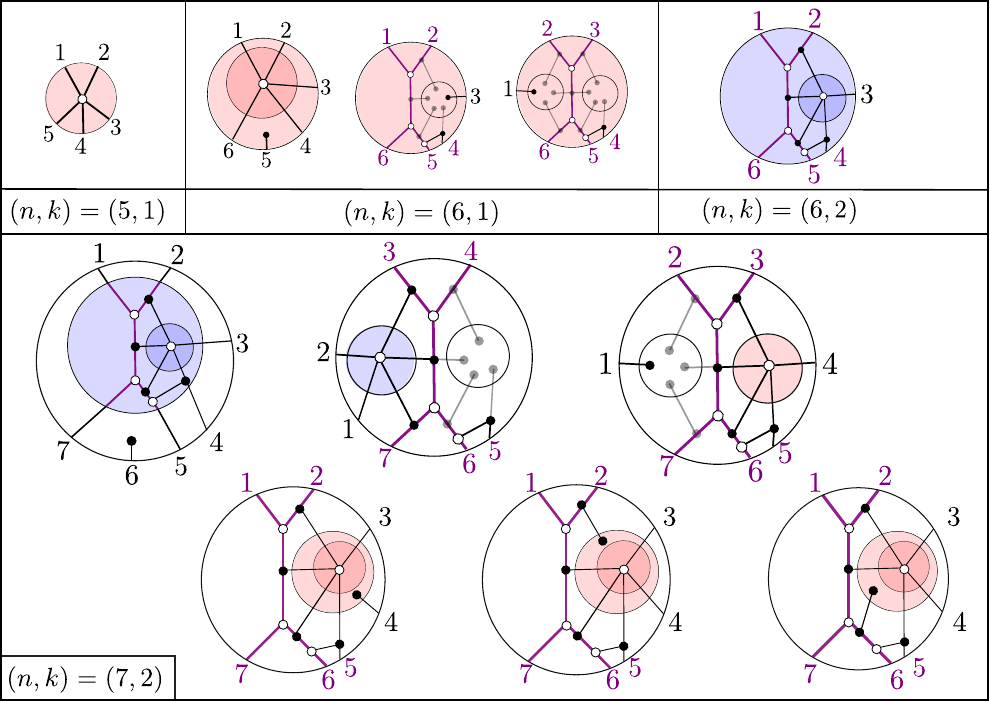}
	\caption{Some BCFW collections in terms of plabic graphs. For $(n,k)=(5,1)$ and $(n, k)= (6,2)$, there 
	is a unique BCFW collection which contains the unique BCFW cell $\Gr^{>0}_{1,\{1,2,3,4,5\}}$ and $\Gr_{0,\{6,1,2\}} \bowtie \Gr_{1,\{2,3,4,5,6\}}$, respectively. For $(n,k)=(6,1)$, we show a BCFW collection where $\T_{pre}$ and each $\T_{k_L, k_R,b,n}$ consists of a single cell. 
		For $(n,k)=(7,2)$ we show a $6$-element BCFW collection.
	 \label{fig:bcfw-tiling-ex}}
\end{figure}

\begin{theorem}[BCFW tilings]\label{thm:BCFWtiling}
Let $\{S_1,\dots,S_{\ell}\}$ be a BCFW collection of cells for $\mathcal{A}_{n,k,4}$. Then for all $Z$, the corresponding collection of tiles $\{Z_{S_1},\dots,Z_{S_{\ell}}\}$
is a tiling of the amplituhedron $\mathcal{A}_{n,k,4}(Z)$, which we refer to as a \emph{BCFW tiling}. 
\end{theorem}

\begin{table}[h] 
\centering 
\begin{tabular}{|l||c|c|c|c|c|c|} 
\hline 
\diagbox{$n$}{$k$} & $0$ & $1$ & $2$ & $3$ & $4$ & $5$  \\
\hline\hline 
$5$ & $1$ & $1$ & $ $ & $ $ & $ $ & $ $  \\
\hline
 $6$ & $1$ & $ 2$ & $ 1$ & $ $ & $ $ & $ $  \\
 \hline 
 $7$ & $1$ & $7$ & $ 7$ & $ 1$ & $ $ & $ $ \\
\hline 
 $8$ & $1$ & $40$ & $2624$ & $40 $ & $ 1$ & $ $ \\
\hline 
$9$ & $1$ & $297$ & $>10^7$ & $>10^7$ & $297$ & $ 1$ \\
\hline
\end{tabular}
\vspace{1em}
\caption{Number of BCFW tilings of $\mathcal{A}_{n,k,4}(Z)$} 
\label{tab:BCFW_tilings}
\end{table}

\cref{thm:BCFWtiling} is a special case of the more 
general \cref{thm:triangulations}, which we will prove below. 

\begin{definition}\label{def:pretiling}
	A collection $\mathcal{T}$ of $4k$-dimensional
BCFW cells in $\Gr_{k,n}^{\geq 0}$ is 
called a
\emph{quasi-BCFW collection of cells}
for $\mathcal{A}_{n,k,4}$ 
if there is a decomposition
	\[\mathcal{T}=\mathcal{T}_{\pre}\sqcup\bigsqcup_{k_L, k_R, b}{\mathcal{T}_{k_L,k_R,b,n}},\]
	where
\begin{itemize}
 \item $b$ ranges from 
	$b_{min}$ to $n-3-k_R$, and $b_{min}, k_L,k_R$ as in \cref{not:bcfw_tiling},
	\item 
 $\mathcal{T}_{pre}=\{\pre_{d}(S^i)|i\in\mathcal{F}\},$ where the collection
			$\{S^i|i\in\mathcal{F}\}$ gives a tiling of $\mathcal{A}_{[n]\setminus\{d\},k, 4}$ for all $Z$,
\item 

$\mathcal{T}_{k_L,k_R,b,n}$ has either of the following forms:

\begin{itemize}
	\item $\{S^{i,j}_L \bcfw S_R^i \ \vert \ 
		i \in \mathcal{D}, j\in \mathcal{E}_i \}$, 
		 where the collection $\{S_R^i \subset \Gr_{k_R, N_R}^{\geq 0} \ \vert \ 
			i\in \mathcal{D}\}$ gives a 
			tiling 
			of $\Ampl_{N_R,k_R,4}$ for all $Z$,  and for every $i$ the collection
  $\{S_L^{i,j} \subset \Gr_{k_L,N_L}^{\geq 0}	\ \vert \ j\in \mathcal{E}_i\}$ gives a tiling of $\Ampl_{N_L,k_L,4}$ for all $Z$, 
	\item 
	 $\{S^{i}_L \bcfw S_R^{i,j} \ \vert \ 
		S^{i}_L \subset \Gr_{k_L,N_L}^{\geq 0}, 
		S_R^{i,j} \subset \Gr_{k_R, N_R}^{\geq 0}, 
		i\in \mathcal{D}, j\in \mathcal{E}_j\},$ 
and the conditions on the right and left side in the above point are interchanged.
\end{itemize}
		In the first case we say that $\mathcal{T}_{k_L,k_R,b,n}$ is \emph{fibred on the right side}, and in the second we say it is \emph{fibred on the left side.}
	\end{itemize}
\end{definition}

\begin{figure}
	\includegraphics[width=1.0\textwidth]{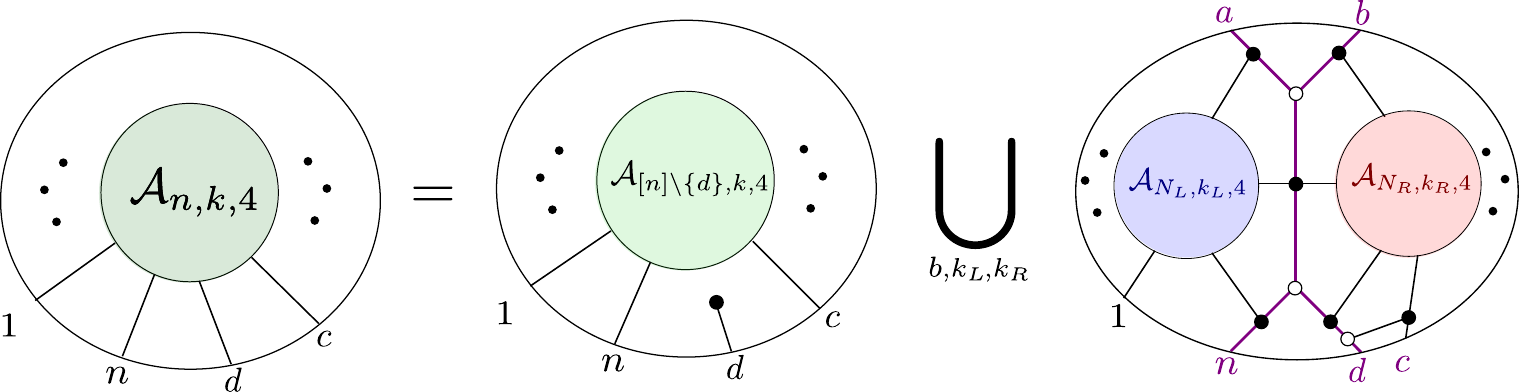}
	\caption{\label{fig:bcfw_tiling} 
	BCFW tiling for $\mathcal{A}_{n,k,4}$. On the right hand side: the first term is obtained by tiling $\mathcal{A}_{[n] \setminus \{d\},k,4}$ (from $\mathcal{T}_{pre}$); the second term is the union over $b,k_L,k_R$ as in \cref{def:pretiling} of the collections of tiles obtained by tiling $\mathcal{A}_{N_L,k_L,4}$ and $\mathcal{A}_{N_R,k_R,4}$ independently (from $\mathcal{T}_{k_L,k_R,b,n}$).}
\end{figure}

\begin{remark}
  \cref{def:pretiling} is similar to \cref{def:BCFWoutput}.
However, while the tiles involved in \cref{def:pretiling} are all BCFW tiles,
their collection does not need to be produced via the recursion in \cref{def:BCFWoutput}, but instead is obtained from tilings of amplituhedra with smaller $n$ and $k$.  
\end{remark}

\begin{theorem}\label{thm:triangulations}
Let $\mathcal{T} = \{S_1,\dots,S_{\ell}\}$ be a quasi-BCFW collection of cells for $\mathcal{A}_{n,k,4}$.  Then for all $Z$,
 the corresponding collection of tiles $\{Z_{S_1},\dots,Z_{S_{\ell}}\}$ 
 is a tiling of the amplituhedron $\Ank$.
\end{theorem}

Our starting point for proving \cref{thm:triangulations}
is the following theorem, proven in \cite[Theorem 9.5, Lemma 9.7]{even2021amplituhedron}
and the comment after the proof of \cite[Proposition 9.9]{even2021amplituhedron}.
\begin{theorem}
	\label{thm:coarse_decomposition}

	Let $Z \in \Mat_{n,k+4}^{>0}$ and
 $$C_{\pre}:=\pre_{d}(\Gr^{>0}_{k,[n]\setminus\{d\}}) \text{ and }
	C_{k_L,k_R,b,n}
	:=\Gr^{>0}_{k_L,N_L} \bcfw 
	\Gr^{>0}_{k_R,N_R}.$$
	Then the sets 
	$$\{\tZ(C_{\pre})\} \cup 
	\{\tZ(C_{k_L,k_R,b,n}) \ \vert \ 
	k_L, k_R \geq 0,   \, k_L+k_R = k-1, \,
	b_{min}\leq b\leq n-3-k_R \}$$
	are pairwise disjoint and their union is dense in
	$\Ampl_{n,k,4}(Z).$

In particular, on $\tZ(C_{\pre})$ all twistor
coordinates $\llrr{i,i+1,n-2,n}$ with $i \in [1,n-4]$ are positive, 
while on each $\tZ(C_{k_L,k_R,b,n})$ at least one such twistor coordinate
	is negative. In addition, for all $b\neq b'$ the quadratic functionaries $\llrr{b'-1,b',n \br n-2, d \br b,b-1,n}$ have opposite signs on $\tZ(C_{k_L,k_R,b,n})$
	and $\tZ(C_{k_L,k_R,b',n}).$

\end{theorem}

\begin{definition}
We say that two open tiles $\gto{S_1},\gto{S_2}$
are \emph{separated by a family $F_1,\ldots, F_N$ of functionaries} if each $F_i$ has a strong sign on one of $\gto{S_1},\gto{S_2},$ and, there are two disjoint sets $A,B\subset \{-1,0,1\}^N$ such that for all $Z$
	\[(\sgn F_1(Y),\sgn F_2(Y),\ldots,\sgn F_N(Y))\in A \text{ for all } Y\in \gto{S_1},\]
	\[\text{while }(\sgn F_1(Y),\sgn F_2(Y),\ldots,\sgn F_N(Y))\in B \text{ for all }Y\in \gto{S_2}.\]
\end{definition}

\begin{remark} \label{rk:famfunc}
If $\gto{S_1},\gto{S_2}$ are separated by a family of functionaries, then $Z_{S_1}^{\circ}\cap Z_{S_2}^{\circ}=\emptyset$ for all $Z.$    
\end{remark}

\begin{corollary}\label{cor:separation of BCFW cells}
	Two BCFW tiles $\gt{S_1}, \gt{S_2}$ satisfy $\gto{S_1}\cap\gto{S_2}=\emptyset$ for all $Z$ if and only if they are separated by a family of functionaries.
\end{corollary}
This corollary is an immediate consequence of \cref{rk:famfunc}, and \cref{thm:BCFW-tile-and-sign-description} and its proof, as we can take the collection of coordinates functionaries
for the two open tiles $\gto{S_1},\gto{S_2}$ as the list of functionaries.

As we already proved BCFW cells give tiles in \cref{thm:BCFW-tile-and-sign-description}, the following two results are enough to complete the proof of \cref{thm:triangulations}. 

\begin{proposition}\label{prop:separation}
Let $\mathcal{T}$ be a quasi-BCFW collection of cells for $\mathcal{A}_{n,k,4}$. 
Then $\gto{S_1} \cap \gto{S_2} = \emptyset$, for any pair of cells $S_1,S_2$ in $\mathcal{T}$.
\end{proposition}
\begin{proposition}\label{prop:surj}
Let $\mathcal{T}$ be a quasi-BCFW collection of cells for $\mathcal{A}_{n,k,4}$. 
Then	$\Ampl_{n,k,4}(Z)=\bigcup_{S\in\mathcal{T}} Z_S.$
\end{proposition}

In the rest of the section, we proceed with proving \cref{prop:separation}, and then \cref{prop:surj}.

\begin{proof}[Proof of Proposition \ref{prop:separation}]
	By the definition of a BCFW collection, 
	every $S\in \mathcal{T}$ belongs to either $\mathcal{T}_{\pre}$ or to $\mathcal{T}_{k_L,k_R,b,n}.$

\noindent \textbf{Case (1):}	Consider $S_1,S_2\in\mathcal{T}.$
 If $S_1,S_2\in\mathcal{T}$ both belong to $\mathcal{T}_{\pre},$ then they can be written as
	$$S_1=\pre_d(S'_1),\text{ and }S_2=\pre_d(S'_2),$$where by construction and \cref{cor:separation of BCFW cells} there is a family of functionaries $F_1,\ldots,F_n$ separating $S'_1,S'_2$. The same functionaries are easily seen to separate $S_1,S_2$ using \cref{thm:signs-under-cyc-rot-pre}. Hence the claim follows by \cref{cor:separation of BCFW cells}.
 
\noindent \textbf{Case (2):}	If $S_1, S_2 \in \mathcal{T}_{k_L,k_R,b,n}$, 
	then 
	without loss of generality, 
	$$S_1=S_L^{i,j} \bcfw S_R^i\text{ and } S_2=S_L^{i',j'} \bcfw S_R^{i'},$$ 
	where $\{S_R^i\}$ gives a tiling of $\mathcal{A}_{N_R,k_R,4}(Z_R)$ for all $Z_R$ and $\{S_L^{i,j}\},~\{S_L^{i',j'}\}$ for fixed $i,i'$ gives a tiling of $\mathcal{A}_{N_L,k_L,4}(Z_L)$ for all $Z_L$.
 
Since $S_1\neq S_2$ either $i\neq i'$ or $i=i'$ but $j\neq j'.$ We treat the first case, the second is treated similarly. By \cref{cor:separation of BCFW cells}, $\gto{S_R^i},\gto{S_R^{i'}}$ can be separated by a family $F_1,\ldots,F_N$ of functionaries. Using \cref{it:strong_sign}, \cref{it:for_separation_theorem} of \cref{prop:vanishing_and_sign_of_functionaries_under_promotion}, we will show that $\gto{S_1}$ and $\gto{S_2}$ can be separated by the family $\Psi({F}_1),\ldots,\Psi({F}_N)$ of functionaries and so are disjoint for all $Z$ by \cref{cor:separation of BCFW cells}. 
Without loss of generality, let $F_1,\ldots,F_M$ have a strong sign $+$ on $\gto{S_R^i}$ and $F_{M+1},\ldots, F_N$ have a strong sign $+$ on $\gto{S_R^{i'}}.$
This means that on $\gto{S_R^i}$ the signs of $(F_1,\ldots,F_N)$ are of the form \[(+1,\ldots,+1,s_{M+1},\ldots,s_N),~\text{where }(s_{M+1},\ldots,s_N)\in U\subseteq\{\pm1,0\}^{\{M+1,\ldots,N\}},\]while on $\gto{S_R^{i'}}$ these signs are all of the form
\[(s_1,\ldots,s_M,+1,\ldots,+1),~\text{where }(s_1,\ldots,s_M)\in V\subseteq\{\pm1,0\}^{\{1,\ldots,M\}}.\]
Since $F_1, \dots, F_N$ is separating, at most one of $U,V$ contains the all $+1$ vector. By ~\cref{it:strong_sign} and \cref{it:for_separation_theorem} of \cref{prop:vanishing_and_sign_of_functionaries_under_promotion}, on every point of $\gto{S_1}$ the promotions $(\Psi(F_1),\ldots,\Psi(F_N))$ have a sign vector of the form
\[(+1,\ldots,+1,s_{M+1},\ldots,s_N),~\text{where }(s_{M+1},\ldots,s_N)\in \cl(U)\subseteq\{\pm1,0\}^{\{M+1,\ldots,N\}},\]while on every point of $\gto{S_2}$ the sign vector is of the form \[(s_1,\ldots,s_M,+1,\ldots,+1),~\text{where }(s_1,\ldots,s_M)\in \cl(V)\subseteq\{\pm1,0\}^{\{1,\ldots,M\}}.\]

Notice that a sign vector is simultaneously of the form $(+1,\ldots,+1,s_{M+1},\ldots,s_N)$ and of the form $(s_1,\ldots,s_M,+1,\ldots,+1)$ if and only if it is the all $+1$ vector. At most one of $\cl(U), \cl(V)$ contains the all $+1$ vector, since at most one of $U, V$ contains the all $+1$ vector. 
 
Thus, the sets  
\[\{(+1,\ldots,+1,s_{M+1},\ldots,s_N): (s_{M+1},\ldots,s_N)\in \text{cl}(U)\},\]\[\{(s_1,\ldots,s_M,+1,\ldots,+1): (s_1,\ldots,s_M)\in \text{cl}(V)\}\]do not intersect.

\noindent \textbf{Case (3):} 
Suppose $S_1$ and $S_2$ belong to different sets among
$\mathcal{T}_{\pre}, \{\mathcal{T}_{k_L, k_R, b}\}_{k_L,k_R,b}$. Denote these sets by $\mathcal{T}_1,\mathcal{T}_2$ respectively, and let $C_1,C_2$ such that $C_i=C_{\pre}$ if $\mathcal{T}_i=\mathcal{T}_{\pre}$, and $C_i=C_{k_L,k_R,b,n}$ if $\mathcal{T}_i=\mathcal{T}_{k_L,k_R,b,n}$ ($C_{pre}$ and $C_{k_L,k_R,b,n}$ are as in \cref{thm:coarse_decomposition}.)

Assume towards contradiction that for some $Z$ the intersection $\gto{S_1}\cap\gto{S_2}$ is non empty. Then, since $\gto{S_1},\gto{S_2}$ are open (\cref{cor:open_map}), also their intersection $U\neq\emptyset$ is open. \cref{lem:opennes} below shows that also $\tZ(C_1),\tZ(C_{2})$
are open. Since 
\[\gto{S_1}\subseteq\overline{\tZ(C_{1})},~\gto{S_2}\subseteq\overline{\tZ(C_{2})},\]this implies $U$ is contained in $\overline{\tZ(C_{1})}\cap\overline{\tZ(C_{2})},$ and hence intersects ${\tZ(C_{1})}\cap{\tZ(C_{2})}$ in an open nonempty set. But this contradicts \cref{thm:coarse_decomposition}.

\end{proof}

\begin{lemma}\label{lem:opennes} Let $S_L^\triangleright,~S_R^\triangleleft$ be as in \cref{not:bullet}. Then:
	
	\begin{itemize}\item the sets $\gto{S_L^\triangleright}$ and $\gto{S_R^\bulR}$ are open;
		\item the sets $\tZ(C_{\pre})$ and $\tZ(C_{k_L,k_R,b,n})$ are open.
	\end{itemize}
\end{lemma}
\begin{proof}
	We start with the first statement. We prove it for $\gto{S_R^\bulR},$ the proof for $\gto{S_L^\bulL}$ is similar. We will show that the amplituhedron map is submersive on $S_R^\bulR,$ and since submersions are open it will prove our claim.
	
	Note that for $|N_L| \leq k_L +4$, $S_R^\bulR$ is a BCFW cell, hence $\tZ|_{S_R^\bulR}$ is open and submersive by \cref{cor:open_map}.

	Let $|N_L| > k_L +4$ and denote $W_L=\text{Span}(Z_i:i\in N_L),$ $d=\dim(W_L).$ Since every $k+4$ rows of $Z$ are linearly independent, $d=\min(N_L,k+4)$  
	
	Let $\pi:F\to\Gr_{k_L}(W_L)$ be the fiber bundle whose fiber over $W \in \Gr_{k_L}(W_L)$ is \[F_W=\pi^{-1}(W)=\Gr_{k_R+1}(\R^{k+4}/W).\]
	The amplituhedrom map $\tZ:S_R^\bulR\to \Gr_{k,k+4}$ decomposes as
	\[S_R^\bulR\to F\to\Gr_{k,k+4},\]where the map $\Phi:F\to\Gr_{k,k+4}$ is given by\[\Phi(U,W)=U+W,\]
	and the map $\Xi:S_R^\bulR\to F$ is given by
	\[\Xi(V)=(\tZ(V)/(V_LZ),V_LZ),\]
	where $V_L=V\cap \text{Span}(\mathbf{e}_i:i\in N_L).$
	
	We will show that these maps are well defined and submersive. This will imply that also their composition, $\tZ|_{S_R^\bulR}$ is a submersion.
	\\\textbf {The map $\Phi$:} $\Phi$ is well defined, since if $U_1,U_2$ are different liftings of a $(k_R+1)$-subspace of $\R^{k+4}/W$ to $k$-spaces $U'_1,U'_2\subseteq \R^{k+4},$ then
	\[U'_1+W=U'_2+W\in\Gr_{k,k+4}.\] It surjective, since any $k-$space $V\subseteq\R^{k+4}$ intersects $W_L$ in a subspace of dimension at least $d-4\geq k_L.$ 
	Thus we can decompose $V$ (not uniquely) as $U+W$ where $W$ is a $k_L-$space in $W_L\cap V.$ $\Phi$ is then seen to be everywhere submersive, by standard arguments.
	\\\textbf{The map $\Xi$:}
	Let $V$ be an element of $S_R^\bulR,$ by the proof of \cref{prop:bcfw-map-injective} and \cref{rmk:def-of-w}, $V$ can be written uniquely as
	\[V=V_L\oplus V_R,\]
	where $V_L\in\pre_{b+1,\ldots,d-1}\Gr^{>0}_{k_L,N_L}$ and $V_R$ belongs to the BCFW cell $S'_R$ of $\pre_{1,\ldots,a-1}\Gr^{\geq 0}_{k_R+1,N_R}$ obtained by applying the upper BCFW map with indices $a,b,c,d,n$ to $S_R,$ and then adding zero columns.

	Let $Y_L:=V_LZ$, then $Y_L\subseteq W_L.$
	In addition, $V_L\in\pre_{b+1,\ldots,d-1}\Gr^{>0}_{k_L,N_L}\subseteq\Gr^{\geq0}_{k_L,n}$. On $\Gr^{\geq0}_{k_L,n},$ right multiplication by the positive matrix $Z$ is just the amplituhedron map from $\Gr^{\geq0}_{k_L,n}$ to $\Gr_{k_L,k+4}$. Hence $Y_L \in \mathcal{A}_{n,k_L,k+4-k_L}(Z)$, so in particular $Y_L$ is of dimension $k_L$ and $Y_L\in\Gr_{k_L}(W_L)$.

	Since $\Xi(V)=(Y/Y_L,Y_L),$ $\Xi$ indeed maps $S_R^\bulR$ to $F.$
	
	Let's denote by $T_p M$ the tangent space of the differentiable manifold $M$ at the point $p$. We can conclude that $\Xi$ is a submersion if we show that:
	\begin{enumerate}
		\item \label{item:proof1} at $V\in S_R^\bulR$, the $V_L$ directions of $T_VS_R^\bulR$ surject on the base directions $T_{Y_L}\Gr_{k_L}(W_L)$;
		\item \label{item:proof2} at $V\in S_R^\bulR$, the $V_R$ directions surject on the fiber directions $T_{Y/Y_L}\Gr_{k_R+1}(\R^{k+4}/Y_L)$.
	\end{enumerate}
	\cref{item:proof1} can be proved by using standard arguments. We will now prove  \cref{item:proof2}.
	
	In order to show that $\Xi$ is a submersion, we will show that at $V\in S_R^\bulR,$ the $V_L$ directions of $T_VS_R^\bulR$ surject on the base directions $T_{Y_L}\Gr_{k_L}(W_L),$ while the $V_R$ directions subrject on the fiber directions $T_{Y/Y_L}\Gr_{k_R+1}(\R^{k+4}/Y_L).$

		Let us fix $V_L$ and note that 
		\[Y_R:=Y/Y_L\subseteq \R^{k+4}/Y_L\]is a $k_R+1$ space, which lies in the image of $S'_R$ under the map
		\begin{equation}\label{eq:quotient_ampli}\Gr^{\geq0}_{k_R+1,N_R}\to\Gr_{k_R+1}(\R^{k+4}/Y_L),\end{equation}obtained by the right multiplication of 
		$Z^{V_L},$ whose rows are 
		\[Z_i/Y_L,~i=a,b,b+1,\ldots,n.\]
		Once fixing an identification of $\R^{k+4}/Y_L$ with $\R^{(k_R+1)+4},$ this map becomes the amplituhedron map, as long as $Z^{V_L}$ is a positive matrix under this identification.
		Note that different identifications leave the maximal minors of $Z^{V_L}$ and the twistor coordinates of a point in $\Gr_{k_R+1}(\R^{k+4}/Y_L)$ invariant, up to a common multiplicative scalar.
		In particular, whether a functionary vanishes at a given point is independent of the identification with $\R^{k_R+5}.$
		
		We now show that all Pl\"ucker coordinates of $Z^{V_L}$ have the same sign, hence there is an identification which makes this matrix positive.
		Using linear algebra, it is easy to show that the Pl\"ucker coordinates $\langle I \rangle_{Z^{V_L}}$ equal (up to multiplication of a common scalar) the  determinants $\det(Y'_L|Z'_I)$ of the $(k+4)\times (k+4)$ matrix obtained by stacking $Y'_L=C_LZ$, where $C_L$ is a matrix representative of $V_L$, with a lift $Z'_I$ of the rows $I$ of $Z^{V_L}$ to $\R^{k+4}$. Here $I$ runs over subsets of $\{a,\ldots,n\}$ of size $k_R+5=k+4-k_L$.   
	
		The determinant $\det(Y_L|Z'_I)$ can be expressed as the determinant of the matrix $C'_LZ,$ where $C'_L$ is the $(k+4)\times n$ matrix whose first $k_L$ rows are $C_L,$ and the remaining rows are the standard basis elements $\mathbf{e}_i,~i\in I.$
		$C'_L$ is non-negative, $Z$ is positive, and by \cref{rk:Lapexpansion} the determinant of $C'_LZ,$ is positive. Hence $\det(Y_L|Z'_I)$ and the Pl\"ucker coordinates $\langle I \rangle_{Z^{V_L}}$ are positive.
		
		We can now show that the map in \cref{eq:quotient_ampli} is submersive: its inverse map is smooth near $Y/Y_L,$ since we can run the same pre-image algorithm of \cref{thm:BCFW-tile-and-sign-description} for the cell $S'_R,$ in analogy with the usual amplituhedron map. The algorithm does not stop as long as certain functionaries are nowhere zero, see the proof of \cref{thm:BCFW-tile-and-sign-description}. Recall that whether a functionary is zero is independent of the identification with $\R^{k_R+5}.$

		The proof regarding $\tilde{Z}(C_{k_L,k_R,b,n})$ is similar to the one of $\gto{S_R^\bulR}$, but one can use standard arguments to prove both the analogue of \cref{item:proof1} and \cref{item:proof2}.
		
		Finally, we complete the proof regarding $C_{\pre}$. The fact that the amplituhedron map from $\Gr^{>0}_{k,[n]\setminus\{n-1\}}$ to $\Gr_{k,k+4}$ is a submersion was shown in \cite[Theorem 1.5]{even2021amplituhedron}. It is straightforward to conclude that also the amplituhedron map restricted to $C_{\pre}=\pre_{n-1}\Gr^{>0}_{k,[n]\setminus\{n-1\}}$ submerses.
		
	\end{proof}

In order to prove \cref{prop:surj},  we need the following results from \cite[Proposition 3.3, Corollary 3.2, Lemma 3.1]{BaoHe}:

\begin{theorem}\label{thm:BaoHe} Let $\{S_\ell\}_{\ell \in \mathcal{C}}$ be a collection of positroid cells in $\Gr^{\ge0}_{k,[n]\setminus\{i\}}$. 
	\begin{enumerate}
		\item\label{it:BH_pre}
		If  $\{\gt{S_\ell}\}_{\ell \in \mathcal{C}}$ tiles $\tZ(\Gr_{k,[n]\setminus\{i\}}^{\geq 0})$ for all  $Z$, 
			then 
		$\{\gt{\pre_i(S_\ell)}\}_{\ell \in \mathcal{C}}$ tiles 
  ${\tZ(\pre_i(\Gr^{\ge0}_{k,[n]\setminus\{i\}}))}$ 
			for all $Z$
		\item\label{it:BH_inc}
		If $\{\gt{S_\ell}\}_{\ell \in \mathcal{C}}$ 
			tiles $\tZ(\Gr_{k,[n]\setminus\{i\}}^{\geq 0})$ for all $Z,$ 
			then 		$\{\gt{\inc_i(S_\ell)}\}_{\ell \in \mathcal{C}}$ tiles $\tZ({\inc_i(\Gr^{\ge0}_{k+1,[n]\setminus\{i\}})})$ for all $Z$
		\item\label{it:BH_cyc}
			Let  $\{S_j\}_{j \in \mathcal{D}}$ be a collection of positroid cells in $\Grk$.
		If $\{\gt{S_j}\}_{j\in \mathcal{D}}$ tiles $\Ampl_{n,k,4}(Z)$ for all $Z,$ 
			then $\{\gt{\cyc_i(S_j)}\}_{j \in \mathcal{D}}$ tiles $\Ampl_{n,k,4}(Z)$ for all $Z.$
		\item\label{it:BH_xy}
		Let $S$ be a positroid cell in $\Grk,$ and 
			let $S_1,\ldots,S_r$ be 
			positroid cells contained in $\overline{S}.$  \\
	  If the union of $\gt{S_1},\ldots,\gt{S_r}$ equals $\gt{S}$ 
		for all $Z$, then for any $j \in [n]$ we have:
     \begin{equation*}
	     Z_{x_j(\R_+).S}=\bigcup_{i=1}^r{Z_{x_j(\R_+).S_i}} \qquad \text{and} \qquad Z_{y_j(\R_+).S}=\bigcup_{i=1}^r Z_{y_j(\R_+).S_i}, \quad \mbox{for all } Z.
     \end{equation*}
	\end{enumerate}
\end{theorem}

Recall from \cref{not:bullet} the definition of $S^\bulL$ and $S^\bulR$.

\begin{lemma}\label{lem:decomp_of_S-bullet}
	Let $\mathcal{C}$ be a collection of BCFW cells.
	If $\{\gt{S}\}_{S \in \mathcal{C}}$  tiles 
	$\mathcal{A}_{N_L, k_L, 4}$ 
	(respectively, $\mathcal{A}_{N_R, k_R, 4}$),
	then
	$\{\gt{S^\bulL}\}_{S \in \mathcal{C}}$ (respectively, $\{\gt{S^\bulR }\}_{S \in \mathcal{C}}$) covers   $\gt{C_{k_L,k_R,b,n}}$ 
	for all $Z$.
\end{lemma}
\begin{proof}
  By \cref{lem:9_6_elt} this statement follows by subsequent applications of \cref{thm:BaoHe}. More specifically, there is a sequence $\mathfrak{s}$ of operations of the form $\pre_i,\inc_i,x_i(\R_+),y_i(\R_+)$, such that for each $S \in \C$ $\mathfrak{s}.S=S^\bulR$ and also $\mathfrak{s}.\Gr^{\ge0}_{k_R,N_R}=\Gr^{\ge0}_{k_L,N_L} \bcfw \Gr^{\ge0}_{k_R,N_R}$. This means that $Z_{\mathfrak{s}.\Gr^{\ge0}_{k_R,N_R}}=\cup_{S \in \C} Z_{\mathfrak{s}.S}$. An analogous argument works for $S^\bulL$.
\end{proof}

\begin{lemma}\label{lem:BCFW_cell_as_intersection}
	For every BCFW cell $S_R$ in $\Gr^{\ge0}_{k_R,N_R},$ and every BCFW cell $S_L$ in $\Gr^{\ge0}_{k_L,N_L},$
 \[\overline{\gto{S_L^\bulL}\cap\gt{S_R^\bulR}}=\overline{\gt{S_L^\bulL}\cap\gto{S_R^\bulR}}=\overline{\gto{S_L^\bulL}\cap\gto{S_R^\bulR}}\subseteq\gt{S_L\bcfw S_R}\]

\end{lemma}
\begin{proof}
 The equalities hold since $\gto{S_L^\bulL},~\gto{S_R^\bulR}$ are open, by \cref{lem:opennes}.
  We now show the inclusion.
	Let $F_1,\ldots,F_{5k}$ be the coordinate functionaries of $\gto{S_L\bcfw S_R}$ 
 where the first $5k_L$ are the promotions of the coordinate functionaries of $S_L,$ the last $5k_R$ are the promotions of the coordinate functionaries of $S_R$ and the remaining $5$ are the (signed) twistors introduced in the $k$-th step-tuple. 
By \cref{lem:sign_of_chord_twistors}, the $5$ (signed) twistors are positive on $\gto{S_L\bcfw S_R},\gto{S_L^\bulL},\gto{S_R^\bulR}.$ By \cref{prop:vanishing_and_sign_of_functionaries_under_promotion}, the first $5k_L$ are positive on $\gto{S_L\bcfw S_R},\gto{S_L^\bulL}$ and the last $5k_R$ are positive on $\gto{S_L\bcfw S_R},\gto{S_R^\bulR}.$
 By \cref{thm:BCFW-tile-and-sign-description}
  every point for which all the $F_i$ are positive belongs to $\gto{S_L \bcfw S_R}.$ Thus, every point of $\gto{S_L^\bulL}\cap\gto{S_R^\bulR}$ belongs to $\gto{S_L\bcfw S_R}.$ We finish by taking closures.
\end{proof}
\begin{proof}[Proof of Proposition \ref{prop:surj}]
	By our assumptions, and \cref{it:BH_pre} of \cref{thm:BaoHe},
	, $\mathcal{T}_{\pre}$ gives a tiling of $\gt{\overline{C}_{\pre}},$ for all $Z$. We will now show that the collection $\{Z_S\}_{S \in \mathcal{T}_{k_L,k_R,b,n}}$ covers the subspace $\gt{{C}_{k_L,k_R,b,n}}$. The result will then follow from Theorem \ref{thm:coarse_decomposition}.
	
	Fix $k_L,k_R,b,n$ and suppose  without loss of generality that $\mathcal{T}_{k_L,k_R,b,n}$ is fibred on the right side:
 \begin{equation*}
    \mathcal{T}_{k_L,k_R,b,n}= \{S^{i,j}_L \bcfw S_R^i \ \vert \ i \in \mathcal{D}, j\in \mathcal{E}_i \},
 \end{equation*}
where the collection $\{\gt{S_R^i} \ \vert \ 
				i\in \mathcal{D}\}$ tiles of $\Ampl_{N_R,k_R,4}(Z_R)$ for all $Z_R$ and for every $i$ the collection
  $\{\gt{S_L^{i,j}}	\ \vert \ j\in \mathcal{E}_i\}$ covers $\Ampl_{N_L,k_L,4}(Z_L)$ for all $Z_L$.
 By \cref{lem:decomp_of_S-bullet}, $\{\gt{(S_R^i)^\bulR} \ \vert \ 
				i\in \mathcal{D}\}$ covers  $\gt{C_{k_L,k_R,b,n}}$ for all $Z$. Similarly, for every $i,$ $\{(S_L^{i,j})^\bulL	\ \vert \ j\in \mathcal{E}_i\}$ gives a tiling of $\gt{C_{k_L,k_R,b,n}}$ for all $Z$. Thus, for all $Z,$ and every $i\in \mathcal{D}.$
	\[\gt{(S_R^i)^\bulR}=\overline{\gto{(S_R^i)^\bulR}}=\overline{\bigcup_{j\in \mathcal{E}_i}\gto{(S_R^i)^\bulR}\cap\gt{(S_L^{i;j})^\bulL}}\subseteq\overline{\bigcup_{j\in \mathcal{E}_i}\gt{S_L^{i,j}\bcfw S_R^i}}=\bigcup_{j\in \mathcal{E}_i}\gt{S_L^{i,j}\bcfw S_R^i},\]
	where the inclusion is by \cref{lem:BCFW_cell_as_intersection} and the last equality uses that $Z_S$ are closed sets. Therefore 
	\[\gt{C_{k_L,k_R,b,n}}\subseteq\bigcup_{i \in \mathcal{D},j \in \mathcal{E}_i}\gt{S_L^{i,j}\bcfw S_R^i}.\]
 This completes the proof.
\end{proof}

\appendix
\section{Background on plabic graphs}\label{app:plabic}

In \cite{postnikov}, 
Postnikov 
classified the cells of the positive Grassmannian, 
using equivalence classes of \emph{reduced plabic graphs}, 
and \emph{decorated (or affine) permutations}. 
Here we review some of this technology, following \cite{FW7},
\cite{postnikov}, and \cite{PSWv1}.

\begin{definition}[Plabic graphs]
   \label{def:plabic}
A {\it planar bicolored graph} (or ``plabic graph'')
is a planar graph $G$ properly embedded into a closed disk, such that 
		each internal vertex is colored black or white;
		each internal vertex is connected by 
		a path to some boundary vertex;
		there are (uncolored) vertices lying on the 
		boundary of the disk labeled $1,\dots, n$
		for some positive $n$;
and each of the boundary vertices is incident to a single 
		edge.
See Figure \ref{G25} for an example.

\end{definition}
\vspace{-.4cm}
\begin{figure}[h]
\centering
\includegraphics[height=1in]{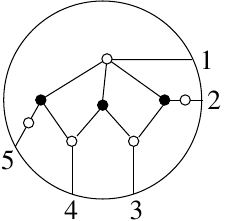}
\caption{A plabic graph}
\label{G25}
\end{figure}
\vspace{-.4cm}

If $G$ has an internal  leaf which is incident to a boundary 
vertex, 
we call this a  
\emph{lollipop}.  
\emph{We will require that our plabic graphs have no internal leaves
except for lollipops.}

There is a natural set of local transformations (moves) of plabic graphs:

(M1) \emph{Square move} (or \emph{urban renewal}).  If a plabic graph has a square formed by
four trivalent vertices whose colors alternate,
then we can switch the
colors of these four vertices.

(M2) \emph{Contracting/expanding a vertex}.
Two adjacent internal vertices of the same color can be merged or unmerged.

(M3) \emph{Middle vertex insertion/removal}.
We can remove/add degree $2$ vertices.

See \cref{fig:M1} for depictions of these three moves.

\begin{figure}[h]
\centering
\includegraphics[height=.5in]{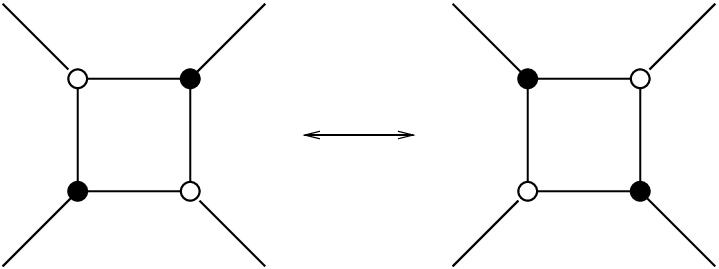}
\hspace{.3in}
\raisebox{6pt}{\includegraphics[height=.3in]{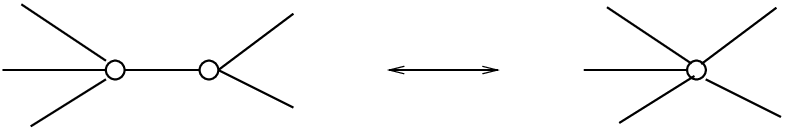}}
\hspace{.3in}
\raisebox{16pt}{\includegraphics[height=.07in]{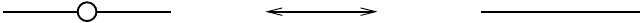}}
\caption{%
	Local moves (M1), (M2), (M3) on plabic graphs.}
\label{fig:M1}
\end{figure}
\begin{definition}\label{def:move}
Two plabic graphs are  \emph{move-equivalent} if they can be obtained
from each other by moves (M1)-(M3).  The \emph{move-equivalence class}
of a given plabic graph $G$ is the set of all plabic graphs which are move-equivalent
to $G$.
A plabic graph 
is called \emph{reduced} if there is no graph in its move-equivalence
	in which there is a \emph{bubble}, that is, 
	two adjacent vertices $u$ and $v$ 
	which are connected by 
		more than one edge.  
\end{definition}

 \begin{definition}[Decorated permutations]\label{defn:decperm}
 A \emph{decorated permutation} on $[n]$ is a bijection $\pi : [n] \to [n]$ whose fixed points are each coloured either black (loop) or white (coloop). We denote a black fixed point $i$ by $\pi(i) = \underline{i}$, and a white fixed point $i$ by $\pi(i) = \overline{i}$.
 An \emph{anti-excedance} of the decorated permutation $\pi$ is an element $i \in [n]$ such that either $\pi^{-1}(i) > i$ or $\pi(i)=\overline{i}$.  We say that a decorated permutation
on $[n]$ is of
\emph{type $(k,n)$} if it has $k$ anti-excedances.
 \end{definition}

\begin{definition}
   \label{def:rules}
	Let $G$ be a reduced plabic graph as above with boundary vertices $1,\dots, n$. For each boundary vertex $i\in [n]$, we follow a path along the edges of $G$ starting at $i$, turning (maximally) right at every internal black vertex, and (maximally) left at every internal white vertex. This path ends at some boundary vertex $\pi(i)$. By \cite[Proposition 7.4.22]{FW7}, 
the fact that $G$ is reduced implies that each fixed point of $\pi$ is attached to a lollipop; 
we color each fixed point by the color of its lollipop. In this way we obtain the \emph{decorated permutation} $\pi_G = \pi$ of $G$. We say that $G$ is of {\itshape type} $(k,n)$, where $k$ is the number of anti-excedances of $\pi_G$. 
\end{definition}

The decorated trip permutation of the plabic graph $G$ of 
\cref{G25}
is $\pi_G = (3,4,5,1,2)$, which has 
$k = 2$ anti-excedances.

\begin{theorem}[Fundamental theorem of reduced plabic graphs] \cite[Theorem 7.4.25]{FW7} \cite{postnikov}
Let $G$ and $G'$ be reduced plabic graphs.  The following statements are equivalent:
	\begin{itemize}
		\item $G$ and $G'$ are move-equivalent;
		\item $G$ and $G'$ have the same decorated trip permutation.
	\end{itemize}
\end{theorem}

%
%

\begin{definition}[Perfect orientation]\label{def:orientation}
A {\it perfect orientation\/} $\O$ of a plabic graph $G$ is a
choice of orientation of each edge such that each
black internal vertex $u$ is incident to exactly one edge
directed away from $u$; and each white internal vertex $v$ is incident
to exactly one edge directed towards $v$.
A plabic graph 
is called {\it perfectly orientable\/} if it has a perfect orientation.
Let $G_\O$ denote the directed graph associated with a perfect orientation $\O$ of $G$. The {\it source set\/} $I_\O \subset [n]$ of a perfect orientation $\O$ is the set of boundary vertices $i$ for which $i$
	is a source of the directed graph $G_\O$.  The complementary set of \emph{sinks} 
	will be denoted by $\bar I_\O$.
\end{definition}

\begin{proposition}\cite[Proposition 11.7, Lemma 11.10]{postnikov}\label{prop:positroid}
Let $G$ be a plabic graph.  Then there is a matroid 
$\mathcal{M}(G)$ called a \emph{positroid} whose bases
are precisely the subsets
$$\{I \ \vert \ I = I_{\O} \text{ for some perfect 
orientation }\O\text{ of }G\}.$$

\end{proposition}

\begin{lemma} \cite[Lemma 3.2]{PSWv1} \label{acycliclemma}
Each reduced plabic graph $G$ on $[n]$ has an acyclic perfect orientation. Moreover,
for each total order $<_i$ on $[n]$ defined by 
$i <_i i+1 <_i \dots <_i n <_i 1 <_i \dots <_i i-1$,
$G$ has an acyclic perfect orientation $\O$ whose source set $I_{\O}$
is the lexicographically minimal basis with respect to $<_i$.
\end{lemma}

Given a perfect orientation $\O$ of a reduced plabic graph $G$,
one can 
 give a parameterization of an associated cell $S_G$ of $\Grk$; the cell depends only 
on the move-equivalence class of $G$, not on $G$ or the choice of perfect orientation. 
 The parameterization is particularly simple 
when $\O$ is acyclic.

\begin{definition}[Path matrix]\label{def:pathmatrix}
Let $G$ be a reduced plabic graph on $[n]$
and let $\O$ be an acyclic perfect orientation of $G$.  Set $k=|I_{\O}|$.
Let us associate a variable $x_e$ with each edge of $G$.
 For $i\in I_\O$ and $j\in \bar I_\O$,
define the {\it boundary measurement\/} $M_{ij}$ as the following expression:
$$
	M_{ij}:=\sum_{P} \prod_{e\in P} x_e,
$$
where the sum is over all directed paths in $G_\O$ that start 
and end at the boundary vertices $i$ and $J$,
and the product 
is over all edges $e$ in $P$.
The \emph{path matrix} $A = A(G,\O) = (a_{ij})$
is the unique $k \times n$ matrix such that 
\begin{itemize}
\item The $k \times k$ submatrix of $A$ in the column set $I_\O$ is the identity matrix.
\item For any 
	$i\in I_\O$ and $j\in \bar I_\O$, the entry $a_{ij}$ equals $\pm M_{ij}$.
\item All Pl\"ucker coordinates of $A$ are nonnegative.
\end{itemize}
\end{definition}

\begin{theorem}\cite{postnikov}\label{thm:positroidcell}
Let $G$ and $\O$ be as in \cref{def:pathmatrix}.
Then for any positive real values of the edge variables $x_e$,
the realizable matroid associated to the path matrix $A(G,\O)$ is the 
{positroid}
$\mathcal{M}_G$ from 
\cref{prop:positroid}.
Let $S_G \subset \Grk$ denote
the set of all $k$-planes in $\R^n$ spanned by the
path matrices $A = A(G,\O)$, as each edge variable $x_e$ ranges over $\R_{>0}$.
Then $S_G$ is homeomorphic to an open ball, of dimension 
$r(G)-1$, where $r(G)$ is the number of regions of $G$,
and $S_G$ is called a \emph{positroid cell}.
We have that $\Grk$ is a disjoint union of positroid cells.
\end{theorem}

If $G$ is a perfectly orientable graph which fails to be reduced,
it still gives rise to a positroid cell $S_G$; however, 
its dimension will be less than $r(G)$.

We can perform certain operations, called \emph{gauge transformations}, 
which preserve the boundary measurements $M_{ij}$.  The following lemma is 
easy to verify.
\begin{lemma}[Gauge transformations]\label{lem:gauge}
Let $G$ and $\O$ be as in \cref{def:pathmatrix}, and let $N$
denote \emph{network} consisting of 
 $G_\O$ together with the edge weights $x_e$.
Pick a collection of positive real numbers
$t_v >0$ associated to each internal vertex $v$, and 
for each boundary vertex $i$, 
set $t_{i} = 1$.  Let $N'$ be the network obtained from $N$ by replacing 
each edge weight $x_e$ (where $e=(u,v)$ is an edge directed from $u$ to $v$)
by $x'_e = x_e t_u t_v^{-1}$.  In other words, for each interval vertex $v$
we multiply by $t_v$ the weights of all edges directed away from $v$,
and we divide by $t_v$ the weights of all edges directed towards $v$.  Then
the network $N'$ has the same boundary measurements as the network $N$.
\end{lemma}

\begin{theorem}\cite[Theorem 18.5]{postnikov}\label{thm:closure}
If $S_G$ 
is a cell associated to plabic graph $G$,
then 
every cell in the closure of $S_G$ comes from a plabic graph 
obtained from $G$ by deleting some edges, and 
conversely,
if we delete some edges of $G$, obtaining a perfectly orientable graph $H$,
it corresponds 
 to a cell $S_H$ in the 
closure of $S_G$.
\end{theorem}

Any positroid cell $S$ can be built up using certain atomic operations, $x_i(t), y_i(t), \inc_i, \pre_i$, which we now review, see e.g. \cite{lam2015totally}. Recall from \cref{def:pre} the definition of $\pre_i$. 

\begin{definition}
Let $E_{a,b}$ be the $n \times n$ matrix whose $(a,b)$ entry is $1$ and all other entries are zero. For $i=1, \dots, n-1$, we define the upper triangular matrix
$x_i(t):= I + t E_{i, i+1} $ and the lower triangular matrix $y_i(t):= I + t E_{i+1, i}$. We define $x_n(t), y_n(t)$ as in \cite[Definition 3.4]{even2021amplituhedron}. For fixed $t$, we have maps $x_i(t),y_i(t): \Gr_{k,n} \to \Gr_{k,n}$ to itself, which are induced by right multiplication:
\[x_i(t): C \mapsto C x_i(t) \qquad \text{and} \qquad y_i(t): C \mapsto C y_i(t).\]
For $i \in [n]$, we also define the map $\inc_i: \Mat_{k, [n] \setminus \{i\}} \to \Mat_{k, n}$ by 
\[\inc_i : \begin{bmatrix}
	\vline &   & \vline & \vline&  & \vline \\
	C_1 & \cdots & C_{i-1} & C_{i+1} & \cdots & C_n \\
		\vline & & \vline & \vline&   & \vline
\end{bmatrix} \mapsto 
\begin{bmatrix}
	0 & \cdots& 0 & 1 & 0& \cdots & 0\\ 
		\vline &  & \vline &0& \vline&  & \vline \\
-C_1 & \cdots & -C_{i-1} & \vdots& C_{i+1} & \cdots & C_n \\
		\vline &  & \vline & 0& \vline&  & \vline \\
\end{bmatrix}.\]
This descends to a map $\inc_i: \Gr_{k, [n]\setminus \{i\}} \to \Gr_{k, n}$.
\end{definition}

The maps $x_i(t), y_i(t), \inc_i, \pre_i$ have plabic graph analogues, namely adding \emph{bridges} and  \emph{lollipops}. 

\begin{definition}
	Let $G$ be a planar graph. The operation of \emph{adding a black-white bridge at $i$} modifies $G$ as follows.
	\begin{center}
		\includegraphics[width=0.3\textwidth]{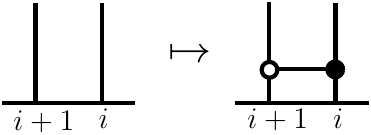}
	\end{center}	
	There is an analagous operation of \emph{adding a white-black bridge at $i$}, where the white vertex is at $i$ and the black vertex is at $i+1$. For $i$ not in the index set of $G$, the operation of \emph{adding a (black or white) lollipop} adds a new boundary vertex $i$ to $G$ and adds an interior (black or white) leaf adjacent to $i$. The boundary vertex $i$ is placed so that the boundary labels of $G$ are cyclically increasing when read clockwise.
\end{definition}

\begin{theorem}
	\cite[Proposition 2.5 and Section 2.6]{Karpman}
	\label{thm:bridges-and-lollipops}
	\begin{enumerate}
		\item Let $S_G \subset \Grk$ be a positroid cell. Then 
		\[x_i(\R_+).S_G:= \{Cx_i(t): C \in S_G, t \in \R_+\} \qquad \text{and} \qquad y_i(\R_+). S_G:=\{Cy_i(t): C \in S_G, t \in \R_+\} \]
		are positroid cells of $\Grk$ whose postroid contains that of $S_G$. A plabic graph corresponding to these cells is, respectively, $G$ with a white-black bridge added at $i$ and $G$ with a black-white bridge added at $i$.
		\item Let $S_G \subset \Gr_{k, [n] \setminus \{i\}}.$ Then $\pre_i S_G \subset \Grk$ and $\inc_i S_G \subset \Gr_{k+1, n}^{\ge 0}$ are positroid cells. A plabic graph corresponding to these cells is, respectively, $G$ with a black lollipop added at $i$ and $G$ with a white lollipop added at $i$.
	\end{enumerate}

\end{theorem}

\bibliographystyle{alpha}
	\bibliography{ClusterTilesPromotion}
\end{document}